\documentclass[11pt,reqno]{amsart}

\usepackage[text={160mm,240mm},centering]{geometry}            % See geometry.pdf to learn the layout options. There are lots.
\usepackage{graphicx}
\usepackage{listings}
\usepackage{amssymb}
\usepackage{epstopdf}
\usepackage{dsfont}
\usepackage{url}
\usepackage{verbatim}

\usepackage{amssymb,amsmath,setspace,geometry,indentfirst,changepage,inputenc}
\usepackage{mathrsfs,amsfonts,savesym,graphicx,bm,amsthm}

\usepackage[mathscr]{eucal}
\usepackage[all]{xy}
\usepackage{mathrsfs}
\usepackage{hyperref}
\usepackage{color}

\usepackage{dcolumn}

\usepackage{booktabs}%加载 booktabs 宏包之后可以使用 \toprule 和 \bottomrule 命令分别画出表格头和表格底的粗横线，而用 \midrule 画出表格中的横线
\usepackage{diagbox} %加载 diagbox 宏包之后可以绘制斜线表头
\usepackage{multirow} %加载 multirow 宏包后可以让表格单元实现合并与分割
\usepackage{makecell} %加载 makecell 宏包可以在表项中使用"\\"命令自由的换行。在不打算固定表列宽度时，它比p{宽度}选项更为灵活。
\usepackage{longtable} %加载 longtable 宏包可以处理行数非常多的长表格
\usepackage{array}
\usepackage{tabularx}
\usepackage{color}

\usepackage{amsmath,amsthm}

\usepackage[mathscr]{eucal}
\usepackage[all]{xy}
\usepackage{mathrsfs}
\usepackage{hyperref}
\usepackage{color}

\newtheorem{theorem}{Theorem}[section]
\newtheorem{lemma}[theorem]{Lemma}

\newtheorem{conjecture}[theorem]{Conjecture}
\newtheorem{corollary}[theorem]{Corollary}
\newtheorem{proposition}[theorem]{Proposition}

\newtheorem{definition-lemma}[theorem]{Definition-Lemma}
\newtheorem{definition-theorem}[theorem]{Definition-Theorem}
\newtheorem{algorithm}[theorem]{Algorithm}

\theoremstyle{definition}

\newtheorem{definition}[theorem]{Definition}

\newtheorem{remark}[theorem]{Remark}

\allowdisplaybreaks

\title{Spectrum, Tjurina Spectrum, and Hertling Conjecture for Singularities of Modality $\leq 3$}

\usepackage{fancyhdr}
\author{Quan Shi}
\address{
	Zhili College,
	Tsinghua University,
	Beijing, 100084, P. R. China.}
\email{shiq20@mails.tsinghua.edu.cn}
\author{Yang Wang}
\address{Department of Mathematical Sciences,
	Tsinghua University,
	Beijing, 100084, P. R. China.}
\email{wangyang21@mails.tsinghua.edu.cn}

\author{Huaiqing Zuo}
\address{Department of Mathematical Sciences,
Tsinghua University,
Beijing, 100084, P. R. China.}
\email{hqzuo@mail.tsinghua.edu.cn}
\thanks{Zuo is supported by NSFC Grant 12271280.}
\begin{document}
\maketitle

\begin{abstract}
	Spectrum is an important numerical invariant of an isolated hypersurface singularity, connecting its topological and analytic structures. The well-known Hertling conjecture tells the relation of range and variance of exponents i.e. elements of spectrum. For trimodal singularities, we compute their spectra and verify Hertling conjecture for them. 
	
	Jung, Kim, Saito and Yoon recently defined Tjurina spectrum, stemming from Hodge ideals. This set of numerical invariants is a subset of spectrum in Steenbrink's sense. We give an estimation of exponents not in Tjurina spectrum and propose a similar Generalized Hertling Conjecture for Tjurina Spectrum. Moreover, we prove the conjecture for singularities of modality $\leq 3$. 
	%Here, we compute this invariant for singularities of modality $\leq 3$. Moreover, we propose ``disappearance conjecture'' that the biggest exponent of spectrum is not in Tjurina spectrum and ``generalzied Hertling conjecture for Tjurina spectrum (GHCTS)'' that Tjurina spectrum fits the same inequality as in generalized Hertling conjecture. As main theorems, we provide an estimation for exponents of Tjurina
	%spectrum not in the spectrum and we prove ``disappearance conjecture'' and ``GHCTS'' for singularities of modality $\leq 3$ except for four cases to be verified.\\
	
	\vspace{0.5em}
	\vspace{0.5em}
	
	\noindent MSC(2020) 14B05, 32S05.
\end{abstract}
\tableofcontents

\newpage
%\subsection*{Acknowledgement}
%
%We thank Morihiko Saito for his valuable comments and suggestions. They are key to the improvement of this paper.  

\subsection*{Notations}

Throughout the paper, we adopt multi-index. That is, for $\bm \alpha = (\alpha_0,...,\alpha_n)\in \mathbb N^{n+1}$, $\bm x^{\bm \alpha}$ refers to monomial $x_0^{\alpha_0}...x_n^{\alpha_n}$. $\bm x$ denotes the sequence of $n+1$ variables $x_0,...,x_n$. $d\bm x$ means the $(n+1)$-form $dx_0\wedge...\wedge dx_n$. For an $\bm \alpha\in \mathbb N^{n+1}$, we define $\vert \bm \alpha \vert := \sum_{i=0}^n \alpha_i$. $\mathbb S_\varepsilon$ and $\mathbb B_\varepsilon$ will denote the $(2n+1)$-sphere and $(2n+2)$-ball in $\mathbb C^{n+1}$ of radius $\varepsilon$ respectively. $\mathbb S_\delta^1$ and $\mathbb D_\delta$ will denote the circle and disk in $\mathbb C$ of radius $\delta$ respectively.

\section{Introduction}

Spectrum is an important invariant for isolated hypersurface singularities. This invariant is first defined by Steenbrink (see \cite{MR0485870} or \cite{https://doi.org/10.48550/arxiv.2003.00519}). For a holomorphic function germ $f$ at $0\in \mathbb C^{n+1}$ that defines an isolated singularity, the spectrum of $f$ is $\mu$ rational numbers $\mathrm{sp}(f) = \{\alpha_1 \leq \alpha_2 \leq ... \leq \alpha_\mu\}$ (called \textit{exponents} or \textit{spectrum numbers}) or a polynomial $\mathrm{Sp}(f) = \sum_{i=1}^\mu t^{\alpha_i}$, where $\mu = \mu(f)$ is the Milnor number of $f$. In Steenbrink's original definition, those rational numbers come from the mixed Hodge structure of Milnor fibration. There are many beautiful properties of spectrum, for example, symmetry and invariant under $\mu$-constant deformation. We review this part in \textbf{Subsection \ref{Spectrum and Hertling Conjecture subsection}}. 

The computation for spectrum is largely credited to M. Saito's theorem that when $f$ is Newton non-degenerated, then its exponents are exactly the jumping numbers of Newton filtration on $\Omega_f^{n+1}$ (see \cite{MR954149}). We review this part in \textbf{Subsection \ref{Newton Filtration}} and give a detailed explanation on computation.

In 2000, Hertling proposed a well-known conjecture when studying $F$-manifolds. It says the variance of all spectrum numbers are no greater than a twelfth of their range (see \cite{MR1849311}) i.e. 
\begin{conjecture}[Hertling]\label{Hertling Conjecture}
	Let $f\in \mathcal O_{n+1}$ be a germ which defines an isolated singularity at the origin, with Milnor number $\mu$. Suppose $\alpha_1\leq ...\leq \alpha_\mu$ are spectrum numbers of $f$, then the following inequality holds:
	\begin{displaymath}
		\frac{1}{\mu}\sum_{i = 1}^{\mu} (\alpha_i-\frac{n+1}{2})^2 \leq \frac{\alpha_\mu-\alpha_1}{12}.
	\end{displaymath}
\end{conjecture}

The 25-year-old problem has witnessed some progresses (see \textbf{Subsection \ref{Spectrum and Hertling Conjecture subsection}}). In this paper, we first push it a step forward, proving the conjecture for trimodal singularities.

\vspace{0.5em}
\noindent\textbf{Theorem A} (Theorem \ref{main theorem A})\textbf{.} \textit{Hertling Conjecture holds for singularities of modality $3$.}
\vspace{0.5em}

%Besides, K. Saito gave two conjectures for the range of spectrum numbers, we also prove them for singularities of modality $\leq 3$.
%
%\begin{conjecture}[\cite{MR0715651}]\label{range conjecture}
%	Let $\alpha_1\leq ...\leq \alpha_\mu$ be the spectrum numbers of $f$. Then:
%	
%	\noindent(1) $\alpha_\mu-\alpha_1 < 1$ if and only if $(V(f),0)$ is a simple singularity i.e. $ADE$ singularity.
%	
%	\noindent(2) $\alpha_\mu-\alpha_1 = 1$ if and only if $(V(f),0)$ is simple elliptic or cusp (i.e. $\tilde E_6,\tilde E_7,\tilde E_8,T_{p,q,r}$, see \cite{MR3146513}).
%\end{conjecture}
%
%
%\noindent\textbf{Theorem B} \textit{(Theorem \ref{K. Saito's range conjecture})\textbf{.}
	%	For all singularities of modality $\leq 3$, the ranges of their spectra are all greater than $1$. That is, \textbf{Conjecture \ref{range conjecture}} holds for them.}
%\vspace{0.5em}

Most recently, S. Jung, I. Kim, M. Saito, and Y. Yoon defined Hodge ideal spectrum and Tjurina spectrum for isolated singularities (see \cite{MR4448602}). Those new spectra come from Hodge ideals and are highly related to the original object. It is rather hard to describe Hodge ideals in a few lines. Roughly speaking, let $Z = \mathrm{div}(f)$ be principal divisor defined by $f$, then Hodge ideal $I_k(\alpha Z)$ is an ideal of $\mathcal O_{n+1}$ indexed by $k\in \mathbb N$ and $\alpha \in \mathbb Q\cap (0,1]$. Hodge ideals induce a filtration on $\mathcal O_{n+1}$ by
\begin{displaymath}
	V_{HI}^\beta := \sum_{\alpha+p\geq \beta} I_p(\alpha Z).
\end{displaymath}
Naturally, we have induced filtrations on Milnor algebra $\mathcal O_{n+1}/(\frac{\partial f}{\partial x_0},...,\frac{\partial f}{\partial x_n})$ and Tjurina algebra  $\mathcal O_{n+1}/(f,\frac{\partial f}{\partial x_0},...,\frac{\partial f}{\partial x_n})$ by canonical surjections. We mainly focus on the filtration $V_{HI}^\bullet$ on Tjurina algebra. The Poincar\'e polynomial with respsect to this filtered vector space is called Tjurina spectrum, denoted as $\mathrm{Sp}^\tau(f)$. Or equivalently, written as $\tau$ rational numbers $\mathrm{sp}^\tau(f) = \{\beta_1 \leq \beta_2 \leq ... \leq \beta_\tau\}$, where $\tau = \tau(f)$ is the Tjurina number of $f$. We call $\beta_1,...,\beta_\tau$ exponents of Tjurina spectrum. We will review in \textbf{Subsection \ref{Brieskorn Lattice}} that Tjurina spectrum is a subset of spectrum.

The study of Tjurina spectrum is a brand new direction in mathematics, which may play a significant role in understanding of modern $\mathscr D$-module theory, especially Hodge ideal and mixed Hodge module.

Besides, to obtain data for concrete examples, we propose a method and a code to compute the Tjurina spectrum through $V$-filtration, and complete the computation for singularities of modality $\leq 3$ except for four cases (see \textbf{Subsection \ref{Computation for Tjurina spectrum}}). Results of the four cases are left unkown since these cases are Newton degenerate, and so far there is not yet an effective method to compute $V$-filtration for Newton degenerate singularities.

The study of Tjurina spectrum is equivalent to the estimation of spectrum numbers not in Tjurina spectrum. We hence provide a way to estimate them. The following is the estimation.

\vspace{0.5em}
\noindent\textbf{Theorem B} (Theorem \ref{f induces fitration lemma})\textbf{.} \textit{Given an isolated hypersurface $(V(f), 0)$, Let $\mu$ and $\tau$ be its Milnor and Tjurina number respectively. Let $d = \mu-\tau$
	, $\alpha_1\leq \alpha_2 \leq ... \leq \alpha_\mu$ be all spectrum numbers, $\alpha_{j_1} < ... < \alpha_{j_t}$ be all jumpings i.e. $\alpha_{j_i} < \alpha_{j_{i}+1}$ and $h_l = \mathrm{Gr}_V^{\alpha_{j_l}}\Omega_f^{n+1} = j_l -j_{l-1}$($j_0 = 0$) be the multiplicity of $\alpha_{j_l}$.
	Besides, let $\alpha_{i_1} \leq \alpha_{i_2} \leq ... \leq \alpha_{i_d}$ be all spectrum numbers not in Tjurina spectrum.}

\textit{For each $1\leq k\leq d$, let $r_k$ be the maximal integer such that $\#\{j\mid \alpha_j\leq \alpha_{r_k}\} \leq k$, then $\alpha_{i_k} \geq \alpha_{j_{r_k}}+1$. In particular, we have $\alpha_{i_1} \geq \alpha_1+1$ and $\alpha_{i_2} \geq \alpha_2+1$ if $d \geq 2$.}
\vspace{0.5em}

In the first version of this paper, we conjectured the greatest spectrum number is not in the Tjurina spectrum and proved it for normal forms of singularities of modality $\leq 3$ (see \textbf{Conjecture \ref{DConjecture}} below). Recently, Jung, Kim, Saito, and Yoon gave a proof (see \textbf{Subsection \ref{Maximal Exponent and Generalized Hertling Conjecture for Tjurina Spectrum}}).

\vspace{0.5em}

\begin{conjecture}\label{DConjecture}
	Suppose $f \in \mathcal{O}_{n+1}$ defines an isolated singularity at the origin and the difference between its Milnor number and Tjurina number is positive i.e. $\mu-\tau \geq 1$, then the maximal spectrum number of $f$ is not contained in the Tjurina spectrum of $f$.
\end{conjecture}
\vspace{0.5em}

Furthermore, we propose a conjecture for Tjurina spectrum parallel to Hertling conjecture, called \textit{Generalized Hertling Conjecture for Tjurina Spectrum}.

\begin{conjecture}
	\label{GHCTS}
	Suppose $f$ is itself a $\tau$-max form. Let $\beta_1,...,\beta_\tau$ be exponents of Tjurina spectrum. Then the following inequality holds:
	\begin{align*}
		\frac{1}{\tau}\sum_{i = 1}^{\tau} (\beta_i-\bar \beta)^2 \leq \frac{\beta_\tau-\beta_1}{12},\tag{GHCTS}
	\end{align*}
	where $\bar \beta$ is the average of $\beta_1,...,\beta_\tau$ i.e. $\bar \beta = \frac{1}{\tau}\sum_{j=1}^\tau \beta_j$.
\end{conjecture}
\vspace{0.5em}

%\noindent\textbf{\textit{\textbf{Generalized Hertling Conjecture for Tjurina Spectrum \ref*{GHCTS}}}.} 
%\textit{Suppose $f$ is itself a $\tau$-max form. Let $\beta_1,...,\beta_\tau$ be exponents of Tjurina spectrum. Then the following inequality holds:}
%\begin{align*}
%	\frac{1}{\tau}\sum_{i = 1}^{\tau} (\beta_i-\bar \beta)^2 \leq \frac{\beta_\tau-\beta_1}{12},\tag{GHTS}
%\end{align*}
%\textit{where $\bar \beta$ is the average of $\beta,...,\beta_\tau$ i.e. $\bar \beta = \frac{1}{\tau}\sum_{j=1}^\tau \beta_j$.}
%\vspace{0.5em}

The $\tau$-max form is defined in \textbf{Definition \ref{tau-max}}. For an arbitrary isolated singularity $(V(f),0)$, we say the conjecture holds for $f$ if it holds for all $\tau$-max form of $f$. In the last three subsections of this paper, we prove the conjecture for singularities of modality $\leq 3$.

\vspace{0.5em}
\noindent\textbf{Theorem C.} \textit{\textbf{Conjecture \ref{GHCTS}} holds for singularities of modality $\leq 3$}.
\vspace{0.5em}

%\vspace{0.5em}
%\noindent\textbf{\textit{\textbf{Generalized Hertling Conjecture for Tjurina Spectrum \ref*{GHCTS}}}.} \textit{Let $f\in \mathcal O_{n+1}$ be a germ which defines an isolated singularity at the origin, with Tjurina number $\tau$. Suppose $\alpha_1\leq ...\leq \alpha_\tau$ are exponents of Tjurina spectrum of $f$, then the following inequality holds:}
%\begin{displaymath}
%	\frac{1}{\tau}\sum_{i = 1}^{\mu} (\alpha_i-\bar \alpha)^2 \leq \frac{\alpha_\tau-\alpha_1}{12},
%\end{displaymath}
%\textit{where $\bar \alpha$ is the average of $\alpha_1,...,\alpha_\tau$ i.e. $\bar \alpha = \frac{1}{\tau}\sum_{j=1}^\tau \alpha_j$.}

%\vspace{0.5em}
%\noindent\textbf{Theorem D} \textit{(Theorem \ref{HCTS1},\ref{HCTS2},\ref{HCTS3})\textbf{.} \textbf{\textit{\textbf{Generalized Hertling Conjecture for Tjurina Spectrum \ref*{GHCTS}}}} holds for singularities of modality $\leq 3$ except for four execptional cases to be verified.}
%\vspace{0.5em}

\section{Preliminaries}
\subsection{Rings of Local Analytic Germs, Hypersurface Singularities and Modalities}
\label{basic analytic germs}
Let $\mathbb C\{\bm x\} = \mathbb C\{x_0,...,x_n\}$ denote the ring of convergent power series of $n+1$ variables. $\mathbb C\{\bm x\}$ is often denoted as $\mathcal O_{n+1}$. It coincides with the ring of complex analytic germs at the origin. $\mathcal O_{n+1}$ is a local ring by definition, with  maximal ideal
$\mathfrak m = (\bm x)$. In this paper, we will mainly use notation $\mathcal O_{n+1}$ and sometimes use $\mathbb C\{\bm x\}$ during review of results from references, in compliance with original notations.

An analytic space germ $(X,0) \subset (\mathbb C^{n+1},0)$ is defined to be the common zero locus of some analytic germs $f_1,...,f_r\in \mathcal O_{n+1}$. A morphism between two analytic space germs $(X,0)\subset (\mathbb C^{m+1},0)$ and $(Y,0)\subset (\mathbb C^{m+1},0)$ is a restriction of a holomorphic map germ $f:(\mathbb C^{n+1},0) \to (\mathbb C^{m+1},0)$ with $(f(X),0)\subset (Y,0)$. Such definition provides the category of analytic space germs. In particular, analytic space germs $(X,0)$ and $(Y,0)$ are isomorphic if there are two morphisms inverse to each other.

An analytic germ $(X,0) \subset (\mathbb C^{n+1},0)$ is called a hypersurface singularity if $X$ is the zero locus of one analytic germ $f\in \mathfrak m \subset \mathcal O_{n+1}$. Such a $(X,0)$ is often denoted as $(V(f),0)$. The singular locus of $(X,0)$ is defined to be the common zero locus of $f,\frac{\partial f}{\partial x_0},...,\frac{\partial f}{\partial x_n}$, often denoted as $(\mathrm{Sing}(X),0)$. We say $X$ is an \textit{isolated singularity} or $f$ defines an \textit{isolated singularity} if $(\mathrm{Sing}(X),0)$ consists of one point. 

There is an algebraic characterization for isolation of singularities by considering the Milnor number and Tjurina number of $f$. Let $J(f) = (\frac{\partial f}{\partial x_0},...,\frac{\partial f}{\partial x_n})$ be the Jacobian ideal of $f$ and $T(f) = (f,J(f))$ be the Tjurina algera of $f$. $\mu(f) := \dim_{\mathbb C} \mathcal O_{n+1}/J(f)$ and $\tau(f) := \dim_{\mathbb C} \mathcal O_{n+1}/T(f)$ are defined to be the Milnor number and Tjurina number of $f$ respectively. $M_f:= \mathcal O_{n+1}/J(f)$ is called the Milnor algebra of $f$ while $T_f := \mathcal O_{n+1}/T(f)$ is called the Tjurina ideal of $f$ (also called the moduli algebra of $f$).

\begin{lemma}[\cite{MR2290112}, Lemma 2.3] \label{gls tjurina and milnor}$U\subseteq \mathbb C^n$ is an open neighborhood of $0$. Let $f:U\to \mathbb C$ be holomorphic, then the following are equivalent:\\
	(a) $0$ is an isolated critical point of $f$.\\
	(b) $\mu(f,0) < \infty$.\\
	(c) $0$ is an isolated singularity of $f^{-1}(f(0)) = V(f-f(0))$.\\
	(d) $\tau(f-f(0),0) < \infty$.
\end{lemma}
\begin{remark}
	Here $(f,0)$ means the analytic germ at $0\in \mathbb C^n$ given by $f$.
\end{remark}

Two analytic germs $f,g \in \mathcal O_{n+1}$ are called \textit{right equivalent} if there exists a $\mathbb C$-algebra automorphism such that $\varphi(f) = g$, called \textit{contact equivalent} if $\varphi(f) = ug$, where $\varphi$ is a $\mathbb C$-algebra automorphism and $u \in \mathcal O_{n+1}^*$ is a unit. Note that the two types of equivalence induce isomorphisms of hypersurface singularities since $\varphi$ is always given by an isomorphism of analytic space germs. Mather and Yau proved in \cite{MR674404} that two hypersurface singularities are isomorphic if and only if their moduli algebras are isomorphic. The Mather-Yau theorem is slightly generalized in \cite{MR2290112}, stated as below:
\begin{theorem}
	(\cite{MR2290112}, Theorem 2.26) Let $f,g\in \mathfrak m\subset \mathcal O_{n+1}$, the following are equivalent:\\
	(1) $f$ is contact equivalent to $g$.\\
	(2) For all $k\geq 1$, $\mathcal O_{n+1}/T_k(f) \simeq \mathcal O_{n+1}/T_k(g)$.\\
	(3) There is some $k\geq 1$ such that $\mathcal O_{n+1}/T_k(f) \simeq \mathcal O_{n+1}/T_k(g)$.
	
	Here, $T_k$ is the $k$-th Tjurina ideal $T_k(f) := (f)+\mathfrak m^k J(f)$. In particular $T_0(f) = T(f)$.
	
	Moreover, if $f$ has an isolated singularity, then $f$ is contact equivalent to $g$ if and only if $T(f) \simeq T(g)$.
\end{theorem}

The introduction of Mather-Yau theorem initiates the massive study of classification of hypersurface singularities since it proves the existence of a convenient complete invariant. The classification is essentially an invariant theory. One hopes to distinguish all singularities by finding enough invariants under isomorphisms of hypersurface singularities. 

So far, there are many well-known invariants of hypersurface singularities. Milnor number and Tjurina number introduced above are two important ones. Apart from invariants given by analytic algebra structures, some invariants may come from other structures. (Local) Berstein-Sato polynomial (see \cite{MR4394404}) is from the $\mathcal D$-module structure. (Local) Igusa Zeta function (see \cite{MR1743467}) is from the local integration. Topological Zeta function (also see \cite{MR1743467}) is from the embedded resolution. The spectrum of singularities (see \cite{MR1042806}, \cite{MR453735} or \cite{https://doi.org/10.48550/arxiv.2003.00519}) rise from the mixed Hodge structures of singularities. We will give a more detailed introduction to spectrum in the next subsection. We point out here that these invariants are highly related through the Monodromy conjecture (see \cite{vanproeyen2009monodromy}), which is, possibly, a century problem.

We call $f\in \mathcal O_{n+1}$ \textit{right $k$-determined} (\textit{contact $k$-determined} resp.) if $f$ is right equivalent to (contact equivalent to resp.) any germ in the set $f+\mathfrak m^{k+1}$. 

\begin{theorem}[Finite Determinancy Theorem, \cite{MR2290112}]\label{Finite Determinacy Theorem} 
	Let $f\in \mathfrak m \subset  \mathcal O_{n+1}$, then
	
	\noindent(1) $f$ is right $k$-determined if
	\begin{displaymath}
		\mathfrak m^{k+1} \subset \mathfrak m^2 J(f).
	\end{displaymath}
	\noindent(1) $f$ is contact $k$-determined if
	\begin{displaymath}
		\mathfrak m^{k+1} \subset \mathfrak m^2 J(f)+\mathfrak m(f).
	\end{displaymath}
\end{theorem}

As a corollary, we have:
\begin{corollary}\label{Corollary of Finite Determinacy Theorem} 
	Let $f\in \mathfrak m\subset  \mathcal O_{n+1}$ and $\mu$ be its Milnor number, then $f$ is right equivalent to any germ in $f+\mathfrak m^{\mu+2}$. Since Milnor number is a right (equivalence) invariant, we therefore know the Milnor numbers of $f$ and any $\tilde f\in f+\mathfrak m^{\mu+2}$ coincide.
\end{corollary}

For $f\in \mathcal O_{n+1}$, we call it quasihomogeneous if $f\in J(f)$ (equivalently, $T(f) = J(f)$). Next, we define the semi-quasihomogeneous germ. For $f\in \mathcal O_{n+1}$, suppose $f = \sum_{v\in \mathbb N^{n+1}} a_{\bm v} \bm x^{\bm v}$ and the support of $f$ is defined to be $\mathrm{Supp}(f) = \{\bm v\in \mathbb N^{n+1}\mid a_{\bm v} \neq 0\}$. Let $\bm w\in \mathbb Q_{> 0}^{n+1}$ and $d = \min_{\bm v\in \mathrm{Supp}(f)}\{\bm w\cdot \bm v\}$. Call $f_{\bm w} = \sum_{\bm v\cdot \bm w = d} a_{\bm v} {\bm x}^{\bm v}$ the principal part of $f$ with respect to $\bm w$. We say $f$ is semi-quasihomogeneous with weight $\bm w$ if $f_{\bm w} \in \mathcal O_{n+1}$ defines an isolated singularity at the origin. The following theorem shows that the Milnor numbers of $f$ and $f_{\bm w}$ coincide.

\begin{theorem}[\cite{MR2290112}, \textit{Corollary 2.18}]
	Let $f\in \mathcal O_{n+1}$ be semiquasihomogeneous with weight $\bm w$. Then $f$ has an
	isolated singularity at 0 and $\mu(f) = \mu(f_{\bm w})$.
\end{theorem}

The usual classification of hypersurface singularities is according to modality, which is an invariant introduced by Arnol'd etc. (see \cite{MR0777682}). It is roughly the dimension of moduli space of a hypersurface singularity. Here is a precise definition.
\begin{definition}
	Let $G$ be a Lie group and $X$ be a smooth manifold (real or complex), and $G$ acts on $X$. For a point in $x\in X$, orbits of the action of $G$ on $X$ may form finite stratification near $x$. The modality of $x$ is defined to be the least number of parameters so that the family covers orbits near $x$.
\end{definition}

Hypersurface singularities of modality $0,1,2,3$ have been completely classified (see \cite{MR0356124},\\ 
\cite{MR0397777}, \cite{MR0553709}, and \cite{MR1721628}). We also refer to \cite{MR3146513}, which lists the classification for modality $\leq 2$.

\subsection{Spectrum Numbers and Hertling Conjecture}
\label{Spectrum and Hertling Conjecture subsection}
%	Spectrum is also called exponents.

Roughly speaking, spectrum numbers of an isolated singularity $(V(f),0)$ come from the mixed Hodge structure on the $n$-th cohomology group of the canonical Milnor fiber. Or equivalently, they can be derived from the Briekorn module. In this subsection, we will give a brief review of Milnor fibration, monodromy, mixed Hodge structure, Briekorn module, spectrum numbers and Hertling conjecture. 

We point out that spectrum numbers are sometimes called exponents.

Suppose $f : (\mathbb C^{n+1},0) \to (\mathbb C,0)$ is an analytic germ that defines an isolated singularity at the origin. Let $\varepsilon >0$ be a sufficiently small real number and $\delta > 0$ is sufficiently small with respect to $\varepsilon$. In \cite{MR0239612}, $f/\vert f\vert$ gives a $C^{\infty}$-fiber bundle over $\mathbb S^1$. 

\begin{theorem}[Milnor Fibration I, \cite{MR0239612}]\label{fibration theorem}
	Let $K = \mathbb S_\varepsilon\cap f^{-1}(0)$, then the following map is a $C^\infty$-bundle bundle.
	\begin{displaymath}
		\pi := \frac{f}{\vert f\vert} : \mathbb S_\varepsilon\setminus K \longrightarrow \mathbb S^1
	\end{displaymath}
\end{theorem}
Since $\pi_1(\mathbb S^1) = \mathbb Z$, we have natural monodromy action up to sign on each fiber $F_\theta := \pi^{-1}(e^{i\theta})$, $h(-,2\pi)$ (see \cite{MR4261550}, Chapter 6, or \cite{MR4367434}, Chapter 6). Milnor also proved the that each fiber $F_\theta := (f/\vert f\vert)^{-1}(e^{i\theta})$ is homotopic to the wedge sum of $\mu$ $n$-spheres, where $\mu$ is the algebraically Milnor number of $f$.
\begin{theorem}[Bouquet Theorem, \cite{MR0239612}]\label{boquet theorem}
	Each fiber $F_\theta$ is homotopy equivalent to $\dot{\bigvee}_\mu \mathbb S^n$, where $\mu$ is the Milnor number of $f$.
\end{theorem}

Therefore, we know the reduced cohomology group of $F_\theta$ is 
\begin{displaymath}
	\tilde H_q(F_\theta;\mathbb Z) = \begin{cases}
		\mathbb Z^\mu,\text{if }q = n,\\
		0,\text{else}.
	\end{cases}
\end{displaymath}
We hence consider the monodromy action on $\tilde H_n(F_\theta;\mathbb Z)$ and this gives an invertible $\mu\times \mu$ matrix $h_*$. The conjugacy class of this matrix is independent of the choice of the flow and hence all eigenvalues of $h_*$ are invariants of the hypersurface singularity $(V(f),0)$ (see also \cite{MR4367434}, Chapter 6).

For $c\in \mathbb C,0< \vert c\vert < \delta$, Milnor showed the following.
\begin{theorem}[\cite{MR0239612}]\label{fiber type}
	The fiber of $c$ i.e. $F_{\arg c}$ is diffeomorphic to $f^{-1}(c) \cap \mathbb B_{\varepsilon}$.
\end{theorem}

We have a way to construct `canonical fiber' as below. Let $\mathbb H =\{z\in \mathbb C\mid \Im z> 0\}$ be the universal covering of $\mathbb D_\delta$ though the covering map $z \mapsto \delta e^{2\pi iz}$. Take $X_\infty$ be the fiber product $X\times_{D} \mathbb H$, where $X = f^{-1}(\mathbb D_\delta\setminus\{0\})$ and $D = \mathbb D_{\delta}\setminus\{0\}$. 
\begin{displaymath}
	\xymatrix{
		X_\infty \ar[rr] \ar[d] & & X \ar^f[d]\\
		\mathbb H \ar^{\delta e^{2\pi iz}}[rr] & & D
	}
\end{displaymath}
Since $\mathbb H$ is contractible, $X_\infty$ is homotopic equivalent to each $f^{-1}(c)\cap \mathbb B_{\varepsilon}$. Then by \textbf{Theorem \ref{fiber type}} and \textbf{Theorem \ref{boquet theorem}}, $H^n:=H^{n}(X_\infty;\mathbb C)$ is a $\mu$-dimensional $\mathbb C$-vector space. The monodromy $h_*$ also acts on $H^n$. In \cite{MR0485870}, Steenbrink introduced a mixed Hodge structure on $ H^n$. We denote it as $(H^n,F^\bullet,W_\bullet)$, where $F^\bullet$ is decreasing and $W_\bullet$ is increasing. $F^\bullet$ and $W_\bullet$ are called Hodge filtration and weighted filtration respectively. Set
\begin{align*}
	H^n_\lambda & := \{u\in H^n\mid (h_*-\lambda)^ku = 0,\text{ for some } k\in \mathbb N\},\\
	H^{p,q}_\lambda & := \mathrm{Gr}_F^p\mathrm{Gr}_{p+q}^W H^n, h_{p,q}^\lambda := \dim_{\mathbb C} H^{p,q}_\lambda.
\end{align*}
Since the decomposition $H^n = \bigoplus_{\lambda \in \mathbb S^1} H^n_{\lambda}$ is compatible with both filtrations, we see that $\sum_{p,q,\lambda} h^{p,q}_\lambda = \mu$. Next we provide the definition of spectrum numbers.

\begin{definition}[Spectrum Numbers, \cite{MR0485870}]
	Suppose $f : (\mathbb C^{n+1},0) \to (\mathbb C,0)$ is an analytic germ that defines an isolated singularity at the origin. Then the specturm numbers of $f$ are $\mu$ rational numbers $0 < \alpha_1 \leq ... \leq \alpha_\mu < n+1$ satisfying the following.
	
	\noindent(1) For all $1\neq \lambda\in \mathbb C,p\in \mathbb Z$, $\#\{j\mid e^{-2\pi i\alpha_j} = \lambda,[\alpha_j]] = n-p\} = \sum_q h^{p,q}_{\lambda}$.
	
	\noindent(2) For all $p\in \mathbb Z$, $\#\{j\mid \alpha_j = n-p+1\} = \sum_q h_1^{p,q}$.
	
	Spectrum numbers can be recorded in a `polynomial' in one variable: 
	\begin{displaymath}
		\mathrm{Sp}(f) = \sum_{j = 1}^\mu t^{\alpha_j}.
	\end{displaymath}
	It is called the spectrum of $f$.
\end{definition}

Here are some basic properties of spectrum numbers.

\begin{theorem}[Symmetry, \cite{MR0485870}]\label{symmetry of spectrum}
	For all $i = 1,...,\mu$, we have $\alpha_i+\alpha_{\mu-i+1} =n+1$.
\end{theorem}
\begin{theorem}[Formula for Quasihomogeneous Singularities, \cite{MR1621831}]\label{Formula for Quasihomogeneous Singularities}
	We further suppose $f\in \mathcal O_{n+1}$ is a weighted homogeneous polynomial of weight type $\bm w = (w_0,...,w_n)$. That is $f(t^{w_0}x_0,...,t^{w_n}x_n) = tf(x_0,...,x_n)$. Let $\bm x^{\bm \alpha_i},i=1,...,\mu$ be a monomial basis of $\mathcal O_{n+1}/J(f)$. Set $\bm e = (1,...,1)\in \mathbb N^{n+1}$. Then the spectrum numbers of $f$ are $(\bm \alpha_i+\bm e)\cdot \bm w,i=1,...,\mu$.
\end{theorem}

\begin{theorem}[Varchenko, \cite{MR648803}]\label{Varchenko's constant deformation}
	The exponents are constant under a deformation of isolated hypersurface
	singularities with constant Milnor number (under a $\mu$-constant deformation).
\end{theorem}

In 2000, Hertling proposed a conjecture that the variance of all spectrum numbers is no greater than a twelfth of their range, as stated \textbf{Conjecture \ref{Hertling Conjecture}}.

%\noindent\textbf{Conjecture \ref{Hertling Conjecture}} (Hertling)\textbf{.} \textit{Let $f\in \mathcal O_{n+1}$ be a germ which defines an isolated singularity at the origin, with Milnor number $\mu$. Suppose $\alpha_1\leq ...\leq \alpha_\mu$ are spectrum number of $f$, then the following inequality holds:}
%\begin{displaymath}
%	\frac{1}{\mu}\sum_{i = 1}^{\mu} (\alpha_i-\frac{n+1}{2})^2 \leq \frac{\alpha_\mu-\alpha_1}{12}.
%\end{displaymath}
%\vspace{0.5em}

So far, there has been some progress in Hertling Conjecture. In 2001, Hertling himself proved the conjecture for quasihomogeneous singularities (\cite{MR1849311}). More than that, the inequality is actually equality in this case. In 2004, Br\'elivet proved the conjecture for irreducible curve singularities and in 2020, they gave a generalized description of the conjecture (\cite{MR4123594}). In 2013, Yau and Zuo proved the conjecture for singularities of modality no greater than $2$ by applying the existing classification (\cite{MR3146513}). We list all the above progress as follows.

\begin{theorem}[\cite{MR1849311}]\label{Hertlin Conjecture for quasihomogeneous singularities}
	\label{3.15}
	Equality in Hertling Conjecture holds for quasihomogeneous singularities. That is, the variance equates to one-twelfth of the range.
\end{theorem}

\begin{theorem}[\cite{MR4123594}] \label{Hertling Conjecture for curve singularities}
	Hertling Conjecture holds for irreducible curve singularities.
\end{theorem}

\begin{theorem}[\cite{MR3146513}]
	Hertling Conjecture holds for singularities of modality $\leq 2$.
\end{theorem}

Hertling's theorem for quasihomogeneous singularities can be slightly generalized. We note that a semi-quasihomogeneous singularities (see \cite{MR2290112}) is actually the ending of $\mu$-constant deformation. In fact, suppose $f = f_{\bm w}+g$ defines a semi-quasihomogeneous singularity, where ${\bm w}\in \mathbb N_{>0}^{n+1}$ is a weight and $f_{\bm w}$ is the principal part of $f$. Consider $f(\bm x,t) = t^{-d}f(t^{w_0}x_0,...,t^{w_n}x_n)$. It is a deformation of $f$. Here $d$ is the integer such that $f_{\bm w}(t^{w_0}x_0,...,t^{w_n}x_n) = t^d f_{\bm w}$. All $f(\bm x,t)$ are semi-quasihomogeneous of weight $w$ with principal part $f_{\bm w}$ and hence by \cite{MR2290112}, they share the same Milnor number. So $f(\bm x,t)$ is a $\mu$-constant deformation. Since $f(\bm x,1) = f$ and $f(\bm x,0) = f_{\bm w}$, by \textbf{Theorem \ref{Varchenko's constant deformation}}, the spectrum of $f$ and $f_{\bm w}$ coincide. In particular, Hertling Conjecture holds for $f$.

\begin{corollary}
	Equality in Hertling Conjecture holds for semi-quasihomogeneous singularities. That is, the variance equates to one-twelfth of the range.
\end{corollary}

In this paper, we would like to push the conjecture a step forward, proving its correctness for singularities of modality $3$ (see \textbf{Theorem \ref{main theorem A}}). 

\subsection{Hodge Ideal and Tjurina Spectrum}
\label{Hodge Ideal and Tjurina Spectrum}

In \cite{MR3959075} and \cite{MR4089396}, the authors extend the notion of Hodge ideals to the case when \(D\) is an arbitrary effective \(\mathbb{Q}\)-divisor on $X$. Hodge ideals \(\left\{I_{k}(D)\right\}_{k \in \mathbb{N}}\) are defined in terms of the Hodge filtration \(F_{\bullet}\) on some \(\mathscr{D}_{X}\)-module associated with \(D\). When \(D\) is an integral and reduced divisor, this recovers the original definition of Hodge ideals \(I_{k}(D)\) in \cite{MR4044463}.

Let $X$ be a smooth complex variety,  and \(\mathscr{D}_{X}\) be the sheaf of differential operators on $X$.  
If $ H $ is an integral and reduced effective divisor on \(X\),  $ D = \alpha H, \alpha \in \mathbb{Q} \cap ( 0, 1], $
let \(\mathcal{O}_{X}(* D)\) be the sheaf of rational functions with poles along \(D\).  
{It } is also a left \(\mathscr{D}_{X}\)-module underlying the mixed Hodge module \(j_{*} \mathbb{Q}_{U}^{H}[n]\), 
where \(U=X \backslash D\) and \(j: U \hookrightarrow X\) is the inclusion map. Any \(\mathscr{D}_{X}\)-module associated {with} a mixed Hodge module has a good filtration \(F_{\bullet}\), the Hodge filtration of the mixed Hodge module \cite{MR1047415}.

{In order to } study the Hodge filtration of \(\mathcal{O}_{X}(* D)\), it seems easier to consider a series of ideal sheaves, defined by Musta\c{t}\v{a} and Popa \cite{MR4044463}, which can be considered to be a generalization of multiplier ideals of divisors. The Hodge ideals \(\left\{I_{k}(D)\right\}_{k \in \mathbb{N}}\) of the divisor \(D\) are defined by:
$$
F_k {\mathcal{O}}_X(*D) = I_k (D) \otimes {\mathcal{O}}_X \bigl( (k+1)D\bigr), \quad \text{for~all }  k \in \mathbb{N}.
$$
These are coherent sheaves of ideals. See \cite{MR4044463} for details {\color{black}and an extensive study} of the ideals $I_k (D)$. Hodge ideals are indexed by the non-negative integers;  at the 0-th step, they essentially coincide with multiplier ideals.
It turns out that \(I_{0}(D)=\mathcal{J}((1-\epsilon) D)\), the multiplier ideal of the divisor \((1-\epsilon) D,\;  0<\epsilon \ll 1 .\)  The multiplier ideal sheaves are ubiquitous objects in birational geometry, encoding local numerical invariants of singularities, and satisfying Kodaira-type vanishing theorems in the global setting. The Hodge ideals are interesting invariants of the singularities, they have similar properties as multiplier ideals. The reivew below follows the notations in \cite{MR4448602}.

We assume that $(X,0)$ is $(\mathbb{C}^{n+1},0)$, $Z$ is the local divisor defined by $f \in \mathbb C\{\bm x\}$, which has an isolated singularity at the origin. 
For $ \alpha \in \mathbb{Q} \cap (0, 1]$, let \(I_{k}(\alpha Z) \subset \mathbb{C}\{\bm x\}\) be the Hodge ideals for \( k \in \mathbb{N}\). 
Since \(f\) has an isolated singularity at 0 and the \(I_{k}(\alpha Z)\) are coherent, these are \(\mathfrak{m}_{X, 0}\)-primary ideals, 
that is, \(I_{k}(\alpha Z) \supset \mathfrak{m}_{X, 0}^{p}\) for some \(p \in \mathbb{Z}_{>0}\) (depending on \(k\) and $ \alpha$) 
with \(\mathfrak{m}_{X, 0} \subset \mathbb{C}\{\bm x\}\) the maximal ideal, and \(\mathbb{C}\{\bm x\} / I_{k}(\alpha Z)\) is a finite-dimensional $\mathbb C$-vector space.  
Equivalently,  {the Hodge ideal $I_k(\alpha Z)$ is Artinian}.
{The} sequence of ideals $I_k(\alpha Z)$ are refined invariant of singularities than the multiplier ideal $I_0(\alpha Z)$ alone (cf. \cite{MR3966789}, Remark 17.13). 
The Hodge spectrum $ \text{Sp}_f ^\text{HI} (t)$ is defined as the Poincar\'{e} series of the finite dimensional filtered vector space 
\begin{equation*}
	( \Omega_f ^n, V_{HI}) \simeq ( \mathbb{C}\{\bm x\} / ( J(f), V_{HI}),
\end{equation*} 
with $ V_{HI}$ defined by $ V_{HI} ^\beta ( \mathbb{C} \{\bm x\}/ J(f)) := \sum_{ \alpha + p \geq \beta} I_p (\alpha Z) \mod J( f)$. The Hodge spectrum is defined to be $ \mathrm{Sp}_f ^\mathrm{HI} (t) = \sum_{ i = 1 } ^{\mu_f} t^{\alpha_{f, i} ^\mathrm{HI}} $
with $ \# \{ i \mid \alpha_{f, i} ^\mathrm{HI} = \beta \} = \dim_\mathbb{C} \mathrm{Gr}_\mathrm{V_{HI}} ^\beta ( \mathbb{ C} \{\bm x\} /J (f))$. Here for a decreasingly filered vector space $G^\bullet V$, $\mathrm{Gr}^\beta V := V_\beta / \bigcup_{\gamma > \beta} G_\gamma$. Here the $ \alpha_{f, i} ^\mathrm{HI}$ are assumed weakly increasing.

The above definition of Hodge ideal spectrum differs from the original one in \cite{jung2018hodge}, where\\ $ V_\mathrm{HI} ^\beta ( \mathbb{C}\{ \bm x\} / J(f))$ was defined by $ I_p ( \alpha Z) \mod J(f)$
for $ \alpha + p = \beta $ with $ \alpha \in ( 0, 1], p \in \mathbb{ N}$ without taking the above summation. 
However, the Hodge ideals $ I_p ( \alpha Z) \mod J(f)$ are not necessarily weakly decreasing for $ \alpha \in (0, 1]$ (with $ p \geq 1$ fixed), see Example 4.2 in \cite{MR4448602}. The fact that $I_{p}(\alpha Z)$ is not necessarily a filtration was also observed in \cite{MR3959075} and in \cite{MR4456022}.

We can define also the Tjurina spectrum $ \mathrm{Sp}_f ^\mathrm{Tj} (t) $ by 
\begin{equation*}
	\mathrm{Sp}_f ^\mathrm{Tj} (t) = \sum_{ i =1} ^{ \tau_f} t^{\alpha_{f, i} ^\mathrm{Tj}} 
\end{equation*}
with $ \#  \{ i \mid \alpha_{ f,i } ^\mathrm{Tj} = \beta\} = \dim_\mathbb{C} \mathrm{Gr}_{V_\mathrm{HI}} ^\beta ( \mathbb{ C} \{ \bm x\} / ( J(f), f)),$
where $ \tau_f $ is the Tjurina number of $ f$, and the $ \alpha_{f, i} ^\mathrm{Tj}$ are assumed weakly increasing. This gives a link among $ \mathrm{Sp}_f (t), \mathrm{Sp}^{\mathrm{Tj}}$ and $ \mathrm{Sp}_f ^\mathrm{HI} (t). $
Indeed, there is a subset $ I \subseteq \{ 1, \ldots, \mu_f\}$ such that $ | I| = \tau_f $ and 
\begin{equation*}
	\mathrm{Sp}^\mathrm{Tj} (t) = \sum_{ i \in I}  t^{ \alpha_{f, i} ^\mathrm{HI}}, \qquad \mathrm{Sp}_f ^\mathrm{HI} (t ) - \mathrm{Sp}_f ^\mathrm{Tj} (t) = \sum_{ i \notin I} t^{ \alpha_{f, i} ^\mathrm{HI}}.
\end{equation*}
We will prove this fact in the next subsection. In this paper, we write $\mathrm{Sp}^{\tau}$ instead  of $\mathrm{Sp}^{\mathrm{Tj}}$ to mean the Tjurina spectrum. We adopt this notation in hope of avoiding too many letters appearing in a formula.

So far, the geometric meaning of the Tjurina spectrum is still mysterious. So it is worthwhile to do some systematical computation about this new object.

\subsection{Brieskorn Lattice and $V$-filtration}\label{Brieskorn Lattice}

In this subsection, we review an equivalent definition for the spectrum, given by Brieskorn lattice. This alternative definition can be used to relate the specturm and Tjurina spectrum. Let $\Omega = \Omega_{\mathbb C^{n+1},\bm 0} = \bigoplus_{j=0}^n\mathcal O_{n+1} dx_j$ be the local K\"ahler differential and $\Omega^k = \bigwedge^k \Omega$ be the module of local $k$-forms. As in \cite{MR954149}, Brieskorn lattice $\mathcal H^{(0)}$ of isolated hypersurface singularity $(V(f),\bm 0)$ is the quotient space of $\Omega^{n+1} = \mathcal O_{n+1} dx_0\wedge...\wedge dx_n$, modulo subspace $df \wedge d\Omega^{n-1}$. There is a natrual action named $\partial_t^{-1}$ on $\mathcal H^{(0)}$ defined by
\begin{displaymath}
	\partial_t^{-1} w = df\wedge \varpi,
\end{displaymath}  
where $w \in \Omega^{n+1}$ and $\varpi \in \Omega^n$ such that $d\varpi = w$. Besides, we have another action $t$ on $\mathcal H^{(0)}$ by $tw = fw$. One can easily check $[t,\partial_t^{-1}] = \partial_t^{-2}$. $\mathcal H^{(0)}$ is a free $\mathbb C\{\{\partial_t^{-1}\}\}$-module of rank $\mu = \mu(f)$ (see \cite{MR0790119} or \cite{MR553954}), where
\begin{displaymath}
	\mathbb C\{\{\partial_t^{-1}\}\} = \{\sum_{j=0}^{\infty} c_i \partial_t^{-i} \mid \sum_{j=0}^\infty \vert c_j\vert\varepsilon^j/j! < \infty,\text{ for some } \varepsilon>0\}.
\end{displaymath}
Let $\mathcal H$ be the localization of $\mathcal H^{(0)}$ by $\partial_t^{-1}$. It is called the Gau\ss-Manin system. $\mathcal H$ is a holonomic and regular $\mathcal D_{\mathbb C^{n+1},\bm 0}$-module and admits a natural $V$-filtration as below 
\begin{align*}
	H^{\alpha} & := \{u\in \mathcal H \mid (\partial_tt-\alpha)^m u = 0,\text{ for some }m >0\},\alpha\in \mathbb Q;\\
	V^{\alpha}\mathcal H & :=  \sum_{\beta \geq \alpha} \mathbb C\{\{\partial_t^{-1}\}\} H^{\beta}.
\end{align*}
$V^{\bullet}$ is a decreasing filtration with $\bigcap_{\alpha\in \mathbb Q} V^{\alpha} = 0$ and $\bigcup_{\alpha\in \mathbb Q} V^{\alpha}\mathcal H = \mathcal H$ (see \cite{MR954149}). 

Since $[t,\partial_t^{-1}] = \partial^{-2}_t$, we have $[\partial_t,t] = 1$ and hence $t(\partial_tt-\alpha)^m = (\partial_tt-\alpha-1)^mt$ and $\partial_t(\partial_tt-\alpha)^m = (\partial_tt-\alpha-1)^m \partial_t$. Therefore, $t H^{\alpha+1} \subseteq H^\alpha$ and $\partial_t H^{\alpha+1} \subseteq H^\alpha$. The latter is an isomorphism since $\partial_t$ is invertible in $\mathbb C\{\{\partial_t^{-1}\}\}[\partial_t]$. 

The filtration $V^{\bullet}$ can be induced to $\mathcal H^{(0)}$ and $\Omega_f^{n+1} := \Omega^{n+1}/df \wedge \Omega^n$ through canonical maps $\mathcal H^{(0)} \hookrightarrow \mathcal H$ and $\mathcal H^{(0)} \twoheadrightarrow \Omega^{n+1}_{f}$. The induced filtration is also called $V^{\bullet}$. Since $tH^{\beta} \subset H^{\beta+1}$, we have $f\cdot V^{\alpha}\mathcal H \subseteq \mathcal H^{\alpha+1}$ and hence 
\begin{displaymath}
	f\Omega_{f}^{n+1} \subseteq V^{\alpha+1}\Omega_{f}^{n+1}. \tag{2.1}
\end{displaymath}
By \cite{MR954149}, the spectrum $(V(f),\bm 0)$ coincides with the Poincar\'e polynomial of the filtration $V^{\bullet}$ on $\Omega_f$ i.e. 
\begin{displaymath}
	\mathrm{Sp}(f) = \sum_{\alpha\in \mathbb Q} \dim \mathrm{Gr}^\alpha_V(\Omega_f^{n+1}).
\end{displaymath}

The following theorems together show how Tjurina spectrum is derived from $V^\bullet\Omega_f^{n+1}$.

\begin{theorem}[\cite{MR4089396}]
	The Hodge ideal $I_k(\alpha Z)$ is equal to $\tilde V^{\alpha+k} \mathcal O_{n+1}$ if modulo $(f)$ i.e. 
	\begin{displaymath}
		I_k(\alpha Z) +(f) = \tilde V^{\alpha+k} \mathcal O_{n+1}+(f),
	\end{displaymath}
	whence $Z(f) = \mathrm{div}(f)$ is the divisor defined by $f$ and $\tilde V^\bullet$ is the microlocal $V$-fitration (for definition, also see \cite{MR4089396}).
\end{theorem}
\begin{theorem}[\cite{MR4448602}]\label{microlocal V-filtration = V-filtration on Brieskorn lattice}
	Let $\tilde V^\bullet$ be the microlocal filtration on $\mathcal O_{n+1}$. Then the induced filtration $\tilde V^\bullet$ on $\mathcal O_{n+1}/(J(f))$ by canonical surjection $\mathcal O_{n+1} \twoheadrightarrow \mathcal O_{n+1}/(J(f))$ coincides with the filtration $V^{\bullet}$ on $\Omega^{n+1}_f$ under canonical isomorphism $\Omega_f^{{n+1}} \overset{\simeq}{\longrightarrow} \mathcal O_{n+1}/(J(f))$.
\end{theorem}

Therefore, we have the following corollary.
\begin{corollary}\label{Tjurina spectrum is a subsepctrum}
	Tjurina spectrum is the Poincar\'e polynomial of the quotient filtration $V^{\bullet}(\Omega_{f}^{n+1}/f\cdot\Omega_{f}^{n+1})$. In particuler, Tjurina spectrum is a subspectrum of spectrum. In the sense of spectrum numbers, Tjurina spectrum is a subset of spectrum. 
\end{corollary}

\begin{remark}
	We point out that the filtration $V^\bullet$ on $\Omega_f^{n+1}$ is different from the $V$-filtration given by another kind of $\mathscr D$-module or multiplier ideals (for definition, see \cite{budur2004multiplier}) induced on $\mathcal O_{n+1}/(J(f))$. They give the same jumping numbers in $[0,1]$, but not equal in general when index is greater thatn $1$. Otherwise, there is no exponents greater than $1$ for Tjurina spectrum (see \cite{MR4448602}, (1.14) and (1.17)).
\end{remark}
%
%There is another useful tool on $\Omega_f^{n+1}$, called Grothendieck residue pairing. The general version of it is introduced in \cite{MR1011977}. In this paper, we consider the form described as follows. Properties below also comes from \cite{MR1011977}. 
%
%The Grothendieck residue pairing is a perfect symmetric $\mathbb C$-pairing
%\begin{displaymath}
%	\mathbb S : \Omega^{n+1}_f \times \Omega_f^{n+1} \to \mathbb C.
%\end{displaymath}
%Let $V^\bullet$ be the filtration on $\Omega_f^{n+1}$ shown above. We ommit $\Omega_f^{n+1}$ in $V^{\bullet}\Omega_f^{n+1}$. For each $p\in \mathbb Q$, we have  
%\begin{displaymath}
%	\mathbb S(V^{>p},V^{n+1-p}) = 0,\ \mathbb S(V^{p},V^{>n+1-p}) = 0,
%\end{displaymath}
%where $V^{>p}$ means $\bigcup_{\alpha > p} V^\alpha$. This tells $\mathbb S$ induces a pairing
%\begin{displaymath}
%	\mathbb S : \mathrm{Gr}^p_V \times \mathrm{Gr}_V^{n+1-p} \to \mathbb C,
%\end{displaymath}
%which is also perfect.
%
%\textbf{\color{red}{We can write more about the Grothendieck residue pair.}}

\subsection{Newton Filtration}
\label{Newton Filtration}

Let $f \in \mathcal O_{n+1}$ be an analytic germ which defines an isolated singularity at the origin. Here we give a brief review of Newton diagram and Newton filtration associated with $f$. It can also be seen in \cite{MR954149} and \cite{MR0704395}.

Let $\Gamma_+(f)$ be the Newton diagram of $f$, say the convex hull of $\bigcup_{\bm v\in \mathrm{Supp}(f)} (\bm v+\mathbb R_{\geq 0}^{n+1})$. Then $Fc(f)$ denotes the set of compact facets of $\Gamma_+(f)$. Recall that for a convex set $X$ in $\mathbb R^n$, a face of $X$ is the intersection of $X$ and a hyperplane $L$, where $X$ lies on one side of $L$. A facet of $X$ is a face of affine dimension $n-1$.

For a facet $\sigma$ in $Fc(f)$, define $A_\sigma$ to be the $\mathbb C$-subalgbra of $\mathbb C[\bm x]$ generated by $\bm x^{\bm v},\bm v\in \mathrm{Cone}(0,\sigma)$. Let $F_\sigma = \sum_{\bm v\in \mathrm{Cone}(0,\sigma)} a_{\bm v}\bm x^{\bm v}$ and $F_{\sigma,i} = x_i\frac{\partial F_\sigma}{\partial x_i},i=0,...,n$. $f$ is called \textit{Newton non-degenerate} (or ``\textit{non-degenerate}'') \textit{at facet} $\sigma \subset \Gamma(f)$, if $\dim_{\mathbb C} A_\sigma/\sum_{i} A_\sigma F_{\sigma,i} < \infty$. Note that if $\mathrm{Cone}(0,\sigma)$ is simplicial, that is, generated by $n+1$ linearly independent vectors, and $\#(\mathrm{Cone}(0,\sigma)\cap \mathrm{Supp}(f) \cap \partial \Gamma_+(f)) = n+1$, then $f$ is non-degenerate at $\sigma$. $f$ is called \textit{non-degenerate} if it is non-degenerate at every $\sigma\in Fc(f)$.

We further assume $f$ is non-degenerate and for every $i=0,...,n$, some power $x_i$ appears in $f$. We say such $f$ \textit{convenient}. Under this circumstance, $\mathbb R_+^{n+1} = \bigcup_{\sigma\in Fc(f)} \mathrm{Cone}(0,\sigma)$. For every $\sigma\in Fc(f)$, there exists a unique linear function $\varphi_{\sigma} : \mathbb R^n \to \mathbb R,  {\bm v}  \mapsto \sum_{i=0}^{n+1} a_{\sigma,i} v_i$ such that $\varphi_\sigma|_\sigma = 1$. Moreover, for any two facets $\sigma,\sigma'$ and $\bm w\in \mathrm{Cone}(0,\sigma) \cap \mathrm{Cone}(0,\sigma')$, $\varphi_{\sigma}(\bm w) = \varphi_{\sigma'}(\bm w)$. Therefore, we are able to define a function $\varphi : \mathbb R_{\geq 0}^{n+1}\to \mathbb R$ such that $\varphi|_{\mathrm{Cone}(0,\sigma)} = \varphi_{\sigma}$. It is called the \textit{valuation of }$\Gamma_+(f)$. Next, we define the \textit{shifted Newton filtration} of $f$.

Let $\bm e = (1,1,...,1)\in \mathbb R_+^{n+1}$. For $g\in \mathcal O_{n+1}$, also define $\varphi(g) := \min_{\bm v\in \mathrm{Supp}(g)} \varphi(v)$. Hence, let $N_{a}\mathcal O_{n+1} = \{g\in \mathcal O_{n+1}\mid \varphi(\bm x^{\bm e}\cdot g) \geq a\}$ and we obtain a decreasing $\mathbb Q$-indexed filtration:
\begin{displaymath}
	... \supset N_{a}\mathcal O_{n+1} \supset ...\supset N_b\mathcal O_{n+1} \supset ...(a< b).
\end{displaymath}

Let $M_f = \mathcal O_{n+1}/J(f)$ be the Milnor algbebra of $f$ and $\mu$ be the Milnor numebr. $N_\bullet\mathcal O_{n+1}$ induces a natural filtration on $M_f$, say
\begin{displaymath}
	N_a M_f = (N_a\mathcal O_{n+1}+J(f))/J(f).
\end{displaymath}

With everything clarified, we state the theorem of M. Saito(\cite{MR954149}):
\begin{theorem}[M. Saito, see \cite{MR954149} or \cite{https://doi.org/10.48550/arxiv.2003.00519}] \label{Saito's non-degenerated theorem}
	Suppose $f\in \mathcal O_{n+1}$ is convenient and Newton non-degenerate, then the $V$-filtration on $\Omega_f^{n+1}$ coincides with the shifted Newton filtration on $M_f$ under canonical isomorphism $\Omega_f^{n+1} \overset{\simeq}{\longrightarrow} M_f$. In particular.
	\begin{displaymath}
		\mathrm{Sp}(f) = \sum_{a\in \mathbb Q} \dim_{\mathbb C} \mathrm{Gr}_a^N(M_f) t^{a},
	\end{displaymath}
	where $\mathrm{Gr}_a^N(M_f) = N_aM_f/ \bigcup_{b> a}N_bM_f$, the grading of $N_\bullet M_f$.
\end{theorem}

Therefore, for many cases, by applying \textbf{Theorem \ref{Saito's non-degenerated theorem}}, it suffices to compute the dimension of each grading. Here we give the method to compute those dimensions. Let $v : \mathcal O_{n+1} \to \mathbb R_{\geq 0}, g \mapsto \varphi(\bm x^{\bm e}\cdot g)$, then $N_a\mathcal O_{n+1} = \{g\in \mathcal O_{n+1}\mid v(g) \geq a\}$. $v$ is called the \textit{shifted valuation of} $f$. For a monomial $\bm x^{\bm \alpha}\in\mathcal O_{n+1}$ and an ideal $I$ of $\mathcal O_{n+1}$, we will also write $\bm x^{\bm \alpha}$ to denote its image in $\mathcal O_{n+1}/I$.

\begin{definition}[Monomial Filtration]
	Suppose $I$ is an ideal of $\mathcal O_{n+1}$ such that $\sqrt{I} = \mathfrak m$ and $F_\bullet\mathcal O_{n+1}$ is a $\mathbb Q$-indexed filtration of $\mathcal O_{n+1}$. It induces a natural filtration on $A := \mathcal O_{n+1}/I$, denoted as $F_\bullet A$. $F_\bullet A$ is called a monomial filtration, if each $F_a A,a\in \mathbb Q$ permits a monomial basis. That is, for all $a\in \mathbb Q$, there is $\bm x^{\bm \rho_1},...,\bm x^{\bm \rho_q}\in \mathcal O_{n+1}/I$ forming a basis of $F_a(\mathcal O_{n+1}/I)$.
\end{definition}
\begin{remark}
	Since $\dim_\mathbb C A< \infty$, by a simple argument of linear algebra, we have the following proposition.
\end{remark}
\begin{proposition}
	Suppose $I$ is an ideal of $\mathcal O_{n+1}$ with $\sqrt{I} = \mathfrak m$ and $F_\bullet\mathcal O_{n+1}$ induces a monomial filtration $A = \mathcal O_{n+1}/I$. Then there is a monomial basis $\{\bm x^{\bm \alpha_i}\}_{i\in \Lambda}$ satisfying the following:
	
	For every $a\in \mathbb Q$, there is a subset $\Lambda_a$ of $\Lambda$ such that $\{\bm x^{\bm \alpha_i}\}_{i\in \Lambda_a}$ is a basis of $\mathrm{Gr}_a^F(A)$. 
	
	Such monomial basis is called an \textit{admissible basis o}f $F_\bullet A$.
\end{proposition}

As a corollary, we have a way to compute $\mathrm{Gr}_a^N(M_f)$.

\begin{corollary}
	Let $\{\bm x^{\bm \alpha_i}\}_{i\in \Lambda}$ be an adimissible basis of $M_f$, then for $a\in \mathbb Q$, $\dim_{\mathbb C}\mathrm{Gr}_a^N(M_f) = \#\{i\in \Lambda\mid \bm x^{\bm \alpha_i} \in F_aM_f\}-\#\{i\in \Lambda\mid \bm x^{\bm \alpha_i} \in \bigcup_{b>a}F_bM_f\}$.
\end{corollary}

To compute an admissible basis, we give another notion, ``maximal basis". The two notions of basis are actually equivalent, which can simply be checked by definition. And hence Newton filtration is a monomial filtration.

\begin{definition}[maximal basis]
	Let $B = \{\bm x^{\bm \alpha_i}\}_{i=1}^\mu$ be a monomial basis of $M_f$, where $v(\bm \alpha_1) \leq v(\bm \alpha_2) \leq ... \leq v(\bm \alpha_\mu)$. For $a\in \mathbb Q$, let $B_a = \{\bm x^{\bm \alpha_i}\}_{v(\bm \alpha_i) \geq a}$.
	
	$B$ is called \textit{maximal} if for any monomial $\bm x^{\bm \beta} \in \mathcal O_{n+1}\setminus J(f)$, $\bm x^{\bm\beta}$ is the linear combination of monomials in $B_{v(\bm \beta)}$ modulo $J(f)$.
\end{definition}

We hence propose a straightforward algorithm to compute maximal (hence admissible) basis of $N_\bullet M_f$, the shifted Newton filtration. Our algorithm turns out to be effective and will be applied frequently in next few pages.

\begin{algorithm}\label{algorithm for a maximal basis}
	Let $f \in \mathcal O_{n+1}$ define an isolated singularity at the origin. $M_f$ denotes the Milnor algebra of $f$ and $v = \varphi(\bm x^{\bm e}\cdot-)$ be the shifted valuation of $f$. Running the following process, we will obtain a maximal basis of $N_\bullet M_f$.
	
	\noindent(1) Let $S = \{\bm x^{\bm \alpha} \in \mathcal O_{n+1}\setminus J(f)\}$ and $T = \varnothing$.
	
	\noindent(2) If $S = \varnothing$, return $T$.
	
	\noindent(3) Let $\bm x^{\bm \alpha} \in S$ such that $v(\bm x^{\bm \alpha})$ is the maximum among all $v(\bm x^{\bm \beta}),\bm x^{\bm \beta}\in S$. Add $\bm x^{\bm \alpha}$ to $T$ and delete all $\bm x^{\bm \gamma} \in \mathrm{span}_\mathbb C T + J(f)$.
	
	\noindent(4) Back to (2).
\end{algorithm}

\section{Spectrum and Hertling Conjecture for Trimodal Singularities}

\subsection{Spectrum of Trimodal Singularities}

Our classification of singularities of modality $3$ follows from \cite{MR0553709} and \cite{MR1721628}. The first reference classifies all quasi-homogeneous trimodal singularities and the second classifies other trimodal singularities. In this subsection, we will compute the spectra of non-quasihomogeneous trimodal singularities (listed below). To avoid too much repetition, we only provide processes for the cases $NA_{r,0}$, $NA_{r,s}$ and $VA_{r,s}$.    

Singularties of modality $3$ in \cite{MR1721628} are classified into two types $4N$ and $4V$. 

% \noindent\textbf{I. 4N Type}

% \begin{tabular}{lccc}
	% 	\hline
	% 	Name & $\mu$ & Representative $f_0$ & Further Terms\\
	% 	\hline
	% 	$NA_{r,0}$ & $16+r$ & $x^{r+5}+x^3y^2+dxy^{4}+by^5$ & $xy^5,y^6$\\
	% 	$NA_{r,s}$ & $16+r+s$ & $ax^{r+5}+x^3y^2+x^2y^3+by^{s+5}$ & $x^3y^3$\\
	% 	$NB_{(-1)}^r$ & $18+r$ & $ax^{r+5}+x^3y^2+y^6$ & $xy^5,xy^6$\\
	% 	$NB_{(0)}^r$ & $19+r$ & $ax^{r+5}+x^3y^2+xy^5$ & $x^2y^4,x^2y^5$\\
	% 	$NB_{(1)}^r$ & $20+r$ & $ax^{r+5}+x^3y^2+y^7$ & $xy^6,xy^7$\\
	% 	$NC_{19}$ & $19$ & $x^4y+y^6$ & $xy^5,x^2y^4,x^2y^5$\\
	% 	$NC_{20}$ & $20$ & $x^4y+xy^5$ & $x^2y^4,x^3y^3,x^3y^4$\\
	% 	$NF_{20}$ & $20$ & $x^5+y^6$ & $x^2y^4,x^3y^3,x^3y^4$\\
	% 	$NF_{21}$ & $21$ & $x^5+xy^5$ & $x^3y^3,x^4y^2,x^4y^3$\\
	% 	\hline
	% \end{tabular}
% ~\\
\begin{table}[!htbp]
	\centering
	\caption{4N Type}
	%\label{频率型、强度型和相对比型指标}也可以用这个
	\renewcommand\arraystretch{1.5}
	\begin{tabular}{lp{2cm}p{6cm}p{3cm}}
		\toprule  %添加表格头部粗线
		Name & $\mu$ & Representative $f_0$ & Further Terms\\
		\midrule  %添加表格中横线
		$NA_{r,0}$ & $16+r$ & $x^{r+5}+x^3y^2+dxy^{4}+by^5$ & $xy^5,y^6$\\
		$NA_{r,s}$ & $16+r+s$ & $ax^{r+5}+x^3y^2+x^2y^3+by^{s+5}$ & $x^3y^3$\\
		$NB_{(-1)}^r$ & $18+r$ & $ax^{r+5}+x^3y^2+y^6$ & $xy^5,xy^6$\\
		$NB_{(0)}^r$ & $19+r$ & $ax^{r+5}+x^3y^2+xy^5$ & $x^2y^4,x^2y^5$\\
		$NB_{(1)}^r$ & $20+r$ & $ax^{r+5}+x^3y^2+y^7$ & $xy^6,xy^7$\\
		$NC_{19}$ & $19$ & $x^4y+y^6$ & $xy^5,x^2y^4,x^2y^5$\\
		$NC_{20}$ & $20$ & $x^4y+xy^5$ & $x^2y^4,x^3y^3,x^3y^4$\\
		$NF_{20}$ & $20$ & $x^5+y^6$ & $x^2y^4,x^3y^3,x^3y^4$\\
		$NF_{21}$ & $21$ & $x^5+xy^5$ & $x^3y^3,x^4y^2,x^4y^3$\\
		\bottomrule  
	\end{tabular}
\end{table}
% \noindent\textbf{II. 4V Type}
% \begin{tabular}{lccc}
	% 	\hline
	% 	Name & $\mu$ & Representative $f_0$ & Further Terms\\
	% 	\hline
	% 	$VA_{r,s}$ & $15+s+r$ & $ax^{r+4}+x^2y^2+(x+y)z^2+by^{s+4}$ & $xyz^2$\\
	% 	$VA_{r,0}^\#$ & $15+r$ & $a\phi_{r+8}+x^3y+2dx^2y^2+yz^2+xy^3$ & $x^2y^3$\\
	% 	$VA_{r,s}^\#$ & $15+r+s$ & $a\phi_{r+8}+x^3y+x^2y^2+yz^2+by^{s+4}$ & $x^3y^2$\\
	% 	$VB_{(-1)}^s$ & $17+s$ & $x^3z+x^2y^2+yz^2+by^{s+4}$ & $x^5,x^6$\\
	% 	$VB_{(0)}^s$ & $18+s$ & $x^5+x^2y^2+yz^2+by^{s+4}$ & $x^4y,x^5y$\\
	% 	$VB_{(-1)}^{\#,r}$ & $17+r$ & $a\phi_{r+8}+x^3y+yz^2+y^5$ & $xy^4,xy^5$\\
	% 	$VB_{(0)}^{\#,r}$ & $18+r$ & $a\phi_{r+8}+x^3y+yz^2+xy^4$ & $y^6,y^7$\\
	% 	$VB_{(1)}^{\#,r}$ & $19+r$ & $a\phi_{r+8}+x^3y+yz^2+y^6$ & $xy^5,xy^6$\\
	% 	$VC_{18}$ & $18$ & $x^4+yz^2+y^5$ & $xy^4,x^2y^3,x^2y^4$\\
	% 	$VC_{19}$ & $19$ & $x^4+yz^2+xy^4$ & $x^2y^3,x^3y^2,x^3y^3$\\
	% 	$VC_{18}^\#$ & $18$ & $x^3z+yz^2+xy^3$ & $x^5,x^4y,x^5y$\\
	% 	$VC_{19}^\#$ & $19$ & $x^5+yz^2+xy^3$ & $x^4y,x^3y^2,x^4y^2$\\
	% 	$VF_{19}$ & $19$ & $x^3z+yz^2+y^4$ & $x^4y,x^3y^2,x^4y^2$\\
	% 	$VF_{20}$ & $20$ & $x^5+yz^2+y^4$ & $x^3y^2,x^2y^3,x^3y^3$\\
	% 	\hline
	% \end{tabular}
% ~\\
~\\

\begin{table}[!htbp]
	\centering
	\caption{4V Type}
	%\label{频率型、强度型和相对比型指标}也可以用这个
	\renewcommand\arraystretch{1.5}
	\begin{tabular}{lp{2cm}p{6cm}p{3cm}}
		\toprule  %添加表格头部粗线
		Name & $\mu$ & Representative $f_0$ & Further Terms\\
		\midrule  %添加表格中横线
		$VA_{r,s}$ & $15+s+r$ & $ax^{r+4}+x^2y^2+(x+y)z^2+by^{s+4}$ & $xyz^2$\\
		$VA_{r,0}^\#$ & $15+r$ & $a\phi_{r+8}+x^3y+2dx^2y^2+yz^2+xy^3$ & $x^2y^3$\\
		$VA_{r,s}^\#$ & $15+r+s$ & $a\phi_{r+8}+x^3y+x^2y^2+yz^2+by^{s+4}$ & $x^3y^2$\\
		$VB_{(-1)}^s$ & $17+s$ & $x^3z+x^2y^2+yz^2+by^{s+4}$ & $x^5,x^6$\\
		$VB_{(0)}^s$ & $18+s$ & $x^5+x^2y^2+yz^2+by^{s+4}$ & $x^4y,x^5y$\\
		$VB_{(-1)}^{\#,r}$ & $17+r$ & $a\phi_{r+8}+x^3y+yz^2+y^5$ & $xy^4,xy^5$\\
		% 		\midrule  %添加表格中横线
		% 	\end{tabular}
	% \end{table}
% \begin{table}[!htbp]
	% 	\centering
	% 	\renewcommand\arraystretch{1.5}
	% 	\begin{tabular}{lp{2cm}p{6cm}p{3cm}}
		% 		\midrule  %
		$VB_{(0)}^{\#,r}$ & $18+r$ & $a\phi_{r+8}+x^3y+yz^2+xy^4$ & $y^6,y^7$\\
		$VB_{(1)}^{\#,r}$ & $19+r$ & $a\phi_{r+8}+x^3y+yz^2+y^6$ & $xy^5,xy^6$\\
		$VC_{18}$ & $18$ & $x^4+yz^2+y^5$ & $xy^4,x^2y^3,x^2y^4$\\
		$VC_{19}$ & $19$ & $x^4+yz^2+xy^4$ & $x^2y^3,x^3y^2,x^3y^3$\\
		$VC_{18}^\#$ & $18$ & $x^3z+yz^2+xy^3$ & $x^5,x^4y,x^5y$\\
		$VC_{19}^\#$ & $19$ & $x^5+yz^2+xy^3$ & $x^4y,x^3y^2,x^4y^2$\\
		$VF_{19}$ & $19$ & $x^3z+yz^2+y^4$ & $x^4y,x^3y^2,x^4y^2$\\
		$VF_{20}$ & $20$ & $x^5+yz^2+y^4$ & $x^3y^2,x^2y^3,x^3y^3$\\
		\bottomrule  
	\end{tabular}
\end{table}
Here $\phi_{m+8} = \begin{cases}
	x^{\frac{m}{2}+4}, \text{if } 2\mid m,\\
	x^{\frac{m-1}{2}+3}z, \text{if } 2\nmid m.
\end{cases}$

Parameters $r,s$ are positive integers and $a,b,d$ are complex numbers (they only influence the choice of monomial basis of Milnor algebra) with $a,b\neq 0$. Further terms mean the singularities is a sum of $f_0$ and a $\mathbb C$-linear combination of further terms.

One may find under such circumstances, all representatives above are Newton non-degenerate, since each cone containg $n+1$ vertices and are simplicial and of dimension $n+1$.

For a specific singularity $(V(f),0),f\in \mathcal O_{n+1}$, one can use SINGULAR to compute its spectrum as below.

\begin{verbatim}
	LIB"gmssing.lib"
	ring r = 0, (x,y), ds;
	poly f = x14+x3y2+xy4+y5;
	spectrum(f);
\end{verbatim}

Next we compute the spectrum of $NA_{r,0}$ and $NA_{r,s}$ in detail, as examples of our algorithm (\textbf{Algorithm \ref{algorithm for a maximal basis}}). For other types of singularities of modality $3$, we merely list results.

To run our algorithm, we need to first determine monomials in $\mathcal O_{n+1}\setminus J(f)$ and second know their relations modulo $J(f)$. We list the Gr\"obner basis of $J(f)$ so that one can easily realize the two things. As for an admissible order (see \cite{MR1075338}), we choose $x>y>1$ and the local order. That is, we apply the lexicographic order for monomials and list monomials terms from small to large. In \cite{MR1721628}, monmial bases of the Milnor algebras of singularities of modality $3$ have been listed. So one can compute maximal bases in a more convenient way.

\begin{lemma}[Gr\"obner basis of $NA_{r,0}$]
	The following is a Gr\"obner basis of $J(f)$, where $f = x^{r+5}+x^3y^2+y^5(r>0)$.
	\begin{align*}
		& 2x^3y+5y^4, 3x^2y^2+x^{r+4}, 15y^5-2(r+5)x^{r+5},x^{r+7}.
	\end{align*}
\end{lemma}

\begin{proposition}[$NA_{r,0}$]
	The following is a maximal basis of the Milnor algebra of $f = x^{r+5}+x^3y^2+y^5$($r>0$):
	\begin{align*}
		x^k(1\leq k \leq r+5);
		x^2y;
		y^i(0 \leq i \leq 4);
		xy^{i+1}( 0\leq i \leq 4).
	\end{align*}
	And the spectrum of $f$ is hence given below:
	\begin{displaymath}
		\frac{2}{5},\frac{3}{5},\frac{4}{5},\frac{4}{5},1,1,1,\frac{6}{5},\frac{6}{5},\frac{7}{5},\frac{8}{5}, \frac{2k+r+4}{2(r+5)}(1\leq k\leq r+5).
	\end{displaymath}
	%where multiplicities of $1,\frac{4}{5},\frac{6}{5}$ are $4,2,2$ respectively.
\end{proposition}
\begin{proof}
	With the help of the Gr\"obner basis, one can show all monomials in $\mathcal O_2\setminus J(f) = \{\bm x^{\bm \alpha \in \mathcal O_X \mid \bm x^{\bm \alpha}\not\in J(f)}\}$ are contained in the following set:
	\begin{align*}
		W = \{& 1,x,...,x^{r+6}, y,xy,...,x^4y,y^2,xy^2,y^3,xy^3,y^4,xy^4,y^5,xy^5\}.
	\end{align*}
	Their shifted valuations are presented as follow:
	\begin{align*}
		V = \{& \frac{2}{5},\frac{2k+r+4}{2(r+5)}(1\leq k\leq r+6), \frac{3}{5}, \frac{4}{5}, 1, \frac{r+6}{r+5}, \frac{r+7}{r+5}, \frac{4}{5},1,1,\frac{6}{5},\frac{6}{5},\frac{7}{5}, \frac{7}{5},\frac{8}{5}\}. 
	\end{align*}
	
	By \cite{MR1721628}, the following is a basis of $\mathcal O_2/J(f)$:
	\begin{align*}
		K = \{x^k(1\leq k \leq r+4);
		x^2y;
		y^i(0 \leq i \leq 4);
		xy^{i+1}( 0\leq i \leq 4); y^5\}.
	\end{align*}
	
	Next, we run \textbf{Algorithm \ref{algorithm for a maximal basis}} on $W$ and seek for a basis by picking or deleting elements of $W$. Since $y^5$ is linearly dependent with $x^{r+5}$, we slightly can adjust $K$ as
	\begin{align*}
		K' = \{x^k(1\leq k \leq r+5);
		x^2y;
		y^i(0 \leq i \leq 4);
		xy^{i+1}( 0\leq i \leq 4)\}.
	\end{align*}
	$K'$ is also a basis. We note that $\frac{8}{5}$ is the biggest in $V$ and $xy^5\in K'$, thus $xy^5$ is picked. Next, since $x^{r+6}$ is linearly dependent with $xy^5$, it is ignored. So far, we have only picked $xy^5$ and deleted $x^{r+6}$.
	
	Since $\frac{2(r+5)+r+4}{2(r+5)} = \frac{3}{2}-\frac{1}{2(r+5)} > \frac{7}{5}$, $x^{r+5}\in K'$ and $y^5$ is linearly dependent with $x^{r+5}$, we pick $x^{r+5}$ and delete $y^5$.
	
	Next, by our algorithm, we then pick $x^k$ for all $k$ such that $\frac{2k+r+4}{2(r+5)} \geq \frac{7}{5}$, since $\frac{r+7}{r+5} \leq \frac{7}{5}$. Let $k_0$ be the smallest one of such $k$. Till now, we have picked $K_{k_0} = \{x^{k_0},...,x^{r+5},xy^5\} \subset K'$. We then confront $y^4$, which is also in $K'$. Hence we pick $y^4$ and delete $x^3y,x^4y$. Eventually, the rest elements of $W$ are exactly $K'-(K_{k_0}\cup \{y^4\})$. $K'$ is a maximal basis by our wish.
\end{proof}

Next we consider $NA_{r,s}$.
\begin{lemma}[Gr\"obner basis of $NA_{r,s}$]
	The following is a Gr\"obner basis of $J(f)$, where 
	\begin{equation*}
		f = x^{r+4}+x^3y^2+x^2y^3+y^{s+4}(r\geq s>0).
	\end{equation*}
	\begin{align*}
		\begin{split}
			& 2x^3y+3x^2y^2+(s+4)y^{s+3},3x^2y^2+2xy^3+(r+4)x^{r+4},\\
			&10xy^4-9(s+4)y^{s+4}+6(r+4)x^{r+4}+5(r+4)x^{r+3}y, (s+4)y^{s+5}-(r+4)x^{r+5},x^{r+6}.
		\end{split}
	\end{align*}
\end{lemma}
\begin{proposition}[$NA_{r,s}$]
	The following is a maximal basis of the Milnor algebra of $f = x^{r+5}+x^3y^2+x^2y^3+y^{s+5}$ when $r,s > 1$:
	\begin{align*}
		x^k(1\leq k \leq r+5);
		x^2y,xy^2;
		y^l(1 \leq l \leq s+5);
		(xy)^{i}( 0\leq i \leq 3).
	\end{align*}
	
	And the spectrum of $f$ is hence given below:
	\begin{displaymath}
		\frac{2}{5},\frac{4}{5},\frac{6}{5},\frac{8}{5},1,1,\frac{2k+r+4}{2(r+5)}(1\leq k\leq r+5),\frac{2l+r+4}{2(s+5)}(1\leq l\leq s+5).
	\end{displaymath}
\end{proposition}
\begin{proof}
	A direct computation shows that all monomials in $\mathcal O_2\setminus J(f)$ are contained in the following set:
	\begin{align*}
		W = \{& x,...,x^{r+6},1,xy,...,(xy)^3, x^2y,x^3y, x^4y,x^5y,xy^2,xy^3,xy^4,xy^5, \\
		&x^3y^2,x^2y^3,x^4y^2,x^2y^4, y,y^2,...,y^{s+6}\}.
	\end{align*}
	Their shifted valuations are
	\begin{align*}
		V = \{& \frac{2k+r+4}{2(r+5)}(1\leq k\leq r+6), \frac{2}{5},\frac{4}{5},\frac{6}{5},\frac{8}{5}, 1,\frac{r+6}{r+5},\frac{r+7}{r+5},\frac{r+8}{r+5},1,\\
		&\frac{s+6}{s+5},\frac{s+7}{s+5},\frac{s+8}{s+5}, \frac{7}{5},\frac{7}{5},\frac{3r+16}{2(r+5)},\frac{3s+16}{2(s+5)}, \frac{2l+r+4}{2(s+5)}(1\leq l\leq s+6)\}.
	\end{align*}
	
	By \cite{MR1721628}, $K = \{x^k(1\leq k \leq r+5);
	x^2y,xy^2;
	y^l(1 \leq l \leq s+5);
	(xy)^{i}( 0\leq i \leq 3)\}$ is a basis of $\mathcal O_2/J(f)$. So it suffices to proof its maximality.
	
	Since $(xy)^3\in K$ and $\frac{8}{5}$ is the biggest one in $V$, we pick $(xy)^3$. Since $x^5y,xy^5,x^4y^2,x^2y^4,x^{r+6}$ and $x^{s+6}$ are linearly dependent with $(xy)^3$, we delete them. Next, running the same procedure as $NA_{r,0}$, we will pick $x^{r+5}$ and $y^{s+5}$ since $\frac{3r+14}{2(r+5)}$ and $\frac{3s+14}{2(r+5)}$ are both greater than $\frac{7}{5}$. We hence delete $x^3y^2,x^2y^3$. Moreover, $x^4y$ and $xy^4$ are linearly dependent with $x^3y^2,x^2y^3$, so we also delete them. Because $\frac{6}{5} > \max\{\frac{r+6}{r+5},\frac{s+6}{s+5}\}$ and $(xy)^2$ is linearly independent with $(xy)^3,x^k,y^l(1\leq k\leq r+5,1\leq l\leq s+5)$, we pick $(xy)^2$. We note that $x^{r+4}$ and $y^{s+4}$ are also picked since $\frac{3(r+4)}{2(r+5)},\frac{3(s+4)}{2(s+5)} > \frac{6}{5}$. Then $x^3y$ and $xy^3$ are deleted, for there dependence with $\{x^{r+4},y^{s+r}\}$. The remaining elements of $W$ and what we have picked are exactly elements of $K$. So we are done. 
\end{proof}
\begin{remark}
	As for the exceptional cases $(r,s) = (1,1),(1,2),(2,2)$, we used SINGULAR and find their spectrum also satisfies the results of cases for $r,s$ larger.
\end{remark}

We list the spectrum of curve singularities of modality $3$ as below.

\begin{table}[!htbp]
	\centering
	\caption{Spectrum of curve Singularities of Modality 3}
	%\label{频率型、强度型和相对比型指标}也可以用这个
	\renewcommand\arraystretch{1.5}
	\begin{tabular}{lp{2cm}p{12cm}}
		\toprule  %添加表格头部粗线
		\textbf{Class}   &$\mu$         & Spectrum Number\\
		\midrule  %添加表格中横线
		$NA_{r,0}$ & $16+r$ & $\frac{2}{5},\frac{3}{5},\frac{4}{5},\frac{4}{5},1,1,1,\frac{6}{5},\frac{6}{5},\frac{7}{5},\frac{8}{5},\frac{2k+r+4}{2(r+5)}(1\leq k\leq r+5)$ \\
		$NA_{r,s}$ & $16+r+s$ & $\frac{2}{5},\frac{4}{5},\frac{6}{5},\frac{8}{5},1,1,\frac{2k+r+4}{2(r+5)}(1\leq k\leq r+5),\frac{2l+r+4}{2(s+5)}(1\leq l\leq s+5)$ \\
		$NB_{(-1)}^r$ & $18+r$ & 
		$\frac{7}{18},\frac{5}{9},\frac{13}{18},\frac{7}{9},\frac{8}{9},\frac{17}{18},1,\frac{19}{18},\frac{10}{9},\frac{11}{9},\frac{23}{18},\frac{13}{9},\frac{29}{18},\frac{2k+r+4}{2(r+5)}(1\leq k\leq r+5)$  \\
		$NB_{(0)}^r$ & $19+r$ &  $\frac{5}{13},\frac{7}{13},\frac{9}{13},\frac{10}{13},\frac{11}{13},\frac{12}{13},1,1,\frac{14}{13},\frac{15}{13},\frac{16}{13},\frac{17}{13},\frac{19}{13},\frac{21}{13},\frac{2k+r+4}{2(r+5)}(1\leq k \leq r+5)$\\
		$NB_{(1)}^r$ & $20+r$ & $\frac{8}{21},\frac{11}{21},\frac{2}{3},\frac{16}{21},\frac{17}{21},\frac{19}{21},\frac{20}{21},1,\frac{22}{21},\frac{23}{21},\frac{25}{21},\frac{26}{21},\frac{4}{3},\frac{31}{21},\frac{34}{21},\frac{2k+r+4}{2(r+5)}(1\leq k \leq r+5)$ \\
		$NC_{19}$ & $19$ & $\frac{3}{8},\frac{13}{24},\frac{7}{12},\frac{17}{24},\frac{3}{4},\frac{19}{24},\frac{7}{8},\frac{11}{12},\frac{23}{24},1,\frac{25}{24},\frac{13}{12},\frac{9}{8},\frac{29}{24},\frac{5}{4},\frac{31}{24},\frac{17}{12},\frac{35}{24},\frac{13}{8}$ \\
		$NC_{20}$ & $20$ & $\frac{7}{19},\frac{10}{19},\frac{11}{19},\frac{13}{19},\frac{14}{19},\frac{15}{19},\frac{16}{19},\frac{17}{19},\frac{18}{19},1,1,\frac{20}{19},\frac{21}{19},\frac{22}{19},\frac{23}{19},\frac{24}{19},\frac{25}{19},\frac{27}{19},\frac{28}{19},\frac{31}{19}$ \\
		$NF_{20}$ & $20$ & $\frac{11}{30},\frac{8}{15},\frac{17}{30},\frac{7}{10},\frac{11}{15},\frac{23}{30},\frac{13}{15},\frac{9}{10},\frac{14}{15},\frac{29}{30},\frac{31}{30},\frac{16}{15},\frac{11}{10},\frac{17}{15},\frac{37}{30},\frac{19}{15},\frac{13}{10},\frac{43}{30},\frac{22}{15},\frac{49}{30}$ \\
		$NF_{21}$ & $21$ & $\frac{9}{25},\frac{13}{25},\frac{14}{25},\frac{17}{25},\frac{18}{25},\frac{19}{25},\frac{21}{25},\frac{22}{25},\frac{23}{25},\frac{24}{25},1,\frac{26}{25},\frac{27}{25},\frac{28}{25},\frac{29}{25},\frac{31}{25},\frac{32}{25},\frac{33}{25},\frac{36}{25},\frac{37}{25},\frac{41}{25}$ \\
		\bottomrule  
	\end{tabular}
\end{table}
We point out that if a rational number appears $q$-times in a row, then it has multiplicity $q$ in the spectrum of the corresponding singularity.

For surface singularities of modality $3$, we only give a sketch of computation for $VA_{r,s}$ and the results of other singularities. The following is a Gr\"obner basis of $J(f),f = x^{r+4}+x^2y^2+(x+y)z+y^{s+4}$.

\begin{lemma}[Gr\"obner basis of $NA_{r,s}$]
	The following is a Gr\"obner basis of $J(f)$, where 
	\begin{equation*}
		f = x^{r+4}+x^2y^2+(x+y)z+y^{s+4}(r,s>0).
	\end{equation*}
	\begin{align*}
		& xz+yz,z^2+2x^2y+(s+4)y^{s+3},2x^2y-2xy^2+(s+4)y^{s+3}-(r+4)x^{r+3},\\
		& 4xy^3-(s+4)y^{s+4}+(r+4)x^{r+4}+2(r+4)x^{r+3}y, 4y^3z+(s+4)y^{s+3}z-(r+4)x^{r+3}z,\\
		& 4(s+4)y^{s+5}+(s+4)^2y^{s+4}-4(r+4)x^{r+5}, x^{r+6}.
	\end{align*}		
\end{lemma}

\begin{proposition}[$VA_{r,s}$]
	The following is a maximal basis of the Milnor algebra of $f = x^{r+4}+x^2y^2+(x+y)z+y^{s+4}$($r,s>0$):
	\begin{align*}
		1,z,xy,xz,z^2,xyz,xyz^2;
		x^k(1\leq k \leq r+4);
		y^l(1 \leq l \leq s+4).
	\end{align*}
	
	\begin{displaymath}
		\frac{7}{8},\frac{5}{4},\frac{11}{8},\frac{3}{2},\frac{13}{8},\frac{7}{4},\frac{17}{8},\frac{2u+2r+7}{2(r+4)}(1\leq u\leq r+4),\frac{2v+2s+7}{2(s+4)}(1\leq v \leq s+4)
	\end{displaymath}
\end{proposition}
\begin{proof}
	By \cite{MR1721628} and a simple computation through the Gr\"obner basis above, we find that
	\begin{displaymath}
		K = \{1,z,xy,xz,z^2,xyz,xyz^2;
		x^k(1\leq k \leq r+4);
		y^l(1 \leq l \leq s+4)\}
	\end{displaymath}
	is a monomial basis of $M$. All monomials not vanishing in $\mathcal O_3/J(f)$ are listed as follow:
	\begin{align*}
		W = \{
		& x,...,x^{r+5}, xy,xy^2,xy^3,xy^4, x^2y,x^2y^2, x^2y^3, x^3y,x^2y^2,x^4y, y,y^2,...,y^{s+5},\\
		& z,z^2,xz,x^2z,xz^2,x^2z^2,xyz, yz,y^2z,yz^2,y^2z^2,xyz^2\}.
	\end{align*}
	%		\begin{displaymath}
		%			W = \{& \frac{2k+r+4}{2(r+5)}(1\leq k\leq r+5),\\
		%			& \frac{2}{5},\frac{4}{5},\frac{6}{5},\frac{8}{5},\\
		%			& 1,\frac{r+6}{r+5},\frac{r+7}{r+5},\frac{r+8}{r+5},1,\frac{s+6}{s+5},\frac{s+7}{s+5},\frac{s+8}{s+5},\\
		%			& \frac{7}{5},\frac{7}{5},\frac{3r+16}{2(r+5)},\frac{3s+16}{2(s+5)},\\
		%			& \frac{2l+r+4}{2(s+5)}(1\leq l\leq s+6)\}.
		%		\end{displaymath}
	Let $m \gg 0$ be a sufficiently large integer. Since spectrum  is stable under $\mu$-constant deformation, and by \textbf{Corollary \ref{Corollary of Finite Determinacy Theorem}}, it suffices to consider the Newton diagram of $f+ z^m$. There are $6$ vertices $A(r+4,0,0),B(0,s+4,0),C(0,0,m),D(2,2,0),E(1,0,2),F(0,1,2)$, generating $4$ facets:
	\begin{align*}
		ADE: & \frac{X}{r+4}+\frac{(r+2)Y}{2(r+4)}+\frac{(r+3)Z}{2(r+4)} = 1\\
		DEB: & \frac{(s+2)X}{2(s+4)}+\frac{Y}{s+4}+\frac{(s+3)}{2(s+4)} = 1\\
		EFC: & \frac{(m-2)X}{m}+\frac{(m-2)Y}{m}+\frac{Z}{m} = 1\\
		DEF: & \frac{X}{4}+\frac{Y}{4}+\frac{3Z}{8} = 1
	\end{align*}
	The valuation $v$ of the Newton polyhegon of $f$ is defined by the above equations. Following the same method in curve singularities and applying our algorithm on the above valuation, we find $K$ is a maximal basis.
\end{proof}

Finally, we list the spectrum of surface singularities of modality $3$ as below.

\begin{table}[!htbp]
	\centering
	\caption{Spectrum of Surface Singularities of Modality 3}
	%\label{频率型、强度型和相对比型指标}也可以用这个
	\renewcommand\arraystretch{1.5}
	\begin{tabular}{lp{2cm}p{12cm}}
		\toprule  %添加表格头部粗线
		\textbf{Class}   &$\mu$         & Spectrum Number\\
		\midrule  %添加表格中横线
		$VA_{r,s}$ & $15+s+r$ & $\frac{7}{8},\frac{5}{4},\frac{11}{8},\frac{3}{2},\frac{13}{8},\frac{7}{4},\frac{17}{8},\frac{2u+2r+7}{2(r+4)}(1\leq u\leq r+4),\frac{2v+2s+7}{2(s+4)}(1\leq v \leq s+4)$ \\
		$VA_{2k,0}^\#$ & $15+2k$ & $\frac{7}{8},\frac{9}{8},\frac{11}{8},\frac{11}{8},\frac{13}{8},\frac{13}{8},\frac{15}{8},\frac{17}{8},\frac{l+2k+8}{2(k+4)}(1\leq l\leq 2k+7)$ \\
		$VA_{2k+1,0}^\#$ & $16+2k$ & $\frac{7}{8},\frac{9}{8},\frac{11}{8},\frac{11}{8},\frac{13}{8},\frac{13}{8},\frac{15}{8},\frac{17}{8},\frac{l+2k+9}{2k+9}(1\leq l \leq 2k+8)$ \\
		$VA_{2k,s}^\#$ & $15+2k+s$ & $\frac{7}{8},\frac{11}{8},\frac{13}{8},\frac{17}{8},\frac{u+2k+8}{2(k+4)}(1\leq u\leq 2k+7),\frac{2v+2s+7}{2(s+4)}(1\leq v\leq s+4)$ \\
		$VA_{2k+1,s}^\#$ & $16+2k+s$ & $\frac{7}{8},\frac{11}{8},\frac{13}{8},\frac{17}{8},\frac{u+2k+9}{2k+9}(1\leq u\leq 2k+8),\frac{2v+2s+7}{2(s+4)}(1\leq v\leq s+4)$ \\
		$VB_{(-1)}^s$ & $17+s$ & $\frac{6}{7},\frac{15}{14},\frac{17}{14},\frac{9}{7},\frac{19}{14},\frac{10}{7},\frac{3}{2},\frac{11}{7},\frac{23}{14},\frac{12}{7},\frac{25}{14},\frac{27}{14},\frac{15}{7},\frac{2l+2s+7}{2(s+4)}(1\leq l\leq s+4)$ \\
		$VB_{(0)}^s$ & $18+s$ & $\frac{17}{20},\frac{21}{20},\frac{6}{5},\frac{5}{4},\frac{27}{20},\frac{7}{5},\frac{29}{20},\frac{31}{20},\frac{8}{5},\frac{33}{20},\frac{7}{4},\frac{9}{5},\frac{39}{20},\frac{43}{20},\frac{2l+2s+7}{2(s+4)}(1\leq l \leq s+4)$ \\
		$VB_{(-1)}^{\#,2k}$ & $17+2k$ & $\frac{13}{15},\frac{16}{15},\frac{19}{15},\frac{4}{3},\frac{22}{15},\frac{23}{15},\frac{5}{3},\frac{26}{15},\frac{29}{15},\frac{32}{15},\frac{l+2k+8}{2(k+4)}(1\leq l\leq 2k+7)$\\
		$VB_{(-1)}^{\#,2k+1}$ & $18+2k$ & $\frac{13}{15},\frac{16}{15},\frac{19}{15},\frac{4}{3},\frac{22}{15},\frac{23}{15},\frac{5}{3},\frac{26}{15},\frac{29}{15},\frac{32}{15},\frac{l+2k+9}{2k+9}(1\leq k\leq 2k+8)$\\
		$VB_{(0)}^{\#,2k}$ & $18+2k$ & $\frac{19}{22},\frac{23}{22},\frac{27}{22},\frac{29}{22},\frac{31}{22},\frac{3}{2},\frac{35}{22},\frac{37}{22},\frac{39}{22},\frac{43}{22},\frac{47}{22},\frac{l+2k+8}{2(k+4)}(1\leq l\leq 2k+7)$\\
		$VB_{(0)}^{\#,2k+1}$ & $19+2k$ & $\frac{19}{22},\frac{23}{22},\frac{27}{22},\frac{29}{22},\frac{31}{22},\frac{3}{2},\frac{35}{22},\frac{37}{22},\frac{39}{22},\frac{43}{22},\frac{47}{22},\frac{l+2k+9}{2k+9}(1\leq l\leq 2k+8)$\\
		\midrule  %添加表格中横线
	\end{tabular}
\end{table}
~\\
\begin{table}[!htbp]
	\centering
	\renewcommand\arraystretch{1.5}
	\begin{tabular}{lp{2cm}p{12cm}}
		\midrule  %
		$VB_{(1)}^{\#,2k}$ & $19+2k$ & $\frac{31}{36},\frac{37}{36},\frac{43}{36},\frac{47}{36},\frac{49}{36},\frac{53}{36},\frac{55}{36},\frac{59}{36},\frac{61}{36},\frac{65}{36},\frac{71}{36},\frac{77}{36},\frac{l+2k+8}{2(k+4)}(1\leq l\leq 2k+7)$\\
		$VB_{(1)}^{\#,2k+1}$ & $20+2k$ & $\frac{31}{36},\frac{37}{36},\frac{43}{36},\frac{47}{36},\frac{49}{36},\frac{53}{36},\frac{55}{36},\frac{59}{36},\frac{61}{36},\frac{65}{36},\frac{71}{36},\frac{77}{36},\frac{l+2k+9}{2k+9}(1\leq l\leq 2k+8)$\\
		$VC_{18}$ & $18$ & $\frac{17}{20},\frac{21}{20},\frac{11}{10},\frac{13}{10},\frac{5}{4},\frac{5}{4},\frac{27}{20},\frac{29}{20},\frac{3}{2},\frac{3}{2},\frac{31}{20},\frac{33}{20},\frac{17}{10},\frac{7}{4},\frac{7}{4},\frac{19}{10},\frac{39}{20},\frac{43}{20}$\\
		$VC_{19}$ & $19$ & $\frac{27}{32},\frac{33}{32},\frac{35}{32},\frac{39}{32},\frac{5}{4},\frac{41}{32},\frac{43}{32},\frac{45}{32},\frac{47}{32},\frac{3}{2},\frac{49}{32},\frac{51}{32},\frac{53}{32},\frac{55}{32},\frac{7}{4},\frac{57}{32},\frac{61}{32},\frac{63}{32},\frac{69}{32}$\\
		$VC_{18}^\#$ & $18$ & $\frac{16}{19},\frac{20}{19},\frac{21}{19},\frac{23}{19},\frac{24}{19},\frac{25}{19},\frac{26}{19},\frac{27}{19},\frac{28}{19},\frac{29}{19},\frac{30}{19},\frac{31}{19},\frac{32}{19},\frac{33}{19},\frac{34}{19},\frac{36}{19},\frac{37}{19},\frac{41}{19}$\\
		$VC_{19}^\#$ & $19$ & $\frac{5}{6},\frac{31}{30},\frac{11}{10},\frac{6}{5},\frac{37}{30},\frac{13}{10},\frac{41}{30},\frac{7}{5},\frac{43}{30},\frac{3}{2},\frac{47}{30},\frac{8}{5},\frac{49}{30},\frac{17}{10},\frac{53}{30},\frac{9}{5},\frac{19}{10},\frac{59}{30},\frac{13}{6}$\\
		$VF_{19}$ & $19$ & $\frac{5}{6},\frac{25}{24},\frac{13}{12},\frac{29}{24},\frac{5}{4},\frac{31}{24},\frac{4}{3},\frac{17}{12},\frac{35}{24},\frac{3}{2},\frac{37}{24},\frac{19}{12},\frac{5}{3},\frac{41}{24},\frac{7}{4},\frac{43}{24},\frac{23}{12},\frac{47}{24},\frac{13}{6}$\\
		$VF_{20}$ & $20$ & $\frac{33}{40},\frac{41}{40},\frac{43}{40},\frac{6}{5},\frac{49}{40},\frac{51}{40},\frac{53}{40},\frac{7}{5},\frac{57}{40},\frac{59}{40},\frac{61}{40},\frac{63}{40},\frac{8}{5},\frac{67}{40},\frac{69}{40},\frac{71}{40},\frac{9}{5},\frac{77}{40},\frac{79}{40},\frac{87}{40}$\\
		\bottomrule %添加表格底部粗线
	\end{tabular}
\end{table}

%As a simple corollary, we find spectrum is a complete invariant for singularities of modality $\leq 3$ by a direct observation. For spectrum of singularities of modality $0,1,2$, we refer to \cite{MR0631501}.

%\begin{theorem}
%	Spectrum is a complete invariant for singularities of modality $\leq 3$.
%\end{theorem}

%So far one can easily see \textbf{Conjecture \ref{range conjecture}} of K. Saito holds for singularities of modality $3$. We write it down as below.
%
%\begin{theorem}\label{K. Saito's range conjecture}
%	For all singularities of modality $3$, the ranges of their spectra are all greater than $1$. That is, \textbf{Conjecture \ref{range conjecture}} holds for them.
%\end{theorem}

\subsection{Hertling Conjecture for Singularities of Modality $3$}

\begin{theorem}\label{main theorem A}
	Hertling Conjecture holds for singularities of modality $3$.
\end{theorem}
\begin{proof}
	By \textbf{Theorem \ref{Hertling Conjecture for curve singularities}}, we only need to consider surface cases. Furthermore, by \textbf{Theorem \ref{Varchenko's constant deformation}} and \textbf{Theorem \ref{Hertlin Conjecture for quasihomogeneous singularities}}, exceptional cases, say $VC,VF$ series satisfies Hertling Conjecture. Hence, it suffices to check the conjecture for $VA,VB$ series. We still use $f_0$ to refer to the representative germ for each type.
	
	\noindent\textbf{(1) $VA_{r,s}$.}
	\begin{align*}
		& \mu\cdot \mathrm{Var}[\mathrm{Sp}(f_0)] \\
		= & (\frac{7}{8}-\frac{3}{2})^2+(\frac{5}{4}-\frac{3}{2})^2+(\frac{11}{8}-\frac{3}{2})^2+(\frac{3}{2}-\frac{3}{2})^2+(\frac{13}{8}-\frac{3}{2})^2+(\frac{7}{4}-\frac{3}{2})^2+(\frac{17}{8}-\frac{3}{2})^2\\
		& + \sum_{u = 1}^{r+4} (\frac{2u+2r+7}{2(r+4)}-\frac{3}{2})^2+\sum_{v = 1}^{s+4}(\frac{2v+2s+7}{2(s+4)}-\frac{3}{2})^2 \\
		= & \frac{15}{16}+\frac{(r+4)^2-1}{12(r+4)}+\frac{(s+4)^2-1}{12(s+4)}.
	\end{align*}
	
	The range of $\mathrm{Sp}(f_0)$ is $\frac{5}{4}$. Let $a = r+4,b+4$, then it suffices to prove
	\begin{displaymath}
		\frac{15}{16}+\frac{a^2-1}{12a}+\frac{b^2-1}{12b} \leq \frac{5}{48}(7+a+b).
	\end{displaymath}
	Equivalently,
	\begin{displaymath}
		\frac{1}{48}(a+b)+\frac{1}{12a}+\frac{1}{12b} \geq \frac{5}{24}.
	\end{displaymath}
	We have $LHS \geq \frac{1}{48}(5+5)+\frac{1}{60}+\frac{1}{60} \geq \frac{5}{24}$. Done.

	\noindent\textbf{(2) $VA_{2k,0}^\#$.}
	\begin{align*}
		& \mu\cdot \mathrm{Var}[\mathrm{Sp}(f_0)] \\
		= & (\frac{7}{8}-\frac{3}{2})^2+(\frac{9}{8}-\frac{3}{2})^2+(\frac{11}{8}-\frac{3}{2})^2+(\frac{11}{8}-\frac{3}{2})^2+(\frac{3}{2}-\frac{3}{2})^2+(\frac{13}{8}-\frac{3}{2})^2+(\frac{13}{8}-\frac{3}{2})^2\\
		& +(\frac{7}{4}-\frac{3}{2})^2+(\frac{17}{8}-\frac{3}{2})^2 + \sum_{l=1}^{2k+7}(\frac{l+2k+8}{2(k+4)}-\frac{3}{2})^2\\
		= & \frac{9}{8}+\frac{(2k+7)(k+3)}{12(k+4)}.
	\end{align*}
	
	The range of $\mathrm{Sp}(f_0)$ is $\frac{5}{4}$. It suffices to prove
	\begin{displaymath}
		\frac{9}{8}+\frac{(2k+7)(k+3)}{12(k+4)} \leq \frac{5}{48}(15+2k).
	\end{displaymath}
	Equivalently,
	\begin{displaymath}
		2k^2+9k\geq 0,
	\end{displaymath}
	which is obvious.

	\noindent\textbf{(3) $VA_{2k+1,0}^\#$.}
	\begin{align*}
		& \mu\cdot \mathrm{Var}[\mathrm{Sp}(f_0)] \\
		= & (\frac{7}{8}-\frac{3}{2})^2+(\frac{9}{8}-\frac{3}{2})^2+(\frac{11}{8}-\frac{3}{2})^2+(\frac{11}{8}-\frac{3}{2})^2+(\frac{3}{2}-\frac{3}{2})^2+(\frac{13}{8}-\frac{3}{2})^2+(\frac{13}{8}-\frac{3}{2})^2\\
		& +(\frac{7}{4}-\frac{3}{2})^2+(\frac{17}{8}-\frac{3}{2})^2 + \sum_{l=1}^{2k+8}(\frac{l+2k+9}{2k+9}-\frac{3}{2})^2\\
		= & \frac{9}{8}+\frac{(2k+7)(k+4)}{6(2k+9)}.
	\end{align*}
	
	The range of $\mathrm{Sp}(f_0)$ is $\frac{5}{4}$. It suffices to prove
	\begin{displaymath}
		\frac{9}{8}+\frac{(2k+7)(k+4)}{6(2k+9)} \leq \frac{5}{48}(16+2k).
	\end{displaymath}
	Equivalently,
	\begin{displaymath}
		2k^2+11k+5\geq 0,
	\end{displaymath}
	which is obvious.

	\noindent\textbf{(4) $VA_{2k,s}^\#$.}
	\begin{align*}
		& \mu\cdot \mathrm{Var}[\mathrm{Sp}(f_0)] \\
		= & (\frac{7}{8}-\frac{3}{2})^2+(\frac{11}{8}-\frac{3}{2})^2+(\frac{13}{8}-\frac{3}{2})^2+(\frac{17}{8}-\frac{3}{2})^2\\
		& + \sum_{u=1}^{2k+7}(\frac{u+2k+8}{2(k+4)}-\frac{3}{2})^2+\sum_{v = 1}^{s+4}(\frac{2v+2s+7}{2(s+4)}-\frac{3}{2})^2 \\
		= & \frac{13}{16}+\frac{(2k+7)(k+3)}{12(k+4)}+\frac{(s+4)^2-1}{12(s+4)}.
	\end{align*}
	
	The range of $\mathrm{Sp}(f_0)$ is $\frac{5}{4}$. Let $a = k+4, b = s+4$, then it suffices to prove
	\begin{displaymath}
		\frac{13}{16}+\frac{(2a-1)(a-1)}{12a}+\frac{b^2-1}{12b} \leq \frac{5}{48}(3+2a+b).
	\end{displaymath}
	Equivalently,
	\begin{displaymath}
		\frac{1}{48}b+\frac{1}{12b}+\frac{a^2+6a-2}{24a} \geq \frac{1}{2}.
	\end{displaymath}
	We have $LHS \geq \frac{5}{48}+\frac{1}{60}+\frac{53}{120} \geq \frac{1}{2}$. Done.

	\noindent\textbf{(5) $VA_{2k+1,s}^\#$.}
	\begin{align*}
		& \mu\cdot \mathrm{Var}[\mathrm{Sp}(f_0)] \\
		= & (\frac{7}{8}-\frac{3}{2})^2+(\frac{11}{8}-\frac{3}{2})^2+(\frac{13}{8}-\frac{3}{2})^2+(\frac{17}{8}-\frac{3}{2})^2\\
		& + \sum_{u=1}^{2k+8}(\frac{u+2k+9}{2k+9}-\frac{3}{2})^2+\sum_{v = 1}^{s+4}(\frac{2v+2s+7}{2(s+4)}-\frac{3}{2})^2 \\
		= & \frac{13}{16}+\frac{(2k+7)(k+4)}{6(2k+9)}+\frac{(s+4)^2-1}{12(s+4)}.
	\end{align*}
	
	The range of $\mathrm{Sp}(f_0)$ is $\frac{5}{4}$. Let $a = k+4, b = s+4$, then it suffices to prove
	\begin{displaymath}
		\frac{13}{16}+\frac{(2a-1)a}{6(2a+1)}+\frac{b^2-1}{12b} \leq \frac{5}{48}(4+2a+b).
	\end{displaymath}
	Equivalently,
	\begin{displaymath}
		\frac{1}{48}b+\frac{1}{12b}+\frac{a(2a+9)}{24(2a+1)} \geq \frac{19}{48}.
	\end{displaymath}
	We have $LHS \geq \frac{5}{48}+\frac{1}{60}+\frac{68}{216} \geq \frac{19}{48}$. Done.

	\noindent\textbf{(6) $VB_{(-1)}^s$.}
	\begin{align*}
		& \mu\cdot \mathrm{Var}[\mathrm{Sp}(f_0)] \\
		= & (\frac{6}{7}-\frac{3}{2})^2+(\frac{13}{14}-\frac{3}{2})^2+(\frac{17}{14}-\frac{3}{2})^2+(\frac{9}{7}-\frac{3}{2})^2+(\frac{19}{14}-\frac{3}{2})^2+(\frac{10}{7}-\frac{3}{2})^2+(\frac{3}{2}-\frac{3}{2})^2\\
		&+(\frac{11}{7}-\frac{3}{2})^2 +(\frac{23}{14}-\frac{3}{2})^2 +(\frac{12}{7}-\frac{3}{2})^2+(\frac{25}{14}-\frac{3}{2})^2+(\frac{29}{14}-\frac{3}{2})^2\\
		&+(\frac{15}{7}-\frac{3}{2})^2+ \sum_{l = 1}^{s+4}(\frac{2l+2s+7}{2(s+4)}-\frac{3}{2})^2 \\
		= & \frac{3}{2}+\frac{(s+4)^2-1}{12(s+4)}.
	\end{align*}
	
	The range of $\mathrm{Sp}(f_0)$ is $\frac{9}{7}$. Let $b = s+4$, then it suffices to prove
	\begin{displaymath}
		\frac{3}{2}+\frac{b^2-1}{12b} \leq \frac{9}{84}(13+b).
	\end{displaymath}
	Equivalently,
	\begin{displaymath}
		2b+\frac{7}{b}- 9\geq 0.
	\end{displaymath}
	Since $b\geq 5$, we have $LHS \geq 10+\frac{7}{5} -9 \geq 0$. Done.

	\noindent\textbf{(7) $VB_{(0)}^s$.}
	\begin{align*}
		& \mu\cdot \mathrm{Var}[\mathrm{Sp}(f_0)] \\
		= & (\frac{17}{20}-\frac{3}{2})^2+(\frac{21}{20}-\frac{3}{2})^2+(\frac{6}{5}-\frac{3}{2})^2+(\frac{5}{4}-\frac{3}{2})^2+(\frac{27}{20}-\frac{3}{2})^2+(\frac{7}{5}-\frac{3}{2})^2\\
		&+(\frac{29}{20}-\frac{3}{2})^2+(\frac{31}{20}-\frac{3}{2})^2 +(\frac{8}{5}-\frac{3}{2})^2+(\frac{33}{20}-\frac{3}{2})^2+(\frac{7}{4}-\frac{3}{2})^2+(\frac{9}{5}-\frac{3}{2})^2\\
		&+(\frac{39}{20}-\frac{3}{2})^2+(\frac{43}{20}-\frac{3}{2})^2+ \sum_{l = 1}^{s+4}(\frac{2l+2s+7}{2(s+4)}-\frac{3}{2})^2 \\
		= & \frac{13}{8}+\frac{(s+4)^2-1}{12(s+4)}.
	\end{align*}

	The range of $\mathrm{Sp}(f_0)$ is $\frac{13}{10}$. Let $b = s+4$, then it suffices to prove
	\begin{displaymath}
		\frac{13}{8}+\frac{b^2-1}{12b} \leq \frac{13}{120}(14+b).
	\end{displaymath}
	Equivalently,
	\begin{displaymath}
		3b+\frac{15}{b}- 13\geq 0.
	\end{displaymath}
	Since $b\geq 5$, we have $LHS \geq 15+\frac{15}{5} -13 \geq 0$. Done.

	\noindent\textbf{(8) $VB_{(-1)}^{\#,2k}$.}
	\begin{align*}
		& \mu\cdot \mathrm{Var}[\mathrm{Sp}(f_0)] \\
		= & (\frac{13}{15}-\frac{3}{2})^2+(\frac{16}{15}-\frac{3}{2})^2+(\frac{19}{15}-\frac{3}{2})^2+(\frac{4}{3}-\frac{3}{2})^2+(\frac{22}{15}-\frac{3}{2})^2+(\frac{23}{15}-\frac{3}{2})^2+(\frac{5}{3}-\frac{3}{2})^2\\
		& +(\frac{26}{15}-\frac{3}{2})^2+(\frac{26}{15}-\frac{3}{2})^2+(\frac{29}{15}-\frac{3}{2})^2+(\frac{32}{15}-\frac{3}{2})^2+ \sum_{l=1}^{2k+7}(\frac{l+2k+8}{2(k+4)}-\frac{3}{2})^2 \\
		= & \frac{121}{90}+\frac{(2k+7)(k+3)}{12(k+4)}.
	\end{align*}

	The range of $\mathrm{Sp}(f_0)$ is $\frac{19}{15}$. Let $a = k+4$, then it suffices to prove
	\begin{displaymath}
		\frac{121}{90}+\frac{(2a-1)(a-1)}{12a} \leq \frac{19}{180}(9+2a).
	\end{displaymath}
	Equivalently,
	\begin{displaymath}
		8a-\frac{15}{a}-26\geq 0.
	\end{displaymath}
	Since $a\geq 5$, we have $LHS \geq 40-\frac{15}{5} -26 \geq 0$. Done.

	\noindent\textbf{(9) $VB_{(-1)}^{\#,2k+1}$.}
	\begin{align*}
		& \mu\cdot \mathrm{Var}[\mathrm{Sp}(f_0)] \\
		= & (\frac{13}{15}-\frac{3}{2})^2+(\frac{16}{15}-\frac{3}{2})^2+(\frac{19}{15}-\frac{3}{2})^2+(\frac{4}{3}-\frac{3}{2})^2+(\frac{22}{15}-\frac{3}{2})^2+(\frac{23}{15}-\frac{3}{2})^2+(\frac{5}{3}-\frac{3}{2})^2\\
		& +(\frac{26}{15}-\frac{3}{2})^2+(\frac{26}{15}-\frac{3}{2})^2+(\frac{29}{15}-\frac{3}{2})^2+(\frac{32}{15}-\frac{3}{2})^2+ \sum_{l=1}^{2k+8}(\frac{l+2k+9}{2k+9}-\frac{3}{2})^2 \\
		= & \frac{121}{90}+\frac{(2k+7)(k+4)}{6(2k+9)}.
	\end{align*}

	The range of $\mathrm{Sp}(f_0)$ is $\frac{19}{15}$. Let $a = k+4$, then it suffices to prove
	\begin{displaymath}
		\frac{121}{90}+\frac{(2a-1)a}{6(2a+1)} \leq \frac{19}{180}(10+2a).
	\end{displaymath}
	Equivalently,
	\begin{displaymath}
		\frac{a(4a+17)}{45(2a+1)} \geq \frac{13}{45}.
	\end{displaymath}
	Since $a\geq 4$, we have $LHS \geq \frac{4\cdot 33}{45\cdot 9} \geq \frac{13}{45}$. Done.

	\noindent\textbf{(10) $VB_{(0)}^{\#,2k}$.}
	\begin{align*}
		& \mu\cdot \mathrm{Var}[\mathrm{Sp}(f_0)] \\
		= & (\frac{19}{22}-\frac{3}{2})^2+(\frac{23}{22}-\frac{3}{2})^2+(\frac{27}{22}-\frac{3}{2})^2+(\frac{29}{22}-\frac{3}{2})^2+(\frac{31}{22}-\frac{3}{2})^2+(\frac{3}{2}-\frac{3}{2})^2+(\frac{35}{22}-\frac{3}{2})^2\\
		& +(\frac{37}{22}-\frac{3}{2})^2+(\frac{39}{22}-\frac{3}{2})^2+(\frac{43}{22}-\frac{3}{2})^2+(\frac{47}{22}-\frac{3}{2})^2+ \sum_{l=1}^{2k+7}(\frac{l+2k+8}{2(k+4)}-\frac{3}{2})^2 \\
		= & \frac{16}{11}+\frac{(2k+7)(k+3)}{12(k+4)}.
	\end{align*}

	The range of $\mathrm{Sp}(f_0)$ is $\frac{14}{11}$. Let $a = k+4$, then it suffices to prove
	\begin{displaymath}
		\frac{16}{11}+\frac{(2a-1)(a-1)}{12a} \leq \frac{7}{66}(10+2a).
	\end{displaymath}
	Equivalently,
	\begin{displaymath}
		\frac{1}{22}a+\frac{1}{4}-\frac{1}{12a} \geq \frac{13}{33}.
	\end{displaymath}
	Since $a\geq 5$, we have $LHS \geq \frac{5}{22}+\frac{1}{4}-\frac{1}{60} \geq \frac{13}{33}$. Done.

	\noindent\textbf{(11) $VB_{(-1)}^{\#,2k+1}$.}
	\begin{align*}
		& \mu\cdot \mathrm{Var}[\mathrm{Sp}(f_0)] \\
		= & (\frac{19}{22}-\frac{3}{2})^2+(\frac{23}{22}-\frac{3}{2})^2+(\frac{27}{22}-\frac{3}{2})^2+(\frac{29}{22}-\frac{3}{2})^2+(\frac{31}{22}-\frac{3}{2})^2+(\frac{3}{2}-\frac{3}{2})^2+(\frac{35}{22}-\frac{3}{2})^2\\
		& +(\frac{37}{22}-\frac{3}{2})^2+(\frac{39}{22}-\frac{3}{2})^2+(\frac{43}{22}-\frac{3}{2})^2+(\frac{47}{22}-\frac{3}{2})^2+ \sum_{l=1}^{2k+8}(\frac{l+2k+9}{2k+9}-\frac{3}{2})^2 \\
		= & \frac{16}{11}+\frac{(2k+7)(k+4)}{6(2k+9)}.
	\end{align*}

	The range of $\mathrm{Sp}(f_0)$ is $\frac{14}{11}$. Let $a = k+4$, then it suffices to prove
	\begin{displaymath}
		\frac{16}{11}+\frac{(2a-1)a}{6(2a+1)} \leq \frac{7}{66}(11+2a).
	\end{displaymath}
	Equivalently,
	\begin{displaymath}
		\frac{a(6a+25)}{66(2a+1)} \geq \frac{19}{66}.
	\end{displaymath}
	Since $a\geq 4$, we have $LHS \geq \frac{4\cdot 49}{66\cdot 9} \geq \frac{19}{66}$. Done.

	\noindent\textbf{(12) $VB_{(1)}^{\#,2k}$.}
	\begin{align*}
		& \mu\cdot \mathrm{Var}[\mathrm{Sp}(f_0)] \\
		= & (\frac{31}{36}-\frac{3}{2})^2+(\frac{37}{36}-\frac{3}{2})^2+(\frac{43}{36}-\frac{3}{2})^2+(\frac{47}{36}-\frac{3}{2})^2+(\frac{49}{36}-\frac{3}{2})^2+(\frac{53}{36}-\frac{3}{2})^2+(\frac{55}{36}-\frac{3}{2})^2\\
		& +(\frac{59}{36}-\frac{3}{2})^2+(\frac{61}{36}-\frac{3}{2})^2+(\frac{65}{36}-\frac{3}{2})^2+(\frac{71}{36}-\frac{3}{2})^2+(\frac{77}{36}-\frac{3}{2})^2+ \sum_{l=1}^{2k+7}(\frac{l+2k+8}{2(k+4)}-\frac{3}{2})^2 \\
		= & \frac{169}{108}+\frac{(2k+7)(k+3)}{12(k+4)}.
	\end{align*}

	The range of $\mathrm{Sp}(f_0)$ is $\frac{23}{18}$. Let $a = k+4$, then it suffices to prove
	\begin{displaymath}
		\frac{169}{108}+\frac{(2a-1)(a-1)}{12a} \leq \frac{23}{216}(11+2a).
	\end{displaymath}
	Equivalently,
	\begin{displaymath}
		\frac{5}{108}a+\frac{1}{4}-\frac{1}{12a} \geq \frac{85}{216}.
	\end{displaymath}
	Since $a\geq 5$, we have $LHS \geq \frac{25}{108}+\frac{1}{4}-\frac{1}{60} \geq \frac{85}{216}$. Done.

	\noindent\textbf{(13) $VB_{(1)}^{\#,2k+1}$.}
	\begin{align*}
		& \mu\cdot \mathrm{Var}[\mathrm{Sp}(f_0)] \\
		= & (\frac{31}{36}-\frac{3}{2})^2+(\frac{37}{36}-\frac{3}{2})^2+(\frac{43}{36}-\frac{3}{2})^2+(\frac{47}{36}-\frac{3}{2})^2+(\frac{49}{36}-\frac{3}{2})^2+(\frac{53}{36}-\frac{3}{2})^2+(\frac{55}{36}-\frac{3}{2})^2\\
		& +(\frac{59}{36}-\frac{3}{2})^2+(\frac{61}{36}-\frac{3}{2})^2+(\frac{65}{36}-\frac{3}{2})^2+(\frac{71}{36}-\frac{3}{2})^2+(\frac{77}{36}-\frac{3}{2})^2+\sum_{l=1}^{2k+8}(\frac{l+2k+9}{2k+9}-\frac{3}{2})^2 \\
		= &  \frac{169}{108}+\frac{(2k+7)(k+4)}{6(2k+9)}.
	\end{align*}
	
	\begin{align*}
		& \tau \cdot \mathrm{Var}[\mathrm{Sp}^\tau(f_0)] \\
		= & (\frac{31}{36}-\frac{3}{2})^2+(\frac{37}{36}-\frac{3}{2})^2+(\frac{43}{36}-\frac{3}{2})^2+(\frac{47}{36}-\frac{3}{2})^2+(\frac{49}{36}-\frac{3}{2})^2+(\frac{53}{36}-\frac{3}{2})^2+(\frac{55}{36}-\frac{3}{2})^2\\
		& +(\frac{59}{36}-\frac{3}{2})^2+(\frac{61}{36}-\frac{3}{2})^2+(\frac{65}{36}-\frac{3}{2})^2+(\frac{71}{36}-\frac{3}{2})^2+\sum_{l=1}^{2k+8}(\frac{l+2k+9}{2k+9}-\frac{3}{2})^2 \\
		= &  \frac{169}{108}+\frac{(2k+7)(k+4)}{6(2k+9)}.
	\end{align*}
	The range of $\mathrm{Sp}(f_0)$ is $\frac{23}{18}$. Let $a = k+4$, then it suffices to prove
	\begin{displaymath}
		\frac{169}{108}+\frac{(2a-1)a}{6(2a+1)} \leq \frac{23}{216}(12+2a).
	\end{displaymath}
	Equivalently,
	\begin{displaymath}
		\frac{a(10a+41)}{108(2a+1)} \geq \frac{31}{108}.
	\end{displaymath}
	Since $a\geq 4$, we have $LHS \geq \frac{4\cdot 81}{108\cdot 9} \geq \frac{31}{108}$. Done.
\end{proof}

\section{Propoerties and Generalized Hertling Conjecture for Tjurina Spectrum}

As before, let $f\in \mathcal O_{n+1}$ define an isolated singularity at the origin. Let $\mathrm{sp}(f) = \{\alpha_1\leq ... \leq \alpha_\mu\}$ denote the collection of its spectrum numbers  and $\mathrm{sp}^{\tau}(f) = \{\beta_1,...,\beta_\tau\}$ denote its Tjurina spectrum (both counting multiplicity). By \textbf{Corollary \ref{Tjurina spectrum is a subsepctrum}}, $\mathrm{sp}^\tau(f)$ is a subset of $\mathrm{sp}(f)$. $I = \{i_1 < ... < i_{\mu-\tau}\} \subset \{1,...,n\}$ are those $k$ such that $\mathrm{sp}(f)\setminus \mathrm{sp}^\tau(f) = \{\alpha_k\mid k\in I\} =: \mathrm{R}(f)$.

In the first subsection, we give a computation method for $\mathrm{sp}^\tau(f)$. In the second subsection, we give an estimation for $\mathrm{R}(f)$ and show that $\alpha_\mu \not \in \mathrm{R}(f)$.

%The following is the generalzied Hertling conjecture for Tjurina spectrum.
%\begin{conjecture}[GHCTS]\label{GHCTS}
%	Let $f\in \mathcal O_{n+1}$ be a germ which defines an isolated singularity at the origin, with Tjurina number $\tau$. Suppose $\alpha_1\leq ...\leq \alpha_\tau$ are exponents of Tjurina spectrum of $f$, then the following inequality holds:
%	\begin{displaymath}
	%		\frac{1}{\tau}\sum_{i = 1}^{\mu} (\alpha_i-\bar \alpha)^2 \leq \frac{\alpha_\tau-\alpha_1}{12},
	%	\end{displaymath}
%	where $\bar \alpha$ is the average of $\alpha_1,...,\alpha_\tau$ i.e. $\bar \alpha = \frac{1}{\tau}\sum_{j=1}^\tau \alpha_j$.
%\end{conjecture}

\subsection{Computation for Tjurina Spectrum}

\label{Computation for Tjurina spectrum}
By \textbf{Corollary \ref{Tjurina spectrum is a subsepctrum}}, we can compute Tjurina spectrum through $V$-filtration. Due to this result, we programed the following code in SINGULAR.

\begin{verbatim}
	LIB"gmssing.lib";
	LIB"sing.lib";
	ring r = 0, (x,y), ds;
	poly f = x22+x3y2+y7;
	list v = vfilt(f);
	ideal I = v[4];
	list sp = spectrum(f);
	list vf = v[3];
	int m = size(sp[2]);
	int mu = milnor(f);
	intvec t;
	ideal J = f,jacob(f);
	int a;
	int j;
	module temp;
	for(a = m ; a >= 1 ; a = a-1){
		temp = vf[a];
		temp = std(temp);
		for(j = 1 ; j <= mu ; j = j+1){
			if(reduce(gen(j),temp) == 0){
				J = J,I[j];
			}
		}
		J = std(J);
		t[a] = vdim(J);
	}
	int tau = tjurina(f);
	intvec w;
	w[m] = tau-t[m];
	for(a = m-1 ; a >= 1 ; a = a-1){
		w[a] = t[a+1]-t[a];
	}
	w;
	spectrum(f);
	mu-tau;
\end{verbatim}

Moreover, if $f$ is Newton non-degenerate (see \textbf{Subsection \ref{Newton Filtration}}), $V^{\bullet}\Omega_f$ also coincides with the shifted Newton filtration $N_\bullet M_f$. Therefore, one can also apply \textbf{Algorithm \ref{algorithm for a maximal basis}}, merely by replacing $J(f)$ with $(f,J(f))$.

As examples, we present the compution results of $\mathrm{R}(f)$ for normal forms when modality $f$ is no greater than $3$. One can recover $\mathrm{sp}^\tau(f)$ from $\mathrm{R}(f)$ and $\mathrm{sp}(f)$ (for modality $\leq 2$, we refer to \cite{MR3146513} and \cite{MR0631501}). In the following tables, we mark a $(!)$ behind a spectrum number to mean that it is the maximal one among $\mathrm{sp}(f)$. 

\begin{remark}\label{a extra remark}
	We point out that among all singularities of modality no greater than $3$, there are four exceptional cases which are Newton degenerate. Their representatives are $(x^2+y^3)^2+xy^{4+q}$($W^\sharp_{1,2q-1}$), $(x^2+y^3)^2+x^2y^{3+q}$($W^\sharp_{1,2q}$), $x^3+xz^2+xy^3+y^{1+q}z^2$($U_{1,2q-1}$) and $x^3+xz^2+xy^3+y^{3+q}z$($U_{1,2q}$). So we have to compute their Tjurina spectrum exceptionally. However, there seems no effective methods to compute $\mathrm{R}(f)$ for them. But through \textbf{Lemma \ref{f induces fitration lemma}} in the next subsection, we can easily show each of their largest expoent is in $\mathrm{R}(f)$ respectively. Also, we conjecture that $\mathrm{R}(f)$ is exactly the largest and the second larget exponents.
\end{remark}

\begin{conjecture}\label{Conjecture for four exceptional}
	For four exceptional cases $W^\sharp_{1,2q-1},W^\sharp_{1,2q},U_{1,2q-1}$ and $U_{1,2q}$, $\mathrm{R}(f)$ is exactly the largest and the second larget exponents.
\end{conjecture}
\begin{table}[!htbp]
	\centering
	\caption{Spectrum of Singularities of Modality 1}
	%\label{频率型、强度型和相对比型指标}也可以用这个
	\renewcommand\arraystretch{1.5}
	\begin{tabular}{lp{2cm}p{2cm}p{6cm}}
		\toprule  %添加表格头部粗线
		\textbf{Class}   &$\tau$         & $ \mu-\tau $       & \textbf{R}(f)  \\
		\midrule  %添加表格中横线
		$ P_8 $            & 8            & 0               &      $ \varnothing$                      \\
		$ X_9 $            &  9           & 0               &      $ \varnothing$                      \\
		$ J_{10}$          & 10           & 0               & $ \varnothing$                        \\  
		$ Q_{10}$          & 9            & 1               & $\frac{49}{24} (!) $                        \\
		$ Q_{11}$          & 10           & 1               & $\frac{37}{18} (!)$                      \\
		$ Q_{12}$          & 11           & 1               & $\frac{31}{15} (!) $                            \\
		$ S_{11}$          & 10           & 1               & $\frac{33}{16} (!) $                     \\
		$ S_{12}$          & 11           & 1               & $\frac{27}{13} (!)  $                     \\
		$ U_{12}$          & 11           & 1               & $\frac{25}{12} (!)  $                     \\
		$ Z_{11}$          & 10           & 1               & $\frac{23}{15} (!) $                     \\
		$ Z_{12}$          & 11           & 1               & $\frac{17}{11} (!) $                    \\
		$ Z_{13}$          & 12           & 1               & $\frac{14}{9}   (!) $                   \\
		$ W_{12}$          & 8            & 0               &  $ \varnothing      $             \\
		$ W_{13}$          & 12           & 1               & $\frac{25}{16} (!) $                      \\
		$ E_{12}$          & 11           & 1               & $\frac{32}{21} (!) $                     \\
		% 						\midrule  %添加表格中横线
		% 			\end{tabular}
	% 	\end{table}
% % ~\\
% \begin{table}[!htbp]
	% 		\centering
	% 		\renewcommand\arraystretch{1.5}
	% 		\begin{tabular}{lp{2cm}p{2cm}p{6cm}}
		% \midrule  %
		$ E_{13}$          & 12           & 1               & $\frac{23}{15} (!) $                     \\
		$ E_{14}$          & 13           & 1               & $\frac{37}{24} (!)  $                     \\
		$ T_{p,q,r}$       & $p+q+r-2$    & 1               & $ 2 (!)$                           \\
		\bottomrule %添加表格底部粗线
	\end{tabular}
\end{table}

% \vskip 20pt %加空行调整图文位置
% ~\\%加空行

\begin{table}[!htbp]
	\centering
	\caption{Spectrum of Curve Singularities of Modality 2}
	\renewcommand\arraystretch{1.5}
	\begin{tabular}{lp{2cm}p{2cm}p{6cm}}
		\toprule  %添加表格头部粗线
		\textbf{Class}                   &$\tau$              & $\mu- \tau $       & \textbf{R}(f)   \\
		\midrule  %添加表格中横线
		$ J_{3,0} $                      &  15                & 1              & $ \frac{37}{18} (!)    $                           \\
		$ J_{3,p} $                      &  14+ $p$           & 2              & $ \frac{37}{18} (!), 2 - \frac{1}{2(p+9)} $      \\               
		$ Z_{1,0} $                      &  14                & 1              & $ \frac{29}{14} (!) $  \\
		$ Z_{1,p} $                      &  13 + $p$          & 2              & $ \frac{29}{14} (!), 2 - \frac{1}{2(p+7)} $       \\
		$ W_{1,0} $                      &  14                & 1              & $ \frac{25}{12} (!) $                       \\
		$ W_{1,p} $                      & 13 + $p$           & 2              & $ \frac{25}{12} (!), 2 - \frac{1}{2(p+6)} $              \\
		$ W_{1,2q-1} ^\sharp $           &                    &                &                            \\
		$ W_{1,2q} ^\sharp $             &                    &                &                            \\
		$ Q_{2,0} $                      & 13                 & 1              & $ \frac{25}{12} (!) $                          \\
		$ Q_{2,p} $                      & 12 + $p$           & 2              & $ \frac{25}{12} (!), 2 - \frac{1}{2(p+6)} $                          \\
		$ S_{1,0} $                      & 13                 & 1              & $ \frac{21}{10} (!) $                          \\
		$ S_{1,p} $                      & 12 + $p$           & 2              & $ \frac{21}{10} (!), 2 - \frac{1}{2(p+5)} $                           \\
		$ S_{1,p} ^\sharp $         	 & 12 + $p$           & 2              & $ \frac{21}{10} (!), 2 - \frac{1}{p+10} $                          \\
		% $ S_{1,2q-1} ^\sharp $         & 12 + $2q-1$           & 2               &  $ \frac{12}{10} (!), \frac{1}{10}+\frac{2q+7}{2q+9}$                          \\
		% $ S_{1,2q} ^\sharp $           & 12 + $2q$             & 2               &  $ \frac{12}{10} (!), \frac{1}{10} + \frac{q+4}{q+5}$                          \\
		$ U_{1,0} $                      & 13                 & 1              & $ \frac{19}{9} (!)$                          \\
		$ U_{1,p} $                      & 12 + $ p $         & 2              & $ \frac{19}{9} (!), 2 - \frac{1}{2q+9}$                          \\
		% $ U_{1,2q-1} $                 & 12 + $ 2q-1 $                  & 2               &  $ \frac{11}{9} (!), \frac{1}{9}+\frac{q+3}{q+4}$                          \\
		% $ U_{1,2q} $                   & 12 + $ 2q  $                  & 2               &  $ \frac{11}{9} (!), \frac{1}{9} +\frac{2q+7}{2q+9}$                          \\
		$ E_{18} $          & 16               & 2                 & $\frac{34}{30} (!), \frac{28}{30}$                          \\
		$ E_{19} $          & 17               & 2                 & $ \frac{8}{7} (!), \frac{20}{21}$                           \\
		$ E_{20} $          & 18               & 2                 & $ \frac{38}{33} (!), \frac{32}{33}$                           \\      
		$ Z_{17} $          & 15               & 2                 & $ \frac{14}{12} (!), \frac{11}{12}$                           \\
		$ Z_{18} $          & 16               & 2                 & $ \frac{20}{17} (!), \frac{16}{17}$                           \\
		$ Z_{19} $          & 17               & 2                 & $ \frac{32}{27} (!), \frac{26}{27}$                           \\
		$ W_{17} $          & 15               & 2                 & $ \frac{6}{5} (!), \frac{19}{20}$                            \\
		$ W_{18} $          & 16               & 2                 & $ \frac{34}{28} (!), \frac{27}{28}$                           \\
		$ Q_{16} $          & 14               & 2                 & $ \frac{25}{21} (!), \frac{19}{21}$                           \\
		$ Q_{17} $          & 15               & 2                 & $ \frac{12}{10} (!), \frac{28}{30}$                           \\
		$ Q_{18} $          & 16               & 2                 & $ \frac{58}{48} (!), \frac{46}{48}$                           \\
		% 						\midrule  %添加表格中横线
		% 			\end{tabular}
	% 	\end{table}
% % ~\\
% \begin{table}[!htbp]
	% 		\centering
	% 		\renewcommand\arraystretch{1.5}
	% 		\begin{tabular}{lp{2cm}p{2cm}p{6cm}}
		% \midrule  %
		$ S_{16} $          & 14               & 2                 & $ \frac{21}{17} (!), \frac{16}{17}$                           \\
		$ S_{17} $          & 15               & 2                 & $ \frac{10}{8} (!), \frac{23}{24}$                          \\
		$ U_{16} $          & 14               & 2                 & $ \frac{19}{15} (!), \frac{14}{15}$                           \\
		\bottomrule %添加表格底部粗线
	\end{tabular}
\end{table}

% \noindent\large{\textbf{Trimodal Singularities}}

% \noindent\textbf{I. Spectrum of Curve Singularities of Modality 3}

% \begin{tabular}{lccc}
	% 	\hline
	% 	Class & $\tau$ & $\mu-\tau$& $\mathrm{R}(f)$\\
	% 	\hline
	% 	$NA_{r,0}$ & $14+r$ & $2$ & $\frac{8}{5}(!),\frac{3}{2}-\frac{1}{2(r+5)}$ \\
	% 	$NA_{r,s}$ & $14+r+s$ &  $2$ &$\frac{8}{5}(!),\frac{3}{2}-\max\{\frac{1}{2(r+5)},\frac{1}{2(s+5)}\}$ \\
	% 	$NB_{(-1)}^r$ & $16+r$ &  $2$ &
	% 	$\frac{29}{18}(!),\frac{3}{2}-\frac{1}{2(r+5)}$  \\
	% 	$NB_{(0)}^r$ & $17+r$ &   $2$ & $\frac{21}{13}(!),\frac{3}{2}-\frac{1}{2(r+5)}$\\
	% 	$NB_{(1)}^r$ & $18+r$ &   $2$ & $\frac{34}{21}(!),\frac{3}{2}-\frac{1}{2(r+5)}$ \\
	% 	$NC_{19}$ & $19$ &  $0$ &$\varnothing$ \\
	% 	$NC_{20}$ & $20$ & $0$ &$\varnothing$ \\
	% 	$NF_{20}$ & $20$ & $0$ &$\varnothing$ \\
	% 	$NF_{21}$ & $21$ & $0$ &$\varnothing$ \\
	% 	\hline
	% \end{tabular}

\begin{table}[!htbp]
	\centering
	\caption{Spectrum of Curve Singularities of Modality 3}
	\renewcommand\arraystretch{1.5}
	\begin{tabular}{lp{2cm}p{2cm}p{6cm}}
		\toprule  %添加表格头部粗线
		\textbf{Class}      &$ \tau$         & $ \mu-\tau $     &\textbf{R}(f)    \\
		\midrule  %添加表格中横线
		$NA_{r,0}$ & $14+r$ & $2$ & $\frac{8}{5}(!),\frac{3}{2}-\frac{1}{2(r+5)}$ \\
		$NA_{r,s}$ & $14+r+s$ &  $2$ &$\frac{8}{5}(!),\frac{3}{2}-\max\{\frac{1}{2(r+5)},\frac{1}{2(s+5)}\}$ \\
		$NB_{(-1)}^r$ & $16+r$ &  $2$ &
		$\frac{29}{18}(!),\frac{3}{2}-\frac{1}{2(r+5)}$  \\
		$NB_{(0)}^r$ & $17+r$ &   $2$ & $\frac{21}{13}(!),\frac{3}{2}-\frac{1}{2(r+5)}$\\
		$NB_{(1)}^r$ & $18+r$ &   $2$ & $\frac{34}{21}(!),\frac{3}{2}-\frac{1}{2(r+5)}$ \\
		$NC_{19}$ & $19$ &  $0$ &$\varnothing$ \\
		$NC_{20}$ & $20$ & $0$ &$\varnothing$ \\
		$NF_{20}$ & $20$ & $0$ &$\varnothing$ \\
		$NF_{21}$ & $21$ & $0$ &$\varnothing$ \\
		\bottomrule %添加表格底部粗线
	\end{tabular}
\end{table}

\begin{table}[!htbp]
	\centering
	\caption{Spectrum of Surface Singularities of Modality 3}
	\renewcommand\arraystretch{1.5}
	\begin{tabular}{lp{2cm}p{2cm}p{6cm}}
		\toprule  %添加表格头部粗线
		\textbf{Class}      &$ \tau$         & $ \mu-\tau $     &\textbf{R}(f)    \\
		\midrule  %添加表格中横线
		$VA_{r,s}$ & $13+s+r$ & $2$ & $\frac{17}{8}(!),2-\max\{\frac{1}{2(r+4)},\frac{1}{2(s+4)}\}$ \\
		$VA_{2k,0}^\#$ & $13+2k$ & $2$ & $\frac{17}{8}(!),2-\frac{1}{2(k+4)}$ \\
		$VA_{2k+1,0}^\#$ & $14+2k$ & $2$ & $\frac{17}{8}(!),2-\frac{1}{2k+9}$ \\
		$VA_{2k,s}^\#$ & $13+2k+s$ & $2$ & $\frac{17}{8}(!),2-\max\{\frac{1}{2(r+4)},\frac{1}{2(s+4)}\}$ \\
		$VA_{2k+1,s}^\#$ & $ 14+2k+s$ & $2$ & $\frac{17}{8}(!),2-\max\{\frac{1}{2k+9},\frac{1}{2(s+4)}\}$ \\
		$VB_{(-1)}^s$ & $15+s$ & $2$ & $\frac{15}{7}(!),2-\frac{1}{2(s+4)}$ \\
		$VB_{(0)}^s$ & $16+s$ & $2$ & $\frac{43}{20}(!),2-\frac{1}{2(s+4)}$ \\
		$VB_{(-1)}^{\#,2k}$ & $15+2k$ & $2$ & $\frac{32}{15}(!),2-\frac{1}{2(k+4)}$\\
		$VB_{(-1)}^{\#,2k+1}$ & $16+2k$ & $2$ & $\frac{32}{15}(!), 2-\frac{1}{2k+9}$\\
		$VB_{(0)}^{\#,2k}$ & $16+2k$ & $2$ & $\frac{47}{22}(!),2-\frac{1}{2(k+4)}$\\
		$VB_{(0)}^{\#,2k+1}$ & $17+2k$ & $2$ & $\frac{47}{22}(!), 2-\frac{1}{2k+9}$\\
		$VB_{(1)}^{\#,2k}$ & $17+2k$ & $2$ & $\frac{77}{36}(!),2-\frac{1}{2(k+4)}$\\
		$VB_{(1)}^{\#,2k+1}$ & $18+2k$ & $2$ & $\frac{77}{36}(!), 2-\frac{1}{2k+9}$\\
		% 		\midrule  %添加表格中横线
		% 	\end{tabular}
	% \end{table}
% % ~\\
% \begin{table}[!htbp]
	% \centering
	% \renewcommand\arraystretch{1.5}
	% \begin{tabular}{lp{2cm}p{2cm}p{6cm}}
		% 		\midrule  %
		$VC_{18}$ & $18$ & $0$ & $\varnothing$\\
		$VC_{19}$ & $19$ & $0$ & $\varnothing$\\
		$VC_{18}^\#$ & $18$ & $0$ & $\varnothing$\\
		$VC_{19}^\#$ & $19$ & $0$ & $\varnothing$\\
		$VF_{19}$ & $19$ & $0$ & $\varnothing$\\
		$VF_{20}$ & $20$ & $0$ & $\varnothing$\\
		\bottomrule %添加表格底部粗线
	\end{tabular}
\end{table}

% \noindent\textbf{II. Spectrum of Surface Singularities of Modality 3}

% \begin{tabular}{lccc}
	% 	\hline
	% 	Class & $\tau$ & $\mu-\tau$ & $\mathrm{R}(f)$\\
	% 	\hline
	% 	$VA_{r,s}$ & $13+s+r$ & $2$ & $\frac{17}{8}(!),2-\max\{\frac{1}{2(r+4)},\frac{1}{2(s+4)}\}$ \\
	% 	$VA_{2k,0}^\#$ & $13+r$ & $2$ & $\frac{17}{8}(!),2-\frac{1}{2(k+4)}$ \\
	% 	$VA_{2k+1,0}^\#$ & $13+r$ & $2$ & $\frac{17}{8}(!),2-\frac{1}{2k+9}$ \\
	% 	$VA_{2k,s}^\#$ & $13+r+s$ & $2$ & $\frac{17}{8}(!),2-\max\{\frac{1}{2(r+4)},\frac{1}{2(s+4)}\}$ \\
	% 	$VA_{2k+1,s}^\#$ & $ 13+r+s$ & $2$ & $\frac{17}{8}(!),2-\max\{\frac{1}{2k+9},\frac{1}{2(s+4)}\}$ \\
	% 	$VB_{(-1)}^s$ & $15+s$ & $2$ & $\frac{15}{7}(!),2-\frac{1}{2(s+4)}$ \\
	% 	$VB_{(0)}^s$ & $16+s$ & $2$ & $\frac{43}{20}(!),2-\frac{1}{2(s+4)}$ \\
	% 	$VB_{(-1)}^{\#,2k}$ & $15+r$ & $2$ & $\frac{32}{15}(!),2-\frac{1}{2(k+4)}$\\
	% 	$VB_{(-1)}^{\#,2k+1}$ & $15+r$ & $2$ & $\frac{32}{15}(!), 2-\frac{1}{2k+9}$\\
	% 	$VB_{(0)}^{\#,2k}$ & $16+r$ & $2$ & $\frac{47}{22}(!),2-\frac{1}{2(k+4)}$\\
	% 	$VB_{(0)}^{\#,2k+1}$ & $16+r$ & $2$ & $\frac{47}{22}(!), 2-\frac{1}{2k+9}$\\
	% 	$VB_{(1)}^{\#,2k}$ & $17+r$ & $2$ & $\frac{77}{36}(!),2-\frac{1}{2(k+4)}$\\
	% 	$VB_{(1)}^{\#,2k+1}$ & $17+r$ & $2$ & $\frac{77}{36}(!), 2-\frac{1}{2k+9}$\\
	% 	$VC_{18}$ & $18$ & $0$ & $\varnothing$\\
	% 	$VC_{19}$ & $19$ & $0$ & $\varnothing$\\
	% 	$VC_{18}^\#$ & $18$ & $0$ & $\varnothing$\\
	% 	$VC_{19}^\#$ & $19$ & $0$ & $\varnothing$\\
	% 	$VF_{19}$ & $19$ & $0$ & $\varnothing$\\
	% 	$VF_{20}$ & $20$ & $0$ & $\varnothing$\\
	% 	\hline
	% \end{tabular}

It is worth mentioning that Tjurina spectrum is not stable under $\tau$-constant deformation (see \cite{tjsp0b}). But we conjecture that it remains constant in the connected stratum containing the normal form of $f$. Besides, we do not consider singularities in \cite{MR0553709}. The reason lies in our formulation of \textit{\textbf{Conjecture \ref{GHCTS}}} in the introduction, since semi-quasihomogeneous singularities can be $\mu$-constantly deformed to quasihomogeneous ones. In these cases $\tau$ increases.

\subsection{Maximal Exponent and Generalized Hertling Conjecture for Tjurina Spectrum}\label{Maximal Exponent and Generalized Hertling Conjecture for Tjurina Spectrum}

Notations are as before. $f\in \mathcal O_{n+1}$ defines an isolated singularity at the origin with $\mu$ and $\tau$ its Milnor and Tjurina numbers respectively.
\begin{theorem}\label{f induces fitration lemma}
	Let $d = \mu-\tau$
	, $\alpha_1\leq \alpha_2 \leq ... \leq \alpha_\mu$ be all spectrum numbers, $\alpha_{j_1} < ... < \alpha_{j_t}$ be all jumpings i.e. $\alpha_{j_i} < \alpha_{j_{i}+1}$ and $h_l = \mathrm{Gr}_V^{\alpha_{j_l}}\Omega_f^{n+1} = j_l -j_{l-1}$ be the multiplicity of $\alpha_{j_l}$.
	Besides, let $\alpha_{i_1} \leq \alpha_{i_2} \leq ... \leq \alpha_{i_d}$ be all exponents not in the Tjurina spectrum.
	
	For each $1\leq k\leq d$, let $r_k$ be the maximal integer such that $\#\{j\mid \alpha_j\leq \alpha_{r_k}\} \leq k$, then $\alpha_{i_k} \geq \alpha_{j_{r_k}}+1$. In particular, we have $\alpha_{i_1} \geq \alpha_1+1$ and $\alpha_{i_2} \geq \alpha_2+1$ if $d \geq 2$.
\end{theorem}
\begin{proof}
	By (2.1), we have $f V^\beta\Omega_f^{n+1} \subset V^{\beta+1}\Omega_f^{n+1}$ for all $\beta\in \mathbb Q$.  We abbreviate $r_k$ as $r$. Take $f : v\mapsto f\cdot v$ as an endomorphipsm of $\Omega_f^{n+1}$. Pick $e_1,...,e_\mu$ to be a basis of $\Omega_f^{n+1}$ with $e_{h_1+...+h_{r-1}+s} \in V^{\alpha_r}$ for all $r$ and $1\leq s \leq h_r$. Then for each $k$, since $\mathrm{rank}(f\cdot) > k-1$, $f$ does not act trivially on $\mathrm{span}\{e_{k+1},...,e_{\mu}\}$ and hence $0\neq fV^{\alpha_{j_r}}\Omega_f^{n+1} \subset V^{\alpha_{j_r}+1}\Omega_f^{n+1}$. Hence there must be some $\beta \geq \alpha_{j_r+1}$ such that $\mathrm{Gr}_V^\beta(\Omega_f^{n+1}/f\Omega_f^{n+1}) < \mathrm{Gr}_V^\beta(\Omega_f^{n+1})$. By a simple induction, we get the result.
	%we have in particular $f \Omega_f^{n+1} = f V^{\alpha_1} \Omega_f^{n+1} \subset V^{\alpha_1+1}\Omega_f^{n+1}$. Hence, for all $\beta < \alpha_1+1$, $\mathrm{Gr}_V^\beta\Omega_f^{n+1} = \mathrm{Gr}_V^\beta (\Omega_f^{n+1}/f\Omega_f^{n+1})$, or $\beta_j = \alpha_j$ for all $\alpha_j < \alpha_1+1$.
	
	%If we further assume $\mu-\tau \geq 2$, $f : v\mapsto fv$ can be viewed as a linear endomorphism of $\Omega_f^{n+1}$ with rank $\mu-\tau \geq 2$. By \cite{dimca2014koszul}, 4.11, the multiplicity of $\alpha_1$ is $1$, or $\dim \mathrm{Gr}_V^{\alpha_1}\Omega_f^{n+1} = 1$. Let $e_1\in V^{\alpha_1}$ whose image in $\mathrm{Gr}_V^{\alpha_1}$ is a basis. Take $e_2,...,e_n\in V^{\alpha_2}$
	%Let $e_1,...,e_n\in \Omega_f^{n+1}$ be a basis under which $f$ is represented in Jordan canonical form. Let $\lambda_k = \sup\{l \mid e_k\in V^{\alpha_l}\Omega_f^{n+1}\}$ and we find $\{k\mid \lambda_k = 1\} \leq 1$. Otherwise, $\dim \mathrm{Gr}_V^{\alpha_1}\Omega_f^{n+1} \geq 2$, a contradiction. Since the rank of this endomorphism is no smaller than $2$, there must be a $k$ such that $\lambda_k \geq \alpha_2$ and $V^{\alpha_{2}+1} \ni f\cdot e_k \neq 0$. Hence, there exists $\beta\geq \alpha_2+1$ such that $\mathrm{Gr}_V^\beta(\Omega_f^{n+1}/f\Omega_f^{n+1}) < \mathrm{Gr}_V^\beta \Omega_f^{n+1}$ i.e. $\beta$ is in $\mathrm{R}(f)$. In particular, $\alpha_{i_{\mu-\tau}} \geq \alpha_2+1$.  

	As for the last assertion, we need the following theorem.
	%As for the last assertion, by Saito's formula (see \cite{MR1621831}, 8.6), $p_g = \#\{\alpha_j \mid \alpha_j \leq 1\}$. Since $\alpha_1 >0$, we are done.
\end{proof}
\begin{theorem}[\cite{dimca2014koszul}, 4.11] \label{multiplicity 1 of the minimal exponent}
	The multiplicity of $\alpha_1$ is $1$. So is $\alpha_\mu$.
\end{theorem}

As a corollary straightforward, we have the following.
\begin{corollary}
	If $\tau < \mu$, then the minimal exponent $\alpha_{i_1}$ not in the Tjurina spectrum satisfies
	\begin{displaymath}
		\alpha_{i_1} \geq \alpha_1+1
	\end{displaymath}
\end{corollary}

From the computation in the subsection before, we find that $\alpha_\mu \not\in \mathrm{R}(f)$ is all those cases. In the first version of the manuscript, we conjectured it holds generally. Later on, in our communication with Jung, Kim, Saito, and Yoon, they wrote specially a paper, \cite{tjsp0b}, and gave a proof the Conjecture \ref{DConjecture} in the paper.

\begin{theorem}[\cite{tjsp0b}]\label{difference conjecture for sptau and sp}
	Suppose $f\in \mathcal O_{n+1}$ defines an isolated singularity at the origin and the difference between its Milnor number and Tjurina number is positive i.e. $\mu-\tau \geq 1$, then the maximal spectrum number of $f$ is not contained in the Tjurina spectrum of $f$. Morover, the minimal and maximal exponents $\alpha_{i_1},\alpha_{i_{\mu-\tau}}$ not in the Tjurina spectrum both have multiplicity $1$.
\end{theorem}

Nevertheless, \textbf{Theorem \ref{f induces fitration lemma}} still has its advantages, especially for large $\mu-\tau$.

Next, we propose a conjecture for Tjurina spectrum parallel to Hertling conjecture. One may simply consider the same inequality, replacing spectrum by Tjurina spectrum and $\frac{n+1}{2}$ by their average. Unfortunately, it turns false in semi-quasihomogeneous cases (see \cite{tjsp0b}). However, the potential problem inside seems to be ``the form is not good enough''. For example, semi-quasihomogeneous singularities can be $\mu$-constantly deformed to quasihomogeneous ones. The latter are good singularities in our context. Therefore, we have to pick out the suitable singularities (but not lossing generality).

\begin{definition}\label{tau-max}
	We call isolated singularities germs $f_1,f_2\in \mathcal O_{n+1}$ $\mu$-accessible, denoted as $f_1 \sim_\mu f_2$, if $f_1$ can be tranformed to $f_2$ after fintie times of $\mu$-constant deformation. More precisely, there is a sequence $f_1 = g_0 \sim g_1 \sim ... \sim g_m = f_2$, where $u \sim v$ means there is a $\mu$-constant deformation $F \in \mathcal O_{n+2}$ such that $F(\bm x,0) = u$ and $F(\bm x,1) = v$. 
	
	Let $\tau^{\mathrm{max}}:= \tau^{\mathrm{max}}(f) = \sup_{g \sim_\mu f} \tau(g)$. We call $g\sim_\mu f$ a $\tau$-max form if $\tau(g) = \tau^{\mathrm{max}}(f)$.
\end{definition}

So far, there is no ambiguity in the satement of \textit{\textbf{Conjecture \ref{GHCTS}}}.

There is an interesting corollary if \textit{\textbf{Conjecture \ref{GHCTS}}} holds.
\begin{corollary}
	Assume \textit{\textbf{Conjecture \ref{GHCTS}}} holds. If the inequality (GHCTS) fails for $(V(f),0)$ and $\mu-\tau = 1$, then $f$ can be tranformed to a quasihomogeneous singularity after finite times of $\mu$-constant deformation.
\end{corollary}
\begin{proof}
	If the assertion fails, then $f$ itself is a $\tau$-max form, contradicting to \textit{\textbf{Conjecture \ref{GHCTS}}}.
\end{proof}

For a general isolated singularity $(V(f),0)$, we say \textit{\textbf{Conjecture \ref{GHCTS}}} holds for $f$ if it holds for all $\tau$-max forms of $f$. In the rest pages of the paper, we prove the conjecture for singularities of modality not greater than $3$.

The following elementary inequality is helpful.
\begin{lemma}\label{a simple reduction}
	\label{5.10}
	Suppose $0 < x_1 \leq x_2 \leq ... \leq x_m$ is a sequence of real numbers and $1\leq k \leq m-2$ be an integer. We further assume $\sum_{i=1}^m x_i = bm$, $m\geq 3$, and $b\geq \frac{1}{2}(1-\frac{4}{m})^{-1}(x_m+x_{m-1})$. Let $V_1$ and $V_2$ be the following number.
	\begin{align*}
		V_1 & = \frac{1}{m-2}\sum_{j\neq m-1,m} x_j^2-\frac{1}{(m-2)^2}(\sum_{j\neq m-1,m} x_j)^2-\frac{x_{m-2}-x_1}{12}\\
		V_2 & = \frac{1}{m-2}\sum_{j\neq k,m} x_j^2-\frac{1}{(m-2)^2}(\sum_{j\neq k,m} x_j)^2-\frac{x_{m-1}-x_1}{12}
	\end{align*}
	Then we have $V_1 \geq V_2$.
\end{lemma}
\begin{proof}
	The difference of $V_1$ and $V_2$ is
	\begin{align*}
		&V_1-V_2 \\
		& = \frac{1}{m-2}(x_k^2-x_{m-1}^2)-\frac{1}{(m-2)^2}(2bm-x_k-x_{m-1}-2x_m)(x_k-x_{m-1})-\frac{x_{m-2}-x_{m-1}}{12}\\
		& \geq (x_{m-1}-x_k)(\frac{2bm-x_k-x_{m-1}-2x_m}{(m-2)^2}-\frac{x_k+x_{m-1}}{m-2})\\
		& \geq (x_{m-1}-x_k)(\frac{2bm-2(x_{m-1}+x_m)}{(m-2)^2}-\frac{x_{m-1}+x_{m}}{m-2}) \geq 0
	\end{align*}
\end{proof}

Next, we prove \textit{\textbf{Generalized Hertling Conjecture for Tjurina Spectrum}} holds for all singularities of modality $\leq 3$. \textbf{Lemma \ref{a simple reduction}} performs useful in the proof. Since $ \tau^{\mathrm{max}} (f) \geq \tau_{f_0}, $ where $ f_0 $ is the normal form of $ f$. We know $ \tau^{\mathrm{max}} (f) \in \{ \mu_f -2, \mu_f - 1, \mu_f \}$ for $ f $ a singularity of modality $ \leq 3, $
where $ \tau_f $ and $ \mu_f $ are Tjurina number and Milnor number of $ f. $
If $ \tau^{\mathrm{max}} (f) = \mu_f, $ i.e. $f \sim_{\mu} g$ with $ g$ quasihomogeneous, Hertling Conjecture holds by\textbf{ Theorem \ref{3.15}}.
If $ \tau^{\mathrm{max}}(f) = \mu_f -1 $, then we only need to show Hertling inequality of the first $ \mu_f - 1 $ spectrum exponents of $ f$ holds, that is,
\begin{equation*}		
	\frac{1}{\mu_f - 1}\sum_{i = 1}^{\mu_f -1} (\alpha_i-\bar{\alpha})^2 \leq \frac{\alpha_{\mu_f -1}-\alpha_1}{12},
\end{equation*}
where $\alpha_1\leq ...\leq \alpha_{\mu_f}$ are spectrum numbers of $f$ and $ \bar{\alpha} = \frac{1}{ \mu_f -1 } \cdot \sum_{i =1 } ^{\mu_f - 1} \alpha_i, $
since the maximal spectrum exponent must disappear.
If $ \tau^{\mathrm{max}}(f) = \mu_f - 2, $  then we only need to show Hertling inequality of the first $ \mu_f - 2 $ spectrum exponents of $ f$ holds
by \textbf{Lemma \ref{5.10}}, that is,
\begin{equation*}		
	\frac{1}{\mu_f - 2}\sum_{i = 1}^{\mu_f -2} (\alpha_i-\bar{\alpha})^2 \leq \frac{\alpha_{\mu_f -2}-\alpha_1}{12},
\end{equation*}
where $\alpha_1\leq ...\leq \alpha_{\mu_f}$ are spectrum numbers of $f$ and $ \bar{\alpha} = \frac{1}{ \mu_f -2 } \cdot \sum_{i =1 } ^{\mu_f - 2} \alpha_i. $

%Since we have $\mu-\tau \leq 2$ for singularities of modality $\leq 3$, $\mu-\tau^{\mathrm{max}} \in \{0,1,2\}$. $0$ means quasihomogeneous and is no need to check. For $1$, by \textbf{Theorem \ref{difference conjecture for sptau and sp}}, the only missing exponent is the largest one. For $2$, by \textbf{Lemma \ref{a simple reduction}}, it suffices to check the (GHCTS) for the first $\tau$-exponents (one can easily see the condition of \textbf{Lemma \ref{a simple reduction}} holds in these cases).

\subsection{Generalized Hertling Conjecture for Tjurina Spectrum for  Unimodal Singularities}

\begin{theorem}\label{HCTS1}
	\textbf{\textit{\textbf{Generalized Hertling Conjecture for Tjurina Spectrum}}} holds for singularities of modality 1 with related to Tjurina Spectrum.
\end{theorem}
\begin{proof}
	\noindent\textbf{(1) $Q_{10}$.}
	
	Since $ \tau_{f_0} = \mu_f - 1 = 9, $ we only need to verify the case of  $\tau^{\mathrm{max}} = \mu_f -1 = 9 . $ 
	\begin{align*}
		& \bar{\alpha} =\frac{1}{\mu_f - 1}\cdot ( \frac{23}{24}+\frac{29}{24}+\frac{31}{24} + \frac{4}{3} + \frac{35}{24}
		+ \frac{37}{24} + \frac{5}{3} + \frac{41}{24} + \frac{43}{24} )=\frac{311}{216}.
	\end{align*}
	\begin{align*}
		& (\mu_f -  1) \cdot \mathrm{Var}[\mathrm{Sp^{\tau^{\mathrm{max}} }}(f)] \\
		=& (\frac{23}{24}-\bar{\alpha})^2+(\frac{29}{24}-\bar{\alpha})^2+(\frac{31}{24}-\bar{\alpha})^2+ (\frac{4}{3} - \bar{\alpha})^2 +  ( \frac{35}{24}-\bar{\alpha})^2 +(\frac{37}{24} - \bar{\alpha})^2 \\
		&+  ( \frac{5}{3} -\bar{\alpha})^2 + ( \frac{41}{24} - \bar{\alpha})^2 +( \frac{43}{24} - \bar{\alpha})^2 \\
		=& \frac{1495}{2592}  .
	\end{align*}
	And $ \alpha_{\mu_f -1}-\alpha_1=  \frac{43}{24} - \frac{23}{24} = \frac{20}{24}.  $
	Then it suffices to prove $\frac{1495}{2592} - \frac{9}{12} \cdot \frac{20}{24} \leq 0,  $ i.e., $ - \frac{125}{2592} \leq 0, $
	which holds obviously.
	%	Considering the average value $ \bar{\alpha}=\frac{3}{2}, $
	%	\begin{align*}
		%		& (\mu_f -  1) \cdot \mathrm{Var}[\mathrm{Sp^{\tau^{\mathrm{max}} }}(f)] \\
		%		=& (\frac{23}{24}-\bar{\alpha})^2+(\frac{29}{24}-\bar{\alpha})^2+(\frac{31}{24}-\bar{\alpha})^2+ (\frac{4}{3} - \bar{\alpha})^2 +  ( \frac{35}{24}-\bar{\alpha})^2 +(\frac{37}{24} - \bar{\alpha})^2 \\
		%		&+  ( \frac{5}{3} -\bar{\alpha})^2 + ( \frac{41}{24} - \bar{\alpha})^2 +( \frac{43}{24} - \bar{\alpha})^2 \\
		%		=& \frac{39}{64}  .
		%	\end{align*}
	%	Then it suffices to prove $\frac{39}{64}- \frac{9}{12} \cdot \frac{20}{24} \leq 0,  $ i.e., $ - \frac{1}{64} \leq 0, $
	%	which holds obviously.
	
	\noindent\textbf{(2) $Q_{11}$.}
	
	Since $ \tau_{f_0} = \mu_f - 1 = 10, $ we only need to verify the case of  $\tau^{\mathrm{max}} = \mu_f -1 = 10 . $ 
	\begin{align*}
		& \bar{\alpha} =\frac{1}{\mu_f -1 }\cdot ( \frac{17}{18}+\frac{7}{6}+\frac{23}{18} + \frac{4}{3} + \frac{25}{18}
		+\frac{3}{2} + \frac{29}{18} + \frac{5}{3} + \frac{31}{18} + \frac{11}{6} ) =\frac{13}{9}.
	\end{align*}
	\begin{align*}
		& (\mu_f -  1) \cdot \mathrm{Var}[\mathrm{Sp^{\tau^{\mathrm{max}} }}(f)] \\
		=& (\frac{17}{18}-\bar{\alpha})^2+(\frac{7}{6}-\bar{\alpha})^2+(\frac{23}{18}-\bar{\alpha})^2+ (\frac{4}{3} - \bar{\alpha})^2 +  ( \frac{25}{18}-\bar{\alpha})^2 +(\frac{3}{2} - \bar{\alpha})^2  +  ( \frac{29}{18} -\bar{\alpha})^2\\
		& + ( \frac{5}{3} - \bar{\alpha})^2 +( \frac{31}{18} - \bar{\alpha})^2 + ( \frac{11}{6} - \bar{\alpha})^2 \\
		=& \frac{55}{81}  .
	\end{align*}
	And $ \alpha_{\mu_f -1}-\alpha_1= \frac{11}{6} - \frac{17}{18} = \frac{16}{18}.  $
	Then it suffices to prove $\frac{55}{81} - \frac{10}{12} \cdot \frac{16}{18} \leq 0,  $ i.e., $ - \frac{5}{81} \leq 0, $
	which holds obviously.
	%	
	%	Considering the average value $ \bar{\alpha}=\frac{3}{2}, $
	%	\begin{align*}
		%		& (\mu_f -  1) \cdot \mathrm{Var}[\mathrm{Sp^{\tau^{\mathrm{max}} }}(f)] \\
		%		=& (\frac{17}{18}-\bar{\alpha})^2+(\frac{7}{6}-\bar{\alpha})^2+(\frac{23}{18}-\bar{\alpha})^2+ (\frac{4}{3} - \bar{\alpha})^2 +  ( \frac{25}{18}-\bar{\alpha})^2 +(\frac{3}{2} - \bar{\alpha})^2  \\
		%		&+  ( \frac{29}{18} -\bar{\alpha})^2 + ( \frac{5}{3} - \bar{\alpha})^2 +( \frac{31}{18} - \bar{\alpha})^2 + ( \frac{11}{6} - \bar{\alpha})^2 \\
		%		=& \frac{115}{162}  .
		%	\end{align*}
	%	Then it suffices to prove $\frac{115}{162} - \frac{10}{12} \cdot \frac{16}{18} \leq 0,  $ i.e., $ - \frac{5}{162} \leq 0, $
	%	which holds obviously.
	%	
	
	\noindent\textbf{(3) $Q_{12}$.}
	
	Since $  \tau_{f_0} = \mu_f - 1 = 11, $ we only need to verify the case of  $\tau^{\mathrm{max}} = \mu_f -1 = 11. $ 
	\begin{align*}
		& \bar{\alpha} =\frac{1}{\mu_f - 1}\cdot ( \frac{14}{15}+\frac{17}{15}+\frac{19}{15} + 2\cdot \frac{4}{3} + \frac{22}{15}
		+ \frac{23}{15} + 2\cdot \frac{5}{3} + \frac{26}{15} + \frac{28}{15} )=\frac{239}{165}.
	\end{align*}
	\begin{align*}
		& (\mu_f -  1) \cdot \mathrm{Var}[\mathrm{Sp^{\tau^{\mathrm{max}} }}(f)] \\
		=& (\frac{14}{15}-\bar{\alpha})^2+(\frac{17}{15}-\bar{\alpha})^2+(\frac{19}{15}-\bar{\alpha})^2+ 2\cdot (\frac{4}{3} - \bar{\alpha})^2 +  ( \frac{22}{15}-\bar{\alpha})^2\\
		& +  ( \frac{23}{15} -\bar{\alpha})^2 + 2\cdot ( \frac{5}{3} - \bar{\alpha})^2 +( \frac{26}{15} - \bar{\alpha})^2 + ( \frac{28}{15} - \bar{\alpha})^2 \\
		=& \frac{646}{825}  .
	\end{align*}
	And $ \alpha_{\mu_f -1}-\alpha_1= \frac{28}{15} - \frac{14}{15} = \frac{14}{15}.  $
	Then it suffices to prove $\frac{646}{825} - \frac{11}{12} \cdot \frac{14}{15} \leq 0,  $ i.e., $ - \frac{359}{4950} \leq 0, $
	which holds obviously.
	%	
	%	
	%	Considering the average value $ \bar{\alpha}=\frac{3}{2}, $
	%	\begin{align*}
		%		& (\mu_f -  1) \cdot \mathrm{Var}[\mathrm{Sp^{\tau^{\mathrm{max}} }}(f)] \\
		%		=& (\frac{14}{15}-\bar{\alpha})^2+(\frac{17}{15}-\bar{\alpha})^2+(\frac{19}{15}-\bar{\alpha})^2+ 2\cdot (\frac{4}{3} - \bar{\alpha})^2 +  ( \frac{22}{15}-\bar{\alpha})^2\\
		%		& +  ( \frac{23}{15} -\bar{\alpha})^2 + 2\cdot ( \frac{5}{3} - \bar{\alpha})^2 +( \frac{26}{15} - \bar{\alpha})^2 + ( \frac{28}{15} - \bar{\alpha})^2 \\
		%		=& \frac{731}{900}  .
		%	\end{align*}
	%	Then it suffices to prove $ \frac{731}{900}- \frac{11}{12} \cdot \frac{14}{15} \leq 0,  $ i.e., $ - \frac{13}{300} \leq 0, $
	%	which holds obviously.
	%	
	%	
	%	
	
	\noindent\textbf{(4) $S_{11}$.}
	
	Since $  \tau_{f_0} = \mu_f - 1 = 10, $ we only need to verify the case of  $\tau^{\mathrm{max}} = \mu_f -1 = 10. $ 
	\begin{align*}
		& \bar{\alpha} =\frac{1}{\mu_f - 1}\cdot ( \frac{15}{16}+\frac{19}{16}+\frac{5}{4} +\frac{21}{16} + \frac{23}{16}
		+ \frac{3}{2}+ \frac{25}{16} +  \frac{27}{16} + \frac{7}{4} + \frac{29}{16} )=\frac{231}{160}.
	\end{align*}
	\begin{align*}
		& (\mu_f -  1) \cdot \mathrm{Var}[\mathrm{Sp^{\tau^{\mathrm{max}} }}(f)] \\
		=& (\frac{14}{15}-\bar{\alpha})^2+(\frac{17}{15}-\bar{\alpha})^2+(\frac{19}{15}-\bar{\alpha})^2+ 2\cdot (\frac{4}{3} - \bar{\alpha})^2 +  ( \frac{22}{15}-\bar{\alpha})^2 +  ( \frac{23}{15} -\bar{\alpha})^2 \\
		& + 2\cdot ( \frac{5}{3} - \bar{\alpha})^2 +( \frac{26}{15} - \bar{\alpha})^2 + ( \frac{28}{15} - \bar{\alpha})^2 \\
		=& \frac{1749}{2560}  .
	\end{align*}
	And $ \alpha_{\mu_f -1}-\alpha_1= \frac{29}{16} - \frac{15}{16} = \frac{14}{16}.  $
	Then it suffices to prove $\frac{1749}{2560}  - \frac{10}{12} \cdot \frac{14}{16} \leq 0,  $ i.e., $ - \frac{353}{7680} \leq 0, $
	which holds obviously.
	%	
	%	Considering the average value $ \bar{\alpha}=\frac{3}{2}, $
	%	\begin{align*}
		%		& (\mu_f -  1) \cdot \mathrm{Var}[\mathrm{Sp^{\tau^{\mathrm{max}} }}(f)] \\
		%		=& (\frac{14}{15}-\bar{\alpha})^2+(\frac{17}{15}-\bar{\alpha})^2+(\frac{19}{15}-\bar{\alpha})^2+ 2\cdot (\frac{4}{3} - \bar{\alpha})^2 +  ( \frac{22}{15}-\bar{\alpha})^2 +  ( \frac{23}{15} -\bar{\alpha})^2 \\
		%		& + 2\cdot ( \frac{5}{3} - \bar{\alpha})^2 +( \frac{26}{15} - \bar{\alpha})^2 + ( \frac{28}{15} - \bar{\alpha})^2 \\
		%		=& \frac{183}{256}  .
		%	\end{align*}
	%	Then it suffices to prove $\frac{183}{256}  - \frac{10}{12} \cdot \frac{14}{16} \leq 0,  $ i.e., $ - \frac{11}{768} \leq 0, $
	%	which holds obviously.
	%	
	%	
	
	\noindent\textbf{(5) $S_{12}$.}
	
	Since $  \tau_{f_0}  = \mu_f - 1 = 11, $ we only need to verify the case of  $\tau^{\mathrm{max}} = \mu_f -1 = 11. $ 
	\begin{align*}
		& \bar{\alpha} =\frac{1}{\mu_f -1}\cdot ( \frac{12}{13}+\frac{15}{13}+\frac{16}{13} +\frac{17}{13} + \frac{18}{13} + \frac{19}{13}
		+ \frac{20}{13} +  \frac{21}{13} + \frac{22}{13} + \frac{23}{13} + \frac{24}{13}) =\frac{231}{160}.
	\end{align*}
	\begin{align*}
		& (\mu_f -  1) \cdot \mathrm{Var}[\mathrm{Sp^{\tau^{\mathrm{max}} }}(f)] \\
		=& (\frac{14}{15}-\bar{\alpha})^2+(\frac{17}{15}-\bar{\alpha})^2+(\frac{19}{15}-\bar{\alpha})^2+ 2\cdot (\frac{4}{3} - \bar{\alpha})^2 +  ( \frac{22}{15}-\bar{\alpha})^2 +  ( \frac{23}{15} -\bar{\alpha})^2\\
		&  + 2\cdot ( \frac{5}{3} - \bar{\alpha})^2 +( \frac{26}{15} - \bar{\alpha})^2 + ( \frac{28}{15} - \bar{\alpha})^2 \\
		=& \frac{1470}{1859}  .
	\end{align*}
	And $ \alpha_{\mu_f -1}-\alpha_1= \frac{24}{13} - \frac{12}{13} = \frac{12}{13}.  $
	Then it suffices to prove $\frac{1470}{1859}   - \frac{11}{12} \cdot \frac{12}{13} \leq 0,  $ i.e., $ - \frac{103}{1859} \leq 0, $
	which holds obviously.
	%	
	%	Considering the average value $ \bar{\alpha}=\frac{3}{2}, $
	%	\begin{align*}
		%		& (\mu_f -  1) \cdot \mathrm{Var}[\mathrm{Sp^{\tau^{\mathrm{max}} }}(f)] \\
		%		=& (\frac{14}{15}-\bar{\alpha})^2+(\frac{17}{15}-\bar{\alpha})^2+(\frac{19}{15}-\bar{\alpha})^2+ 2\cdot (\frac{4}{3} - \bar{\alpha})^2 +  ( \frac{22}{15}-\bar{\alpha})^2 +  ( \frac{23}{15} -\bar{\alpha})^2\\
		%		&  + 2\cdot ( \frac{5}{3} - \bar{\alpha})^2 +( \frac{26}{15} - \bar{\alpha})^2 + ( \frac{28}{15} - \bar{\alpha})^2 \\
		%		=& \frac{555}{676}  .
		%	\end{align*}
	%	Then it suffices to prove $ \frac{555}{676}  - \frac{11}{12} \cdot \frac{12}{13} \leq 0,  $ i.e., $ - \frac{17}{676} \leq 0, $
	%	which holds obviously.
	%	
	
	\noindent\textbf{(6) $U_{12}$.}
	
	Since $  \tau_{f_0} = \mu_f - 1 = 11, $ we only need to verify the case of  $\tau^{\mathrm{max}} = \mu_f -1 = 11 . $ 
	\begin{align*}
		& \bar{\alpha} =\frac{1}{\mu_f -1 }\cdot ( \frac{11}{12}+\frac{7}{6}+2\cdot \frac{5}{4} +\frac{17}{12} + 2\cdot \frac{3}{2}
		+ \frac{19}{12} +  2\cdot \frac{7}{4} + \frac{11}{6} ) =\frac{191}{132}.
	\end{align*}
	\begin{align*}
		& (\mu_f -  1) \cdot \mathrm{Var}[\mathrm{Sp^{\tau^{\mathrm{max}} }}(f)] \\
		=& (\frac{14}{15}-\bar{\alpha})^2+(\frac{17}{15}-\bar{\alpha})^2+(\frac{19}{15}-\bar{\alpha})^2+ 2\cdot (\frac{4}{3} - \bar{\alpha})^2 +  ( \frac{22}{15}-\bar{\alpha})^2 +  ( \frac{23}{15} -\bar{\alpha})^2\\
		&  + 2\cdot ( \frac{5}{3} - \bar{\alpha})^2 +( \frac{26}{15} - \bar{\alpha})^2 + ( \frac{28}{15} - \bar{\alpha})^2 \\
		=& \frac{35}{44}  .
	\end{align*}
	And $ \alpha_{\mu_f -1}-\alpha_1= \frac{11}{6} - \frac{11}{12} = \frac{11}{12}.  $
	Then it suffices to prove $\frac{35}{44}  - \frac{11}{12} \cdot \frac{11}{12} \leq 0,  $ i.e., $ - \frac{71}{1584} \leq 0, $
	which holds obviously.
	%	
	%	
	%	Considering the average value $ \bar{\alpha}=\frac{3}{2}, $
	%	\begin{align*}
		%		& (\mu_f -  1) \cdot \mathrm{Var}[\mathrm{Sp^{\tau^{\mathrm{max}} }}(f)] \\
		%		=& (\frac{14}{15}-\bar{\alpha})^2+(\frac{17}{15}-\bar{\alpha})^2+(\frac{19}{15}-\bar{\alpha})^2+ 2\cdot (\frac{4}{3} - \bar{\alpha})^2 +  ( \frac{22}{15}-\bar{\alpha})^2 +  ( \frac{23}{15} -\bar{\alpha})^2 \\
		%		& + 2\cdot ( \frac{5}{3} - \bar{\alpha})^2 +( \frac{26}{15} - \bar{\alpha})^2 + ( \frac{28}{15} - \bar{\alpha})^2 \\
		%		=& \frac{119}{144}  .
		%	\end{align*}
	%	Then it suffices to prove $\frac{119}{144}  - \frac{11}{12} \cdot \frac{11}{12} \leq 0,  $ i.e., $ - \frac{1}{72} \leq 0, $
	%	which holds obviously.
	%	
	
	\noindent\textbf{(7) $Z_{11}$.}
	
	Since $  \tau_{f_0}  = \mu_f - 1 = 10, $ we only need to verify the case of  $\tau^{\mathrm{max}} = \mu_f -1 = 10 . $ 
	\begin{align*}
		& \bar{\alpha} =\frac{1}{\mu_f -1 }\cdot ( \frac{7}{15}+\frac{2}{3}+\frac{11}{15} +\frac{13}{15} + \frac{14}{15}
		+1 + \frac{16}{15} +  \frac{17}{15} + \frac{19}{15} + \frac{4}{3} ) =\frac{71}{75}.
	\end{align*}
	\begin{align*}
		& (\mu_f -  1) \cdot \mathrm{Var}[\mathrm{Sp^{\tau^{\mathrm{max}} }}(f)] \\
		=& (\frac{7}{15}-\bar{\alpha})^2+(\frac{2}{3}-\bar{\alpha})^2+(\frac{11}{15}-\bar{\alpha})^2+ (\frac{13}{15} - \bar{\alpha})^2 +  ( \frac{14}{15}-\bar{\alpha})^2 + (1 - \bar{\alpha})^2 \\
		& +  ( \frac{16}{15} -\bar{\alpha})^2 + ( \frac{17}{15} - \bar{\alpha})^2 +( \frac{19}{15} - \bar{\alpha})^2 + ( \frac{4}{3} - \bar{\alpha})^2 \\
		=& \frac{748}{1125}  .
	\end{align*}
	And $ \alpha_{\mu_f -1}-\alpha_1= \frac{4}{3} - \frac{7}{15} = \frac{13}{15}.  $
	Then it suffices to prove $ \frac{748}{1125}  - \frac{10}{12} \cdot \frac{13}{15} \leq 0,  $ i.e., $ - \frac{43}{750} \leq 0, $
	which holds obviously.
	%	
	%	Considering the average value $ \bar{\alpha}=1, $
	%	\begin{align*}
		%		& (\mu_f -  1) \cdot \mathrm{Var}[\mathrm{Sp^{\tau^{\mathrm{max}} }}(f)] \\
		%		=& (\frac{7}{15}-\bar{\alpha})^2+(\frac{2}{3}-\bar{\alpha})^2+(\frac{11}{15}-\bar{\alpha})^2+ (\frac{13}{15} - \bar{\alpha})^2 +  ( \frac{14}{15}-\bar{\alpha})^2 + (1 - \bar{\alpha})^2 \\
		%		& +  ( \frac{16}{15} -\bar{\alpha})^2 + ( \frac{17}{15} - \bar{\alpha})^2 +( \frac{19}{15} - \bar{\alpha})^2 + ( \frac{4}{3} - \bar{\alpha})^2 \\
		%		=& \frac{52}{75}  .
		%	\end{align*}
	%	Then it suffices to prove $ \frac{52}{75}  - \frac{10}{12} \cdot \frac{13}{15} \leq 0,  $ i.e., $ - \frac{13}{450} \leq 0, $
	%	which holds obviously.
	%	
	%	
	
	\noindent\textbf{(8) $Z_{12}$.}
	
	Since $  \tau_{f_0} = \mu_f - 1 = 11, $ we only need to verify the case of  $\tau^{\mathrm{max}} = \mu_f -1 = 11 . $ 
	\begin{align*}
		& \bar{\alpha} =\frac{1}{\mu_f -1}\cdot ( \frac{5}{11}+\frac{7}{11}+\frac{8}{11} +\frac{9}{11} + \frac{10}{11} + 2 \cdot 1
		+\frac{12}{11} + \frac{13}{11} + \frac{14}{11} + \frac{15}{11}  )=\frac{115}{121}.
	\end{align*}
	\begin{align*}
		& (\mu_f -  1) \cdot \mathrm{Var}[\mathrm{Sp^{\tau^{\mathrm{max}} }}(f)] \\
		=& (\frac{5}{11}-\bar{\alpha})^2+(\frac{7}{11}-\bar{\alpha})^2+(\frac{8}{11}-\bar{\alpha})^2+ (\frac{9}{11} - \bar{\alpha})^2 +  ( \frac{10}{11}-\bar{\alpha})^2 +2 \cdot (1 - \bar{\alpha})^2\\
		&  +  ( \frac{12}{11} -\bar{\alpha})^2 + ( \frac{13}{11} - \bar{\alpha})^2 +( \frac{14}{11} - \bar{\alpha})^2 + ( \frac{15}{11} - \bar{\alpha})^2 \\
		=& \frac{1020}{1331}  .
	\end{align*}
	And $ \alpha_{\mu_f -1}-\alpha_1= \frac{15}{11} - \frac{5}{11} = \frac{10}{11}.  $
	Then it suffices to prove $ \frac{1020}{1331} - \frac{11}{12} \cdot \frac{10}{11} \leq 0,  $ i.e., $ - \frac{535}{7986} \leq 0, $
	which holds obviously.
	%	
	%	Considering the average value $ \bar{\alpha}=1, $
	%	\begin{align*}
		%		& (\mu_f -  1) \cdot \mathrm{Var}[\mathrm{Sp^{\tau^{\mathrm{max}} }}(f)] \\
		%		=& (\frac{5}{11}-\bar{\alpha})^2+(\frac{7}{11}-\bar{\alpha})^2+(\frac{8}{11}-\bar{\alpha})^2+ (\frac{9}{11} - \bar{\alpha})^2 +  ( \frac{10}{11}-\bar{\alpha})^2  +2 \cdot (1 - \bar{\alpha})^2 \\
		%		& +  ( \frac{12}{11} -\bar{\alpha})^2 + ( \frac{13}{11} - \bar{\alpha})^2 +( \frac{14}{11} - \bar{\alpha})^2 + ( \frac{15}{11} - \bar{\alpha})^2 \\
		%		=& \frac{96}{121}  .
		%	\end{align*}
	%	Then it suffices to prove $ \frac{96}{121} - \frac{11}{12} \cdot \frac{10}{11} \leq 0,  $ i.e., $ - \frac{29}{726} \leq 0, $
	%	which holds obviously.
	%	
	%	
	
	\noindent\textbf{(9) $Z_{13}$.}
	
	Since $  \tau_{f_0} = \mu_f - 1 = 12, $ we only need to verify the case of  $\tau^{\mathrm{max}} = \mu_f -1 = 12. $ 
	\begin{align*}
		& \bar{\alpha} =\frac{1}{\tau}\cdot ( \frac{4}{9}+\frac{11}{18}+\frac{13}{18}+\frac{7}{9} +\frac{8}{9} + \frac{17}{18} +  1
		+\frac{19}{18} + \frac{10}{9} + \frac{11}{9} + \frac{23}{18} + \frac{25}{18}   )=\frac{103}{108}.
	\end{align*}
	\begin{align*}
		& (\mu_f -  1) \cdot \mathrm{Var}[\mathrm{Sp^{\tau^{\mathrm{max}} }}(f)] \\
		=& (\frac{4}{9}-\bar{\alpha})^2+(\frac{11}{18}-\bar{\alpha})^2+(\frac{7}{9}-\bar{\alpha})^2+ (\frac{8}{9} - \bar{\alpha})^2 +  ( \frac{17}{18}-\bar{\alpha})^2 + (1 - \bar{\alpha})^2\\
		&  +  ( \frac{19}{18} -\bar{\alpha})^2 + ( \frac{10}{9} - \bar{\alpha})^2 +( \frac{11}{9} - \bar{\alpha})^2 + ( \frac{23}{18} - \bar{\alpha})^2 +( \frac{25}{18} - \bar{\alpha})^2\\
		=& \frac{845}{972}  .
	\end{align*}
	And $ \alpha_{\mu_f -1}-\alpha_1= \frac{25}{18} - \frac{4}{9} = \frac{17}{18}.  $
	Then it suffices to prove $ \frac{845}{972} - \frac{12}{12} \cdot \frac{17}{18} \leq 0,  $ i.e., $ - \frac{73}{972} \leq 0, $
	which holds obviously.
	%	
	%	
	%	Considering the average value $ \bar{\alpha}=1, $
	%	\begin{align*}
		%		& (\mu_f -  1) \cdot \mathrm{Var}[\mathrm{Sp^{\tau^{\mathrm{max}} }}(f)] \\
		%		=& (\frac{4}{9}-\bar{\alpha})^2+(\frac{11}{18}-\bar{\alpha})^2+(\frac{7}{9}-\bar{\alpha})^2+ (\frac{8}{9} - \bar{\alpha})^2 +  ( \frac{17}{18}-\bar{\alpha})^2 + (1 - \bar{\alpha})^2 \\
		%		& +  ( \frac{19}{18} -\bar{\alpha})^2 + ( \frac{10}{9} - \bar{\alpha})^2 +( \frac{11}{9} - \bar{\alpha})^2 + ( \frac{23}{18} - \bar{\alpha})^2 +( \frac{25}{18} - \bar{\alpha})^2\\
		%		=& \frac{145}{162}  .
		%	\end{align*}
	%	Then it suffices to prove $ \frac{145}{162} - \frac{12}{12} \cdot \frac{17}{18} \leq 0,  $ i.e., $ - \frac{4}{81} \leq 0, $
	%	which holds obviously.
	%	
	%	
	
	\noindent\textbf{(10) $W_{13}$.}
	
	Since $  \tau_{f_0} = \mu_f - 1 = 12, $ we only need to verify the case of  $\tau^{\mathrm{max}} = \mu_f -1 = 12 . $ 
	\begin{align*}
		& \bar{\alpha} =\frac{1}{\mu_f -1}\cdot ( \frac{7}{16}+\frac{5}{8}+\frac{11}{16} +\frac{13}{16} + \frac{7}{8} +  \frac{15}{16}
		+1 + \frac{17}{16} + \frac{9}{8} + \frac{19}{16} + \frac{21}{16} + \frac{11}{8}   )=\frac{61}{64}.
	\end{align*}
	\begin{align*}
		& (\mu_f -  1) \cdot \mathrm{Var}[\mathrm{Sp^{\tau^{\mathrm{max}} }}(f)] \\
		=& (\frac{7}{16}-\bar{\alpha})^2+(\frac{5}{8}-\bar{\alpha})^2+(\frac{11}{16}-\bar{\alpha})^2+ (\frac{13}{16} - \bar{\alpha})^2 +  ( \frac{7}{8}-\bar{\alpha})^2 + (\frac{15}{16} - \bar{\alpha})^2\\
		& + (1 - \bar{\alpha})^2 +  ( \frac{17}{16} -\bar{\alpha})^2 + ( \frac{9}{8} - \bar{\alpha})^2 +( \frac{19}{16} - \bar{\alpha})^2 
		+ ( \frac{21}{16} - \bar{\alpha})^2 +( \frac{11}{8} - \bar{\alpha})^2\\
		=& \frac{897}{1024}  .
	\end{align*}
	And $ \alpha_{\mu_f -1}-\alpha_1= \frac{11}{8} - \frac{7}{16} = \frac{15}{16}.  $
	Then it suffices to prove $  \frac{897}{1024}- \frac{12}{12} \cdot \frac{15}{16} \leq 0,  $ i.e., $ - \frac{63}{1024} \leq 0, $
	which holds obviously.
	%	
	%	
	%	Considering the average value $ \bar{\alpha}=1, $
	%	\begin{align*}
		%		& (\mu_f -  1) \cdot \mathrm{Var}[\mathrm{Sp^{\tau^{\mathrm{max}} }}(f)] \\
		%		=& (\frac{7}{16}-\bar{\alpha})^2+(\frac{5}{8}-\bar{\alpha})^2+(\frac{11}{16}-\bar{\alpha})^2+ (\frac{13}{16} - \bar{\alpha})^2 +  ( \frac{7}{8}-\bar{\alpha})^2 + (\frac{15}{16} - \bar{\alpha})^2 \\
		%		&+ (1 - \bar{\alpha})^2 +  ( \frac{17}{16} -\bar{\alpha})^2 + ( \frac{9}{8} - \bar{\alpha})^2 +( \frac{19}{16} - \bar{\alpha})^2 
		%		+ ( \frac{21}{16} - \bar{\alpha})^2 +( \frac{11}{8} - \bar{\alpha})^2\\
		%		=& \frac{231}{256}  .
		%	\end{align*}
	%	Then it suffices to prove $  \frac{231}{256} - \frac{12}{12} \cdot \frac{15}{16} \leq 0,  $ i.e., $ - \frac{9}{256} \leq 0, $
	%	which holds obviously.
	%	
	%	
	
	\noindent\textbf{(11) $E_{12}$.}
	
	Since $ \tau_{f_0} = \mu_f - 1 = 11, $ we only need to verify the case of  $\tau^{\mathrm{max}} = \mu_f -1 = 11 . $ 
	\begin{align*}
		& \bar{\alpha} =\frac{1}{\mu_f -1}\cdot ( \frac{10}{21}+\frac{13}{21}+\frac{16}{21} +\frac{17}{21} + \frac{19}{21} +  \frac{20}{21}
		+ \frac{22}{21} + \frac{23}{21} + \frac{25}{21} + \frac{26}{21} + \frac{29}{21}   )=\frac{20}{21}.
	\end{align*}
	\begin{align*}
		& (\mu_f -  1) \cdot \mathrm{Var}[\mathrm{Sp^{\tau^{\mathrm{max}} }}(f)] \\
		=& (\frac{10}{21}-\bar{\alpha})^2+(\frac{13}{21}-\bar{\alpha})^2+(\frac{16}{21}-\bar{\alpha})^2+ (\frac{17}{21} - \bar{\alpha})^2 +  ( \frac{19}{21}-\bar{\alpha})^2 + (\frac{20}{21} - \bar{\alpha})^2 \\
		&+ (\frac{22}{21} - \bar{\alpha})^2 +  ( \frac{23}{21} -\bar{\alpha})^2 + ( \frac{25}{21} - \bar{\alpha})^2 +( \frac{26}{21} - \bar{\alpha})^2 
		+ ( \frac{29}{21} - \bar{\alpha})^2 \\
		=& \frac{110}{147}  .
	\end{align*}
	And $ \alpha_{\mu_f -1}-\alpha_1= \frac{29}{21} - \frac{10}{21} = \frac{19}{21}.  $
	Then it suffices to prove $  \frac{110}{147}- \frac{11}{12} \cdot \frac{19}{21} \leq 0,  $ i.e., $ - \frac{143}{1764} \leq 0, $
	which holds obviously.
	%	
	%	
	%	Considering the average value $ \bar{\alpha}=1, $
	%	\begin{align*}
		%		& (\mu_f -  1) \cdot \mathrm{Var}[\mathrm{Sp^{\tau^{\mathrm{max}} }}(f)] \\
		%		=& (\frac{10}{21}-\bar{\alpha})^2+(\frac{13}{21}-\bar{\alpha})^2+(\frac{16}{21}-\bar{\alpha})^2+ (\frac{17}{21} - \bar{\alpha})^2 +  ( \frac{19}{21}-\bar{\alpha})^2 +(\frac{20}{21} - \bar{\alpha})^2\\
		%		&  + (\frac{22}{21} - \bar{\alpha})^2 +  ( \frac{23}{21} -\bar{\alpha})^2 + ( \frac{25}{21} - \bar{\alpha})^2 +( \frac{26}{21} - \bar{\alpha})^2 
		%		+ ( \frac{29}{21} - \bar{\alpha})^2 \\
		%		=& \frac{341}{441}  .
		%	\end{align*}
	%	Then it suffices to prove $  \frac{341}{441}- \frac{11}{12} \cdot \frac{19}{21} \leq 0,  $ i.e., $ - \frac{11}{196} \leq 0, $
	%	which holds obviously.
	%	
	%	
	
	\noindent\textbf{(12) $E_{13}$.}
	
	Since $  \tau_{f_0} = \mu_f - 1 = 12, $ we only need to verify the case of  $\tau^{\mathrm{max}} = \mu_f -1 = 12 . $ 
	\begin{align*}
		& \bar{\alpha} =\frac{1}{\mu_f -1}\cdot ( \frac{7}{15}+\frac{3}{5}+\frac{11}{15} +\frac{4}{5} + \frac{13}{15} +  \frac{14}{15}
		+ 1 + \frac{16}{15} + \frac{17}{15} + \frac{6}{5} + \frac{19}{15} + \frac{7}{5}   )=\frac{43}{45}.
	\end{align*}
	\begin{align*}
		& (\mu_f -  1) \cdot \mathrm{Var}[\mathrm{Sp^{\tau^{\mathrm{max}} }}(f)] \\
		=& (\frac{7}{15}-\bar{\alpha})^2+(\frac{3}{5}-\bar{\alpha})^2+(\frac{11}{15}-\bar{\alpha})^2+ (\frac{4}{5} - \bar{\alpha})^2 +  ( \frac{13}{15}-\bar{\alpha})^2 + (\frac{14}{15} - \bar{\alpha})^2 \\
		&+ (1- \bar{\alpha})^2 +  ( \frac{16}{15} -\bar{\alpha})^2 + ( \frac{17}{15} - \bar{\alpha})^2 +( \frac{6}{5} - \bar{\alpha})^2 
		+ ( \frac{19}{15} - \bar{\alpha})^2 +( \frac{7}{5} - \bar{\alpha})^2 \\
		=& \frac{572}{675}  .
	\end{align*}
	And $ \alpha_{\mu_f -1}-\alpha_1= \frac{7}{5} - \frac{7}{15} = \frac{14}{15}.  $
	Then it suffices to prove $ \frac{572}{675}- \frac{12}{12} \cdot \frac{14}{15} \leq 0,  $ i.e., $ - \frac{58}{675} \leq 0, $
	which holds obviously.
	%	
	%	
	%	Considering the average value $ \bar{\alpha}=1, $
	%	\begin{align*}
		%		& (\mu_f -  1) \cdot \mathrm{Var}[\mathrm{Sp^{\tau^{\mathrm{max}} }}(f)] \\
		%		=& (\frac{7}{15}-\bar{\alpha})^2+(\frac{3}{5}-\bar{\alpha})^2+(\frac{11}{15}-\bar{\alpha})^2+ (\frac{4}{5} - \bar{\alpha})^2 +  ( \frac{13}{15}-\bar{\alpha})^2 + (\frac{14}{15} - \bar{\alpha})^2 \\
		%		&+ (1- \bar{\alpha})^2 +  ( \frac{16}{15} -\bar{\alpha})^2 + ( \frac{17}{15} - \bar{\alpha})^2 +( \frac{6}{5} - \bar{\alpha})^2 
		%		+ ( \frac{19}{15} - \bar{\alpha})^2 +( \frac{7}{5} - \bar{\alpha})^2 \\
		%		=& \frac{196}{225}  .
		%	\end{align*}
	%	Then it suffices to prove $ \frac{196}{225}- \frac{12}{12} \cdot \frac{14}{15} \leq 0,  $ i.e., $ - \frac{14}{225} \leq 0, $
	%	which holds obviously.
	%	
	%	
	
	\noindent\textbf{(13) $E_{14}$.}
	
	Since $ \tau_{f_0} = \mu_f - 1 = 13, $ we only need to verify the case of  $\tau^{\mathrm{max}} = \mu_f -1 = 13. $ 
	\begin{align*}
		& \bar{\alpha} =\frac{1}{\tau}\cdot ( \frac{11}{24}+\frac{7}{12}+\frac{15}{24} +\frac{17}{24} + \frac{5}{6} +  \frac{11}{12}
		+ \frac{23}{24} + \frac{25}{24} + \frac{13}{12} + \frac{7}{6} + \frac{29}{24} + \frac{31}{24}  +\frac{17}{12} )\\
		&=\frac{23}{24}.
	\end{align*}
	\begin{align*}
		& (\mu_f -  1) \cdot \mathrm{Var}[\mathrm{Sp^{\tau^{\mathrm{max}} }}(f)] \\
		=& (\frac{11}{24}-\bar{\alpha})^2+(\frac{7}{12}-\bar{\alpha})^2+(\frac{15}{24}-\bar{\alpha})^2+ (\frac{17}{24} - \bar{\alpha})^2 +  ( \frac{5}{6}-\bar{\alpha})^2+ (\frac{11}{12} - \bar{\alpha})^2 + (\frac{23}{24}- \bar{\alpha})^2\\
		& +  ( \frac{25}{24} -\bar{\alpha})^2 + ( \frac{13}{12} - \bar{\alpha})^2 +( \frac{7}{6} - \bar{\alpha})^2 
		+ ( \frac{29}{24} - \bar{\alpha})^2 +( \frac{31}{24} - \bar{\alpha})^2 + (\frac{17}{12} - \bar{\alpha})^2 \\
		=& \frac{275}{288}  .
	\end{align*}
	And $ \alpha_{\mu_f -1}-\alpha_1= \frac{17}{12} - \frac{11}{24} = \frac{23}{24}.  $
	Then it suffices to prove $ \frac{275}{288}  - \frac{13}{12} \cdot \frac{23}{24} \leq 0,  $ i.e., $ - \frac{1}{12} \leq 0, $
	which holds obviously.
	
	%	
	%	Considering the average value $ \bar{\alpha}=1, $
	%	\begin{align*}
		%		& (\mu_f -  1) \cdot \mathrm{Var}[\mathrm{Sp^{\tau^{\mathrm{max}} }}(f)] \\
		%		=& (\frac{11}{24}-\bar{\alpha})^2+(\frac{7}{12}-\bar{\alpha})^2+(\frac{15}{24}-\bar{\alpha})^2+ (\frac{17}{24} - \bar{\alpha})^2 +  ( \frac{5}{6}-\bar{\alpha})^2+ (\frac{11}{12} - \bar{\alpha})^2 + (\frac{23}{24}- \bar{\alpha})^2 \\
		%		& +  ( \frac{25}{24} -\bar{\alpha})^2 + ( \frac{13}{12} - \bar{\alpha})^2 +( \frac{7}{6} - \bar{\alpha})^2 
		%		+ ( \frac{29}{24} - \bar{\alpha})^2 +( \frac{31}{24} - \bar{\alpha})^2 + (\frac{17}{12} - \bar{\alpha})^2 \\
		%		=& \frac{559}{576}  .
		%	\end{align*}
	%	Then it suffices to prove $ \frac{559}{576}  - \frac{13}{12} \cdot \frac{23}{24} \leq 0,  $ i.e., $ - \frac{13}{192} \leq 0, $
	%	which holds obviously.
	%	
	%	
	
	\noindent\textbf{(14) $T_{p,q,r}$.}
	
	Since $ \tau_{f_0} = \mu_f - 1 = p+q +r -2, $ we only need to verify the case of  $\tau^{\mathrm{max}} = \mu_f -1 = p+q+r-2 . $ 
	\begin{align*}
		& \bar{\alpha} =\frac{1}{\mu_f -1}\cdot ( 0 + \sum_{k_1 = 1}^{p-1} \frac{k_1}{p} + \sum_{k_2 = 1}^{q-1} \frac{k_2}{q} + \sum_{k_3 = 1}^{r-1} \frac{k_3}{r} )=\frac{p + q + r - 3}{2(p + q + r - 2)}.
	\end{align*}
	\begin{align*}
		& (\mu_f -  1) \cdot \mathrm{Var}[\mathrm{Sp^{\tau^{\mathrm{max}} }}(f)] \\
		=& (0 - \bar{\alpha})^2 + \sum_{k_1 = 1}^{p-1} ( \frac{k_1}{p} - \bar{\alpha})^2 + \sum_{k_2 = 1}^{q-1} ( \frac{k_2}{q} - \bar{\alpha})^2 + \sum_{k_3 = 1}^{r-1} ( \frac{k_3}{r} - \bar{\alpha})^2 \\
		=& \frac{1}{12(p + q + r - 2)pqr} \cdot (p^3rq + (2q^2r + (2r^2 - 8r + 2)q + 2r)p^2 + (q^3r + (2r^2 - 8r + 2)q^2\\
		& + (r^3 - 8r^2 + 15r - 4)q + 2r^2 - 4r)p + 2qr(q + r - 2) )  .
	\end{align*}
	Without loss of generality, we may assume $ 2 \leq r \leq q \leq p. $  
	And $ \alpha_{\mu_f -1}-\alpha_1= \frac{p-1}{p} - 0 = \frac{p-1}{p}.  $
	Then it suffices to prove
	\begin{equation*}
		\begin{split}
			&\frac{1}{12(p + q + r - 2)pqr} \cdot (p^3rq + (2q^2r + (2r^2 - 8r + 2)q + 2r)p^2 + (q^3r + (2r^2 - 8r + 2)q^2\\
			& + (r^3 - 8r^2 + 15r - 4)q + 2r^2 - 4r)p + 2qr(q + r - 2) ) - \frac{p + q + r -2}{12} \cdot \frac{p-1}{p} \leq 0, 
		\end{split}
	\end{equation*}
	i.e., 
	\begin{equation*}
		\begin{split}
			&\frac{1}{12(p + q + r - 2)pqr} \cdot ( q^3r - 2(r - 1)(-r + p)q^2 + (r^3 + (-2p - 2)r^2 \\
			&+ (-3p^2 + 7p)r + 2p^2 - 4p)q + 2pr(p + r - 2) ) \leq 0 .
		\end{split}
	\end{equation*}
	Viewing the numerator
	\begin{equation*}
		\begin{split}
			f(p) =&((-3r + 2)q + 2r) p^2 + ( -2(r - 1)q^2 + (-2r^2 + 7r - 4)q + 2r(r - 2) )p\\
			& + q^3r + 2(r - 1)rq^2 + (r^3 - 2r^2)q
		\end{split}
	\end{equation*}
	of the left hand side of the above inequality as a quadratic function of the variable $ p, $
	it is not diffcult to verify $ f(p) \leq 0, \forall p \geq q $ under the restriction $ \frac{1}{p} + \frac{1}{q} + \frac{1}{r} < 1,  $ i.e., $ p > \frac{qr}{ qr - q - r}.$ 
	Thus the above inequality holds for all $ p, q, r$ which satisfy $ \frac{1}{p} + \frac{1}{q} + \frac{1}{r} < 1 $.
\end{proof}
\subsection{Generalized Hertling Conjecture for Tjurina Spectrum for  Bimodal Singularities}

\begin{theorem}\label{HCTS2}
	\textbf{\textit{\textbf{Generalized Hertling Conjecture for Tjurina Spectrum}}} holds for singularities of modality 2 with related to Tjurina Spectrum, except for $W^\sharp_{1,2q-1},W^\sharp_{1,2q},U_{1,2q-1}$ and $U_{1,2q}$. If \textbf{Conjecture \ref{Conjecture for four exceptional}} holds, then this theorem holds for all bimodal singularities.
\end{theorem}
\begin{proof}
	\noindent\textbf{(1) $J_{3,0}$.}
	
	Since $ \tau_{f_0} = \mu_f - 1 = 15, $ we only need to verify the case of  $\tau^{\mathrm{max}} = \mu_f -1 = 15 . $ 
	\begin{align*}
		& \bar{\alpha} =\frac{1}{\mu_f - 1}\cdot ( \frac{4}{9}+\frac{5}{9}+\frac{6}{9} + 2 \cdot \frac{7}{9} + 2 \cdot \frac{8}{9}+2 \cdot \frac{10}{9} + 2 \cdot \frac{11}{9} + \frac{12}{9} + \frac{13}{9} )
		=\frac{26}{27}.
	\end{align*}
	\begin{align*}
		& (\mu_f -  1) \cdot \mathrm{Var}[\mathrm{Sp^{\tau^{\mathrm{max}} }}(f)] \\
		=& (\frac{4}{9}-\bar{\alpha})^2+(\frac{5}{9}-\bar{\alpha})^2+(\frac{6}{9}-\bar{\alpha})^2+2 \cdot (\frac{7}{9} - \bar{\alpha})^2 + 2 \cdot ( \frac{8}{9} - \bar{\alpha})^2 +2 \cdot(\frac{10}{9} - \bar{\alpha})^2\\
		& + 2 \cdot ( \frac{11}{9} -\bar{\alpha})^2 + ( \frac{12}{9} - \bar{\alpha})^2 + ( \frac{13}{9} - \bar{\alpha})^2\\
		=& \frac{280}{243}  .
	\end{align*}
	And $ \alpha_{\mu_f -1}-\alpha_1= \frac{13}{9} - \frac{4}{9} =1 .  $
	Then it suffices to prove $\frac{280}{243} - \frac{15}{12} \cdot 1 \leq 0,  $ i.e., $ - \frac{95}{972} \leq 0, $
	which holds obviously.
	%	
	%	
	%	Considering the average value $ \bar{\alpha}=1, $
	%	\begin{align*}
		%		& (\mu_f -  1) \cdot \mathrm{Var}[\mathrm{Sp^{\tau^{\mathrm{max}} }}(f)] \\
		%		=& (\frac{4}{9}-\bar{\alpha})^2+(\frac{5}{9}-\bar{\alpha})^2+(\frac{6}{9}-\bar{\alpha})^2+2 \cdot (\frac{7}{9} - \bar{\alpha})^2 + 2 \cdot ( \frac{8}{9} - \bar{\alpha})^2 +2 \cdot(\frac{10}{9} - \bar{\alpha})^2 \\
		%		&+ 2 \cdot ( \frac{11}{9} -\bar{\alpha})^2 + ( \frac{12}{9} - \bar{\alpha})^2 + ( \frac{13}{9} - \bar{\alpha})^2\\
		%		=& \frac{95}{81}  .
		%	\end{align*}
	%	Then it suffices to prove $\frac{95}{81} - \frac{15}{12} \cdot 1 \leq 0,  $ i.e., $ - \frac{25}{324} \leq 0, $
	%	which holds obviously.
	%	
	
	\noindent\textbf{(2) $J_{3,p}$.}

	Since $ \tau_{f_0}  = \mu_f - 2 = 14+p, $ we need to verify the cases of $\tau^{\mathrm{max}} = 14 + p $ or $ 15 + p. $ For simplicity, we only verify the cases of $ p$ even.
	For $ p $ even, if  $\tau^{\mathrm{max}} = 14 + p, $ 
	\begin{align*}
		&\bar{\alpha} =\frac{1}{\mu_f -2}\cdot ( \frac{17}{18}+\frac{23}{18}+\frac{25}{18} + \sum_{ k=1 } ^{ \frac{p+10}{2}} (1+\frac{2k-1}{ 2(p+9)})+ \frac{29}{18} + \frac{31}{18} + \sum_{ k=2 } ^{ \frac{p+10}{2}} (2-\frac{2k-1}{ 2(p+9)}))\\
		&=\frac{27p^2 + 602p + 3240}{18(p + 9)(14 + p)}. 
	\end{align*}
	\begin{align*}
		& (\mu_f -  2) \cdot \mathrm{Var}[\mathrm{Sp^{\tau^{\mathrm{max}} }}(f)] \\
		=& (\frac{17}{18}-\bar{\alpha})^2+(\frac{23}{18}-\bar{\alpha})^2+(\frac{25}{8}-\bar{\alpha})^2+\sum_{ k=1 } ^{ \frac{p+10}{2}} (1+\frac{2k-1}{ 2(p+9)} - \bar{\alpha})^2 + ( \frac{29}{18} - \bar{\alpha})^2\\
		&  + (\frac{31}{18}- \bar{\alpha})^2 + \sum_{ k=2 } ^{ \frac{p+10}{2}} (2 - \frac{2k-1}{ 2(p+9)} - \bar{\alpha})^2\\
		=& \frac{27p^4 + 1166p^3 + 18429p^2 + 128052p + 332100}{324(p + 9)^2(14 + p)}.
	\end{align*}
	And $ \alpha_{\mu_f -2}-\alpha_1= 2 - \frac{3}{2(p+9)} - \frac{17}{18} = \frac{19}{18} - \frac{3}{2(p+9)}. $
	Then it suffices to prove
	\begin{align*}
		&	\frac{27p^4 + 1166p^3 + 18429p^2 + 128052p + 332100}{324(p + 9)^2(14 + p)} - \frac{14+p}{12} \cdot (\frac{19}{18} - \frac{3}{2(p+9)}) \leq 0.
	\end{align*}
	Equivalently,
	\begin{align*}
		&	\frac{-3p^4 - 209p^3 - 4662p^2 - 37980p - 97848}{648(p + 9)^2(14 + p)} \leq 0,
	\end{align*}
	which holds obviously.

	If  $\tau^{\mathrm{max}} = 15 + p, $ 
	\begin{align*}
		&\bar{\alpha} =\frac{1}{\mu_f -1}\cdot ( \frac{17}{18}+\frac{23}{18}+\frac{25}{18} + \sum_{ k=1 } ^{ \frac{p+10}{2}} (1+\frac{2k-1}{ 2(p+9)})+ \frac{29}{18} + \frac{31}{18} + \sum_{ k=1 } ^{ \frac{p+10}{2}} (2-\frac{2k-1}{ 2(p+9)}))\\
		&=\frac{395 + 27p}{270 + 18p}. 
	\end{align*}
	\begin{align*}
		& (\mu_f -  1) \cdot \mathrm{Var}[\mathrm{Sp^{\tau^{\mathrm{max}} }}(f)] \\
		=& (\frac{17}{18}-\bar{\alpha})^2+(\frac{23}{18}-\bar{\alpha})^2+(\frac{25}{8}-\bar{\alpha})^2+\sum_{ k=1 } ^{ \frac{p+10}{2}} (1+\frac{2k-1}{ 2(p+9)} - \bar{\alpha})^2 + ( \frac{29}{18} - \bar{\alpha})^2\\
		&  + (\frac{31}{18}- \bar{\alpha})^2 + \sum_{ k=1 } ^{ \frac{p+10}{2}} (2 - \frac{2k-1}{ 2(p+9)} - \bar{\alpha})^2\\
		=& \frac{27p^3 + 1031p^2 + 12710p + 50400}{324(p + 9)(15 + p)}.
	\end{align*}
	And $ \alpha_{\mu_f -1}-\alpha_1= 2 - \frac{1}{2(p+9)} - \frac{17}{18}. $
	Then it suffices to prove
	\begin{align*}
		&	\frac{27p^3 + 1031p^2 + 12710p + 50400}{324(p + 9)(15 + p)} - \frac{15+p}{12} \cdot (2 - \frac{1}{2(p+9)} - \frac{17}{18}) \leq 0.
	\end{align*}
	Equivalently,
	\begin{align*}
		&	\frac{-3p^3 - 134p^2 - 1985p - 8550}{648(p + 9)(15 + p)} \leq 0,
	\end{align*}
	which holds obviously.
	%	Considering the average value $ \bar{\alpha}=\frac{3}{2}, $
	%	\begin{align*}
		%		& (\mu_f -  1) \cdot \mathrm{Var}[\mathrm{Sp^{\tau^{\mathrm{max}} }}(f)] \\
		%		=& (\frac{17}{18}-\bar{\alpha})^2+(\frac{23}{18}-\bar{\alpha})^2+(\frac{25}{18}-\bar{\alpha})^2+\sum_{ k=1 } ^{ \frac{p+9}{2}} (1+\frac{2k-1}{ 2(p+9)} - \bar{\alpha})^2 + ( \frac{3}{2} - \bar{\alpha})^2\\
		%		&  + ( \frac{29}{18} - \bar{\alpha})^2+ (\frac{31}{18}- \bar{\alpha})^2 + \sum_{ k=2 } ^{ \frac{p+9}{2}} (2 - \frac{2k-1}{ 2(p+9)} - \bar{\alpha})^2\\
		%		=&  \frac{27p^3 + 788p^2 + 7758p + 25596}{324(p + 9)^2}.
		%	\end{align*}
	%	Then it suffices to prove
	%	\begin{align*}
		%		&\frac{27p^3 + 788p^2 + 7758p + 25596}{324(p + 9)^2} - \frac{14+p}{12} \cdot (\frac{19}{18} - \frac{3}{2(p+9)}) \leq 0.
		%	\end{align*}
	%	Equivalently,
	%	\begin{align*}
		%		&\frac{-3p^3 - 167p^2 - 1602p - 3240}{648(p + 9)^2} \leq 0,
		%	\end{align*}
	%	which holds obviously.
	%	
	%	
	
	\noindent\textbf{(3) $Z_{1,0}$.}
	
	Since $  \tau_{f_0}  = \mu_f - 1 = 14, $ we only need to verify the case of  $\tau^{\mathrm{max}} = \mu_f -1 = 14. $ 
	\begin{align*}
		& \bar{\alpha} =\frac{1}{\mu_f -1}\cdot ( \frac{3}{7}+\frac{4}{7}+2 \cdot \frac{5}{7} + 2 \cdot \frac{6}{7} + 3 \cdot 1+2 \cdot \frac{8}{7} + 2 \cdot \frac{9}{7} + \frac{10}{7} ) =\frac{47}{49}.
	\end{align*}
	\begin{align*}
		& (\mu_f -  1) \cdot \mathrm{Var}[\mathrm{Sp^{\tau^{\mathrm{max}} }}(f)] \\
		=& (\frac{3}{7}-\bar{\alpha})^2+(\frac{4}{7}-\bar{\alpha})^2+2\cdot (\frac{5}{7}-\bar{\alpha})^2+2 \cdot (\frac{6}{7} - \bar{\alpha})^2 + 3 \cdot ( 1 -\bar{\alpha})^2 +2 \cdot(\frac{8}{7} - \bar{\alpha})^2\\
		& + 2 \cdot ( \frac{9}{7} -\bar{\alpha})^2 + ( \frac{10}{9} - \bar{\alpha})^2 \\
		=& \frac{370}{343}  .
	\end{align*}
	And $ \alpha_{\mu_f -1}-\alpha_1= \frac{10}{7} - \frac{3}{7} =1 .  $
	Then it suffices to prove $\frac{370}{343} - \frac{14}{12} \cdot 1 \leq 0,  $ i.e., $ - \frac{181}{2058} \leq 0, $
	which holds obviously.
	%	
	%	
	%	Considering the average value $ \bar{\alpha}=1, $
	%	\begin{align*}
		%		& (\mu_f -  1) \cdot \mathrm{Var}[\mathrm{Sp^{\tau^{\mathrm{max}} }}(f)] \\
		%		=& (\frac{4}{9}-\bar{\alpha})^2+(\frac{5}{9}-\bar{\alpha})^2+(\frac{6}{9}-\bar{\alpha})^2+2 \cdot (\frac{7}{9} - \bar{\alpha})^2 + 2 \cdot ( \frac{8}{9} - \bar{\alpha})^2 +2 \cdot(\frac{10}{9} - \bar{\alpha})^2\\
		%		& + 2 \cdot ( \frac{11}{9} -\bar{\alpha})^2 + ( \frac{12}{9} - \bar{\alpha})^2 + ( \frac{13}{9} - \bar{\alpha})^2\\
		%		=& \frac{54}{47}  .
		%	\end{align*}
	%	Then it suffices to prove $\frac{54}{49} - \frac{14}{12} \cdot 1 \leq 0,  $ i.e., $ - \frac{19}{294} \leq 0, $
	%	which holds obviously.
	
	\noindent\textbf{(4) $Z_{1,p}$.}
	Since $  \tau_{f_0} = \mu_f - 2 = 13+p, $ we need to verify the cases of $\tau^{\mathrm{max}} = 13 + p $ or $ 14 + p. $ For simplicity, we only verify the cases of $ p$ even.
	For $ p $ even, if  $\tau^{\mathrm{max}} = 13 + p, $ 
	\begin{align*}
		&\bar{\alpha} =\frac{1}{\mu_f - 2}\cdot ( \frac{13}{14}+\frac{17}{14}+\frac{19}{14} + \sum_{ k=1 } ^{ \frac{p+8}{2}} (1+\frac{2k-1}{ 2(p+7)})+2 \cdot \frac{3}{2} + \frac{23}{14} + \frac{25}{14} +\sum_{ k=2 } ^{ \frac{p+8}{2}} (2-\frac{2k-1}{ 2(p+7)}))\\
		&=\frac{21p^2 + 405p + 1813}{14(p + 7)(13 + p)}.
	\end{align*}
	\begin{align*}
		& (\mu_f -  2) \cdot \mathrm{Var}[\mathrm{Sp^{\tau^{\mathrm{max}} }}(f)] \\
		&=(\frac{13}{14}-\bar{\alpha})^2+(\frac{17}{14}-\bar{\alpha})^2+(\frac{19}{14}-\bar{\alpha})^2+\sum_{ k=1 } ^{ \frac{p+7}{2}} (1+\frac{2k-1}{ 2(p+7)} - \bar{\alpha})^2 + 2\cdot( \frac{3}{2} - \bar{\alpha})^2\\
		&  + ( \frac{23}{14} - \bar{\alpha})^2+ (\frac{25}{14}- \bar{\alpha})^2 + \sum_{ k=2 } ^{ \frac{p+7}{2}} (2 - \frac{2k-1}{ 2(p+7)} - \bar{\alpha})^2\\
		= & \frac{49p^4 + 1831p^3 + 24605p^2 + 144494p + 315168}{588(p + 7)^2(13 + p)}.
	\end{align*}
	And $ \alpha_{\mu_f -2}-\alpha_1= 2 - \frac{3}{2(p+7)} - \frac{13}{14} = \frac{15}{14} - \frac{3}{2(p+7)}. $
	Then it suffices to prove
	\begin{align*}
		&\frac{49p^4 + 1831p^3 + 24605p^2 + 144494p + 315168}{588(p + 7)^2(13 + p)} \leq \frac{13+p}{12} \cdot (\frac{15}{14} - \frac{3}{2(p+7)}).
	\end{align*}
	Equivalently,
	\begin{align*}
		&\frac{-7p^4 - 391p^3 - 7049p^2 - 41615p - 65268}{1176(p + 7)^2(13 + p)} \leq 0,
	\end{align*}
	which holds obviously.

	If  $\tau^{\mathrm{max}} = 14 + p, $ 
	\begin{align*}
		&\bar{\alpha} =\frac{1}{\mu_f - 1}\cdot ( \frac{13}{14}+\frac{17}{14}+\frac{19}{14} + \sum_{ k=1 } ^{ \frac{p+8}{2}} (1+\frac{2k-1}{ 2(p+7)})+2 \cdot \frac{3}{2} + \frac{23}{14} + \frac{25}{14} +\sum_{ k=1 } ^{ \frac{p+8}{2}} (2-\frac{2k-1}{ 2(p+7)}))\\
		&=\frac{286 + 21p}{196 + 14p}.
	\end{align*}
	\begin{align*}
		& (\mu_f -  1) \cdot \mathrm{Var}[\mathrm{Sp^{\tau^{\mathrm{max}} }}(f)] \\
		&=(\frac{13}{14}-\bar{\alpha})^2+(\frac{17}{14}-\bar{\alpha})^2+(\frac{19}{14}-\bar{\alpha})^2+\sum_{ k=1 } ^{ \frac{p+7}{2}} (1+\frac{2k-1}{ 2(p+7)} - \bar{\alpha})^2 + 2\cdot( \frac{3}{2} - \bar{\alpha})^2\\
		&  + ( \frac{23}{14} - \bar{\alpha})^2+ (\frac{25}{14}- \bar{\alpha})^2 + \sum_{ k=1 } ^{ \frac{p+7}{2}} (2 - \frac{2k-1}{ 2(p+7)} - \bar{\alpha})^2\\
		= & \frac{49p^3 + 1684p^2 + 18316p + 62160}{588(p + 7)(14 + p)}.
	\end{align*}
	And $ \alpha_{\mu_f -1}-\alpha_1= 2 - \frac{1}{2(p+7)} - \frac{13}{14}. $
	Then it suffices to prove
	\begin{align*}
		&\frac{49p^3 + 1684p^2 + 18316p + 62160}{588(p + 7)(14 + p)} -\frac{14+p}{12} \cdot (2 - \frac{1}{2(p+7)} - \frac{13}{14}).
	\end{align*}
	Equivalently,
	\begin{align*}
		&\frac{-7p^3 - 258p^2 - 3156p - 10136}{1176(p + 7)(14 + p)} \leq 0,
	\end{align*}
	which holds obviously.
	%	Considering the average value $ \bar{\alpha}=\frac{3}{2}, $
	%	\begin{align*}
		%		& (\mu_f -  1) \cdot \mathrm{Var}[\mathrm{Sp^{\tau^{\mathrm{max}} }}(f)] \\
		%		&= (\frac{13}{14}-\bar{\alpha})^2+(\frac{17}{14}-\bar{\alpha})^2+(\frac{19}{14}-\bar{\alpha})^2+\sum_{ k=1 } ^{ \frac{p+8}{2}} (1+\frac{2k-1}{ 2(p+7)} - \bar{\alpha})^2  + 2\cdot( \frac{3}{2} - \bar{\alpha})^2\\
		%		& + ( \frac{23}{14} - \bar{\alpha})^2+ (\frac{25}{14}- \bar{\alpha})^2 + \sum_{ k=2 } ^{ \frac{p+8}{2}} (2 - \frac{2k-1}{ 2(p+7)} - \bar{\alpha})^2\\
		%		&= \frac{49p^3 + 1194p^2 + 9758p + 26460}{588(p + 7)^2}.
		%	\end{align*}
	%	Then it suffices to prove
	%	\begin{align*}
		%		&\frac{49p^3 + 1194p^2 + 9758p + 26460}{588(p + 7)^2} \leq \frac{13+p}{12} \cdot (\frac{15}{14} - \frac{3}{2(p+7)}).
		%	\end{align*}
	%	Equivalently,
	%	\begin{align*}
		%		&\frac{-7p^3 - 300p^2 - 1799p - 588}{1176(p + 7)^2} \leq 0,
		%	\end{align*}
	%	which holds obviously.
	%	
	
	\noindent\textbf{(5) $W_{1,0}$.}
	
	Since $  \tau_{f_0}  = \mu_f - 1 = 14, $ we only need to verify the case of  $\tau^{\mathrm{max}} = \mu_f -1 = 14 . $ 
	\begin{align*}
		& \bar{\alpha} =\frac{1}{\mu_f -1}\cdot ( \frac{5}{12}+\frac{7}{12}+\frac{2}{3} +\frac{3}{4} + \frac{5}{6} + 2 \cdot \frac{11}{12} + 
		1 + 2 \cdot \frac{13}{12} + \frac{7}{6} + \frac{5}{4} + \frac{4}{3} + \frac{17}{12}  ) =\frac{23}{24}.
	\end{align*}
	\begin{align*}
		& (\mu_f -  1) \cdot \mathrm{Var}[\mathrm{Sp^{\tau^{\mathrm{max}} }}(f)] \\
		=& (\frac{5}{12}-\bar{\alpha})^2+(\frac{7}{12}-\bar{\alpha})^2+ (\frac{2}{3}-\bar{\alpha})^2+ (\frac{3}{4} - \bar{\alpha})^2 + ( \frac{5}{6} -\bar{\alpha})^2 + 2\cdot(\frac{11}{12} - \bar{\alpha})^2 + ( 1 - \bar{\alpha})^2\\
		&+2\cdot (\frac{13}{12} - \bar{\alpha})^2  + ( \frac{7}{6} - \bar{\alpha})^2 + ( \frac{5}{4} - \bar{\alpha})^2 + ( \frac{4}{3} - \bar{\alpha})^2 + ( \frac{17}{12} - \bar{\alpha})^2 \\
		=& \frac{629}{576}  .
	\end{align*}
	
	And $ \alpha_{\mu_f -1}-\alpha_1= \frac{17}{12} - \frac{5}{12} =1 .  $
	Then it suffices to prove $ \frac{629}{576} - \frac{14}{12} \cdot 1 \leq 0,  $ i.e., $ - \frac{43}{576} \leq 0, $
	which holds obviously.
	%	
	%	
	%	Considering the average value $ \bar{\alpha}=1, $
	%	\begin{align*}
		%		& (\mu_f -  1) \cdot \mathrm{Var}[\mathrm{Sp^{\tau^{\mathrm{max}} }}(f)] \\
		%		=& (\frac{5}{12}-\bar{\alpha})^2+(\frac{7}{12}-\bar{\alpha})^2+ (\frac{2}{3}-\bar{\alpha})^2+ (\frac{3}{4} - \bar{\alpha})^2 + ( \frac{5}{6} -\bar{\alpha})^2 + 2\cdot(\frac{11}{12} - \bar{\alpha})^2 + ( 1 - \bar{\alpha})^2\\
		%		&+2\cdot (\frac{13}{12} - \bar{\alpha})^2  + ( \frac{7}{6} - \bar{\alpha})^2 + ( \frac{5}{4} - \bar{\alpha})^2 + ( \frac{4}{3} - \bar{\alpha})^2 + ( \frac{17}{12} - \bar{\alpha})^2 \\
		%		=& \frac{161}{144}  .
		%	\end{align*}
	%	Then it suffices to prove $\frac{161}{144}  - \frac{14}{12} \cdot 1 \leq 0,  $ i.e., $ - \frac{7}{144} \leq 0, $
	%	which holds obviously.
	
	\noindent\textbf{(6) $W_{1,p}$.}

	Since $  \tau_{f_0}  = \mu_f - 2 = 13+p, $ we need to verify the cases of $\tau^{\mathrm{max}} = 13 + p $ or $ 14 + p. $ For simplicity, we only verify the cases of $ p$ even.
	For $ p $ even, if  $\tau^{\mathrm{max}} = 13 + p, $ 
	\begin{align*}
		&\bar{\alpha} =\frac{1}{\mu_f -2 }\cdot ( \frac{11}{12}+\frac{14}{12}+\frac{16}{12} + \frac{17}{12} + \sum_{ k=1 } ^{ \frac{p+6}{2}} (1+\frac{2k-1}{ 2(p+6)})+ \frac{3}{2} + \frac{19}{12} + \frac{20}{12} + \frac{22}{12}+ \\
		&\sum_{ k=2 } ^{ \frac{p+6}{2}} (2-\frac{2k-1}{ 2(p+6)}))\\
		=&\frac{18p^2 + 329p + 1332}{12(p + 6)(13 + p)}.
	\end{align*}
	\begin{align*}
		& (\mu_f -  2) \cdot \mathrm{Var}[\mathrm{Sp^{\tau^{\mathrm{max}} }}(f)] \\
		= & (\frac{11}{12}-\bar{\alpha})^2+(\frac{14}{12}-\bar{\alpha})^2+(\frac{16}{12}-\bar{\alpha})^2+ (\frac{17}{12} -\bar{\alpha})^2 +\sum_{ k=1 } ^{ \frac{p+6}{2}} (1+\frac{2k-1}{ 2(p+6)} - \bar{\alpha})^2\\
		& +(\frac{3}{2} - \bar{\alpha})^2 + ( \frac{19}{12} - \bar{\alpha})^2+ (\frac{20}{12}- \bar{\alpha})^2 + (\frac{22}{12} - \bar{\alpha})^2 + \sum_{ k=2 } ^{ \frac{p+6}{2}} (2 - \frac{2k-1}{ 2(p+6)} - \bar{\alpha})^2\\
		= & \frac{12p^4 + 427p^3 + 5370p^2 + 29232p + 58464}{144(p + 6)^2(13 + p)}.
	\end{align*}
	And   
	\begin{equation*}
		\alpha_{\mu_f -2}-\alpha_1 =
		\begin{cases}
			\frac{22}{12}-\frac{11}{12}=\frac{11}{12}, \qquad \forall 0 \leq p \leq 2 \\
			2- \frac{3}{2(p+6)}-\frac{11}{12} = \frac{13}{12} - \frac{3}{2(p+6)}, \qquad \forall p \geq 3
		\end{cases}.
	\end{equation*}
	Then it suffices to prove
	\begin{equation*}
		\begin{cases}
			\frac{12p^4 + 427p^3 + 5370p^2 + 29232p + 58464}{144(p + 6)^2(13 + p)}-\frac{13+p}{12} \cdot \frac{11}{12} \leq 0, \qquad \forall 0 \leq p \leq 2\\
			\frac{12p^4 + 427p^3 + 5370p^2 + 29232p + 58464}{144(p + 6)^2(13 + p)}-\frac{13+p}{12}\cdot ( \frac{13}{12} - \frac{3}{2(p+6)})\leq 0, \qquad \forall p \geq 3
		\end{cases}.
	\end{equation*}
	Equivalently,
	\begin{equation*}
		\begin{cases}
			\frac{p^4 + 9p^3 - 317p^2 - 3372p - 8460}{144(p + 6)^2(13 + p)} \leq 0, \qquad \forall 0 \leq p \leq 2 \\
			\frac{-p^4 - 49p^3 - 775p^2 - 3450p - 2376}{144(p + 6)^2(13 + p)} \leq 0, \qquad \forall p \geq 3
		\end{cases},
	\end{equation*}
	which hold obviously.

	If  $\tau^{\mathrm{max}} = 14 + p, $ 
	\begin{align*}
		&\bar{\alpha} =\frac{1}{\mu_f -1 }\cdot ( \frac{11}{12}+\frac{14}{12}+\frac{16}{12} + \frac{17}{12} + \sum_{ k=1 } ^{ \frac{p+6}{2}} (1+\frac{2k-1}{ 2(p+6)})+ \frac{3}{2} + \frac{19}{12} + \frac{20}{12} + \frac{22}{12}+ \\
		&\sum_{ k=1 } ^{ \frac{p+6}{2}} (2-\frac{2k-1}{ 2(p+6)}))\\
		=&\frac{245 + 18p}{168 + 12p}.
	\end{align*}
	\begin{align*}
		& (\mu_f -  1) \cdot \mathrm{Var}[\mathrm{Sp^{\tau^{\mathrm{max}} }}(f)] \\
		= & (\frac{11}{12}-\bar{\alpha})^2+(\frac{14}{12}-\bar{\alpha})^2+(\frac{16}{12}-\bar{\alpha})^2+ (\frac{17}{12} -\bar{\alpha})^2 +\sum_{ k=1 } ^{ \frac{p+6}{2}} (1+\frac{2k-1}{ 2(p+6)} - \bar{\alpha})^2\\
		& +(\frac{3}{2} - \bar{\alpha})^2 + ( \frac{19}{12} - \bar{\alpha})^2+ (\frac{20}{12}- \bar{\alpha})^2 + (\frac{22}{12} - \bar{\alpha})^2 + \sum_{ k=1 } ^{ \frac{p+6}{2}} (2 - \frac{2k-1}{ 2(p+6)} - \bar{\alpha})^2\\
		= & \frac{12p^3 + 403p^2 + 4207p + 13230}{144(14 + p)(p + 6)}.
	\end{align*}
	And $ \alpha_{\mu_f -1}-\alpha_1 =2- \frac{1}{2(p+6)}-\frac{11}{12}. $
	Then it suffices to prove
	\begin{equation*}
		\frac{12p^3 + 403p^2 + 4207p + 13230}{144(14 + p)(p + 6)}-\frac{14+p}{12} \cdot (2- \frac{1}{2(p+6)}-\frac{11}{12}) \leq 0.
	\end{equation*}
	Equivalently,
	\begin{equation*}
		\frac{-p^3 - 33p^2 - 357p - 882}{144(14 + p)(p + 6)} \leq 0,
	\end{equation*}
	which holds obviously.
	%	Considering the average value $ \bar{\alpha}=\frac{3}{2}, $
	%	\begin{align*}
		%		& (\mu_f -  1) \cdot \mathrm{Var}[\mathrm{Sp^{\tau^{\mathrm{max}} }}(f)] \\
		%		= & (\frac{11}{12}-\bar{\alpha})^2+(\frac{14}{12}-\bar{\alpha})^2+(\frac{16}{12}-\bar{\alpha})^2+ (\frac{17}{12} -\bar{\alpha})^2 +\sum_{ k=1 } ^{ \frac{p+6}{2}} (1+\frac{2k-1}{ 2(p+6)} - \bar{\alpha})^2\\
		%		& + (\frac{3}{2} - \bar{\alpha})^2 + ( \frac{19}{12} - \bar{\alpha})^2+ (\frac{20}{12}- \bar{\alpha})^2 + (\frac{22}{12} - \bar{\alpha})^2 + \sum_{ k=2 } ^{ \frac{p+6}{2}} (2 - \frac{2k-1}{ 2(p+6)} - \bar{\alpha})^2\\
		%		= & \frac{12p^3 + 271p^2 + 2016p + 4896}{144(p + 6)^2}.
		%	\end{align*}
	%	
	%	Then it suffices to prove
	%	\begin{equation*}
		%		\begin{cases}
			%			\frac{12p^3 + 271p^2 + 2016p + 4896}{144(p + 6)^2}-\frac{13+p}{12} \cdot \frac{11}{12} \leq 0, \qquad \forall 0 \leq p \leq 3 \\
			%			\frac{12p^3 + 271p^2 + 2016p + 4896}{144(p + 6)^2}-\frac{13+p}{12}\cdot ( \frac{13}{12} - \frac{3}{2(p+6)})\leq 0, \qquad \forall p \geq 4 
			%		\end{cases}.
		%	\end{equation*}
	%	Equivalently,
	%	\begin{equation*}
		%		\begin{cases}
			%			\frac{p^3 - 4p^2 - 96p - 252}{144(p + 6)^2} \leq 0, \qquad \forall 0 \leq p \leq 3 \\
			%			\frac{-p^3 - 36p^2 - 138p + 216}{144(p + 6)^2} \leq 0, \qquad \forall p \geq 4
			%		\end{cases},
		%	\end{equation*}
	%	which hold obviously.
	%	
	
	\noindent\textbf{(7) $Q_{2,0}$.}
	
	Since $  \tau_{f_0} = \mu_f - 1 = 13, $ we only need to verify the case of  $\tau^{\mathrm{max}} = \mu_f -1 = 13 . $ 
	\begin{align*}
		& \bar{\alpha} =\frac{1}{\tau}\cdot ( \frac{11}{12}+\frac{13}{12}+2 \cdot \frac{5}{4} + \frac{4}{3} + 2 \cdot \frac{17}{12} + 
		2 \cdot \frac{19}{12} + \frac{5}{3} + 2 \cdot \frac{7}{4} + \frac{23}{12})=\frac{227}{156}.
	\end{align*}
	\begin{align*}
		& (\mu_f -  1) \cdot \mathrm{Var}[\mathrm{Sp^{\tau^{\mathrm{max}} }}(f)] \\
		=& (\frac{11}{12}-\bar{\alpha})^2+(\frac{13}{12}-\bar{\alpha})^2+ 2\cdot (\frac{5}{4}-\bar{\alpha})^2+ (\frac{4}{3} - \bar{\alpha})^2 + 2 \cdot ( \frac{17}{12} -\bar{\alpha})^2 +2\cdot (\frac{19}{12} - \bar{\alpha})^2  \\
		& +( \frac{5}{3} - \bar{\alpha})^2 + 2\cdot( \frac{7}{4} - \bar{\alpha})^2 + ( \frac{23}{12} - \bar{\alpha})^2 \\
		=& \frac{931}{936}  .
	\end{align*}
	And $ \alpha_{\mu_f -1}-\alpha_1= \frac{23}{12} - \frac{11}{12} =1 .  $
	Then it suffices to prove $ \frac{931}{936} - \frac{13}{12} \cdot 1 \leq 0,  $ i.e., $ - \frac{83}{936} \leq 0, $
	which holds obviously.
	%	
	%	
	%	Considering the average value $ \bar{\alpha}=1, $
	%	\begin{align*}
		%		& (\mu_f -  1) \cdot \mathrm{Var}[\mathrm{Sp^{\tau^{\mathrm{max}} }}(f)] \\
		%		=& (\frac{11}{12}-\bar{\alpha})^2+(\frac{13}{12}-\bar{\alpha})^2+ 2\cdot (\frac{5}{4}-\bar{\alpha})^2+ (\frac{4}{3} - \bar{\alpha})^2 + 2 \cdot ( \frac{17}{12} -\bar{\alpha})^2\\
		%		&+2\cdot (\frac{19}{12} - \bar{\alpha})^2  + ( \frac{5}{3} - \bar{\alpha})^2 + 2\cdot( \frac{7}{4} - \bar{\alpha})^2 + ( \frac{23}{12} - \bar{\alpha})^2 \\
		%		=& \frac{49}{48}  .
		%	\end{align*}
	%	Then it suffices to prove $\frac{49}{48}  - \frac{13}{12} \cdot 1 \leq 0,  $ i.e., $ - \frac{1}{16} \leq 0, $
	%	which holds obviously.
	%	
	
	\noindent\textbf{(8) $Q_{2,p}$.}

	Since $  \tau_{f_0} = \mu_f - 2 = 12+p, $ we need to verify the cases of $\tau^{\mathrm{max}} = 12 + p $ or $ 13 + p. $ For simplicity, we only verify the cases of $ p$ even.
	For $ p $ even, if  $\tau^{\mathrm{max}} = 12 + p, $ 
	\begin{align*}
		& \bar{\alpha} =\frac{1}{\mu_f - 2}\cdot ( \frac{11}{12}+\frac{15}{12}+\frac{16}{12} + \frac{17}{12} + \sum_{ k=1 } ^{ \frac{p+6}{2}} (1+\frac{2k-1}{ 2(p+6)}) +  \frac{19}{12} + \frac{20}{12} + \frac{21}{12} \\
		&+ \sum_{ k=2 } ^{ \frac{p+6}{2}} (2-\frac{2k-1}{ 2(p+6)}))\\
		&=\frac{18p^2 + 311p + 1224}{12(p + 6)(12 + p)}.
	\end{align*}
	\begin{align*}
		& (\mu_f -  2) \cdot \mathrm{Var}[\mathrm{Sp^{\tau^{\mathrm{max}} }}(f)] \\
		= & (\frac{11}{12}-\bar{\alpha})^2+(\frac{15}{12}-\bar{\alpha})^2+(\frac{16}{12} - \bar{\alpha})^2 + (\frac{17}{12}-\bar{\alpha})^2+\sum_{ k=1 } ^{ \frac{p+6}{2}} (1+\frac{2k-1}{ 2(p+6)} - \bar{\alpha})^2\\
		& + ( \frac{19}{12} - \bar{\alpha})^2+ (\frac{20}{12}- \bar{\alpha})^2 + (\frac{21}{12} - \bar{\alpha})^2 + \sum_{ k=2 } ^{ \frac{p+6}{2}} (2 - \frac{2k-1}{ 2(p+6)} - \bar{\alpha})^2\\
		= & \frac{12p^4 + 401p^3 + 4763p^2 + 24696p + 47520}{144(p + 6)^2(12 + p)}.
	\end{align*}
	
	And $ \alpha_{\mu_f -2}-\alpha_1= 2 - \frac{3}{2(p+6)} - \frac{11}{12} = \frac{13}{12} - \frac{3}{2(p+6)}. $
	Then it suffices to prove
	\begin{align*}
		&	\frac{12p^4 + 401p^3 + 4763p^2 + 24696p + 47520}{144(p + 6)^2(12 + p)} -\frac{12+p}{12} \cdot (\frac{13}{12} - \frac{3}{2(p+6)}) \leq 0.
	\end{align*}
	Equivalently,
	\begin{align*}
		&	\frac{-p^4 - 49p^3 - 781p^2 - 3816p - 4320}{144(p + 6)^2(12 + p)} \leq 0,
	\end{align*}
	which holds obviously.

	If $\tau^{\mathrm{max}} = 13 + p, $ 
	\begin{align*}
		& \bar{\alpha} =\frac{1}{\mu_f - 1}\cdot ( \frac{11}{12}+\frac{15}{12}+\frac{16}{12} + \frac{17}{12} + \sum_{ k=1 } ^{ \frac{p+6}{2}} (1+\frac{2k-1}{ 2(p+6)}) +  \frac{19}{12} + \frac{20}{12} + \frac{21}{12} \\
		&+ \sum_{ k=1 } ^{ \frac{p+6}{2}} (2-\frac{2k-1}{ 2(p+6)}))\\
		&=\frac{227 + 18p}{156 + 12p}.
	\end{align*}
	\begin{align*}
		& (\mu_f -  1) \cdot \mathrm{Var}[\mathrm{Sp^{\tau^{\mathrm{max}} }}(f)] \\
		= & (\frac{11}{12}-\bar{\alpha})^2+(\frac{15}{12}-\bar{\alpha})^2+(\frac{16}{12} - \bar{\alpha})^2 + (\frac{17}{12}-\bar{\alpha})^2+\sum_{ k=1 } ^{ \frac{p+6}{2}} (1+\frac{2k-1}{ 2(p+6)} - \bar{\alpha})^2\\
		& + ( \frac{19}{12} - \bar{\alpha})^2+ (\frac{20}{12}- \bar{\alpha})^2 + (\frac{21}{12} - \bar{\alpha})^2 + \sum_{ k=1 } ^{ \frac{p+6}{2}} (2 - \frac{2k-1}{ 2(p+6)} - \bar{\alpha})^2\\
		= & \frac{12p^3 + 377p^2 + 3706p + 11172}{144(p + 6)(13 + p)}.
	\end{align*}
	
	And $ \alpha_{\mu_f -1}-\alpha_1= 2 - \frac{1}{2(p+6)} - \frac{11}{12}. $
	Then it suffices to prove
	\begin{equation*}
		\frac{12p^3 + 377p^2 + 3706p + 11172}{144(p + 6)(13 + p)} - \frac{13+p}{12} \cdot ( 2 - \frac{1}{2(p+6)} - \frac{11}{12} ) \leq 0.
	\end{equation*}
	Equivalently,
	\begin{equation*}
		\frac{-p^3 - 33p^2 - 363p - 996}{144(p + 6)(13 + p)}  \leq 0,
	\end{equation*}
	which holds obviously.
	
	%	Considering the average value $ \bar{\alpha}=\frac{3}{2}, $
	%	\begin{align*}
		%		& (\mu_f -  1) \cdot \mathrm{Var}[\mathrm{Sp^{\tau^{\mathrm{max}} }}(f)] \\
		%		= & (\frac{11}{12}-\bar{\alpha})^2+(\frac{15}{12}-\bar{\alpha})^2+(\frac{16}{12} - \bar{\alpha})^2 + (\frac{17}{12}-\bar{\alpha})^2+\sum_{ k=1 } ^{ \frac{p+6}{2}} (1+\frac{2k-1}{ 2(p+6)} - \bar{\alpha})^2\\
		%		&  + ( \frac{19}{12} - \bar{\alpha})^2+ (\frac{20}{12}- \bar{\alpha})^2 + (\frac{21}{12} - \bar{\alpha})^2 + \sum_{ k=2 } ^{ \frac{p+6}{2}} (2 - \frac{2k-1}{ 2(p+6)} - \bar{\alpha})^2\\
		%		= & \frac{12p^3 + 257p^2 + 1848p + 4392}{144(p + 6)^2}.
		%	\end{align*}
	%	Then it suffices to prove
	%	\begin{align*}
		%		&	\frac{12p^3 + 257p^2 + 1848p + 4392}{144(p + 6)^2} - \frac{12+p}{12} \cdot (\frac{13}{12} - \frac{3}{2(p+6)}) \leq 0.
		%	\end{align*}
	%	Equivalently,
	%	\begin{align*}
		%		&\frac{-p^3 - 37p^2 - 168p + 72}{144(p + 6)^2} \leq 0,
		%	\end{align*}
	%	which holds obviously.
	%	
	%	
	
	\noindent\textbf{(9) $S_{1,0}$.}
	
	Since $  \tau_{f_0} = \mu_f - 1 = 13, $ we only need to verify the case of  $\tau^{\mathrm{max}} = \mu_f -1 = 13 . $ 
	\begin{align*}
		& \bar{\alpha} =\frac{1}{\tau}\cdot ( \frac{9}{10}+\frac{11}{10}+\frac{6}{5} + 2 \cdot \frac{13}{10} + \frac{7}{5} + 2 \cdot \frac{3}{2}
		+\frac{8}{5} + 2 \cdot \frac{17}{10}  + \frac{9}{5} + \frac{19}{10} )=\frac{189}{130}.
	\end{align*}
	\begin{align*}
		& (\mu_f -  1) \cdot \mathrm{Var}[\mathrm{Sp^{\tau^{\mathrm{max}} }}(f)] \\
		=& (\frac{9}{10}-\bar{\alpha})^2+(\frac{11}{10}-\bar{\alpha})^2+ (\frac{6}{5}-\bar{\alpha})^2+2\cdot (\frac{13}{10} - \bar{\alpha})^2 +( \frac{7}{5} -\bar{\alpha})^2 +2 \cdot ( \frac{3}{2} - \bar{\alpha})^2\\
		&+\cdot (\frac{8}{5} - \bar{\alpha})^2  +2 \cdot ( \frac{17}{10} - \bar{\alpha})^2 + ( \frac{9}{5} - \bar{\alpha})^2 + ( \frac{19}{10} - \bar{\alpha})^2 \\
		=& \frac{329}{325}  .
	\end{align*}
	And $ \alpha_{\mu_f -1}-\alpha_1= \frac{19}{10} - \frac{9}{10} =1 .  $
	Then it suffices to prove $ \frac{329}{325} - \frac{13}{12} \cdot 1 \leq 0,  $ i.e., $ - \frac{277}{3900} \leq 0, $
	which holds obviously.
	%	
	%	
	%	Considering the average value $ \bar{\alpha}=\frac{3}{2}, $
	%	\begin{align*}
		%		& (\mu_f -  1) \cdot \mathrm{Var}[\mathrm{Sp^{\tau^{\mathrm{max}} }}(f)] \\
		%		=& (\frac{9}{10}-\bar{\alpha})^2+(\frac{11}{10}-\bar{\alpha})^2+ (\frac{6}{5}-\bar{\alpha})^2+2\cdot (\frac{13}{10} - \bar{\alpha})^2 +( \frac{7}{5} -\bar{\alpha})^2 +2 \cdot ( \frac{3}{2} - \bar{\alpha})^2\\
		%		&+\cdot (\frac{8}{5} - \bar{\alpha})^2  +2 \cdot ( \frac{17}{10} - \bar{\alpha})^2 + ( \frac{9}{5} - \bar{\alpha})^2 + ( \frac{19}{10} - \bar{\alpha})^2 \\
		%		=& \frac{26}{25}  .
		%	\end{align*}
	%	Then it suffices to prove $ \frac{26}{25} - \frac{13}{12} \cdot 1 \leq 0,  $ i.e., $ - \frac{13}{300} \leq 0, $
	%	which holds obviously.
	%	
	
	\noindent\textbf{(10) $S_{1,p}$.}

	Since $ \tau_{f_0} = \mu_f - 2 = 12+p, $ we need to verify the cases of $\tau^{\mathrm{max}} = 12 + p $ or $ 13 + p. $ For simplicity, we only verify the cases of $ p$ even.
	For $ p $ even, if  $\tau^{\mathrm{max}} = 12 + p, $ 
	\begin{align*}
		& \bar{\alpha} =\frac{1}{\mu_f -2 }\cdot ( \frac{9}{10}+\frac{12}{10}+\frac{13}{10} + \frac{14}{10} + \sum_{ k=1 } ^{ \frac{p+6}{2}} (1+\frac{2k-1}{ 2(p+5)}) + \frac{16}{10} + \frac{17}{10} + \frac{18}{10} \\
		&+ \sum_{ k=2 } ^{ \frac{p+6}{2}} (2-\frac{2k-1}{ 2(p+5)}))\\
		&=\frac{15p^2 + 244p + 850}{10(p + 5)(12 + p)}.
	\end{align*}
	\begin{align*}
		& (\mu_f -  2) \cdot \mathrm{Var}[\mathrm{Sp^{\tau^{\mathrm{max}} }}(f)] \\
		= & (\frac{9}{10}-\bar{\alpha})^2+(\frac{12}{10}-\bar{\alpha})^2+(\frac{13}{10} - \bar{\alpha})^2 + (\frac{14}{10}-\bar{\alpha})^2+\sum_{ k=1 } ^{ \frac{p+6}{2}} (1+\frac{2k-1}{ 2(p+5)} - \bar{\alpha})^2\\
		& + ( \frac{16}{10} - \bar{\alpha})^2+ (\frac{17}{10}- \bar{\alpha})^2 + (\frac{18}{10} - \bar{\alpha})^2 + \sum_{ k=2 } ^{ \frac{p+6}{2}} (2 - \frac{2k-1}{ 2(p+5)} - \bar{\alpha})^2\\
		= & \frac{25p^4 + 792p^3 + 8711p^2 + 41340p + 71700}{300(p + 5)^2(12 + p)}.
	\end{align*}
	And $ \alpha_{\mu_f -2}-\alpha_1= 2 - \frac{3}{2(p+5)} - \frac{9}{10} = \frac{11}{10} - \frac{3}{2(p+5)}. $
	Then it suffices to prove
	\begin{align*}
		&\frac{25p^4 + 792p^3 + 8711p^2 + 41340p + 71700}{300(p + 5)^2(12 + p)}- \frac{12+p}{12} \cdot (\frac{11}{10} - \frac{3}{2(p+5)})\leq 0.
	\end{align*}
	Equivalently,
	\begin{align*}
		&\frac{-5p^4 - 211p^3 - 2898p^2 - 9720p - 600}{600(p + 5)^2(12 + p)} \leq 0,
	\end{align*}
	which holds obviously.

	If  $\tau^{\mathrm{max}} = 13 + p, $ 
	\begin{align*}
		& \bar{\alpha} =\frac{1}{\mu_f -1 }\cdot ( \frac{9}{10}+\frac{12}{10}+\frac{13}{10} + \frac{14}{10} + \sum_{ k=1 } ^{ \frac{p+6}{2}} (1+\frac{2k-1}{ 2(p+5)}) + \frac{16}{10} + \frac{17}{10} + \frac{18}{10} \\
		&+ \sum_{ k=1 } ^{ \frac{p+6}{2}} (2-\frac{2k-1}{ 2(p+5)}))\\
		&=\frac{189 + 15p}{130 + 10p}.
	\end{align*}
	\begin{align*}
		& (\mu_f -  1) \cdot \mathrm{Var}[\mathrm{Sp^{\tau^{\mathrm{max}} }}(f)] \\
		= & (\frac{9}{10}-\bar{\alpha})^2+(\frac{12}{10}-\bar{\alpha})^2+(\frac{13}{10} - \bar{\alpha})^2 + (\frac{14}{10}-\bar{\alpha})^2+\sum_{ k=1 } ^{ \frac{p+6}{2}} (1+\frac{2k-1}{ 2(p+5)} - \bar{\alpha})^2\\
		& + ( \frac{16}{10} - \bar{\alpha})^2+ (\frac{17}{10}- \bar{\alpha})^2 + (\frac{18}{10} - \bar{\alpha})^2 + \sum_{ k=1 } ^{ \frac{p+6}{2}} (2 - \frac{2k-1}{ 2(p+5)} - \bar{\alpha})^2\\
		= & \frac{25p^3 + 767p^2 + 7198p + 19740}{300(p + 5)(13 + p)}.
	\end{align*}
	And $ \alpha_{\mu_f -1}-\alpha_1= 2 - \frac{1}{2(p+5)} - \frac{9}{10} . $
	Then it suffices to prove
	\begin{align*}
		&\frac{25p^3 + 767p^2 + 7198p + 19740}{300(p + 5)(13 + p)}- \frac{13+p}{12} \cdot (2 - \frac{1}{2(p+5)} - \frac{9}{10} )\leq 0.
	\end{align*}
	Equivalently,
	\begin{align*}
		&\frac{-5p^3 - 146p^2 - 1399p - 2770}{600(p + 5)(13 + p)} \leq 0,
	\end{align*}
	which holds obviously.

	\noindent\textbf{(11) $S_{1,p} ^\#$.}

	Since $  \tau_{f_0} = \mu_f - 2 = 12 + p, $ we need to verify the cases of $\tau^{\mathrm{max}} = 12 + p$ or $ 13 + p. $ 
	If  $\tau^{\mathrm{max}} = 12 + p, $ 
	\begin{align*}
		& \bar{\alpha} =\frac{1}{\mu_f -2}\cdot ( \frac{9}{10}+\frac{13}{10} + \sum_{ k=1 } ^{\frac{p}{2}+5} (1+\frac{k}{ p+10}) + \frac{17}{10} + \sum_{ k=2 } ^{\frac{p}{2}+5} (2-\frac{k}{ p+10}))\\
		& =\frac{15p^2 + 319p + 1700}{10(p + 10)(12 + p)}.
	\end{align*}
	\begin{align*}
		& (\mu_f -  2) \cdot \mathrm{Var}[\mathrm{Sp^{\tau^{\mathrm{max}} }}(f)] \\
		= & (\frac{9}{10}-\bar{\alpha})^2+(\frac{13}{10}-\bar{\alpha})^2+\sum_{ k=1 } ^{\frac{p}{2}+5} (1+\frac{k}{ p + 10} - \bar{\alpha})^2  + (\frac{17}{10} - \bar{\alpha})^2 + \sum_{ k=2 } ^{ \frac{p}{2} +5} (2 - \frac{k}{ p + 10} - \bar{\alpha})^2\\
		= & \frac{25p^4 + 1032p^3 + 15911p^2 + 109680p + 286800}{300(12 + p)(p + 10)^2}.
	\end{align*}
	And $ \alpha_{\mu_f -2}-\alpha_1= 2 - \frac{2}{p + 10} - \frac{9}{10} = \frac{11}{10} - \frac{2}{p + 10}. $
	Then it suffices to prove
	\begin{align*}
		& \frac{25p^4 + 1032p^3 + 15911p^2 + 109680p + 286800}{300(12 + p)(p + 10)^2} -\frac{12 + p}{12} \cdot (\frac{11}{10} - \frac{2}{p + 10}) \leq 0.
	\end{align*}
	Equivalently,
	\begin{align*}
		&\frac{-5p^4 - 256p^3 - 4598p^2 - 32640p - 74400}{600(12 + p)(p + 10)^2} \leq 0,
	\end{align*}
	which holds obviously.
	
	If  $\tau^{\mathrm{max}} = 13 + p, $ 
	\begin{align*}
		& \bar{\alpha} =\frac{1}{\mu_f -1}\cdot ( \frac{9}{10}+\frac{13}{10} + \sum_{ k=1 } ^{\frac{p}{2}+5} (1+\frac{k}{ p+10}) + \frac{17}{10} + \sum_{ k=1 } ^{\frac{p}{2}+5} (2-\frac{k}{ p+10}))\\
		& =\frac{189 + 15p}{130 + 10p}.
	\end{align*}
	\begin{align*}
		& (\mu_f - 1) \cdot \mathrm{Var}[\mathrm{Sp^{\tau^{\mathrm{max}} }}(f)] \\
		= & (\frac{9}{10}-\bar{\alpha})^2+(\frac{13}{10}-\bar{\alpha})^2+\sum_{ k=1 } ^{\frac{p}{2}+5} (1+\frac{k}{ p + 10} - \bar{\alpha})^2  + (\frac{17}{10} - \bar{\alpha})^2 + \sum_{ k=1 } ^{ \frac{p}{2} +5} (2 - \frac{k}{ p + 10} - \bar{\alpha})^2\\
		= & \frac{25p^3 + 882p^2 + 10253p + 39480}{300(p + 10)(13 + p)}.
	\end{align*}
	And $ \alpha_{\mu_f -1}-\alpha_1= 2 - \frac{1}{p + 10} - \frac{9}{10}. $
	Then it suffices to prove
	\begin{align*}
		& \frac{25p^3 + 882p^2 + 10253p + 39480}{300(p + 10)(13 + p)} -\frac{13+ p}{12} \cdot (2 - \frac{1}{p + 10} - \frac{9}{10}) \leq 0.
	\end{align*}
	Equivalently,
	\begin{align*}
		&\frac{-5p^3 - 166p^2 - 1789p - 5540}{600(p + 10)(13 + p)} \leq 0,
	\end{align*}
	which holds obviously.

	%	Considering the average value $ \bar{\alpha}=\frac{3}{2}, $
	%	\begin{align*}
		%		& (\mu_f -  1) \cdot \mathrm{Var}[\mathrm{Sp^{\tau^{\mathrm{max}} }}(f)] \\
		%		= & (\frac{9}{10}-\bar{\alpha})^2+(\frac{13}{10}-\bar{\alpha})^2+\sum_{ k=1 } ^{q+5} (1+\frac{k}{ 2q + 10} - \bar{\alpha})^2  + (\frac{17}{10} - \bar{\alpha})^2+ \sum_{ k=2 } ^{ q +5} (2 - \frac{k}{ 2q + 10} - \bar{\alpha})^2\\
		%		= & \frac{50q^3 + 732q^2 + 3745q + 6600}{300(q + 5)^2}.
		%	\end{align*}
	%	Then it suffices to prove
	%	\begin{align*}
		%		&\frac{50q^3 + 732q^2 + 3745q + 6600}{300(q + 5)^2} -\frac{12 + 2q}{12} \cdot (\frac{11}{10} - \frac{2}{2q + 10}) \leq 0.
		%	\end{align*}
	%	Equivalently,
	%	\begin{align*}
		%		&\frac{-5q^3 - 98q^2 - 380q - 150}{300(q + 5)^2} \leq 0,
		%	\end{align*}
	%	which holds obviously.
	%	
	
	\noindent\textbf{(12) $U_{1,0}$.}

	Since $  \tau_{f_0}  = \mu_f - 1 = 13, $ we only need to verify the case of  $\tau^{\mathrm{max}} = \mu_f -1 = 13 . $ 
	\begin{align*}
		& \bar{\alpha} =\frac{1}{\mu_f -1}\cdot ( \frac{8}{9}+\frac{10}{9}+2 \cdot \frac{11}{9} + \frac{12}{9} + 2 \cdot \frac{13}{9} + 
		2 \cdot \frac{14}{9} + \frac{15}{9} + 2 \cdot \frac{16}{9} + \frac{17}{9}) =\frac{170}{117}.
	\end{align*}
	\begin{align*}
		& (\mu_f -  1) \cdot \mathrm{Var}[\mathrm{Sp^{\tau^{\mathrm{max}} }}(f)] \\
		=& (\frac{8}{9}-\bar{\alpha})^2+(\frac{10}{9}-\bar{\alpha})^2+ 2\cdot (\frac{11}{9}-\bar{\alpha})^2+ (\frac{12}{9} - \bar{\alpha})^2 + 2 \cdot ( \frac{13}{9} -\bar{\alpha})^2\\
		&+2\cdot (\frac{14}{9} - \bar{\alpha})^2  + ( \frac{15}{9} - \bar{\alpha})^2 + 2\cdot( \frac{16}{9} - \bar{\alpha})^2 + ( \frac{17}{9} - \bar{\alpha})^2 \\
		=& \frac{1078}{1053}  .
	\end{align*}
	And $ \alpha_{\mu_f -1}-\alpha_1= \frac{17}{9} - \frac{8}{9} =1 .  $
	Then it suffices to prove $  \frac{1078}{1053} - \frac{13}{12} \cdot 1 \leq 0,  $ i.e., $ - \frac{251}{4212} \leq 0, $
	which holds obviously.
	%	
	%	
	%	Considering the average value $ \bar{\alpha}=\frac{3}{2}, $
	%	\begin{align*}
		%		& (\mu_f -  1) \cdot \mathrm{Var}[\mathrm{Sp^{\tau^{\mathrm{max}} }}(f)] \\
		%		=& (\frac{8}{9}-\bar{\alpha})^2+(\frac{10}{9}-\bar{\alpha})^2+ 2\cdot (\frac{11}{9}-\bar{\alpha})^2+ (\frac{12}{9} - \bar{\alpha})^2 + 2 \cdot ( \frac{13}{9} -\bar{\alpha})^2\\
		%		&+2\cdot (\frac{14}{9} - \bar{\alpha})^2  + ( \frac{15}{9} - \bar{\alpha})^2 + 2\cdot( \frac{16}{9} - \bar{\alpha})^2 + ( \frac{17}{9} - \bar{\alpha})^2 \\
		%		=& \frac{341}{324}  .
		%	\end{align*}
	%	Then it suffices to prove $  \frac{341}{324} - \frac{13}{12} \cdot 1 \leq 0,  $ i.e., $ - \frac{5}{162} \leq 0, $
	%	which holds obviously.
	%	
	
	\noindent\textbf{(13) $U_{1,p}$.}

	Since $  \tau_{f_0}  = \mu_f - 2 = 12+p, $ we need to verify the case of  $\tau^{\mathrm{max}} = \mu_f -1 = 12+ p  $ or $ 13 + p. $
	For $ p $ even, if  $\tau^{\mathrm{max}} = 12 + p, $ 
	\begin{align*}
		& \bar{\alpha} =\frac{1}{\mu_f -2}\cdot ( \frac{8}{9}+\frac{11}{9}+ \frac{13}{9}  + 
		\frac{14}{9} + \frac{16}{9} + \sum_{k =1}^{p/2  + 4 } (1+\frac{k}{p+9})  + \sum_{k=2 } ^{\frac{p}{2} + 4} ( 2- \frac{k}{9})^2 ))\\
		& =\frac{(27p^2 + 547p + 2754}{18(p + 9)(12 + p))}.
	\end{align*}
	\begin{align*}
		& (\mu_f -  2) \cdot \mathrm{Var}[\mathrm{Sp^{\tau^{\mathrm{max}} }}(f)] \\
		=& (\frac{8}{9}-\bar{\alpha})^2+(\frac{11}{9}-\bar{\alpha})^2+ (\frac{13}{9}-\bar{\alpha})^2+ (\frac{14}{9} - \bar{\alpha})^2 + ( \frac{16}{9}-\bar{\alpha})^2 +  \sum_{k =1}^{\frac{p}{2}  + 4 } (1+\frac{k}{p+9} - x)^2 )\\
		& + \sum_{k=2 } ^{\frac{p}{2} + 4} ( 2- \frac{k}{p+9} -x)^2  \\
		=& \frac{27p^4 + 1064p^3 + 15617p^2 + 102816p + 257580}{324(p + 9)^2(12 + p)}  .
	\end{align*}
	And $ \alpha_{\mu_f -2}-\alpha_1= 2 -\frac{2}{p+9} - \frac{8}{9}.  $
	Then it suffices to prove $   \frac{27p^4 + 1064p^3 + 15617p^2 + 102816p + 257580}{324(p + 9)^2(12 + p)} - \frac{12 + p}{12} \cdot (2 -\frac{2}{p+9} - \frac{8}{9}) \leq 0.  $
	Equivalently, $ \frac{-3p^4 - 142p^3 - 2311p^2 - 13824p - 22356}{324(p + 9)^2(12 + p)} \leq 0, $
	which holds obviously.
	
	If  $\tau^{\mathrm{max}} = 13 + p, $ 
	\begin{align*}
		& \bar{\alpha} =\frac{1}{\mu_f -1}\cdot ( \frac{8}{9}+\frac{11}{9}+ \frac{13}{9}  + 
		\frac{14}{9} + \frac{16}{9} + \sum_{k =1}^{p/2  + 4 } (1+\frac{k}{p+9})  + \sum_{k=1 } ^{\frac{p}{2} + 4} ( 2- \frac{k}{9})^2 )) 
		=\frac{340 + 27p}{234 + 18p}.
	\end{align*}
	\begin{align*}
		& (\mu_f -  1) \cdot \mathrm{Var}[\mathrm{Sp^{\tau^{\mathrm{max}} }}(f)] \\
		=& (\frac{8}{9}-\bar{\alpha})^2+(\frac{11}{9}-\bar{\alpha})^2+ (\frac{13}{9}-\bar{\alpha})^2+ (\frac{14}{9} - \bar{\alpha})^2 + ( \frac{16}{9}-\bar{\alpha})^2 +  \sum_{k =1}^{p/2  + 4 } (1+\frac{k}{p+9} - x)^2 )  \\
		&+ \sum_{k=1 } ^{\frac{p}{2} + 4} ( 2- \frac{k}{p+9} -x)^2 \\
		=& \frac{27p^3 + 929p^2 + 10462p + 38808}{324(p + 9)(13 + p)}  .
	\end{align*}
	And $ \alpha_{\mu_f -1}-\alpha_1= 2 -\frac{1}{p+9} - \frac{8}{9}.  $
	Then it suffices to prove $  \frac{27p^3 + 929p^2 + 10462p + 38808}{324(p + 9)(13 + p)} - \frac{13 + p}{12} \cdot ( 2 -\frac{1}{p+9} - \frac{8}{9}) \leq 0.  $
	Equivalently, $ \frac{-3p^3 - 94p^2 - 926p - 2259}{324(p + 9)(13 + p)} \leq 0, $
	which holds obviously.
	
	\noindent\textbf{(14) $E_{18}$.}
	
	Since $  \tau_{f_0}  = \mu_f - 2 = 16, $ we need to verify the cases of $\tau^{\mathrm{max}} = 16 $ or $ 17. $ 
	If  $\tau^{\mathrm{max}} = 16, $ 
	\begin{align*}
		& \bar{\alpha} =\frac{1}{\mu_f - 2 }\cdot ( \frac{13}{30}+\frac{8}{15}+ \frac{19}{30} + \frac{11}{15} 
		+\frac{23}{30} + \frac{5}{6} + \frac{13}{15} + \frac{14}{15} + \frac{29}{30}+\frac{31}{30} + \frac{16}{15} + \frac{17}{15}\\
		& + \frac{7}{6} + \frac{37}{30} + \frac{19}{15} + \frac{41}{30})\\
		&=\frac{449}{480}.
	\end{align*}
	\begin{align*}
		& (\mu_f -  2) \cdot \mathrm{Var}[\mathrm{Sp^{\tau^{\mathrm{max}} }}(f)] \\
		=& (\frac{13}{30}-\bar{\alpha})^2+(\frac{8}{15}-\bar{\alpha})^2+ (\frac{19}{30}-\bar{\alpha})^2+ (\frac{23}{30} - \bar{\alpha})^2 + ( \frac{5}{6} -\bar{\alpha})^2 +(\frac{13}{15}  - \bar{\alpha})^2 + (\frac{14}{15} - \bar{\alpha})^2\\
		& + ( \frac{29}{30}- \bar{\alpha})^2+ (\frac{31}{30} - \bar{\alpha})^2  + ( \frac{16}{15} - \bar{\alpha})^2 + ( \frac{17}{15} - \bar{\alpha})^2 + ( \frac{7}{6} - \bar{\alpha})^2 +( \frac{37}{30} - \bar{\alpha})^2 \\
		&+(\frac{19}{15} -\bar{\alpha})^2 + ( \frac{41}{30} - \bar{\alpha})^2 \\
		=& \frac{1751}{1600}  .
	\end{align*}
	And $ \alpha_{\mu_f -2}-\alpha_1= \frac{41}{30} - \frac{13}{30} =\frac{28}{30} .  $
	Then it suffices to prove $  \frac{1751}{1600} - \frac{16}{12} \cdot \frac{28}{30} \leq 0,  $ i.e., $ - \frac{2161}{14400} \leq 0, $
	which holds obviously.

	If  $\tau^{\mathrm{max}} = 17, $ 
	\begin{align*}
		& \bar{\alpha} =\frac{1}{\tau}\cdot ( \frac{13}{30}+\frac{8}{15}+ \frac{19}{30} + \frac{11}{15} 
		+\frac{23}{30} + \frac{5}{6} + \frac{13}{15} + \frac{14}{15} + \frac{29}{30}+\frac{31}{30} + \frac{16}{15} + \frac{17}{15}\\
		& + \frac{7}{6} + \frac{37}{30} + \frac{19}{15} + \frac{41}{30} + \frac{22}{15})\\
		&=\frac{29}{30}.
	\end{align*}
	\begin{align*}
		& (\mu_f -  1) \cdot \mathrm{Var}[\mathrm{Sp^{\tau^{\mathrm{max}} }}(f)] \\
		=& (\frac{13}{30}-\bar{\alpha})^2+(\frac{8}{15}-\bar{\alpha})^2+ (\frac{19}{30}-\bar{\alpha})^2+ (\frac{23}{30} - \bar{\alpha})^2 + ( \frac{5}{6} -\bar{\alpha})^2 +(\frac{13}{15}  - \bar{\alpha})^2 + (\frac{14}{15} - \bar{\alpha})^2\\
		& + ( \frac{29}{30}- \bar{\alpha})^2+ (\frac{31}{30} - \bar{\alpha})^2  + ( \frac{16}{15} - \bar{\alpha})^2 + ( \frac{17}{15} - \bar{\alpha})^2 + ( \frac{7}{6} - \bar{\alpha})^2 +( \frac{37}{30} - \bar{\alpha})^2 \\
		&+(\frac{19}{15} -\bar{\alpha})^2 + ( \frac{41}{30} - \bar{\alpha})^2 + ( \frac{22}{15} - \bar{\alpha})^2 \\
		=& \frac{34}{25}  .
	\end{align*}
	And $ \alpha_{\mu_f -1}-\alpha_1= \frac{22}{15} - \frac{13}{30} =\frac{31}{30} .  $
	Then it suffices to prove $  \frac{34}{25}  - \frac{17}{12} \cdot \frac{31}{30} \leq 0,  $ 
	i.e., $ - \frac{187}{1800} \leq 0, $
	which holds obviously.
	
	%	Considering the average value $ \bar{\alpha}=1, $
	%	\begin{align*}
		%		& (\mu_f -  1) \cdot \mathrm{Var}[\mathrm{Sp^{\tau^{\mathrm{max}} }}(f)] \\
		%		=& (\frac{13}{30}-\bar{\alpha})^2+(\frac{8}{15}-\bar{\alpha})^2+ (\frac{19}{30}-\bar{\alpha})^2+ (\frac{23}{30} - \bar{\alpha})^2 + ( \frac{5}{6} -\bar{\alpha})^2 +(\frac{13}{15}  - \bar{\alpha})^2 + (\frac{14}{15} - \bar{\alpha})^2\\
		%		&  + ( \frac{29}{30}- \bar{\alpha})^2+ (\frac{31}{30} - \bar{\alpha})^2  + ( \frac{16}{15} - \bar{\alpha})^2 + ( \frac{17}{15} - \bar{\alpha})^2 + ( \frac{7}{6} - \bar{\alpha})^2 +( \frac{37}{30} - \bar{\alpha})^2 \\
		%		&+(\frac{19}{15} -\bar{\alpha})^2 + ( \frac{41}{30} - \bar{\alpha})^2 \\
		%		=& \frac{209}{180}  .
		%	\end{align*}
	%	Then it suffices to prove $ \frac{209}{180}- \frac{16}{12} \cdot \frac{28}{30} \leq 0,  $ i.e., $ - \frac{1}{12} \leq 0, $
	%	which holds obviously.
	
	\noindent\textbf{(15) $E_{19}$.}
	
	Since $  \tau_{f_0}  = \mu_f - 2 = 17, $ we need to verify the cases of $\tau^{\mathrm{max}} = 17 $ or $ 18. $ 
	If  $\tau^{\mathrm{max}} = 17, $ 
	\begin{align*}
		& \bar{\alpha} =\frac{1}{\mu_f - 2}\cdot ( \frac{3}{7}+\frac{11}{21}+ \frac{13}{21} + \frac{5}{7} 
		+\frac{16}{21} + \frac{17}{21} + \frac{6}{7} + \frac{19}{21} + \frac{20}{21} + 1+\frac{22}{21} + \frac{23}{21} + \frac{8}{7} + \frac{25}{21}\\
		& + \frac{26}{21} + \frac{9}{7} + \frac{29}{21})\\
		&=\frac{335}{357}.
	\end{align*}
	\begin{align*}
		& (\mu_f -  2) \cdot \mathrm{Var}[\mathrm{Sp^{\tau^{\mathrm{max}} }}(f)] \\
		=& (\frac{3}{7}-\bar{\alpha})^2+(\frac{11}{21}-\bar{\alpha})^2+ (\frac{13}{21}-\bar{\alpha})^2+ (\frac{5}{7} - \bar{\alpha})^2 + ( \frac{16}{21} -\bar{\alpha})^2+(\frac{17}{21}  - \bar{\alpha})^2 + (\frac{6}{7} - \bar{\alpha})^2 \\
		&+ ( \frac{19}{21}- \bar{\alpha})^2 + (\frac{20}{21} -\bar{\alpha})^2 +( 1 - \bar{\alpha})^2+ (\frac{22}{21} - \bar{\alpha})^2  + ( \frac{23}{21} - \bar{\alpha})^2 + ( \frac{8}{7} - \bar{\alpha})^2 + ( \frac{25}{21} - \bar{\alpha})^2 \\
		&+( \frac{26}{21} - \bar{\alpha})^2 +(\frac{9}{7} -\bar{\alpha})^2 + ( \frac{29}{21} - \bar{\alpha})^2 \\
		=& \frac{2978}{2499}  .
	\end{align*}
	And $ \alpha_{\mu_f -2}-\alpha_1= \frac{29}{21} - \frac{3}{7} =\frac{20}{21} .  $
	Then it suffices to prove $  \frac{2978}{2499}- \frac{17}{12} \cdot \frac{20}{21} \leq 0,  $ i.e., $ - \frac{1181}{7497} \leq 0, $
	which holds obviously.

	If  $\tau^{\mathrm{max}} = 18, $ 
	\begin{align*}
		& \bar{\alpha} =\frac{1}{\mu_f - 1}\cdot ( \frac{3}{7}+\frac{11}{21}+ \frac{13}{21} + \frac{5}{7} 
		+\frac{16}{21} + \frac{17}{21} + \frac{6}{7} + \frac{19}{21} + \frac{20}{21} + 1+\frac{22}{21} + \frac{23}{21} + \frac{8}{7} + \frac{25}{21}\\
		& + \frac{26}{21} + \frac{9}{7} + \frac{29}{21} + \frac{31}{21})\\
		&=\frac{61}{63}.
	\end{align*}
	\begin{align*}
		& (\mu_f -  1) \cdot \mathrm{Var}[\mathrm{Sp^{\tau^{\mathrm{max}} }}(f)] \\
		=& (\frac{3}{7}-\bar{\alpha})^2+(\frac{11}{21}-\bar{\alpha})^2+ (\frac{13}{21}-\bar{\alpha})^2+ (\frac{5}{7} - \bar{\alpha})^2 + ( \frac{16}{21} -\bar{\alpha})^2+(\frac{17}{21}  - \bar{\alpha})^2 + (\frac{6}{7} - \bar{\alpha})^2 \\
		&+ ( \frac{19}{21}- \bar{\alpha})^2 + (\frac{20}{21} -\bar{\alpha})^2 +( 1 - \bar{\alpha})^2+ (\frac{22}{21} - \bar{\alpha})^2  + ( \frac{23}{21} - \bar{\alpha})^2 + ( \frac{8}{7} - \bar{\alpha})^2 + ( \frac{25}{21} - \bar{\alpha})^2 \\
		&+( \frac{26}{21} - \bar{\alpha})^2 +(\frac{9}{7} -\bar{\alpha})^2 + ( \frac{29}{21} - \bar{\alpha})^2  +(\frac{31}{21} - \bar{\alpha})^2\\
		=& \frac{646}{441}  .
	\end{align*}
	And $ \alpha_{\mu_f -1}-\alpha_1= \frac{31}{21} - \frac{3}{7} =\frac{22}{21} .  $
	Then it suffices to prove $ \frac{646}{441} - \frac{18}{12} \cdot \frac{22}{21} \leq 0,  $ 
	i.e., $ - \frac{47}{441} \leq 0, $
	which holds obviously.
	%	Considering the average value $ \bar{\alpha}=1, $
	%	\begin{align*}
		%		& (\mu_f -  1) \cdot \mathrm{Var}[\mathrm{Sp^{\tau^{\mathrm{max}} }}(f)] \\
		%		=& (\frac{3}{7}-\bar{\alpha})^2+(\frac{11}{21}-\bar{\alpha})^2+ (\frac{13}{21}-\bar{\alpha})^2+ (\frac{5}{7} - \bar{\alpha})^2 + ( \frac{16}{21} -\bar{\alpha})^2+(\frac{17}{21}  - \bar{\alpha})^2 + (\frac{6}{7} - \bar{\alpha})^2 \\
		%		&+ ( \frac{19}{21}- \bar{\alpha})^2 + (\frac{20}{21} -\bar{\alpha})^2 +( 1 - \bar{\alpha})^2+ (\frac{22}{21} - \bar{\alpha})^2  + ( \frac{23}{21} - \bar{\alpha})^2 + ( \frac{8}{7} - \bar{\alpha})^2 + ( \frac{25}{21} - \bar{\alpha})^2 \\
		%		&+( \frac{26}{21} - \bar{\alpha})^2 +(\frac{9}{7} -\bar{\alpha})^2 + ( \frac{29}{21} - \bar{\alpha})^2 \\
		%		=& \frac{554}{441}  .
		%	\end{align*}
	%	Then it suffices to prove $\frac{554}{441} - \frac{17}{12} \cdot \frac{20}{21} \leq 0,  $ i.e., $ - \frac{41}{441} \leq 0, $
	%	which holds obviously.
	%	
	%	
	
	\noindent\textbf{(16) $E_{20}$.}

	Since $  \tau_{f_0} = \mu_f - 2 = 18, $ we need to verify the cases of $\tau^{\mathrm{max}} = 18 $ or $ 19. $ 
	If  $\tau^{\mathrm{max}} = 18, $ 
	\begin{align*}
		& \bar{\alpha} =\frac{1}{\mu_f - 2}\cdot ( \frac{14}{33}+\frac{17}{33}+ \frac{20}{33} + \frac{23}{33} 
		+\frac{25}{33} + \frac{26}{33} + \frac{28}{33} + \frac{29}{33} + \frac{31}{33} + \frac{32}{33}+\frac{34}{33} + \frac{35}{33} + \frac{37}{33} + \frac{38}{33} +\\
		& \frac{40}{33} + \frac{41}{33} + \frac{43}{33} + \frac{46}{33})\\
		&=\frac{559}{594}.
	\end{align*}
	\begin{align*}
		& (\mu_f -  2) \cdot \mathrm{Var}[\mathrm{Sp^{\tau^{\mathrm{max}} }}(f)] \\
		=& (\frac{14}{33}-\bar{\alpha})^2+(\frac{17}{33}-\bar{\alpha})^2+ (\frac{20}{33}-\bar{\alpha})^2+ (\frac{23}{33} - \bar{\alpha})^2 + ( \frac{25}{33} -\bar{\alpha})^2+(\frac{26}{33}  - \bar{\alpha})^2 + (\frac{28}{33} - \bar{\alpha})^2 \\
		&+ ( \frac{29}{33}- \bar{\alpha})^2 + (\frac{31}{33} -\bar{\alpha})^2 +( \frac{32}{33} - \bar{\alpha})^2+ (\frac{34}{33} - \bar{\alpha})^2  + ( \frac{35}{33} - \bar{\alpha})^2 + ( \frac{37}{33} - \bar{\alpha})^2 + ( \frac{38}{33} - \bar{\alpha})^2 \\
		&+( \frac{40}{33} - \bar{\alpha})^2 +(\frac{41}{33} -\bar{\alpha})^2 + ( \frac{43}{33} - \bar{\alpha})^2 +( \frac{46}{33}   - \bar{\alpha})^2\\
		=& \frac{209}{162}  .
	\end{align*}
	And $ \alpha_{\mu_f -2}-\alpha_1= \frac{46}{33} - \frac{14}{33} =\frac{32}{33} .  $
	Then it suffices to prove $  \frac{209}{162} - \frac{18}{12} \cdot \frac{32}{33} \leq 0,  $ i.e., $ - \frac{293}{1782} \leq 0, $
	which holds obviously.

	If  $\tau^{\mathrm{max}} = 19, $ 
	\begin{align*}
		& \bar{\alpha} =\frac{1}{\mu_f - 1}\cdot ( \frac{14}{33}+\frac{17}{33}+ \frac{20}{33} + \frac{23}{33} 
		+\frac{25}{33} + \frac{26}{33} + \frac{28}{33} + \frac{29}{33} + \frac{31}{33} + \frac{32}{33}+\frac{34}{33} + \frac{35}{33} + \frac{37}{33} + \frac{38}{33} +\\
		& \frac{40}{33} + \frac{41}{33} + \frac{43}{33} + \frac{46}{33} + \frac{49}{33})\\
		&=\frac{32}{33}.
	\end{align*}
	\begin{align*}
		& (\mu_f -  1) \cdot \mathrm{Var}[\mathrm{Sp^{\tau^{\mathrm{max}} }}(f)] \\
		=& (\frac{14}{33}-\bar{\alpha})^2+(\frac{17}{33}-\bar{\alpha})^2+ (\frac{20}{33}-\bar{\alpha})^2+ (\frac{23}{33} - \bar{\alpha})^2 + ( \frac{25}{33} -\bar{\alpha})^2+(\frac{26}{33}  - \bar{\alpha})^2 + (\frac{28}{33} - \bar{\alpha})^2 \\
		&+ ( \frac{29}{33}- \bar{\alpha})^2 + (\frac{31}{33} -\bar{\alpha})^2 +( \frac{32}{33} - \bar{\alpha})^2+ (\frac{34}{33} - \bar{\alpha})^2  + ( \frac{35}{33} - \bar{\alpha})^2 + ( \frac{37}{33} - \bar{\alpha})^2 + ( \frac{38}{33} - \bar{\alpha})^2 \\
		&+( \frac{40}{33} - \bar{\alpha})^2 +(\frac{41}{33} -\bar{\alpha})^2 + ( \frac{43}{33} - \bar{\alpha})^2 +( \frac{46}{33}   - \bar{\alpha})^2 +(\frac{49}{33} -\bar{\alpha})^2\\
		=& \frac{190}{121}  .
	\end{align*}
	And $ \alpha_{\mu_f -2}-\alpha_1= \frac{49}{33} - \frac{14}{33} =\frac{35}{33} .  $
	Then it suffices to prove $ \frac{190}{121}- \frac{19}{12} \cdot \frac{35}{33} \leq 0,  $ 
	i.e., $ - \frac{475}{4356} \leq 0, $
	which holds obviously.
	%	
	%	Considering the average value $ \bar{\alpha}=1, $
	%	\begin{align*}
		%		& (\mu_f -  1) \cdot \mathrm{Var}[\mathrm{Sp^{\tau^{\mathrm{max}} }}(f)] \\
		%		=& (\frac{14}{33}-\bar{\alpha})^2+(\frac{17}{33}-\bar{\alpha})^2+ (\frac{20}{33}-\bar{\alpha})^2+ (\frac{23}{33} - \bar{\alpha})^2 + ( \frac{25}{33} -\bar{\alpha})^2+(\frac{26}{33}  - \bar{\alpha})^2 + (\frac{28}{33} - \bar{\alpha})^2\\
		%		& + ( \frac{29}{33}- \bar{\alpha})^2 + (\frac{31}{33} -\bar{\alpha})^2 +( \frac{32}{33} - \bar{\alpha})^2+ (\frac{34}{33} - \bar{\alpha})^2  + ( \frac{35}{33} - \bar{\alpha})^2 + ( \frac{37}{33} - \bar{\alpha})^2 + ( \frac{38}{33} - \bar{\alpha})^2 \\
		%		&+( \frac{40}{33} - \bar{\alpha})^2 +(\frac{41}{33} -\bar{\alpha})^2 + ( \frac{43}{33} - \bar{\alpha})^2 +( \frac{46}{33}   - \bar{\alpha})^2\\
		%		=& \frac{491}{363}  .
		%	\end{align*}
	%	Then it suffices to prove $  \frac{491}{363} - \frac{18}{12} \cdot \frac{32}{33} \leq 0,  $ i.e., $ - \frac{37}{363} \leq 0, $
	%	which holds obviously.
	%	
	%	
	
	\noindent\textbf{(17) $Z_{17}$.}

	Since $  \tau_{f_0} = \mu_f - 2 = 15, $ we need to verify the cases of $\tau^{\mathrm{max}} = 15 $ or $ 16. $ 
	If  $\tau^{\mathrm{max}} = 15, $ 
	\begin{align*}
		& \bar{\alpha} =\frac{1}{\mu_f -2 }\cdot ( \frac{5}{12}+\frac{13}{24}+ \frac{2}{3} + \frac{15}{24} 
		+\frac{17}{24} + \frac{5}{6} + \frac{11}{12} +0+\frac{13}{12} + \frac{7}{6} + \frac{29}{24} + \frac{31}{24} + \frac{4}{3} )\\
		=&\frac{67}{72}.
	\end{align*}
	\begin{align*}
		& (\mu_f -  2) \cdot \mathrm{Var}[\mathrm{Sp^{\tau^{\mathrm{max}} }}(f)] \\
		=& (\frac{5}{12}-\bar{\alpha})^2+(\frac{13}{24}-\bar{\alpha})^2+ (\frac{2}{3}-\bar{\alpha})^2+ (\frac{15}{24} - \bar{\alpha})^2 + ( \frac{17}{24} -\bar{\alpha})^2+(\frac{5}{6}  - \bar{\alpha})^2 + (\frac{11}{12} - \bar{\alpha})^2 \\
		&+ ( 0 - \bar{\alpha})^2 + (\frac{13}{12} - \bar{\alpha})^2  + ( \frac{7}{6} - \bar{\alpha})^2 + ( \frac{29}{24} - \bar{\alpha})^2 + ( \frac{31}{24} - \bar{\alpha})^2 +( \frac{4}{3} - \bar{\alpha})^2 \\
		=& \frac{659}{648}  .
	\end{align*}
	And $ \alpha_{\mu_f -2}-\alpha_1= \frac{4}{3} - \frac{5}{12} =\frac{11}{12} .  $
	Then it suffices to prove $  \frac{659}{648} - \frac{15}{12} \cdot \frac{11}{12} \leq 0,  $ i.e., $ - \frac{167}{1296} \leq 0, $
	which holds obviously.

	If  $\tau^{\mathrm{max}} = 16, $ 
	\begin{align*}
		& \bar{\alpha} =\frac{1}{\mu_f -1 }\cdot ( \frac{5}{12}+\frac{13}{24}+ \frac{2}{3} + \frac{15}{24} 
		+\frac{17}{24} + \frac{5}{6} + \frac{11}{12} +0+\frac{13}{12} + \frac{7}{6} + \frac{29}{24} + \frac{31}{24} + \frac{4}{3} + \frac{35}{24})\\
		=&\frac{185}{192}.
	\end{align*}
	\begin{align*}
		& (\mu_f -  1) \cdot \mathrm{Var}[\mathrm{Sp^{\tau^{\mathrm{max}} }}(f)] \\
		=& (\frac{5}{12}-\bar{\alpha})^2+(\frac{13}{24}-\bar{\alpha})^2+ (\frac{2}{3}-\bar{\alpha})^2+ (\frac{15}{24} - \bar{\alpha})^2 + ( \frac{17}{24} -\bar{\alpha})^2+(\frac{5}{6}  - \bar{\alpha})^2 + (\frac{11}{12} - \bar{\alpha})^2 \\
		&+ ( 0 - \bar{\alpha})^2 + (\frac{13}{12} - \bar{\alpha})^2  + ( \frac{7}{6} - \bar{\alpha})^2 + ( \frac{29}{24} - \bar{\alpha})^2 + ( \frac{31}{24} - \bar{\alpha})^2 +( \frac{4}{3} - \bar{\alpha})^2 +(\frac{35}{24} -\bar{\alpha})^2\\
		=& \frac{2975}{2304}  .
	\end{align*}
	And $ \alpha_{\mu_f -1}-\alpha_1= \frac{35}{24} - \frac{5}{12} =\frac{25}{24} .  $
	Then it suffices to prove $   \frac{2975}{2304}- \frac{16}{12} \cdot \frac{25}{24} \leq 0,  $
	i.e., $ - \frac{25}{256} \leq 0, $
	which holds obviously.

	\noindent\textbf{(18) $Z_{18}$.}

	Since $  \tau_{f_0} = \mu_f - 2 = 16, $ we need to verify the cases of $\tau^{\mathrm{max}} = 16 $ or $ 17. $ 
	If  $\tau^{\mathrm{max}} = 16, $ 
	\begin{align*}
		& \bar{\alpha} =\frac{1}{\mu_f -2}\cdot ( \frac{7}{17}+\frac{9}{17}+ \frac{11}{17} + \frac{12}{17} 
		+\frac{13}{17} + \frac{15}{17} + \frac{16}{17} +2 \cdot 0+\frac{18}{17} + \frac{19}{17} + \frac{20}{17} + \frac{21}{17} + \frac{22}{17} + \frac{23}{17})\\
		=&\frac{127}{136}.
	\end{align*}
	\begin{align*}
		& (\mu_f -  2) \cdot \mathrm{Var}[\mathrm{Sp^{\tau^{\mathrm{max}} }}(f)] \\
		=& (\frac{7}{17}-\bar{\alpha})^2+(\frac{9}{17}-\bar{\alpha})^2+ (\frac{11}{17}-\bar{\alpha})^2+ (\frac{12}{17} - \bar{\alpha})^2 + ( \frac{13}{17} -\bar{\alpha})^2+(\frac{15}{17}  - \bar{\alpha})^2 + (\frac{16}{17} - \bar{\alpha})^2 \\
		&+ 2\cdot( 0 - \bar{\alpha})^2 + (\frac{18}{17} - \bar{\alpha})^2  + ( \frac{19}{17} - \bar{\alpha})^2 + ( \frac{20}{17} - \bar{\alpha})^2 + ( \frac{21}{17} - \bar{\alpha})^2\\
		& +( \frac{22}{17} - \bar{\alpha})^2 + (\frac{23}{17} - \bar{\alpha})^2 \\
		=& \frac{1303}{1156}  .
	\end{align*}
	And $ \alpha_{\mu_f -2}-\alpha_1= \frac{23}{17} - \frac{7}{17} =\frac{16}{17} .  $
	Then it suffices to prove $   \frac{1303}{1156}  - \frac{16}{12} \cdot \frac{16}{17} \leq 0,  $ i.e., $ - \frac{443}{3468} \leq 0, $
	which holds obviously.

	If  $\tau^{\mathrm{max}} = 17, $ 
	\begin{align*}
		& \bar{\alpha} =\frac{1}{\mu_f -1}\cdot ( \frac{7}{17}+\frac{9}{17}+ \frac{11}{17} + \frac{12}{17} 
		+\frac{13}{17} + \frac{15}{17} + \frac{16}{17} +2 \cdot 0+\frac{18}{17} + \frac{19}{17} + \frac{20}{17} + \frac{21}{17} + \frac{22}{17} \\
		&+ \frac{23}{17} +\frac{25}{17})\\
		=&\frac{279}{289}.
	\end{align*}
	\begin{align*}
		& (\mu_f -  1) \cdot \mathrm{Var}[\mathrm{Sp^{\tau^{\mathrm{max}} }}(f)] \\
		=& (\frac{7}{17}-\bar{\alpha})^2+(\frac{9}{17}-\bar{\alpha})^2+ (\frac{11}{17}-\bar{\alpha})^2+ (\frac{12}{17} - \bar{\alpha})^2 + ( \frac{13}{17} -\bar{\alpha})^2+(\frac{15}{17}  - \bar{\alpha})^2 + (\frac{16}{17} - \bar{\alpha})^2 \\
		&+ 2\cdot( 0 - \bar{\alpha})^2 + (\frac{18}{17} - \bar{\alpha})^2  + ( \frac{19}{17} - \bar{\alpha})^2 + ( \frac{20}{17} - \bar{\alpha})^2 + ( \frac{21}{17} - \bar{\alpha})^2\\
		& +( \frac{22}{17} - \bar{\alpha})^2 + (\frac{23}{17} - \bar{\alpha})^2 +(\frac{25}{17} - \bar{\alpha})^2  \\
		=& \frac{6870}{4913}  .
	\end{align*}
	And $ \alpha_{\mu_f -1}-\alpha_1= \frac{25}{17} - \frac{7}{17} =\frac{18}{17} .  $
	Then it suffices to prove $   \frac{6870}{4913}   - \frac{17}{12} \cdot \frac{18}{17} \leq 0,  $ 
	i.e., $ - \frac{999}{9826} \leq 0, $
	which holds obviously.
	%	Considering the average value $ \bar{\alpha}=1, $
	%	\begin{align*}
		%		& (\mu_f -  1) \cdot \mathrm{Var}[\mathrm{Sp^{\tau^{\mathrm{max}} }}(f)] \\
		%		=& (\frac{7}{17}-\bar{\alpha})^2+(\frac{9}{17}-\bar{\alpha})^2+ (\frac{11}{17}-\bar{\alpha})^2+ (\frac{12}{17} - \bar{\alpha})^2 + ( \frac{13}{17} -\bar{\alpha})^2+(\frac{15}{17}  - \bar{\alpha})^2 + (\frac{16}{17} - \bar{\alpha})^2 \\
		%		&+ 2\cdot( 0 - \bar{\alpha})^2 + (\frac{18}{17} - \bar{\alpha})^2  + ( \frac{19}{17} - \bar{\alpha})^2 + ( \frac{20}{17} - \bar{\alpha})^2 + ( \frac{21}{17} - \bar{\alpha})^2\\
		%		& +( \frac{22}{17} - \bar{\alpha})^2 + (\frac{23}{17} - \bar{\alpha})^2 \\
		%		=& \frac{346}{289}  .
		%	\end{align*}
	%	Then it suffices to prove $ \frac{346}{289}  - \frac{16}{12} \cdot \frac{16}{17} \leq 0,  $ i.e., $ - \frac{50}{867} \leq 0, $
	%	which holds obviously.
	%	
	%	
	%	
	
	\noindent\textbf{(19) $Z_{19}$.}
	
	Since $  \tau_{f_0}  = \mu_f - 2 = 17, $ we need to verify the cases of $\tau^{\mathrm{max}} = 17 $ or $ 18. $ 
	If  $\tau^{\mathrm{max}} = 17, $ 
	\begin{align*}
		& \bar{\alpha} =\frac{1}{\mu_f - 2}\cdot ( \frac{11}{27}+\frac{14}{27}+ \frac{17}{27} + \frac{19}{27} 
		+\frac{20}{27} + \frac{22}{27} + \frac{23}{27} + \frac{25}{27}+ \frac{26}{27} + 0+\frac{28}{27} + \frac{29}{27} + \frac{31}{27} + \frac{32}{27} \\
		&+ \frac{34}{27} + \frac{35}{27} + \frac{37}{27})\\
		=&\frac{430}{459}.
	\end{align*}
	\begin{align*}
		& (\mu_f -  2) \cdot \mathrm{Var}[\mathrm{Sp^{\tau^{\mathrm{max}} }}(f)] \\
		=& (\frac{11}{27}-\bar{\alpha})^2+(\frac{14}{27}-\bar{\alpha})^2+ (\frac{17}{27}-\bar{\alpha})^2+ (\frac{19}{27} - \bar{\alpha})^2 + ( \frac{20}{27} -\bar{\alpha})^2+(\frac{22}{27}  - \bar{\alpha})^2 + (\frac{23}{27} - \bar{\alpha})^2 \\
		&+ (\frac{25}{27}- \bar{\alpha})^2 +( \frac{26}{27} - \bar{\alpha})^2+( 0 - \bar{\alpha})^2 + (\frac{28}{27} - \bar{\alpha})^2  + ( \frac{29}{27} - \bar{\alpha})^2 + ( \frac{31}{27} - \bar{\alpha})^2 + ( \frac{32}{27} - \bar{\alpha})^2\\
		& +( \frac{34}{27} - \bar{\alpha})^2 + (\frac{35}{27} - \bar{\alpha})^2 +( \frac{37}{27}  - \bar{\alpha})^2 \\
		=& \frac{15190}{12393}  .
	\end{align*}
	And $ \alpha_{\mu_f -2}-\alpha_1= \frac{37}{27} - \frac{11}{27} =\frac{26}{27} .  $
	Then it suffices to prove $  \frac{15190}{12393}  - \frac{17}{12} \cdot \frac{26}{27} \leq 0,  $ i.e., $ - \frac{3433}{24786} \leq 0, $
	which holds obviously.

	If  $\tau^{\mathrm{max}} = 18, $ 
	\begin{align*}
		& \bar{\alpha} =\frac{1}{\mu_f - 1}\cdot ( \frac{11}{27}+\frac{14}{27}+ \frac{17}{27} + \frac{19}{27} 
		+\frac{20}{27} + \frac{22}{27} + \frac{23}{27} + \frac{25}{27}+ \frac{26}{27} + 0+\frac{28}{27} + \frac{29}{27} + \frac{31}{27} + \frac{32}{27} \\
		&+ \frac{34}{27} + \frac{35}{27} + \frac{37}{27} +\frac{40}{27})\\
		=&\frac{235}{243}.
	\end{align*}
	\begin{align*}
		& (\mu_f -  1) \cdot \mathrm{Var}[\mathrm{Sp^{\tau^{\mathrm{max}} }}(f)] \\
		=& (\frac{11}{27}-\bar{\alpha})^2+(\frac{14}{27}-\bar{\alpha})^2+ (\frac{17}{27}-\bar{\alpha})^2+ (\frac{19}{27} - \bar{\alpha})^2 + ( \frac{20}{27} -\bar{\alpha})^2+(\frac{22}{27}  - \bar{\alpha})^2 + (\frac{23}{27} - \bar{\alpha})^2 \\
		&+ (\frac{25}{27}- \bar{\alpha})^2 +( \frac{26}{27} - \bar{\alpha})^2+( 0 - \bar{\alpha})^2 + (\frac{28}{27} - \bar{\alpha})^2  + ( \frac{29}{27} - \bar{\alpha})^2 + ( \frac{31}{27} - \bar{\alpha})^2 + ( \frac{32}{27} - \bar{\alpha})^2\\
		& +( \frac{34}{27} - \bar{\alpha})^2 + (\frac{35}{27} - \bar{\alpha})^2 +( \frac{37}{27}  - \bar{\alpha})^2 +( \frac{40}{27} -\bar{\alpha})^2 \\
		=& \frac{9880}{6561}  .
	\end{align*}
	And $ \alpha_{\mu_f -1}-\alpha_1= \frac{40}{27} - \frac{11}{27} =\frac{29}{27} .  $
	Then it suffices to prove $  \frac{9880}{6561}  - \frac{18}{12} \cdot \frac{29}{27} \leq 0,  $ i.e., $ - \frac{1381}{13122} \leq 0, $
	which holds obviously.
	%	Considering the average value $ \bar{\alpha}=1, $
	%	\begin{align*}
		%		& (\mu_f -  1) \cdot \mathrm{Var}[\mathrm{Sp^{\tau^{\mathrm{max}} }}(f)] \\
		%		=& (\frac{11}{27}-\bar{\alpha})^2+(\frac{14}{27}-\bar{\alpha})^2+ (\frac{17}{27}-\bar{\alpha})^2+ (\frac{19}{27} - \bar{\alpha})^2 + ( \frac{20}{27} -\bar{\alpha})^2+(\frac{22}{27}  - \bar{\alpha})^2 + (\frac{23}{27} - \bar{\alpha})^2\\
		%		& + (\frac{25}{27}- \bar{\alpha})^2 +( \frac{26}{27} - \bar{\alpha})^2+( 0 - \bar{\alpha})^2 + (\frac{28}{27} - \bar{\alpha})^2  + ( \frac{29}{27} - \bar{\alpha})^2 + ( \frac{31}{27} - \bar{\alpha})^2 + ( \frac{32}{27} - \bar{\alpha})^2\\
		%		& +( \frac{34}{27} - \bar{\alpha})^2 + (\frac{35}{27} - \bar{\alpha})^2 +( \frac{37}{27}  - \bar{\alpha})^2 \\
		%		=& \frac{943}{729}  .
		%	\end{align*}
	%	Then it suffices to prove $ \frac{943}{729}  - \frac{17}{12} \cdot \frac{26}{27} \leq 0,  $ i.e., $ - \frac{103}{1458} \leq 0, $
	%	which holds obviously.
	%	
	%	
	%	
	%	
	
	\noindent\textbf{(20) $W_{17}$.}
	
	Since $  \tau_{f_0}  = \mu_f - 2 = 15, $ we need to verify the cases of $\tau^{\mathrm{max}} = 15 $ or $ 16. $ 
	If  $\tau^{\mathrm{max}} = 15, $ 
	\begin{align*}
		& \bar{\alpha} =\frac{1}{\mu_f -2}\cdot ( \frac{2}{5}+\frac{11}{20}+ \frac{13}{20} + \frac{7}{10} 
		+\frac{4}{5} + \frac{17}{20} + \frac{9}{10} + \frac{19}{20}+ 0+\frac{21}{20} + \frac{11}{10} + \frac{23}{20} + \frac{6}{5}\\
		& + \frac{13}{10} + \frac{27}{20}) \\
		=&\frac{24}{25}.
	\end{align*}
	\begin{align*}
		& (\mu_f -  2) \cdot \mathrm{Var}[\mathrm{Sp^{\tau^{\mathrm{max}} }}(f)] \\
		=& (\frac{2}{5}-\bar{\alpha})^2+(\frac{11}{20}-\bar{\alpha})^2+ (\frac{13}{20}-\bar{\alpha})^2+ (\frac{7}{10} - \bar{\alpha})^2 + ( \frac{4}{5} -\bar{\alpha})^2+(\frac{17}{20}  - \bar{\alpha})^2 + (\frac{9}{10} - \bar{\alpha})^2 \\
		&+ (\frac{19}{20}- \bar{\alpha})^2 +( 0 - \bar{\alpha})^2 + (\frac{21}{20} - \bar{\alpha})^2  + ( \frac{11}{10} - \bar{\alpha})^2 + ( \frac{23}{20} - \bar{\alpha})^2 + ( \frac{6}{5} - \bar{\alpha})^2\\
		& +( \frac{13}{10} - \bar{\alpha})^2 +( \frac{27}{20}  - \bar{\alpha})^2 \\
		=& \frac{431}{400}  .
	\end{align*}
	And $ \alpha_{\mu_f -2}-\alpha_1= \frac{27}{20} - \frac{2}{5} =\frac{19}{20} .  $
	Then it suffices to prove $ \frac{431}{400}   - \frac{15}{12} \cdot \frac{19}{20} \leq 0,  $ i.e., $ - \frac{11}{100} \leq 0, $
	which holds obviously.

	If  $\tau^{\mathrm{max}} = 16, $ 
	\begin{align*}
		& \bar{\alpha} =\frac{1}{\mu_f -1}\cdot ( \frac{2}{5}+\frac{11}{20}+ \frac{13}{20} + \frac{7}{10} 
		+\frac{4}{5} + \frac{17}{20} + \frac{9}{10} + \frac{19}{20}+ 0+\frac{21}{20} + \frac{11}{10} + \frac{23}{20} + \frac{6}{5}\\
		& + \frac{13}{10} + \frac{27}{20} + \frac{29}{20}) \\
		=&\frac{77}{80}.
	\end{align*}
	\begin{align*}
		& (\mu_f -  1) \cdot \mathrm{Var}[\mathrm{Sp^{\tau^{\mathrm{max}} }}(f)] \\
		=& (\frac{2}{5}-\bar{\alpha})^2+(\frac{11}{20}-\bar{\alpha})^2+ (\frac{13}{20}-\bar{\alpha})^2+ (\frac{7}{10} - \bar{\alpha})^2 + ( \frac{4}{5} -\bar{\alpha})^2+(\frac{17}{20}  - \bar{\alpha})^2 + (\frac{9}{10} - \bar{\alpha})^2 \\
		&+ (\frac{19}{20}- \bar{\alpha})^2 +( 0 - \bar{\alpha})^2 + (\frac{21}{20} - \bar{\alpha})^2  + ( \frac{11}{10} - \bar{\alpha})^2 + ( \frac{23}{20} - \bar{\alpha})^2 + ( \frac{6}{5} - \bar{\alpha})^2\\
		& +( \frac{13}{10} - \bar{\alpha})^2 +( \frac{27}{20}  - \bar{\alpha})^2 +(\frac{29}{20} - \bar{\alpha})^2 \\
		=& \frac{527}{400}  .
	\end{align*}
	And $ \alpha_{\mu_f -1}-\alpha_1= \frac{29}{20} - \frac{2}{5} =\frac{21}{20} .  $
	Then it suffices to prove $ \frac{527}{400}- \frac{16}{12} \cdot \frac{21}{20} \leq 0,  $
	i.e., $ - \frac{33}{400} \leq 0, $
	which holds obviously.
	
	%	Considering the average value $ \bar{\alpha}=1, $
	%	\begin{align*}
		%		& (\mu_f -  1) \cdot \mathrm{Var}[\mathrm{Sp^{\tau^{\mathrm{max}} }}(f)] \\
		%		=& (\frac{2}{5}-\bar{\alpha})^2+(\frac{11}{20}-\bar{\alpha})^2+ (\frac{13}{20}-\bar{\alpha})^2+ (\frac{7}{10} - \bar{\alpha})^2 + ( \frac{4}{5} -\bar{\alpha})^2+(\frac{17}{20}  - \bar{\alpha})^2 + (\frac{9}{10} - \bar{\alpha})^2\\
		%		& + (\frac{19}{20}- \bar{\alpha})^2 +( 0 - \bar{\alpha})^2 + (\frac{21}{20} - \bar{\alpha})^2  + ( \frac{11}{10} - \bar{\alpha})^2 + ( \frac{23}{20} - \bar{\alpha})^2 + ( \frac{6}{5} - \bar{\alpha})^2\\
		%		& +( \frac{13}{10} - \bar{\alpha})^2 +( \frac{27}{20}  - \bar{\alpha})^2 \\
		%		=& \frac{91}{80}  .
		%	\end{align*}
	%	Then it suffices to prove $ \frac{91}{80}  - \frac{15}{12} \cdot \frac{19}{20} \leq 0,  $ i.e., $ - \frac{1}{20} \leq 0, $
	%	which holds obviously.
	%	
	%	
	%	
	
	\noindent\textbf{(21) $W_{18}$.}
	
	Since $  \tau_{f_0}  = \mu_f - 2 = 16, $ we need to verify the cases of $\tau^{\mathrm{max}} = 16 $ or $ 17. $ 
	If  $\tau^{\mathrm{max}} = 16, $ 
	\begin{align*}
		& \bar{\alpha} =\frac{1}{\mu_f -2}\cdot ( \frac{11}{28}+\frac{15}{28}+ \frac{9}{14} + \frac{11}{14} 
		+\frac{23}{28} + \frac{25}{28} + \frac{13}{14} + \frac{27}{28}+\frac{29}{28} + \frac{15}{14} + \frac{31}{28} + \frac{33}{28} + \frac{17}{14} \\
		&+ \frac{37}{28} +\frac{19}{14})\\
		=&\frac{431}{448}.
	\end{align*}
	\begin{align*}
		& (\mu_f -  2) \cdot \mathrm{Var}[\mathrm{Sp^{\tau^{\mathrm{max}} }}(f)] \\
		=& (\frac{11}{28}-\bar{\alpha})^2+(\frac{15}{28}-\bar{\alpha})^2+ (\frac{9}{14}-\bar{\alpha})^2+ (\frac{11}{14} - \bar{\alpha})^2 + ( \frac{23}{28} -\bar{\alpha})^2+(\frac{25}{28}  - \bar{\alpha})^2 + (\frac{13}{14} - \bar{\alpha})^2 \\
		&+ (\frac{27}{28}- \bar{\alpha})^2  + (\frac{29}{28} - \bar{\alpha})^2  + ( \frac{15}{14} - \bar{\alpha})^2 + ( \frac{31}{28} - \bar{\alpha})^2 + ( \frac{33}{28} - \bar{\alpha})^2\\
		& +( \frac{17}{14} - \bar{\alpha})^2 +( \frac{37}{28}  - \bar{\alpha})^2 +( \frac{19}{14} - \bar{\alpha})^2 \\
		=& \frac{14789}{12544}  .
	\end{align*}
	And $ \alpha_{\mu_f -2}-\alpha_1= \frac{19}{14} - \frac{11}{28} =\frac{27}{28} .  $
	Then it suffices to prove $ \frac{14789}{12544}  - \frac{16}{12} \cdot \frac{27}{28} \leq 0,  $ i.e., $ - \frac{11}{100} \leq 0, $
	which holds obviously.

	If  $\tau^{\mathrm{max}} = 17, $ 
	\begin{align*}
		& \bar{\alpha} =\frac{1}{\mu_f -1}\cdot ( \frac{11}{28}+\frac{15}{28}+ \frac{9}{14} + \frac{11}{14} 
		+\frac{23}{28} + \frac{25}{28} + \frac{13}{14} + \frac{27}{28}+\frac{29}{28} + \frac{15}{14} + \frac{31}{28} + \frac{33}{28} + \frac{17}{14} \\
		&+ \frac{37}{28} +\frac{19}{14} + \frac{41}{28})\\
		=&\frac{27}{28}.
	\end{align*}
	\begin{align*}
		& (\mu_f -  1) \cdot \mathrm{Var}[\mathrm{Sp^{\tau^{\mathrm{max}} }}(f)] \\
		=& (\frac{11}{28}-\bar{\alpha})^2+(\frac{15}{28}-\bar{\alpha})^2+ (\frac{9}{14}-\bar{\alpha})^2+ (\frac{11}{14} - \bar{\alpha})^2 + ( \frac{23}{28} -\bar{\alpha})^2+(\frac{25}{28}  - \bar{\alpha})^2 + (\frac{13}{14} - \bar{\alpha})^2 \\
		&+ (\frac{27}{28}- \bar{\alpha})^2  + (\frac{29}{28} - \bar{\alpha})^2  + ( \frac{15}{14} - \bar{\alpha})^2 + ( \frac{31}{28} - \bar{\alpha})^2 + ( \frac{33}{28} - \bar{\alpha})^2\\
		& +( \frac{17}{14} - \bar{\alpha})^2 +( \frac{37}{28}  - \bar{\alpha})^2 +( \frac{19}{14} - \bar{\alpha})^2 +(\frac{41}{28} -\bar{\alpha})^2\\
		=& \frac{561}{392}  .
	\end{align*}
	And $ \alpha_{\mu_f -1}-\alpha_1= \frac{41}{28} - \frac{11}{28} =\frac{30}{28} .  $
	Then it suffices to prove $ \frac{561}{392}  - \frac{17}{12} \cdot \frac{30}{28} \leq 0,  $ 
	i.e., $ - \frac{17}{196} \leq 0, $
	which holds obviously.
	%	Considering the average value $ \bar{\alpha}=1, $
	%	\begin{align*}
		%		& (\mu_f -  1) \cdot \mathrm{Var}[\mathrm{Sp^{\tau^{\mathrm{max}} }}(f)] \\
		%		=& (\frac{11}{28}-\bar{\alpha})^2+(\frac{15}{28}-\bar{\alpha})^2+ (\frac{9}{14}-\bar{\alpha})^2+ (\frac{11}{14} - \bar{\alpha})^2 + ( \frac{23}{28} -\bar{\alpha})^2+(\frac{25}{28}  - \bar{\alpha})^2 + (\frac{13}{14} - \bar{\alpha})^2\\
		%		& + (\frac{27}{28}- \bar{\alpha})^2  + (\frac{29}{28} - \bar{\alpha})^2  + ( \frac{15}{14} - \bar{\alpha})^2 + ( \frac{31}{28} - \bar{\alpha})^2 + ( \frac{33}{28} - \bar{\alpha})^2\\
		%		& +( \frac{17}{14} - \bar{\alpha})^2 +( \frac{37}{28}  - \bar{\alpha})^2 +( \frac{19}{14} - \bar{\alpha})^2 \\
		%		=& \frac{485}{392}.
		%	\end{align*}
	%	Then it suffices to prove $ \frac{485}{392}- \frac{16}{12} \cdot \frac{27}{28}  \leq 0,  $ i.e., $ - \frac{19}{392} \leq 0, $
	%	which holds obviously.
	%	
	
	\noindent\textbf{(22) $Q_{16}$.}
	
	Since $ \tau_{f_0} = \mu_f - 2 = 14, $ we need to verify the cases of $\tau^{\mathrm{max}} = 14 $ or $ 15. $ 
	If  $\tau^{\mathrm{max}} = 14, $ 
	\begin{align*}
		& \bar{\alpha} =\frac{1}{\mu_f -2 }\cdot ( \frac{19}{21}+\frac{20}{21}+ \frac{25}{21} + \frac{26}{21} 
		+2\cdot \frac{4}{3} + \frac{29}{21} + \frac{31}{21} +\frac{32}{21} + \frac{34}{21} + 2\cdot \frac{5}{3} + \frac{37}{21} + \frac{38}{21} )\\
		=&\frac{139}{98}.
	\end{align*}
	\begin{align*}
		& (\mu_f -  2) \cdot \mathrm{Var}[\mathrm{Sp^{\tau^{\mathrm{max}} }}(f)] \\
		=& (\frac{19}{21}-\bar{\alpha})^2+(\frac{20}{21}-\bar{\alpha})^2+ (\frac{25}{21}-\bar{\alpha})^2+ (\frac{26}{21} - \bar{\alpha})^2 + 2\cdot ( \frac{4}{3} -\bar{\alpha})^2+(\frac{29}{21}  - \bar{\alpha})^2 + (\frac{31}{21} - \bar{\alpha})^2\\
		& + (\frac{32}{21}- \bar{\alpha})^2  + (\frac{34}{21} - \bar{\alpha})^2  + 2\cdot ( \frac{5}{3} - \bar{\alpha})^2 + ( \frac{37}{21} - \bar{\alpha})^2 + ( \frac{38}{21} - \bar{\alpha})^2\\
		=& \frac{6361}{6174}  .
	\end{align*}
	And $ \alpha_{\mu_f -2}-\alpha_1= \frac{38}{21} - \frac{19}{21} =\frac{19}{21} .  $
	Then it suffices to prove $\frac{6361}{6174} - \frac{14}{12} \cdot \frac{19}{21} \leq 0,  $ i.e., $ - \frac{26}{1029} \leq 0, $
	which holds obviously.

	If  $\tau^{\mathrm{max}} = 15, $ 
	\begin{align*}
		& \bar{\alpha} =\frac{1}{\mu_f -1 }\cdot ( \frac{19}{21}+\frac{20}{21}+ \frac{25}{21} + \frac{26}{21} 
		+2\cdot \frac{4}{3} + \frac{29}{21} + \frac{31}{21} +\frac{32}{21} + \frac{34}{21} + 2\cdot \frac{5}{3} + \frac{37}{21} + \frac{38}{21} +\frac{43}{21} )\\
		=&\frac{92}{63}.
	\end{align*}
	\begin{align*}
		& (\mu_f -  1) \cdot \mathrm{Var}[\mathrm{Sp^{\tau^{\mathrm{max}} }}(f)] \\
		=& (\frac{19}{21}-\bar{\alpha})^2+(\frac{20}{21}-\bar{\alpha})^2+ (\frac{25}{21}-\bar{\alpha})^2+ (\frac{26}{21} - \bar{\alpha})^2 + 2\cdot ( \frac{4}{3} -\bar{\alpha})^2+(\frac{29}{21}  - \bar{\alpha})^2 + (\frac{31}{21} - \bar{\alpha})^2\\
		& + (\frac{32}{21}- \bar{\alpha})^2  + (\frac{34}{21} - \bar{\alpha})^2  + 2\cdot ( \frac{5}{3} - \bar{\alpha})^2 + ( \frac{37}{21} - \bar{\alpha})^2 + ( \frac{38}{21} - \bar{\alpha})^2 +(\frac{43}{21} -\bar{\alpha})^2\\
		=& \frac{1852}{1323}  .
	\end{align*}
	And $ \alpha_{\mu_f -1}-\alpha_1= \frac{43}{21} - \frac{19}{21} =\frac{24}{21} .  $
	Then it suffices to prove $ \frac{1852}{1323}  - \frac{15}{12} \cdot \frac{24}{21} \leq 0,  $ 
	i.e., $ - \frac{38}{1323} \leq 0, $
	which holds obviously.
	%	Considering the average value $ \bar{\alpha}=\frac{3}{2}, $
	%	\begin{align*}
		%		=& (\frac{19}{21}-\bar{\alpha})^2+(\frac{20}{21}-\bar{\alpha})^2+ (\frac{25}{21}-\bar{\alpha})^2+ (\frac{26}{21} - \bar{\alpha})^2 + 2\cdot ( \frac{4}{3} -\bar{\alpha})^2+(\frac{29}{21}  - \bar{\alpha})^2 + (\frac{31}{21} - \bar{\alpha})^2 \\
		%		&+ (\frac{32}{21}- \bar{\alpha})^2  + (\frac{34}{21} - \bar{\alpha})^2  + 2\cdot ( \frac{5}{3} - \bar{\alpha})^2 + ( \frac{37}{21} - \bar{\alpha})^2 + ( \frac{38}{21} - \bar{\alpha})^2\\
		%		=& \frac{991}{882}.
		%	\end{align*}
	%	Then it suffices to prove $ \frac{991}{882}- \frac{14}{12} \cdot \frac{19}{21}  \leq 0,  $ i.e., $ \frac{10}{147} \leq 0, $
	%	which holds obviously.
	%	%猜想对Q_16不成立
	%	%
	%	%
	%	%
	%	
	
	\noindent\textbf{(23) $Q_{17}$.}
	Since $  \tau_{f_0}  = \mu_f - 2 = 15, $ we need to verify the cases of $\tau^{\mathrm{max}} = 15 $ or $ 16. $ 
	If  $\tau^{\mathrm{max}} = 15, $ 
	\begin{align*}
		& \bar{\alpha} =\frac{1}{\mu_f -2 }\cdot ( \frac{9}{10}+\frac{31}{30}+ \frac{7}{6} + \frac{37}{30} 
		+\frac{13}{10} + \frac{4}{3} + \frac{41}{30} + \frac{43}{30}  + \frac{3}{2} +\frac{47}{30} + \frac{49}{30} + \frac{5}{3} + \frac{17}{10} \\
		&+ \frac{53}{30} + \frac{11}{6})\\
		=&\frac{643}{450}.
	\end{align*}
	\begin{align*}
		& (\mu_f -  2) \cdot \mathrm{Var}[\mathrm{Sp^{\tau^{\mathrm{max}} }}(f)] \\
		=& (\frac{9}{10}-\bar{\alpha})^2+(\frac{31}{30}-\bar{\alpha})^2+ (\frac{7}{6}-\bar{\alpha})^2+ (\frac{37}{30} - \bar{\alpha})^2+(\frac{13}{10}  - \bar{\alpha})^2 +  ( \frac{4}{3} -\bar{\alpha})^2 \\
		&+ (\frac{41}{30} - \bar{\alpha})^2 + (\frac{43}{30}- \bar{\alpha})^2 + (\frac{3}{2} -\bar{\alpha})^2  + (\frac{47}{30} - \bar{\alpha})^2  +(\frac{49}{30}- \bar{\alpha})^2 + ( \frac{5}{3} - \bar{\alpha})^2 \\
		&+ ( \frac{17}{10} - \bar{\alpha})^2 + ( \frac{53}{30} - \bar{\alpha})^2 + ( \frac{11}{6} - \bar{\alpha})^2\\
		=& \frac{7063}{6750}  .
	\end{align*}
	And $ \alpha_{\mu_f -2}-\alpha_1= \frac{11}{6} - \frac{9}{10} =\frac{28}{30} .  $
	Then it suffices to prove $  \frac{7063}{6750}  - \frac{15}{12} \cdot \frac{28}{30} \leq 0,  $ i.e., $ - \frac{406}{3375} \leq 0, $
	which holds obviously.

	If  $\tau^{\mathrm{max}} = 16, $ 
	\begin{align*}
		& \bar{\alpha} =\frac{1}{\mu_f -1 }\cdot ( \frac{9}{10}+\frac{31}{30}+ \frac{7}{6} + \frac{37}{30} 
		+\frac{13}{10} + \frac{4}{3} + \frac{41}{30} + \frac{43}{30}  + \frac{3}{2} +\frac{47}{30} + \frac{49}{30} + \frac{5}{3} + \frac{17}{10} \\
		&+ \frac{53}{30} + \frac{11}{6} + \frac{59}{30})\\
		=&\frac{117}{80}.
	\end{align*}
	\begin{align*}
		& (\mu_f -  1) \cdot \mathrm{Var}[\mathrm{Sp^{\tau^{\mathrm{max}} }}(f)] \\
		=& (\frac{9}{10}-\bar{\alpha})^2+(\frac{31}{30}-\bar{\alpha})^2+ (\frac{7}{6}-\bar{\alpha})^2+ (\frac{37}{30} - \bar{\alpha})^2+(\frac{13}{10}  - \bar{\alpha})^2 +  ( \frac{4}{3} -\bar{\alpha})^2 \\
		&+ (\frac{41}{30} - \bar{\alpha})^2 + (\frac{43}{30}- \bar{\alpha})^2 + (\frac{3}{2} -\bar{\alpha})^2  + (\frac{47}{30} - \bar{\alpha})^2  +(\frac{49}{30}- \bar{\alpha})^2 + ( \frac{5}{3} - \bar{\alpha})^2 \\
		&+ ( \frac{17}{10} - \bar{\alpha})^2 + ( \frac{53}{30} - \bar{\alpha})^2 + ( \frac{11}{6} - \bar{\alpha})^2 +( \frac{59}{30} -\bar{\alpha})^2\\
		=& \frac{527}{400}  .
	\end{align*}
	And $ \alpha_{\mu_f -1}-\alpha_1= \frac{59}{30} - \frac{9}{10} =\frac{32}{30} .  $
	Then it suffices to prove $  \frac{527}{400} - \frac{16}{12} \cdot \frac{32}{30} \leq 0,  $ 
	i.e., $ - \frac{377}{3600} \leq 0, $
	which holds obviously.
	
	%	Considering the average value $ \bar{\alpha}=\frac{3}{2}, $
	%	\begin{align*}
		%		& (\mu_f -  1) \cdot \mathrm{Var}[\mathrm{Sp^{\tau^{\mathrm{max}} }}(f)] \\
		%		=& (\frac{9}{10}-\bar{\alpha})^2+(\frac{31}{30}-\bar{\alpha})^2+ (\frac{7}{6}-\bar{\alpha})^2+ (\frac{37}{30} - \bar{\alpha})^2+(\frac{13}{10}  - \bar{\alpha})^2 +  ( \frac{4}{3} -\bar{\alpha})^2 \\
		%		&+ (\frac{41}{30} - \bar{\alpha})^2 + (\frac{43}{30}- \bar{\alpha})^2 + (\frac{3}{2} -\bar{\alpha})^2 + (\frac{47}{30} - \bar{\alpha})^2  +(\frac{49}{30}- \bar{\alpha})^2 + ( \frac{5}{3} - \bar{\alpha})^2 \\
		%		&+ ( \frac{17}{10} - \bar{\alpha})^2 + ( \frac{53}{30} - \bar{\alpha})^2 + ( \frac{11}{6} - \bar{\alpha})^2\\
		%		=& \frac{101}{90}  .
		%	\end{align*}
	%	Then it suffices to prove $ \frac{101}{90}  - \frac{15}{12} \cdot \frac{28}{30} \leq 0,  $ i.e., $ - \frac{2}{45} \leq 0, $
	%	which holds obviously.
	%	
	%	
	
	\noindent\textbf{(24) $Q_{18}$.}
	
	Since $ \tau_{f_0} = \mu_f - 2 = 16, $ we need to verify the cases of $\tau^{\mathrm{max}} = 16 $ or $ 17. $ 
	If  $\tau^{\mathrm{max}} = 16, $ 
	\begin{align*}
		& \bar{\alpha} =\frac{1}{\mu_f - 2}\cdot ( \frac{43}{48}+\frac{49}{48}+ \frac{55}{48} + \frac{59}{48} 
		+\frac{61}{48} + \frac{4}{3} + \frac{65}{48} + \frac{67}{48}  + \frac{71}{48} +\frac{73}{48} + \frac{77}{48} + \frac{79}{48} + \frac{5}{3} + \frac{83}{48} \\
		&+ \frac{85}{48} + \frac{89}{48})\\
		=&\frac{275}{192}.
	\end{align*}
	\begin{align*}
		& (\mu_f -  2) \cdot \mathrm{Var}[\mathrm{Sp^{\tau^{\mathrm{max}} }}(f)] \\
		=& (\frac{43}{48}-\bar{\alpha})^2+(\frac{49}{48}-\bar{\alpha})^2+ (\frac{55}{48}-\bar{\alpha})^2+ (\frac{59}{48} - \bar{\alpha})^2+(\frac{61}{48}  - \bar{\alpha})^2 +  ( \frac{4}{3} -\bar{\alpha})^2 \\
		&+ (\frac{65}{48} - \bar{\alpha})^2 + (\frac{67}{48}- \bar{\alpha})^2 + (\frac{71}{48} -\bar{\alpha})^2  + (\frac{73}{48} - \bar{\alpha})^2  +(\frac{77}{48}- \bar{\alpha})^2 + ( \frac{79}{48} - \bar{\alpha})^2 \\
		&+ ( \frac{5}{3} - \bar{\alpha})^2 + ( \frac{83}{48} - \bar{\alpha})^2 + ( \frac{85}{48} - \bar{\alpha})^2 + ( \frac{89}{48} - \bar{\alpha})^2\\
		=& \frac{293}{256}  .
	\end{align*}
	And $ \alpha_{\mu_f -2}-\alpha_1= \frac{89}{48} - \frac{43}{48} =\frac{46}{48} .  $
	Then it suffices to prove $  \frac{293}{256}   - \frac{16}{12} \cdot \frac{46}{48} \leq 0,  $ i.e., $ - \frac{307}{2304} \leq 0, $
	which holds obviously.

	If  $\tau^{\mathrm{max}} = 17, $ 
	\begin{align*}
		& \bar{\alpha} =\frac{1}{\mu_f - 1}\cdot ( \frac{43}{48}+\frac{49}{48}+ \frac{55}{48} + \frac{59}{48} 
		+\frac{61}{48} + \frac{4}{3} + \frac{65}{48} + \frac{67}{48}  + \frac{71}{48} +\frac{73}{48} + \frac{77}{48} + \frac{79}{48} + \frac{5}{3} + \frac{83}{48} \\
		&+ \frac{85}{48} + \frac{89}{48} + \frac{95}{48})\\
		=&\frac{1195}{816}.
	\end{align*}
	\begin{align*}
		& (\mu_f -  1) \cdot \mathrm{Var}[\mathrm{Sp^{\tau^{\mathrm{max}} }}(f)] \\
		=& (\frac{43}{48}-\bar{\alpha})^2+(\frac{49}{48}-\bar{\alpha})^2+ (\frac{55}{48}-\bar{\alpha})^2+ (\frac{59}{48} - \bar{\alpha})^2+(\frac{61}{48}  - \bar{\alpha})^2 +  ( \frac{4}{3} -\bar{\alpha})^2 \\
		&+ (\frac{65}{48} - \bar{\alpha})^2 + (\frac{67}{48}- \bar{\alpha})^2 + (\frac{71}{48} -\bar{\alpha})^2  + (\frac{73}{48} - \bar{\alpha})^2  +(\frac{77}{48}- \bar{\alpha})^2 + ( \frac{79}{48} - \bar{\alpha})^2 \\
		&+ ( \frac{5}{3} - \bar{\alpha})^2 + ( \frac{83}{48} - \bar{\alpha})^2 + ( \frac{85}{48} - \bar{\alpha})^2 + ( \frac{89}{48} - \bar{\alpha})^2 +(\frac{95}{48} -\bar{\alpha})^2\\
		=& \frac{3103}{2176}  .
	\end{align*}
	And $ \alpha_{\mu_f -1}-\alpha_1= \frac{95}{48} - \frac{43}{48} =\frac{52}{48} .  $
	Then it suffices to prove $ \frac{3103}{2176} - \frac{17}{12} \cdot \frac{52}{48} \leq 0,  $ 
	i.e., $ - \frac{2129}{19584} \leq 0, $
	which holds obviously.
	
	%	Considering the average value $ \bar{\alpha}=\frac{3}{2}, $
	%	\begin{align*}
		%		& (\mu_f -  1) \cdot \mathrm{Var}[\mathrm{Sp^{\tau^{\mathrm{max}} }}(f)] \\
		%		=& (\frac{43}{48}-\bar{\alpha})^2+(\frac{49}{48}-\bar{\alpha})^2+ (\frac{55}{48}-\bar{\alpha})^2+ (\frac{59}{48} - \bar{\alpha})^2+(\frac{61}{48}  - \bar{\alpha})^2 +  ( \frac{4}{3} -\bar{\alpha})^2\\
		%		& + (\frac{65}{48} - \bar{\alpha})^2 + (\frac{67}{48}- \bar{\alpha})^2 + (\frac{71}{48} -\bar{\alpha})^2  + (\frac{73}{48} - \bar{\alpha})^2  +(\frac{77}{48}- \bar{\alpha})^2 + ( \frac{79}{48} - \bar{\alpha})^2\\
		%		& + ( \frac{5}{3} - \bar{\alpha})^2 + ( \frac{83}{48} - \bar{\alpha})^2 + ( \frac{85}{48} - \bar{\alpha})^2 + ( \frac{89}{48} - \bar{\alpha})^2\\
		%		=& \frac{1403}{1152}  .
		%	\end{align*}
	%	Then it suffices to prove $  \frac{1403}{1152} - \frac{16}{12} \cdot \frac{46}{48} \leq 0,  $ i.e., $ - \frac{23}{384} \leq 0, $
	%	which holds obviously.
	%	
	
	\noindent\textbf{(25) $S_{16}$.}
	
	Since $  \tau_{f_0} = \mu_f - 2 = 14, $ we need to verify the cases of $\tau^{\mathrm{max}} = 14 $ or $ 15. $ 
	If  $\tau^{\mathrm{max}} = 14, $ 
	\begin{align*}
		& \bar{\alpha} =\frac{1}{\mu_f -2}\cdot ( \frac{15}{17}+\frac{18}{17}+ \frac{20}{17} + \frac{21}{17} 
		+\frac{22}{17} + \frac{23}{17} + \frac{24}{17} + \frac{25}{17}  +\frac{26}{17} + \frac{27}{17} + \frac{28}{17} + \frac{29}{17} + \frac{30}{17} + \frac{31}{17} )\\
		=&\frac{339}{238}.
	\end{align*}
	\begin{align*}
		& (\mu_f -  2) \cdot \mathrm{Var}[\mathrm{Sp^{\tau^{\mathrm{max}} }}(f)] \\
		=& (\frac{15}{17}-\bar{\alpha})^2+(\frac{18}{17}-\bar{\alpha})^2+ (\frac{20}{17}-\bar{\alpha})^2+ (\frac{21}{17} - \bar{\alpha})^2+(\frac{22}{17}  - \bar{\alpha})^2 +  ( \frac{23}{17} -\bar{\alpha})^2 + (\frac{24}{17} - \bar{\alpha})^2 \\
		&+ (\frac{25}{17}- \bar{\alpha})^2  + (\frac{26}{17} - \bar{\alpha})^2  +(\frac{27}{17}- \bar{\alpha})^2 + ( \frac{28}{17} - \bar{\alpha})^2 + ( \frac{29}{17} - \bar{\alpha})^2 + ( \frac{30}{17} - \bar{\alpha})^2 + ( \frac{31}{17} - \bar{\alpha})^2\\
		=& \frac{4009}{4046}  .
	\end{align*}
	And $ \alpha_{\mu_f -2}-\alpha_1= \frac{31}{17} - \frac{15}{17} =\frac{16}{17} .  $
	Then it suffices to prove $  \frac{4009}{4046}  - \frac{14}{12} \cdot\frac{16}{17} \leq 0,  $ i.e., $ - \frac{1301}{12138} \leq 0, $
	which holds obviously.

	If  $\tau^{\mathrm{max}} = 15, $ 
	\begin{align*}
		& \bar{\alpha} =\frac{1}{\mu_f -1}\cdot ( \frac{15}{17}+\frac{18}{17}+ \frac{20}{17} + \frac{21}{17} 
		+\frac{22}{17} + \frac{23}{17} + \frac{24}{17} + \frac{25}{17}  +\frac{26}{17} + \frac{27}{17} + \frac{28}{17} + \frac{29}{17} + \frac{30}{17} \\
		&+ \frac{31}{17}  + \frac{33}{17})\\
		=&\frac{124}{85}.
	\end{align*}
	\begin{align*}
		& (\mu_f -  1) \cdot \mathrm{Var}[\mathrm{Sp^{\tau^{\mathrm{max}} }}(f)] \\
		=& (\frac{15}{17}-\bar{\alpha})^2+(\frac{18}{17}-\bar{\alpha})^2+ (\frac{20}{17}-\bar{\alpha})^2+ (\frac{21}{17} - \bar{\alpha})^2+(\frac{22}{17}  - \bar{\alpha})^2 +  ( \frac{23}{17} -\bar{\alpha})^2 + (\frac{24}{17} - \bar{\alpha})^2 \\
		&+ (\frac{25}{17}- \bar{\alpha})^2  + (\frac{26}{17} - \bar{\alpha})^2  +(\frac{27}{17}- \bar{\alpha})^2 + ( \frac{28}{17} - \bar{\alpha})^2 + ( \frac{29}{17} - \bar{\alpha})^2 + ( \frac{30}{17} - \bar{\alpha})^2 \\
		&+ ( \frac{31}{17} - \bar{\alpha})^2 +(\frac{33}{17} -\bar{\alpha})^2\\
		=& \frac{1792}{1445}  .
	\end{align*}
	And $ \alpha_{\mu_f -1}-\alpha_1= \frac{33}{17} - \frac{15}{17} =\frac{18}{17} .  $
	Then it suffices to prove $  \frac{1792}{1445} - \frac{15}{12} \cdot\frac{18}{17} \leq 0,  $ i.e., $ - \frac{241}{2890} \leq 0, $
	which holds obviously.
	%	Considering the average value $ \bar{\alpha}=\frac{3}{2}, $
	%	\begin{align*}
		%		& (\mu_f -  1) \cdot \mathrm{Var}[\mathrm{Sp^{\tau^{\mathrm{max}} }}(f)] \\
		%		=& (\frac{15}{17}-\bar{\alpha})^2+(\frac{18}{17}-\bar{\alpha})^2+ (\frac{20}{17}-\bar{\alpha})^2+ (\frac{21}{17} - \bar{\alpha})^2+(\frac{22}{17}  - \bar{\alpha})^2 +  ( \frac{23}{17} -\bar{\alpha})^2 + (\frac{24}{17} - \bar{\alpha})^2 \\
		%		&+ (\frac{25}{17}- \bar{\alpha})^2  + (\frac{26}{17} - \bar{\alpha})^2  +(\frac{27}{17}- \bar{\alpha})^2 + ( \frac{28}{17} - \bar{\alpha})^2 + ( \frac{29}{17} - \bar{\alpha})^2\\
		%		& + ( \frac{30}{17} - \bar{\alpha})^2 + ( \frac{31}{17} - \bar{\alpha})^2\\
		%		=& \frac{619}{578}  .
		%	\end{align*}
	%	Then it suffices to prove $  \frac{619}{578}  - \frac{16}{12} \cdot \frac{16}{17}  \leq 0,  $ i.e., $ - \frac{47}{1734} \leq 0, $
	%	which holds obviously.
	%	
	
	\noindent\textbf{(26) $S_{17}$.}
	
	Since $  \tau_{f_0}  = \mu_f - 2 = 15, $ we need to verify the cases of $\tau^{\mathrm{max}} = 15 $ or $ 16. $ 
	If  $\tau^{\mathrm{max}} = 15, $ 
	\begin{align*}
		& \bar{\alpha} =\frac{1}{\mu_f - 2}\cdot ( \frac{7}{8}+\frac{25}{24}+ \frac{7}{6} + \frac{29}{24} 
		+\frac{31}{24} + \frac{4}{3} + \frac{11}{8} + \frac{35}{24} + \frac{3}{2}  +\frac{37}{24} + \frac{13}{8} + \frac{5}{3} + \frac{41}{24} + \frac{43}{24} + \frac{11}{6} )\\
		=&\frac{257}{180}.
	\end{align*}
	\begin{align*}
		& (\mu_f -  2) \cdot \mathrm{Var}[\mathrm{Sp^{\tau^{\mathrm{max}} }}(f)] \\
		=& (\frac{7}{8}-\bar{\alpha})^2+(\frac{25}{24}-\bar{\alpha})^2+ (\frac{7}{6}-\bar{\alpha})^2+ (\frac{29}{24} - \bar{\alpha})^2+(\frac{31}{24}  - \bar{\alpha})^2 +  ( \frac{4}{3} -\bar{\alpha})^2 + (\frac{11}{8} - \bar{\alpha})^2\\
		& + (\frac{35}{24}- \bar{\alpha})^2 + (\frac{3}{2} - \bar{\alpha})^2+ (\frac{37}{24} - \bar{\alpha})^2  +(\frac{13}{8}- \bar{\alpha})^2 + ( \frac{5}{3} - \bar{\alpha})^2 + ( \frac{41}{24} - \bar{\alpha})^2 \\
		&+ ( \frac{43}{24} - \bar{\alpha})^2 + ( \frac{11}{6} - \bar{\alpha})^2\\
		=& \frac{4717}{4320}  .
	\end{align*}
	And $ \alpha_{\mu_f -2}-\alpha_1= \frac{11}{6} - \frac{7}{8} =\frac{23}{24} .  $
	Then it suffices to prove $  \frac{4009}{4046}  - \frac{15}{12} \cdot\frac{23}{24} \leq 0,  $ i.e., $ - \frac{229}{2160} \leq 0, $
	which holds obviously.

	If  $\tau^{\mathrm{max}} = 16, $ 
	\begin{align*}
		& \bar{\alpha} =\frac{1}{\mu_f - 1}\cdot ( \frac{7}{8}+\frac{25}{24}+ \frac{7}{6} + \frac{29}{24} 
		+\frac{31}{24} + \frac{4}{3} + \frac{11}{8} + \frac{35}{24} + \frac{3}{2}  +\frac{37}{24} + \frac{13}{8} + \frac{5}{3} + \frac{41}{24} + \frac{43}{24} \\
		&+ \frac{11}{6} + \frac{47}{24})\\
		=&\frac{187}{128}.
	\end{align*}
	\begin{align*}
		& (\mu_f -  1) \cdot \mathrm{Var}[\mathrm{Sp^{\tau^{\mathrm{max}} }}(f)] \\
		=& (\frac{7}{8}-\bar{\alpha})^2+(\frac{25}{24}-\bar{\alpha})^2+ (\frac{7}{6}-\bar{\alpha})^2+ (\frac{29}{24} - \bar{\alpha})^2+(\frac{31}{24}  - \bar{\alpha})^2 +  ( \frac{4}{3} -\bar{\alpha})^2 + (\frac{11}{8} - \bar{\alpha})^2\\
		& + (\frac{35}{24}- \bar{\alpha})^2 + (\frac{3}{2} - \bar{\alpha})^2+ (\frac{37}{24} - \bar{\alpha})^2  +(\frac{13}{8}- \bar{\alpha})^2 + ( \frac{5}{3} - \bar{\alpha})^2 + ( \frac{41}{24} - \bar{\alpha})^2 \\
		&+ ( \frac{43}{24} - \bar{\alpha})^2 + ( \frac{11}{6} - \bar{\alpha})^2 +(  \frac{47}{24} -\bar{\alpha})^2\\
		=& \frac{4165}{3072}  .
	\end{align*}
	And $ \alpha_{\mu_f -1}-\alpha_1= \frac{47}{24} - \frac{7}{8} =\frac{26}{24} .  $
	Then it suffices to prove $ \frac{4165}{3072}- \frac{16}{12} \cdot\frac{26}{24} \leq 0,  $ 
	i.e., $ - \frac{817}{9216} \leq 0, $
	which holds obviously.
	%	Considering the average value $ \bar{\alpha}=\frac{3}{2}, $
	%	\begin{align*}
		%		& (\mu_f -  1) \cdot \mathrm{Var}[\mathrm{Sp^{\tau^{\mathrm{max}} }}(f)] \\
		%		=& (\frac{7}{8}-\bar{\alpha})^2+(\frac{25}{24}-\bar{\alpha})^2+ (\frac{7}{6}-\bar{\alpha})^2+ (\frac{29}{24} - \bar{\alpha})^2+(\frac{31}{24}  - \bar{\alpha})^2 +  ( \frac{4}{3} -\bar{\alpha})^2 + (\frac{11}{8} - \bar{\alpha})^2 \\
		%		&+ (\frac{35}{24}- \bar{\alpha})^2 + (\frac{3}{2} - \bar{\alpha})^2+ (\frac{37}{24} - \bar{\alpha})^2  +(\frac{13}{8}- \bar{\alpha})^2 + ( \frac{5}{3} - \bar{\alpha})^2 + ( \frac{41}{24} - \bar{\alpha})^2\\
		%		& + ( \frac{43}{24} - \bar{\alpha})^2 + ( \frac{11}{6} - \bar{\alpha})^2\\
		%		=& \frac{337}{288}  .
		%	\end{align*}
	%	Then it suffices to prove $  \frac{337}{288} - \frac{15}{12} \cdot\frac{23}{24} \leq 0,  $ i.e., $ - \frac{1}{36} \leq 0, $
	%	which holds obviously.
	%	
	%	
	%	
	
	\noindent\textbf{(27) $U_{16}$.}
	
	Since $  \tau_{f_0} = \mu_f - 2 = 14, $ we need to verify the cases of $\tau^{\mathrm{max}} = 14 $ or $ 15. $ 
	If  $\tau^{\mathrm{max}} = 14, $ 
	\begin{align*}
		& \bar{\alpha} =\frac{1}{\mu_f -2}\cdot ( \frac{13}{15}+\frac{16}{15}+ 2\cdot\frac{6}{5} + \frac{19}{15} 
		+2\cdot \frac{7}{5} + \frac{22}{15}  +\frac{23}{15} + 2\cdot \frac{8}{5} + \frac{26}{15} + 2 \cdot \frac{9}{5} )\\
		=&\frac{299}{210}.
	\end{align*}
	\begin{align*}
		& (\mu_f -  2) \cdot \mathrm{Var}[\mathrm{Sp^{\tau^{\mathrm{max}} }}(f)] \\
		=& (\frac{13}{15}-\bar{\alpha})^2+(\frac{16}{15}-\bar{\alpha})^2+ 2\cdot (\frac{6}{5}-\bar{\alpha})^2+ (\frac{19}{15} - \bar{\alpha})^2+2\cdot (\frac{7}{5}  - \bar{\alpha})^2 +  ( \frac{22}{15} -\bar{\alpha})^2 \\
		&+ (\frac{23}{15} - \bar{\alpha})^2  +2\cdot (\frac{8}{5}- \bar{\alpha})^2 + ( \frac{26}{15} - \bar{\alpha})^2 + 2\cdot ( \frac{9}{5} - \bar{\alpha})^2\\
		=& \frac{3209}{3150}  .
	\end{align*}
	And $ \alpha_{\mu_f -2}-\alpha_1= \frac{9}{5} - \frac{13}{15} =\frac{14}{15} .  $
	Then it suffices to prove $  \frac{3209}{3150}  - \frac{14}{12} \cdot \frac{14}{15}  \leq 0,  $ i.e., $ - \frac{221}{3150} \leq 0, $
	which holds obviously.

	If  $\tau^{\mathrm{max}} = 15, $ 
	\begin{align*}
		& \bar{\alpha} =\frac{1}{\mu_f -1}\cdot ( \frac{13}{15}+\frac{16}{15}+ 2\cdot\frac{6}{5} + \frac{19}{15} 
		+2\cdot \frac{7}{5} + \frac{22}{15}  +\frac{23}{15} + 2\cdot \frac{8}{5} + \frac{26}{15} + 2 \cdot \frac{9}{5} + \frac{29}{15} )\\
		=&\frac{328}{225}.
	\end{align*}
	\begin{align*}
		& (\mu_f -  1) \cdot \mathrm{Var}[\mathrm{Sp^{\tau^{\mathrm{max}} }}(f)] \\
		=& (\frac{13}{15}-\bar{\alpha})^2+(\frac{16}{15}-\bar{\alpha})^2+ 2\cdot (\frac{6}{5}-\bar{\alpha})^2+ (\frac{19}{15} - \bar{\alpha})^2+2\cdot (\frac{7}{5}  - \bar{\alpha})^2 +  ( \frac{22}{15} -\bar{\alpha})^2 \\
		&+ (\frac{23}{15} - \bar{\alpha})^2  +2\cdot (\frac{8}{5}- \bar{\alpha})^2 + ( \frac{26}{15} - \bar{\alpha})^2 + 2\cdot ( \frac{9}{5} - \bar{\alpha})^2 +( \frac{29}{15} -\bar{\alpha})^2\\
		=& \frac{4256}{3375}  .
	\end{align*}
	And $ \alpha_{\mu_f -1}-\alpha_1= \frac{29}{15} - \frac{13}{15} =\frac{16}{15} .  $
	Then it suffices to prove $   \frac{4256}{3375}  - \frac{15}{12} \cdot \frac{16}{15}  \leq 0,  $ 
	i.e., $ - \frac{244}{3375} \leq 0, $
	which holds obviously.
	
	%	Considering the average value $ \bar{\alpha}=\frac{3}{2}, $
	%	\begin{align*}
		%		& (\mu_f -  1) \cdot \mathrm{Var}[\mathrm{Sp^{\tau^{\mathrm{max}} }}(f)] \\
		%		=& (\frac{13}{15}-\bar{\alpha})^2+(\frac{16}{15}-\bar{\alpha})^2+ 2\cdot (\frac{6}{5}-\bar{\alpha})^2+ (\frac{19}{15} - \bar{\alpha})^2+2\cdot (\frac{7}{5}  - \bar{\alpha})^2 +  ( \frac{22}{15} -\bar{\alpha})^2 \\
		%		&+ (\frac{23}{15} - \bar{\alpha})^2  +2\cdot (\frac{8}{5}- \bar{\alpha})^2 + ( \frac{26}{15} - \bar{\alpha})^2 + 2\cdot ( \frac{9}{5} - \bar{\alpha})^2\\
		%		=& \frac{11}{10}  .
		%	\end{align*}
	%	Then it suffices to prove $   \frac{11}{10}  - \frac{14}{12} \cdot \frac{14}{15}  \leq 0,  $ i.e., $ \frac{1}{90} \leq 0, $
	%	which holds obviously.
	%	%猜想对U_{16}不成立
\end{proof}

\subsection{Generalized Hertling Conjecture for Tjurina Spectrum  for Trimodal Singularities}

\begin{theorem}\label{HCTS3}
	\textbf{\textit{\textbf{Generalized Hertling Conjecture for Tjurina Spectrum}}} holds for singularities of modality 3 with related to Tjurina Spectrum.
\end{theorem}
\begin{proof}
	% By \textbf{Theorem \ref{Hertling Conjecture for curve singularities}}, 
	% we only need to consider surface cases. Furthermore,
	%  by \textbf{Theorem \ref{Varchenko's constant deformation}} and 
	%  \textbf{Theorem \ref{Hertlin Conjecture for quasihomogeneous singularities}},
	%   exceptional cases, say $VC,VF$ series satisfies Hertling Conjecture.
	%    Hence, it suffices to check the conjecture for $VA,VB$ series. 
	%    We still use $f_0$ to refer to the representative germ for each type.
	By \textbf{Theorem 3.8} in \cite{MR4448602}, we only need to consider those cases which are not quasihomogeneous.

	\noindent\textbf{(1) $NA_{r,0}$.}
	
	Since $  \tau_{f_0} = \mu_f - 2 = 14 + r, $ we need to verify the cases of $\tau^{\mathrm{max}} = 14 +r $ or $ 15+r. $ 
	If  $\tau^{\mathrm{max}} = 14+r, $ 
	\begin{align*}
		&\bar{\alpha}  = \frac{1}{\mu_f -2}\cdot( \frac{2}{5} +\frac{3}{5}+\frac{4}{5} + \frac{4}{5} + 1 + 1 +1 + \frac{6}{5} +\frac{6}{5} + \frac{7}{5} + \sum_{ k=1 } ^{ r+4} \frac{2k+r+4}{ 2(r+5)}) \\
		=&\frac{(10r^2 + 179r + 650)}{(10(r + 5)(14 + r))}.
	\end{align*}
	\begin{align*}
		& (\mu_f -  2) \cdot \mathrm{Var}[\mathrm{Sp^{\tau^{\mathrm{max}} }}(f)] \\
		=&(\frac{2}{5}-\bar{\alpha})^2+(\frac{3}{5}-\bar{\alpha})^2+2\cdot(\frac{4}{5}-\bar{\alpha})^2+3\cdot(1-\bar{\alpha})^2+2\cdot(\frac{6}{5}-\bar{\alpha})^2+(\frac{7}{5}-\bar{\alpha})^2\\
		& + \sum_{u = 1}^{r+4} (\frac{2u+r+4}{2(r+5)}-\bar{\alpha})^2 \\
		=& \frac{25r^4+902r^3+11135r^2+57580r+105900}{300(r + 5)^2(14 + r)}.
	\end{align*}
	And  
	\begin{equation*}
		\alpha_{\mu_f -2}-\alpha_1= \begin{cases}
			\frac{7}{5}-\frac{2}{5}=1, \qquad \forall 0 \leq r \leq 10 \\
			\frac{3(r+4)}{2(r+5)}-\frac{2}{5}, \qquad \forall r \geq 11 
		\end{cases}.
	\end{equation*}
	Then it suffices to prove
	\begin{equation*}
		\begin{cases}
			\frac{25r^4+902r^3+11135r^2+57580r+105900}{300(r + 5)^2(14 + r)}-\frac{14+r}{12} \cdot 1\leq 0, \qquad \forall 0 \leq r \leq 10 \\
			\frac{25r^4+902r^3+11135r^2+57580r+105900}{300(r + 5)^2(14 + r)}-\frac{14+r}{12}\cdot (\frac{3(r+4)}{2(r+5)}-\frac{2}{5}) \leq 0, \qquad \forall r \geq 11 
		\end{cases}.
	\end{equation*}
	Equivalently,
	\begin{equation*}
		\begin{cases}
			\frac{-24r^3 - 695r^2 - 4460r - 8300}{150(r + 5)^2(14 + r)} \leq 0, \qquad \forall 0 \leq r \leq 10 \\
			\frac{-5r^4 - 211r^3 - 2810r^2 - 5940r + 15800}{600(r + 5)^2(14 + r)} \leq 0, \qquad \forall r \geq 11 
		\end{cases}
	\end{equation*}
	which hold obviously.

	If  $\tau^{\mathrm{max}} = 15+r, $ 
	\begin{align*}
		&\bar{\alpha}  = \frac{1}{\mu_f -1}\cdot( \frac{2}{5} +\frac{3}{5}+\frac{4}{5} + \frac{4}{5} + 1 + 1 +1 + \frac{6}{5} +\frac{6}{5} + \frac{7}{5} + \sum_{ k=1 } ^{ r+5} \frac{2k+r+4}{ 2(r+5)}) \\
		=&\frac{72 + 5r}{75 + 5r}.
	\end{align*}
	\begin{align*}
		& (\mu_f -  1) \cdot \mathrm{Var}[\mathrm{Sp^{\tau^{\mathrm{max}} }}(f)] \\
		=&(\frac{2}{5}-\bar{\alpha})^2+(\frac{3}{5}-\bar{\alpha})^2+2\cdot(\frac{4}{5}-\bar{\alpha})^2+3\cdot(1-\bar{\alpha})^2+2\cdot(\frac{6}{5}-\bar{\alpha})^2+(\frac{7}{5}-\bar{\alpha})^2\\
		& + \sum_{u = 1}^{r+5} (\frac{2u+r+4}{2(r+5)}-\bar{\alpha})^2 \\
		=& \frac{25r^3 + 877r^2 + 9282r + 27360}{300(r + 5)(15 + r)}.
	\end{align*}
	And  
	\begin{equation*}
		\alpha_{\mu_f -1}-\alpha_1= \frac{3r+14}{2(r+5)}-\frac{2}{5}.
	\end{equation*}
	Then it suffices to prove
	\begin{equation*}
		\frac{25r^3 + 877r^2 + 9282r + 27360}{300(r + 5)(15 + r)}-\frac{15+r}{12}\cdot (\frac{3r+14}{2(r+5)}-\frac{2}{5}) \leq 0.
	\end{equation*}
	Equivalently,
	\begin{equation*}
		\frac{-5r^3 - 146r^2 - 1311r - 1530}{600(r + 5)(15 + r)} \leq 0,
	\end{equation*}
	which holds obviously.
	%	
	%	Considering the average value $ \bar{\alpha}=1, $
	%	\begin{align*}
		%		& (\mu_f -  1) \cdot \mathrm{Var}[\mathrm{Sp^{\tau^{\mathrm{max}} }}(f)] \\
		%		=&(\frac{2}{5}-\bar{\alpha})^2+(\frac{3}{5}-\bar{\alpha})^2+2\cdot(\frac{4}{5}-\bar{\alpha})^2 +3\cdot(1-\bar{\alpha})^2+2\cdot(\frac{6}{5}-\bar{\alpha})^2+(\frac{7}{5}-\bar{\alpha})^2 \\
		%		&+ \sum_{u = 1}^{r+4} (\frac{2u+r+4}{2(r+5)}-\bar{\alpha})^2 \\
		%		=& \frac{25r^3 + 552r^2 + 3770r + 8100}{300(r + 5)^2}.
		%	\end{align*}
	%	The range of $\mathrm{Sp^\tau}(f_0)$ is 
	%	\begin{equation*}
		%		\begin{cases}
			%			\frac{7}{5}-\frac{2}{5}=1, \qquad \forall 0 \leq r \leq 10 \\
			%			\frac{3(r+4)}{2(r+5)}-\frac{2}{5}, \qquad \forall r \geq 11 
			%		\end{cases}.
		%	\end{equation*}
	%	Then it suffices to prove
	%	\begin{equation*}
		%		\begin{cases}
			%			\frac{25r^3 + 552r^2 + 3770r + 8100}{300(r + 5)^2}-\frac{14+r}{12} \cdot 1\leq 0, \qquad \forall 0 \leq r \leq 10 \\
			%			\frac{25r^3 + 552r^2 + 3770r + 8100}{300(r + 5)^2}-\frac{14+r}{12}\cdot (\frac{3(r+4)}{2(r+5)}-\frac{2}{5}) \leq 0, \qquad \forall r \geq 11 
			%		\end{cases}.
		%	\end{equation*}
	%	Equivalently,
	%	\begin{equation*}
		%		\begin{cases}
			%			\frac{-48r^2 - 355r - 650}{300(r + 5)^2} \leq 0, \qquad \forall 0 \leq r \leq 10 \\
			%			\frac{-5r^3 - 141r^2 - 110r + 2200}{600(r + 5)^2} \leq 0, \qquad \forall r \geq 11 
			%		\end{cases},
		%	\end{equation*}
	%	which hold obviously.
	%	
	
	\noindent\textbf{(2) $NA_{r,s}$.}
	
	Since $ \tau_{f_0} = \mu_f - 2 = 14 + r +s, $ we need to verify the cases of $\tau^{\mathrm{max}} = 14 +r +s$ or $ 15+r+s. $ For  $ r \leq s, $ if  $\tau^{\mathrm{max}} = 14+r +s, $ 
	\begin{align*}
		& \bar{\alpha} = \frac{1}{\mu_f-2}\cdot( \frac{2}{5} +\frac{4}{5}+1+1+\frac{6}{5} + \sum_{ u=1 } ^{ r+5} \frac{2u+r+4}{ 2(r+5)}+\sum_{ v=1 } ^{ s+4} \frac{2v+s+4}{ 2(s+5)})\\
		=& \frac{10r^2 + (10s + 179)r + 50s + 650}{10(r + 5)(14 + r + s)}.
	\end{align*}
	\begin{align*}
		& (\mu_f -  2) \cdot \mathrm{Var}[\mathrm{Sp^{\tau^{\mathrm{max}} }}(f)] \\
		= & (\frac{2}{5}-\bar{\alpha})^2+(\frac{4}{5}-\bar{\alpha})^2+2\cdot(1-\bar{\alpha})^2+2\cdot(\frac{6}{5}-\bar{\alpha})^2 + \sum_{u = 1}^{r+5} (\frac{2u+r+4}{2(r+5)}-\bar{\alpha})^2\\
		&  +\sum_{v = 1}^{s+4} (\frac{2v+s+4}{2(s+5)}-\bar{\alpha})^2 \\
		= & \frac{1}{300(14 + r + s)(r + 5)^2(s + 5)^2} \cdot ( 25r^4(s + 5)^2 + (50s^3 + 1407s^2 + 10445s + 23225)r^3 \\
		&+ (25s^4 + 1407s^3 + 23425s^2 + 143855s + 292475)r^2 + (250s^4 + 10295s^3 \\
		&+ 140750s^2 + 777150s + 1489000)r + 625s^4 + 22550s^3 + 278450s^2 + 1439500s + 2647500).
	\end{align*}
	And $ \alpha_{\mu_f -2}-\alpha_1 = \frac{3r+14}{2(r+5)}-\frac{2}{5}. $ 
	Then it suffices to prove
	\begin{align*}
		&\frac{1}{300(14 + r + s)(r + 5)^2(s + 5)^2} \cdot ( 25r^4(s + 5)^2 + (50s^3 + 1407s^2 + 10445s + 23225)r^3\\
		& + (25s^4 + 1407s^3 + 23425s^2 + 143855s + 292475)r^2 + (250s^4 + 10295s^3\\
		& + 140750s^2 + 777150s + 1489000)r + 625s^4 + 22550s^3 + 278450s^2 + 1439500s + 2647500)\\
		&- \frac{14+r+s}{12} \cdot ( \frac{3r+14}{2(r+5)}-\frac{2}{5}) \leq 0.
	\end{align*}
	Equivalently,
	\begin{align*}
		&\frac{1}{600(14 + r + s)(r + 5)^2(s + 5)^2} \cdot (-5r^4(s + 5)^2 + (-10s^3 - 351s^2 - 2510s - 5175)r^3\\
		& + (-5s^4 - 326s^3 - 7155s^2 - 44340s - 83300)r^2 + (-25s^4 - 1860s^3 - 41525s^2 \\
		&- 254700s - 469500)r - 2400s^3 - 69350s^2 - 446000s - 830000)\leq 0,
	\end{align*}
	which  holds obviously.

	If  $\tau^{\mathrm{max}} = 15+r +s, $ 
	\begin{align*}
		& \bar{\alpha} = \frac{1}{\mu_f-1}\cdot( \frac{2}{5} +\frac{4}{5}+1+1+\frac{6}{5} + \sum_{ u=1 } ^{ r+5} \frac{2u+r+4}{ 2(r+5)}+\sum_{ v=1 } ^{ s+5} \frac{2v+s+4}{ 2(s+5)})\\
		=& \frac{72 + 5r + 5s}{75 + 5r + 5s}.
	\end{align*}
	\begin{align*}
		& (\mu_f -  1) \cdot \mathrm{Var}[\mathrm{Sp^{\tau^{\mathrm{max}} }}(f)] \\
		= & (\frac{2}{5}-\bar{\alpha})^2+(\frac{4}{5}-\bar{\alpha})^2+2\cdot(1-\bar{\alpha})^2+2\cdot(\frac{6}{5}-\bar{\alpha})^2 + \sum_{u = 1}^{r+5} (\frac{2u+r+4}{2(r+5)}-\bar{\alpha})^2\\
		&  +\sum_{v = 1}^{s+5} (\frac{2v+s+4}{2(s+5)}-\bar{\alpha})^2 \\
		= & \frac{1}{300(s + 5)(r + 5)(15 + r + s)} \cdot ( (25s + 125)r^3 + (50s^2 + 1132s + 4385)r^2 \\
		&+ (25s^3 + 1132s^2 + 14392s + 46410)r + 125s^3 + 4385s^2 + 46410s + 136800).
	\end{align*}
	And $ \alpha_{\mu_f -1}-\alpha_1=\frac{3s+14}{2(s+5)}-\frac{2}{5}. $ 
	Then it suffices to prove
	\begin{align*}
		& \frac{1}{300(s + 5)(r + 5)(15 + r + s)} \cdot ( (25s + 125)r^3 + (50s^2 + 1132s + 4385)r^2\\
		&+ (25s^3 + 1132s^2 + 14392s + 46410)r  + 125s^3 + 4385s^2 + 46410s + 136800)\\
		&- \frac{14+r+s}{12} \cdot (\frac{3s+14}{2(s+5)}-\frac{2}{5}) \leq 0.
	\end{align*}
	Equivalently,
	\begin{align*}
		&\frac{1}{600(s + 5)(r + 5)(15 + r + s)} \cdot ((-5r - 25)s^3 + (-10r^2 - 186r - 730)s^2 \\
		&+ (-5r^3 - 161r^2 - 1841r - 6555)s + 20r^2 - 930r - 7650)\leq 0,
	\end{align*}
	which  holds obviously.

	\noindent\textbf{(3) $NB_{(-1)} ^r$.}
	
	Since $ \tau_{f_0} = \mu_f - 2 = 16 + r, $ we need to verify the cases of $\tau^{\mathrm{max}} = 16 +r$ or $ 17+r. $ If  $\tau^{\mathrm{max}} = 16+r, $ 
	for $ r \leq 3, $
	\begin{align*}
		&\bar{\alpha} = \frac{1}{\mu_f -2}\cdot( \frac{7}{18} +\frac{5}{9}+\frac{13}{18} + \frac{7}{9} + \frac{8}{9} + \frac{17}{18} +1+\frac{19}{18} + \frac{10}{9} +\frac{11}{9} + \frac{23}{18} + \sum_{ u=1 } ^{ r+5} \frac{2u+r+4}{ 2(r+5)})\\
		&= \frac{269 + 18r}{288 + 18r}.
	\end{align*}
	\begin{align*}
		& (\mu_f -2) \cdot \mathrm{Var}[\mathrm{Sp^{\tau^{\mathrm{max}} }}(f)] \\
		& = (\frac{7}{18}-\bar{\alpha})^2+(\frac{5}{9}-\bar{\alpha})^2+(\frac{13}{18}-\bar{\alpha})^2+(\frac{7}{9}-\bar{\alpha})^2+(\frac{8}{9}-\bar{\alpha})^2+( \frac{17}{18}-\bar{\alpha})^2+(1-\bar{\alpha})^2\\
		&+(\frac{19}{18}-\bar{\alpha})^2+(\frac{10}{9}-\bar{\alpha})^2+(\frac{11}{9}-\bar{\alpha})^2+(\frac{23}{18}-\bar{\alpha})^2+\sum_{u = 1}^{r+5} (\frac{2u+r+4}{2(r+5)}-\bar{\alpha})^2 \\
		&= \frac{27r^3 + 979r^2 + 10424r + 30723}{324(r + 5)(16 + r)}.
	\end{align*}
	
	And $ \alpha_{\mu_f -2}-\alpha_1= \frac{3r+14}{2(r+5)}-\frac{7}{18}. $
	Then it suffices to prove
	\begin{equation*}
		\frac{27r^3 + 979r^2 + 10424r + 30723}{324(r + 5)(16 + r)}-\frac{16+r}{12} \cdot (\frac{3r+14}{2(r+5)}-\frac{7}{18}) \leq 0.
	\end{equation*}
	Equivalently,
	\begin{equation*}
		\frac{-6r^3 - 235r^2 - 3248r - 8442}{648(r + 5)(16 + r)} \leq 0,
	\end{equation*}
	which holds obviously.

	For $ r \geq 4, $
	\begin{align*}
		&\bar{\alpha} = \frac{1}{\mu_f -2}\cdot( \frac{7}{18} +\frac{5}{9}+\frac{13}{18} + \frac{7}{9} + \frac{8}{9} + \frac{17}{18} +1+\frac{19}{18} + \frac{10}{9} +\frac{11}{9} + \frac{23}{18} + \frac{13}{9}+\sum_{ u=1 } ^{ r+4} \frac{2u+r+4}{ 2(r+5)})\\
		&= \frac{(18r^2 + 358r + 1349)}{(18(r + 5)(16 + r))}.
	\end{align*}
	\begin{align*}
		& (\mu_f -2) \cdot \mathrm{Var}[\mathrm{Sp^{\tau^{\mathrm{max}} }}(f)] \\
		& = (\frac{7}{18}-\bar{\alpha})^2+(\frac{5}{9}-\bar{\alpha})^2+(\frac{13}{18}-\bar{\alpha})^2+(\frac{7}{9}-\bar{\alpha})^2+(\frac{8}{9}-\bar{\alpha})^2+( \frac{17}{18}-\bar{\alpha})^2+(1-\bar{\alpha})^2\\
		&+(\frac{19}{18}-\bar{\alpha})^2+(\frac{10}{9}-\bar{\alpha})^2+(\frac{11}{9}-\bar{\alpha})^2+(\frac{23}{18}-\bar{\alpha})^2+(\frac{13}{9}-\bar{\alpha})^2+ \sum_{u = 1}^{r+4} (\frac{2u+r+4}{2(r+5)}-\bar{\alpha})^2 \\
		&= \frac{27r^4 + 1097r^3 + 15000r^2 + 82989r + 159223}{324(r + 5)^2(16 + r)}.
	\end{align*}
	And 
	\begin{equation*}
		\alpha_{\mu_f -2}-\alpha_1= \begin{cases}
			\frac{13}{9}-\frac{7}{18}=\frac{19}{18}, \qquad \forall 4 \leq r \leq 21 \\
			\frac{3(r+4)}{2(r+5)}-\frac{7}{18}, \qquad \forall r \geq 22 
		\end{cases}.
	\end{equation*}
	Then it suffices to prove
	\begin{equation*}
		\begin{cases}
			\frac{27r^4 + 1097r^3 + 15000r^2 + 82989r + 159223}{324(r + 5)^2(16 + r)}-\frac{16+r}{12} \cdot \frac{19}{18} \leq 0, \qquad \forall 4 \leq r \leq 21 \\
			\frac{27r^4 + 1097r^3 + 15000r^2 + 82989r + 159223}{324(r + 5)^2(16 + r)}-\frac{16+r}{12}\cdot (\frac{3(r+4)}{2(r+5)}-\frac{7}{18}) \leq 0, \qquad \forall r \geq 22 
		\end{cases}
	\end{equation*}
	Equivalently,
	\begin{equation*}
		\begin{cases}
			\frac{-3r^4 - 200r^3 - 4257r^2 - 25542r - 46354}{648(r + 5)^2(16 + r)} \leq 0, \qquad \forall 4 \leq r \leq 21 \\
			\frac{-6r^4 - 245r^3 - 3063r^2 - 1926r + 38126}{648(r + 5)^2(16 + r)} \leq 0, \qquad \forall r \geq 22 
		\end{cases},
	\end{equation*}
	which hold obviously.
	
	If  $\tau^{\mathrm{max}} = 17+r, $ 
	\begin{align*}
		&\bar{\alpha} = \frac{1}{\mu_f -1}\cdot( \frac{7}{18} +\frac{5}{9}+\frac{13}{18} + \frac{7}{9} + \frac{8}{9} + \frac{17}{18} +1+\frac{19}{18} + \frac{10}{9} +\frac{11}{9} + \frac{23}{18} + \frac{13}{9}+\sum_{ u=1 } ^{ r+5} \frac{2u+r+4}{ 2(r+5)})\\
		&= \frac{295 + 18r}{306 + 18r}.
	\end{align*}
	\begin{align*}
		& (\mu_f -1) \cdot \mathrm{Var}[\mathrm{Sp^{\tau^{\mathrm{max}} }}(f)] \\
		& = (\frac{7}{18}-\bar{\alpha})^2+(\frac{5}{9}-\bar{\alpha})^2+(\frac{13}{18}-\bar{\alpha})^2+(\frac{7}{9}-\bar{\alpha})^2+(\frac{8}{9}-\bar{\alpha})^2+( \frac{17}{18}-\bar{\alpha})^2+(1-\bar{\alpha})^2\\
		&+(\frac{19}{18}-\bar{\alpha})^2+(\frac{10}{9}-\bar{\alpha})^2+(\frac{11}{9}-\bar{\alpha})^2+(\frac{23}{18}-\bar{\alpha})^2+(\frac{13}{9}-\bar{\alpha})^2+ \sum_{u = 1}^{r+5} (\frac{2u+r+4}{2(r+5)}-\bar{\alpha})^2 \\
		&= \frac{27r^3 + 1070r^2 + 12619r + 39396}{324(r + 5)(17 + r)}.
	\end{align*}
	
	And 
	\begin{equation*}
		\alpha_{\mu_f -1}-\alpha_1= \begin{cases}
			\frac{13}{9}-\frac{7}{18}=\frac{19}{18}, \qquad \forall 0 \leq r \leq 3 \\
			\frac{3r+14}{2(r+5)}-\frac{7}{18}, \qquad \forall r \geq 4 
		\end{cases}.
	\end{equation*}
	Then it suffices to prove
	\begin{equation*}
		\begin{cases}
			\frac{27r^3 + 1070r^2 + 12619r + 39396}{324(r + 5)(17 + r)}-\frac{17+r}{12} \cdot \frac{19}{18} \leq 0, \qquad \forall 0 \leq r \leq 3, \\
			\frac{27r^3 + 1070r^2 + 12619r + 39396}{324(r + 5)(17 + r)}-\frac{17+r}{12}\cdot (\frac{3r+14}{2(r+5)}-\frac{7}{18}) \leq 0, \qquad \forall r \geq 4
		\end{cases}.
	\end{equation*}
	Equivalently,
	\begin{equation*}
		\begin{cases}
			\frac{-3r^3 - 83r^2 - 925r - 3573}{648(r + 5)(17 + r)} \leq 0, \qquad \forall 0 \leq r \leq 3\\
			\frac{-6r^3 - 173r^2 - 1384r - 105}{648(r + 5)(17 + r)} \leq 0, \qquad \forall r \geq 4 
		\end{cases},
	\end{equation*}
	which hold obviously.
	%	Considering the average value $ \bar{\alpha}=1, $
	%	\begin{align*}
		%		& (\mu_f -  1) \cdot \mathrm{Var}[\mathrm{Sp^{\tau^{\mathrm{max}} }}(f)] \\
		%		&=  (\frac{7}{18}-1)^2+(\frac{5}{9}-1)^2+(\frac{13}{18}-1)^2+(\frac{7}{9}-1)^2+(\frac{8}{9}-1)^2+( \frac{17}{18}-1)^2+(1-1)^2\\
		%		&+(\frac{19}{18}-1)^2+(\frac{10}{9}-1)^2+(\frac{11}{9}-1)^2+(\frac{23}{18}-1)^2+(\frac{13}{9}-1)^2 + \sum_{u = 1}^{r+4} (\frac{2u+r+4}{2(r+5)}-1)^2 \\
		%		&= \frac{27r^3 + 665r^2 + 4760r + 10469}{324(r + 5)^2}.
		%	\end{align*}
	%	Then it suffices to prove
	%	\begin{center}
		%		$\begin{cases}
			%			\frac{27r^3 + 665r^2 + 4760r + 10469}{324(r + 5)^2}-\frac{16+r}{12}\cdot \frac{19}{18}, \qquad \forall 0 \leq r \leq 21 \\
			%			\frac{27r^3 + 665r^2 + 4760r + 10469}{324(r + 5)^2}-\frac{16+r}{12}\cdot (\frac{3(r+4)}{2(r+5)}-\frac{7}{18}) \leq 0, \qquad \forall r \geq 22 
			%		\end{cases}
		%		$.
		%	\end{center}
	%	Equivalently,
	%	\begin{center}
		%		$\begin{cases}
			%			\frac{-3r^3 - 152r^2 - 1025r - 1862}{648(r + 5)^2}, \qquad \forall 0 \leq r \leq 21 \\
			%			\frac{-6r^3 - 149r^2 + 121r + 3418}{648(r + 5)^2} \leq 0, \qquad \forall r \geq 22 
			%		\end{cases}
		%		$,
		%	\end{center}
	%	which hold obviously.
	%	

	\noindent\textbf{(4) $NB_{(0)} ^r$.}
	
	Since $  \tau_{f_0}  = \mu_f - 2 = 17 + r, $ we need to verify the cases of $\tau^{\mathrm{max}} = 17 +r$ or $ 18+r. $ If  $\tau^{\mathrm{max}} = 17+r, $
	for $ r \leq 7, $ 
	\begin{align*}
		&\bar{\alpha} = \frac{1}{\mu_f -2}\cdot( \frac{5}{13} +\frac{7}{13}+\frac{9}{13} + \frac{10}{13} + \frac{11}{13} + \frac{12}{13} +1+1 +\frac{14}{13} + \frac{15}{13} +\frac{16}{13} + \frac{17}{13} +\sum_{ u=1 } ^{ r+5} \frac{2u+r+4}{ 2(r+5)})\\
		&= \frac{207 + 13r}{221 + 13r}.
	\end{align*}
	\begin{align*}
		& (\mu_f -2) \cdot\mathrm{Var}[\mathrm{Sp^{\tau^{\mathrm{max}} }}(f)] \\
		& =(\frac{5}{13}-\bar{\alpha})^2+(\frac{7}{13}-\bar{\alpha})^2+(\frac{9}{13}-\bar{\alpha})^2+(\frac{10}{13}-\bar{\alpha})^2+(\frac{11}{13}-\bar{\alpha})^2+( \frac{12}{13}-\bar{\alpha})^2 +2\cdot (1-\bar{\alpha})^2\\
		&+(\frac{14}{13}-\bar{\alpha})^2+(\frac{15}{13}-\bar{\alpha})^2+(\frac{16}{13}-\bar{\alpha})^2+(\frac{17}{13}-\bar{\alpha})^2  + \sum_{u = 1}^{r+5} (\frac{2u+r+4}{2(r+5)}-\bar{\alpha})^2 \\
		&= \frac{169r^3 + 6483r^2 + 72674r + 220392}{2028(r + 5)(17 + r)}.
	\end{align*}
	And $ \alpha_{\mu_f -2}-\alpha_1= \frac{3r+14}{2(r+5)}-\frac{5}{13}.$
	Then it suffices to prove
	\begin{equation*}
		\frac{169r^3 + 6483r^2 + 72674r + 220392}{2028(r + 5)(17 + r)}- \frac{17+r}{12} \cdot (\frac{3r+14}{2(r+5)}-\frac{5}{13})\leq 0.
	\end{equation*}
	Equivalently,
	\begin{equation*}
		\frac{-39r^3 - 1568r^2 - 21949r - 55140}{4056(r + 5)(17 + r)} \leq 0,
	\end{equation*}
	which holds obviously.

	for $ r \geq 8, $ 
	\begin{align*}
		&\bar{\alpha} = \frac{1}{\mu_f -2}\cdot( \frac{5}{13} +\frac{7}{13}+\frac{9}{13} + \frac{10}{13} + \frac{11}{13} + \frac{12}{13} +1+1 +\frac{14}{13} + \frac{15}{13} +\frac{16}{13} + \frac{17}{13} +\frac{19}{13}\\
		& +\sum_{ u=1 } ^{ r+4} \frac{2u+r+4}{ 2(r+5)})\\
		&= \frac{26r^2 + 543r + 2078}{26(r + 5)(17 + r)}.
	\end{align*}
	\begin{align*}
		& (\mu_f -2) \cdot\mathrm{Var}[\mathrm{Sp^{\tau^{\mathrm{max}} }}(f)] \\
		& =(\frac{5}{13}-\bar{\alpha})^2+(\frac{7}{13}-\bar{\alpha})^2+(\frac{9}{13}-\bar{\alpha})^2+(\frac{10}{13}-\bar{\alpha})^2+(\frac{11}{13}-\bar{\alpha})^2+( \frac{12}{13}-\bar{\alpha})^2 +2\cdot (1-\bar{\alpha})^2\\
		&+(\frac{14}{13}-\bar{\alpha})^2+(\frac{15}{13}-\bar{\alpha})^2+(\frac{16}{13}-\bar{\alpha})^2+(\frac{17}{13}-\bar{\alpha})^2+(\frac{19}{13}-\bar{\alpha})^2  + \sum_{u = 1}^{r+4} (\frac{2u+r+4}{2(r+5)}-\bar{\alpha})^2 \\
		&= \frac{169r^4 + 7253r^3 + 103907r^2 + 591490r + 1154184}{2028(r + 5)^2(17 + r)}.
	\end{align*}
	And 
	\begin{equation*}
		\alpha_{\mu_f -2}-\alpha_1= \begin{cases}
			\frac{19}{13}-\frac{5}{13}=\frac{14}{13}, \qquad \forall 8 \leq r \leq 33 \\
			\frac{3(r+4)}{2(r+5)}-\frac{5}{13}, \qquad \forall r \geq 34 
		\end{cases}.
	\end{equation*}
	Then it suffices to prove
	\begin{equation*}
		\begin{cases}
			\frac{169r^4 + 7253r^3 + 103907r^2 + 591490r + 1154184}{2028(r + 5)^2(17 + r)}- \frac{17+r}{12} \cdot \frac{14}{13} \leq 0, \qquad \forall 8 \leq r \leq 33 \\
			\frac{169r^4 + 7253r^3 + 103907r^2 + 591490r + 1154184}{2028(r + 5)^2(17 + r)} -\frac{17+r}{12}\cdot (\frac{3(r+4)}{2(r+5)}-\frac{7}{18}) \leq 0, \qquad \forall r \geq 34 
		\end{cases}.
	\end{equation*}
	Equivalently,
	\begin{equation*}
		\begin{cases}
			\frac{-13r^4 - 755r^3 - 15121r^2 - 89190r - 160766}{2028(r + 5)^2(17 + r)} \leq 0, \qquad \forall 8 \leq r \leq 33 \\
			\frac{-39r^4 - 1575r^3 - 18971r^2 + 5713r + 317158}{4056(r + 5)^2(17 + r)} \leq 0, \qquad \forall r \geq 34 
		\end{cases},
	\end{equation*}
	which hold obviously.

	If  $\tau^{\mathrm{max}} = 18+r, $ 
	\begin{align*}
		&\bar{\alpha} = \frac{1}{\mu_f -1}\cdot( \frac{5}{13} +\frac{7}{13}+\frac{9}{13} + \frac{10}{13} + \frac{11}{13} + \frac{12}{13} +1+1 +\frac{14}{13} + \frac{15}{13} +\frac{16}{13} + \frac{17}{13} +\frac{19}{13}\\
		& +\sum_{ u=1 } ^{ r+5} \frac{2u+r+4}{ 2(r+5)})\\
		&= \frac{226 + 13r}{234 + 13r}.
	\end{align*}
	\begin{align*}
		& (\mu_f -1) \cdot\mathrm{Var}[\mathrm{Sp^{\tau^{\mathrm{max}} }}(f)] \\
		& =(\frac{5}{13}-\bar{\alpha})^2+(\frac{7}{13}-\bar{\alpha})^2+(\frac{9}{13}-\bar{\alpha})^2+(\frac{10}{13}-\bar{\alpha})^2+(\frac{11}{13}-\bar{\alpha})^2+( \frac{12}{13}-\bar{\alpha})^2 +2\cdot (1-\bar{\alpha})^2\\
		&+(\frac{14}{13}-\bar{\alpha})^2+(\frac{15}{13}-\bar{\alpha})^2+(\frac{16}{13}-\bar{\alpha})^2+(\frac{17}{13}-\bar{\alpha})^2+(\frac{19}{13}-\bar{\alpha})^2  + \sum_{u = 1}^{r+5} (\frac{2u+r+4}{2(r+5)}-\bar{\alpha})^2 \\
		&= \frac{169r^3 + 7084r^2 + 87804r + 280848}{2028(r + 5)(18 + r)}.
	\end{align*}
	And 
	\begin{equation*}
		\alpha_{\mu_f -1}-\alpha_1= \begin{cases}
			\frac{19}{13}-\frac{5}{13}=\frac{14}{13}, \qquad \forall 0 \leq r \leq 7 \\
			\frac{3r+14}{2(r+5)}-\frac{5}{13}, \qquad \forall r \geq 8
		\end{cases}.
	\end{equation*}
	Then it suffices to prove
	\begin{equation*}
		\begin{cases}
			\frac{169r^3 + 7084r^2 + 87804r + 280848}{2028(r + 5)(18 + r)}- \frac{18+r}{12} \cdot \frac{14}{13} \leq 0, \qquad \forall 0 \leq r \leq 14 \\
			\frac{169r^3 + 7084r^2 + 87804r + 280848}{2028(r + 5)(18 + r)} -\frac{18+r}{12}\cdot (\frac{3r+14}{2(r+5)}-\frac{7}{18}) \leq 0, \qquad \forall r \geq 15
		\end{cases}.
	\end{equation*}
	Equivalently,
	\begin{equation*}
		\begin{cases}
			\frac{-13r^3 - 378r^2 - 3924r - 13992}{2028(r + 5)(18 + r)} \leq 0, \qquad \forall 0 \leq r \leq 14 \\
			\frac{-39r^3 - 1120r^2 - 8316r + 5712}{4056(r + 5)(18 + r)} \leq 0, \qquad \forall r \geq 15 
		\end{cases},
	\end{equation*}
	which hold obviously.

	%	Considering the average value $ \bar{\alpha}=1, $
	%	\begin{align*}
		%		& (\mu_f -  1) \cdot \mathrm{Var}[\mathrm{Sp^{\tau^{\mathrm{max}} }}(f)] \\
		%		&= (\frac{5}{13}-\bar{\alpha})^2+(\frac{7}{13}-\bar{\alpha})^2+(\frac{9}{13}-\bar{\alpha})^2+(\frac{10}{13}-\bar{\alpha})^2+(\frac{11}{13}-\bar{\alpha})^2+( \frac{12}{13}-\bar{\alpha})^2 +2\cdot (1-\bar{\alpha})^2\\
		%		&+(\frac{14}{13}-\bar{\alpha})^2+(\frac{15}{13}-\bar{\alpha})^2+(\frac{16}{13}-\bar{\alpha})^2+(\frac{17}{13}-\bar{\alpha})^2+(\frac{19}{13}-\bar{\alpha})^2+ \sum_{u = 1}^{r+4} (\frac{2u+r+4}{2(r+5)}-\bar{\alpha})^2 \\
		%		= & \frac{169r^3 + 4380r^2 + 31970r + 70968}{2028(r + 5)^2}.
		%	\end{align*}
	%	Then it suffices to prove
	%	\begin{center}
		%		$\begin{cases}
			%			\frac{169r^3 + 4380r^2 + 31970r + 70968}{2028(r + 5)^2} -\frac{17+r}{12} \cdot \frac{14}{13}, \qquad \forall 0 \leq r \leq 33 \\
			%			\frac{169r^3 + 4380r^2 + 31970r + 70968}{2028(r + 5)^2}-\frac{17+r}{12}\cdot (\frac{3(r+4)}{2(r+5)}-\frac{5}{13}) \leq 0, \qquad \forall r \geq 34 
			%		\end{cases}
		%		$.
		%	\end{center}
	%	Equivalently,
	%	\begin{center}
		%		$\begin{cases}
			%			\frac{-13r^3 - 534r^2 - 3520r - 6382}{2028(r + 5)^2}, \qquad \forall 0 \leq r \leq 33 \\
			%			\frac{-39r^3 - 912r^2 + 1579r + 24806}{4056(r + 5)^2} \leq 0, \qquad \forall r \geq 34
			%		\end{cases}
		%		$,
		%	\end{center}
	%	which hold obviously.
	%	
	
	\noindent\textbf{(5) $NB_{(1)} ^r$.}
	
	Since $ \tau_{f_0} = \mu_f - 2 = 18 + r, $ we need to verify the cases of $\tau^{\mathrm{max}} = 18 +r$ or $ 19+r. $ If  $\tau^{\mathrm{max}} = 18+r, $ 
	for $ r \leq 15,$
	\begin{align*}
		&\bar{\alpha} = \frac{1}{\mu_f -2}\cdot( \frac{8}{21} +\frac{11}{21}+\frac{2}{3} + \frac{16}{21} + \frac{17}{21} + \frac{19}{21} +\frac{20}{21}+1+\frac{22}{21} + \frac{23}{21} +\frac{25}{21} + \frac{26}{21}+ \frac{4}{3}\\
		&+\sum_{ u=1 } ^{ r+5} \frac{2u+r+4}{ 2(r+5)})\\
		&= \frac{355 + 21r}{378 + 21r}.
	\end{align*}
	\begin{align*}
		& (\mu_f -2) \cdot \mathrm{Var}[\mathrm{Sp^{\tau^{\mathrm{max}} }}(f)] \\
		& = (\frac{8}{21}-\bar{\alpha})^2+(\frac{11}{21}-\bar{\alpha})^2+(\frac{2}{3}-\bar{\alpha})^2+(\frac{16}{21}-\bar{\alpha})^2+(\frac{17}{21}-\bar{\alpha})^2+( \frac{19}{21}-\bar{\alpha})^2\\
		&+(\frac{20}{21}-\bar{\alpha})^2+(1-\bar{\alpha})^2+(\frac{22}{21}-\alpha)^2+(\frac{23}{21}-\alpha)^2+(\frac{25}{21}-\alpha)^2+(\frac{26}{21}-\alpha)^2+(\frac{4}{3}-\alpha)^2\\
		&+ \sum_{u = 1}^{r+5} (\frac{2u+r+4}{2(r+5)}-\bar{\alpha})^2 \\
		= & \frac{147r^3 + 5952r^2 + 70100r + 218164}{1764(r + 5)(18 + r)}.
	\end{align*}
	And $ \alpha_{\mu_f -2}-\alpha_1= \frac{3r+14}{2(r+5)}-\frac{8}{21}. $
	Then it suffices to prove
	\begin{equation*}
		\frac{147r^3 + 5952r^2 + 70100r + 218164}{1764(r + 5)(18 + r)}- \frac{18+r}{12} \cdot (\frac{3r+14}{2(r+5)}-\frac{8}{21}) \leq 0.
	\end{equation*}
	Equivalently,
	\begin{equation*}
		\frac{-35r^3 - 1438r^2 - 20324r - 49024}{3528(r + 5)(18 + r)} \leq 0, 
	\end{equation*}
	which holds obviously.

	For $ r \geq 16,$
	\begin{align*}
		&\bar{\alpha} = \frac{1}{\mu_f -2}\cdot( \frac{8}{21} +\frac{11}{21}+\frac{2}{3} + \frac{16}{21} + \frac{17}{21} + \frac{19}{21} +\frac{20}{21}+1+\frac{22}{21} + \frac{23}{21} +\frac{25}{21} + \frac{26}{21}+ \frac{4}{3}+\frac{31}{21}\\
		&+\sum_{ u=1 } ^{ r+4} \frac{2u+r+4}{ 2(r+5)})\\
		&= \frac{26r^2 + 543r + 2078}{26(r + 5)(17 + r)}.
	\end{align*}
	\begin{align*}
		& (\mu_f -2) \cdot \mathrm{Var}[\mathrm{Sp^{\tau^{\mathrm{max}} }}(f)] \\
		& = (\frac{8}{21}-\bar{\alpha})^2+(\frac{11}{21}-\bar{\alpha})^2+(\frac{2}{3}-\bar{\alpha})^2+(\frac{16}{21}-\bar{\alpha})^2+(\frac{17}{21}-\bar{\alpha})^2+( \frac{19}{21}-\bar{\alpha})^2\\
		&+(\frac{20}{21}-\bar{\alpha})^2+(1-\bar{\alpha})^2+(\frac{22}{21}-\alpha)^2+(\frac{23}{21}-\alpha)^2+(\frac{25}{21}-\alpha)^2+(\frac{26}{21}-\alpha)^2+(\frac{4}{3}-\alpha)^2\\
		&+(\frac{31}{21}-\alpha)^2 + \sum_{u = 1}^{r+4} (\frac{2u+r+4}{2(r+5)}-\bar{\alpha})^2 \\
		= & \frac{147r^4 + 6646r^3 + 99501r^2 + 581148r + 1150916}{1764(r + 5)^2(18 + r)}.
	\end{align*}
	And 
	\begin{equation*}
		\alpha_{\mu_f -2}-\alpha_1= \begin{cases}
			\frac{31}{21}-\frac{8}{21}=\frac{23}{21}, \qquad \forall 16 \leq r \leq 57 \\
			\frac{3(r+4)}{2(r+5)}-\frac{8}{21}, \qquad \forall r \geq 58
		\end{cases}.
	\end{equation*}
	Then it suffices to prove
	\begin{equation*}
		\begin{cases}
			\frac{147r^4 + 6646r^3 + 99501r^2 + 581148r + 1150916}{1764(r + 5)^2(18 + r)}- \frac{18+r}{12} \cdot \frac{23}{21} \leq 0, \qquad \forall 16 \leq r \leq 57 \\
			\frac{147r^4 + 6646r^3 + 99501r^2 + 581148r + 1150916}{1764(r + 5)^2(18 + r)} -\frac{18+r}{12}\cdot (\frac{3(r+4)}{2(r+5)}-\frac{8}{21}) \leq 0, \qquad \forall r \geq 58 
		\end{cases}.
	\end{equation*}
	Equivalently,
	\begin{equation*}
		\begin{cases}
			\frac{-7r^4 - 380r^3 - 7324r^2 - 42696r - 76592}{882(r + 5)^2(18 + r)} \leq 0, \qquad \forall 16 \leq r \leq 57 \\
			\frac{-35r^4 - 1401r^3 - 16178r^2 + 22500r + 351352}{3528(r + 5)^2(18 + r)} \leq 0, \qquad \forall r \geq 58 
		\end{cases},
	\end{equation*}
	which hold obviously.

	If  $\tau^{\mathrm{max}} = 19+r, $ 
	\begin{align*}
		&\bar{\alpha} = \frac{1}{\mu_f -1}\cdot( \frac{8}{21} +\frac{11}{21}+\frac{2}{3} + \frac{16}{21} + \frac{17}{21} + \frac{19}{21} +\frac{20}{21}+1+\frac{22}{21} + \frac{23}{21} +\frac{25}{21} + \frac{26}{21}+ \frac{4}{3}+\frac{31}{21}\\
		&+\sum_{ u=1 } ^{ r+5} \frac{2u+r+4}{ 2(r+5)})\\
		&= \frac{386 + 21r}{399 + 21r}.
	\end{align*}
	\begin{align*}
		& (\mu_f -1) \cdot \mathrm{Var}[\mathrm{Sp^{\tau^{\mathrm{max}} }}(f)] \\
		& = (\frac{8}{21}-\bar{\alpha})^2+(\frac{11}{21}-\bar{\alpha})^2+(\frac{2}{3}-\bar{\alpha})^2+(\frac{16}{21}-\bar{\alpha})^2+(\frac{17}{21}-\bar{\alpha})^2+( \frac{19}{21}-\bar{\alpha})^2\\
		&+(\frac{20}{21}-\bar{\alpha})^2+(1-\bar{\alpha})^2+(\frac{22}{21}-\alpha)^2+(\frac{23}{21}-\alpha)^2+(\frac{25}{21}-\alpha)^2+(\frac{26}{21}-\alpha)^2+(\frac{4}{3}-\alpha)^2\\
		&+(\frac{31}{21}-\alpha)^2 + \sum_{u = 1}^{r+5} (\frac{2u+r+4}{2(r+5)}-\bar{\alpha})^2 \\
		= & \frac{147r^3 + 6499r^2 + 84446r + 276072}{1764(r + 5)(19 + r)}.
	\end{align*}
	And 
	\begin{equation*}
		\alpha_{\mu_f -1}-\alpha_1= \begin{cases}
			\frac{31}{21}-\frac{8}{21}=\frac{23}{21}, \qquad \forall 0 \leq r \leq 15 \\
			\frac{3r+14}{2(r+5)}-\frac{8}{21}, \qquad \forall r \geq 16
		\end{cases},
	\end{equation*}
	Then it suffices to prove
	\begin{equation*}
		\begin{cases}
			\frac{147r^3 + 6499r^2 + 84446r + 276072}{1764(r + 5)(19 + r)} - \frac{19+r}{12} \cdot \frac{23}{21} \leq 0, \qquad \forall 0 \leq r \leq 15 \\
			\frac{147r^3 + 6499r^2 + 84446r + 276072}{1764(r + 5)(19 + r)} -\frac{19+r}{12}\cdot (\frac{3r+14}{2(r+5)}-\frac{8}{21}) \leq 0, \qquad \forall r \geq 16 
		\end{cases}.
	\end{equation*}
	Equivalently,
	\begin{equation*}
		\begin{cases}
			\frac{-14r^3 - 424r^2 - 4265r - 14533}{1764(r + 5)(19 + r)} \leq 0, \qquad \forall 0 \leq r \leq 15 \\
			\frac{-35r^3 - 1002r^2 - 6801r + 11366}{3528(r + 5)(19 + r)} \leq 0, \qquad \forall r \geq 16 
		\end{cases},
	\end{equation*}
	which hold obviously.

	%	
	%	Considering the average value $ \bar{\alpha}=1, $
	%	\begin{align*}
		%		& \tau \cdot \mathrm{Var}[\mathrm{Sp^\tau}(f_0)]\\
		%		& = (\frac{8}{21}-\bar{\alpha})^2+(\frac{11}{21}-\bar{\alpha})^2+(\frac{2}{3}-\bar{\alpha})^2+(\frac{16}{21}-\bar{\alpha})^2+(\frac{17}{21}-\bar{\alpha})^2+( \frac{19}{21}-\bar{\alpha})^2\\
		%		&+(\frac{20}{21}-\bar{\alpha})^2+(1-\bar{\alpha})^2+(\frac{22}{21}-\alpha)^2+(\frac{23}{21}-\alpha)^2+(\frac{25}{21}-\alpha)^2+(\frac{26}{21}-\alpha)^2+(\frac{4}{3}-\alpha)^2\\
		%		&+(\frac{31}{21}-\alpha)^2 + \sum_{u = 1}^{r+4} (\frac{2u+r+4}{2(r+5)}-\bar{\alpha})^2 \\
		%		&= \frac{147r^3 + 4000r^2 + 29710r + 66484}{1764(r + 5)^2}.
		%	\end{align*}
	%	Then it suffices to prove
	%	\begin{center}
		%		$\begin{cases}
			%			\frac{147r^3 + 4000r^2 + 29710r + 66484}{1764(r + 5)^2}- \frac{18+r}{12} \cdot \frac{23}{21} \leq 0, \qquad \forall 0 \leq r \leq 57 \\
			%			\frac{147r^3 + 4000r^2 + 29710r + 66484}{1764(r + 5)^2} -\frac{18+r}{12}\cdot (\frac{3(r+4)}{2(r+5)}-\frac{8}{21}) \leq 0, \qquad \forall r \geq 58 
			%		\end{cases}
		%		$
		%	\end{center}
	%	Equivalently,
	%	\begin{center}
		%		$\begin{cases}
			%			\frac{-14r^3 - 508r^2 - 3295r - 5966}{1764(r + 5)^2} \leq 0, \qquad \forall 0 \leq r \leq 33 \\
			%			\frac{-35r^3 - 771r^2 + 2118r + 24608}{3528(r + 5)^2} \leq 0, \qquad \forall r \geq 34
			%		\end{cases}
		%		$,
		%	\end{center}
	%	which hold obviously.
	%	
	%	
	
	\noindent\textbf{(6) $VA_{r,s} $.}
	
	Since $  \tau_{f_0}  = \mu_f - 2 = 13+ r+s, $ we need to verify the cases of $\tau^{\mathrm{max}} = 13 +r+s$ or $ 14+r+s. $ If  $\tau^{\mathrm{max}} = 18+r, $ 
	For $ r \leq  s,$  if  $\tau^{\mathrm{max}} = 13+r +s, $ 
	\begin{align*}
		&\bar{\alpha} =\frac{1}{\mu_f -2}\cdot( \frac{7}{8}+\frac{5}{4}+\frac{11}{8} + \frac{3}{2} + \frac{13}{8} + \frac{7}{4}+\sum_{ u=1 } ^{ r+4} \frac{2u+2r+7}{ 2(r+4)}+\sum_{ v=1 } ^{ s+3} \frac{2v+2s+7}{ 2(s+4)})\\
		&=\frac{12r^2 + (12s + 195)r + 48s + 592}{8(r + 4)(13 + r + s)}.
	\end{align*}
	\begin{align*}
		& (\mu_f -  2) \cdot \mathrm{Var}[\mathrm{Sp^{\tau^{\mathrm{max}} }}(f)] \\
		&= (\frac{7}{8}-\bar{\alpha})^2+(\frac{5}{4}-\bar{\alpha})^2+(\frac{11}{8}-\bar{\alpha})^2+(\frac{3}{2}-\bar{\alpha})^2+(\frac{13}{8}-\bar{\alpha})^2+( \frac{7}{4}-\bar{\alpha})^2\\
		& + \sum_{u = 1}^{r+4} (\frac{2u+2r+7}{2(r+4)}-\bar{\alpha})^2 +\sum_{v =1 } ^{s+3} (\frac{2v+2s+7}{ 2(s+4)}-\bar{\alpha})^2\\
		&=\frac{1}{192(13 + r + s)(s + 4)^2(r + 4)} \cdot ((16r + 64)s^4 + (32r^2 + 649r + 2068)s^3\\
		& + (16r^3 + 713r^2 + 8222r + 22232)s^2 + (128r^3 + 4248r^2 + 39928r + 98528)s\\
		& + 256r^3 + 7584r^2 + 64928r + 152448).
	\end{align*}
	And $ \alpha_{\mu_f -2}-\alpha_1= \frac{4r+15}{2(r+4)}-\frac{7}{8}. $
	Then it suffices to prove
	\begin{align*}
		&\frac{1}{192(13 + r + s)(s + 4)^2(r + 4)} \cdot ((16r + 64)s^4 + (32r^2 + 649r + 2068)s^3 \\
		&+ (16r^3 + 713r^2 + 8222r + 22232)s^2 + (128r^3 + 4248r^2 + 39928r + 98528)s \\
		&+ 256r^3 + 7584r^2 + 64928r + 152448) - \frac{13+r+s}{12} \cdot (\frac{4r+15}{2(r+4)}-\frac{7}{8}) \leq 0.
	\end{align*}
	Equivalently,
	\begin{align*}
		&\frac{1}{192(13 + r + s)(s + 4)^2(r + 4)} \cdot ( -2(s + 4)^2r^3 + (-4s^3 - 107s^2 - 584s - 928)r^2\\
		& + (-2s^4 - 91s^3 - 1540s^2 - 7256s - 10368)r - 108s^3 - 2920s^2 - 14624s - 20608) \leq 0,
	\end{align*}
	which holds obviously.
	
	If  $\tau^{\mathrm{max}} = 14+r +s, $ 
	\begin{align*}
		&\bar{\alpha} =\frac{1}{\mu_f -1}\cdot( \frac{7}{8}+\frac{5}{4}+\frac{11}{8} + \frac{3}{2} + \frac{13}{8} + \frac{7}{4}+\sum_{ u=1 } ^{ r+4} \frac{2u+2r+7}{ 2(r+4)}+\sum_{ v=1 } ^{ s+4} \frac{2v+2s+7}{ 2(s+4)})\\
		&=\frac{12r^2 + (12s + 195)r + 48s + 592}{8(r + 4)(13 + r + s)}.
	\end{align*}
	\begin{align*}
		& (\mu_f -  1) \cdot \mathrm{Var}[\mathrm{Sp^{\tau^{\mathrm{max}} }}(f)] \\
		&= (\frac{7}{8}-\bar{\alpha})^2+(\frac{5}{4}-\bar{\alpha})^2+(\frac{11}{8}-\bar{\alpha})^2+(\frac{3}{2}-\bar{\alpha})^2+(\frac{13}{8}-\bar{\alpha})^2+( \frac{7}{4}-\bar{\alpha})^2\\
		& + \sum_{u = 1}^{r+3} (\frac{2u+2r+7}{2(r+4)}-\bar{\alpha})^2 +\sum_{v =1 } ^{s+4} (\frac{2v+2s+7}{ 2(s+4)}-\bar{\alpha})^2\\
		&=\frac{1}{192(r + 4)(s + 4)(14 + r + s)} \cdot ((16s + 64)r^3 + (32s^2 + 649s + 2068)r^2 \\
		&+ (16s^3 + 649s^2 + 7323s + 19708)r + 64s^3 + 2068s^2 + 19708s + 49200).
	\end{align*}
	And $ \alpha_{\mu_f -1}-\alpha_1= \frac{4s+15}{2(s+4)}-\frac{7}{8}. $
	Then it suffices to prove
	\begin{align*}
		&\frac{1}{192(r + 4)(s + 4)(14 + r + s)} \cdot ((16s + 64)r^3 + (32s^2 + 649s + 2068)r^2\\
		& + (16s^3 + 649s^2 + 7323s + 19708)r + 64s^3 + 2068s^2 + 19708s + 49200)\\
		& - \frac{14+r+s}{12} \cdot (\frac{4s+15}{2(s+4)}-\frac{7}{8}).
	\end{align*}
	Equivalently,
	\begin{align*}
		&\frac{1}{192(r + 4)(s + 4)(14 + r + s)} \cdot ( (-2r - 8)s^3 + (-4r^2 - 63r - 204)s^2 \\
		&+ (-2r^3 - 55r^2 - 525r - 1572)s + 20r^2 - 4r - 976) \leq 0,
	\end{align*}
	which holds obviously.
	\noindent\textbf{(7) $VA_{2k,0}^\#$.}
	
	Since $  \tau_{f_0}  = \mu_f - 2 = 13+ 2k, $ we need to verify the cases of $\tau^{\mathrm{max}} = 13 +2k$ or $ 14+2k. $ If  $\tau^{\mathrm{max}} = 13+2k, $ 
	If  $\tau^{\mathrm{max}} = 13+2k, $ 
	\begin{align*}
		&\bar{\alpha} =\frac{1}{\mu_f -2}\cdot( \frac{7}{8}+\frac{9}{8}+\frac{11}{8} + \frac{11}{8} + \frac{13}{8} + \frac{13}{8}+\frac{15}{8}+\sum_{ u=1 } ^{ 2k+6} \frac{u+2k+8}{ 2(k+4)})\\
		&=\frac{24k^2 + 243k + 592}{16k^2 + 168k + 416}.
	\end{align*}
	\begin{align*}
		& (\mu_f -  2) \cdot \mathrm{Var}[\mathrm{Sp^{\tau^{\mathrm{max}} }}(f)] \\
		&= (\frac{7}{8}-\bar{\alpha})^2+(\frac{9}{8}-\bar{\alpha})^2+(\frac{11}{8}-\bar{\alpha})^2+(\frac{11}{8}-\bar{\alpha})^2+(\frac{13}{8}-\bar{\alpha})^2+( \frac{13}{8}-\bar{\alpha})^2 +(\frac{15}{8} - \bar{\alpha})^2\\
		& + \sum_{u = 1}^{2k+6} (\frac{2u+2k+8}{2(k+4)}-\bar{\alpha})^2 \\
		&= \frac{32k^4 + 637k^3 + 4675k^2 + 15356k + 19056}{96(13 + 2k)(k + 4)^2}.
	\end{align*}
	And \begin{equation*}
		\alpha_{\mu_f -2 } - \alpha_1 =
		\begin{cases}
			\frac{15}{8}-\frac{7}{8}=1, \qquad \forall 0 \leq k \leq 3 \\
			\frac{2k+7}{k+4}-\frac{7}{8}, \qquad \forall k \geq 4
		\end{cases}.
	\end{equation*}
	Then it suffices to prove
	\begin{equation*}
		\begin{cases}
			\frac{32k^4 + 637k^3 + 4675k^2 + 15356k + 19056}{96(13 + 2k)(k + 4)^2}- \frac{13+2k}{12} \cdot 1 \leq 0, \qquad \forall 0 \leq k \leq 3 \\
			\frac{32k^4 + 637k^3 + 4675k^2 + 15356k + 19056}{96(13 + 2k)(k + 4)^2} -\frac{13+2k}{12}\cdot (\frac{2k+7}{k+4}-\frac{7}{8}) \leq 0, \qquad \forall k \geq 4
		\end{cases}
	\end{equation*}
	Equivalently,
	\begin{equation*}
		\begin{cases}
			\frac{-35k^3 - 517k^2 - 2116k - 2576}{96(13 + 2k)(k + 4)^2} \leq 0, \qquad \forall 0 \leq k \leq 3 \\
			\frac{-4k^4 - 87k^3 - 622k^2 - 1284k + 128}{96(13 + 2k)(k + 4)^2} \leq 0, \qquad \forall k \geq 4
		\end{cases},
	\end{equation*}
	which hold obviously.

	If  $\tau^{\mathrm{max}} = 14+2k , $ 
	\begin{align*}
		&\bar{\alpha} =\frac{1}{\mu_f -1}\cdot( \frac{7}{8}+\frac{9}{8}+\frac{11}{8} + \frac{11}{8} + \frac{13}{8} + \frac{13}{8}+\frac{15}{8}+\sum_{ u=1 } ^{ 2k+7} \frac{u+2k+8}{ 2(k+4)})\\
		&=\frac{163 + 24k}{112 + 16k}.
	\end{align*}
	\begin{align*}
		& (\mu_f -  1) \cdot \mathrm{Var}[\mathrm{Sp^{\tau^{\mathrm{max}} }}(f)] \\
		&= (\frac{7}{8}-\bar{\alpha})^2+(\frac{9}{8}-\bar{\alpha})^2+(\frac{11}{8}-\bar{\alpha})^2+(\frac{11}{8}-\bar{\alpha})^2+(\frac{13}{8}-\bar{\alpha})^2+( \frac{13}{8}-\bar{\alpha})^2 +(\frac{15}{8} - \bar{\alpha})^2\\
		& + \sum_{u = 1}^{2k+7} (\frac{2u+2k+8}{2(k+4)}-\bar{\alpha})^2 \\
		&= \frac{64k^3 + 1146k^2 + 6611k + 12300}{384(k + 4)(7 + k)}.
	\end{align*}
	And $ \alpha_{\mu_f -1 } - \alpha_1 = \frac{4k+15}{2(k+4)}-\frac{7}{8}. $
	Then it suffices to prove
	\begin{equation*}
		\frac{64k^3 + 1146k^2 + 6611k + 12300}{384(k + 4)(7 + k)}- \frac{14+2k}{12} \cdot (\frac{4k+15}{2(k+4)}-\frac{7}{8}) \leq 0.
	\end{equation*}
	Equivalently,
	\begin{equation*}
		\frac{-8k^3 - 118k^2 - 501k - 244}{384(k + 4)(7 + k)} \leq 0,
	\end{equation*}
	which holds obviously.

	\noindent\textbf{(8) $VA_{2k,s}^\#$.}
	
	Since $ \tau_{f_0}  = \mu_f - 2 = 13+ 2k + s, $ we need to verify the cases of $\tau^{\mathrm{max}} = 13 +2k +s$ or $ 14+2k  +s. $ If  $\tau^{\mathrm{max}} = 13+2k, $ 
	If  $\tau^{\mathrm{max}} = 13+2k + s, $ 
	for $ k \leq  s,$
	\begin{align*}
		& \bar{\alpha} =\frac{1}{\mu_f -2}\cdot( \frac{7}{8}+\frac{11}{8}+\frac{13}{8} +\sum_{ u=1 } ^{ 2k+6} \frac{u+2k+8}{ 2(k+4)}+\sum_{ v=1 } ^{ s+4} \frac{2v+2s+7}{ 2(s+4)})\\
		&=\frac{24k^2 + (12s + 243)k + 48s + 592}{8(k + 4)(13 + 2k + s)}.
	\end{align*}
	\begin{align*}
		& (\mu_f -  2) \cdot \mathrm{Var}[\mathrm{Sp^{\tau^{\mathrm{max}} }}(f)] \\
		& =(\frac{7}{8}-\bar{\alpha})^2+(\frac{11}{8}-\bar{\alpha})^2+(\frac{13}{8}-\bar{\alpha})^2 + \sum_{u = 1}^{2k+6} (\frac{u+2k+8}{2(k+4)}-\bar{\alpha})^2 +\sum_{v =1 } ^{s+4} (\frac{2v+2s+7}{ 2(s+4)}-\bar{\alpha})^2\\
		&= \frac{1}{192(k + 4)^2(s + 4)^2(13 + 2k + s)} \cdot (64k^4(s + 4)^2 + (64s^3 + 1794s^2 + 11440s + 21056)k^3 \\
		&+ (16s^4 + 1025s^3 + 16786s^2 + 91352s + 157136)k^2+ (128s^4 + 5144s^3 + 64768s^2 \\
		&+ 313568s + 509440)k + 256s^4 + 8272s^3 + 88976s^2 + 394112s + 609792).
	\end{align*}
	And $ \alpha_{\mu_f -2}-\alpha_1= \frac{4k+15}{2(k+4)}-\frac{7}{8}. $
	Then it suffices to prove
	\begin{align*}
		&\frac{1}{192(k + 4)^2(s + 4)^2(13 + 2k + s)} \cdot (64k^4(s + 4)^2 + (64s^3 + 1794s^2 + 11440s + 21056)k^3\\
		& + (16s^4 + 1025s^3 + 16786s^2 + 91352s + 157136)k^2 + (128s^4 + 5144s^3 + 64768s^2 \\
		&+ 313568s + 509440)k + 256s^4 + 8272s^3 + 88976s^2 + 394112s + 609792)\\
		&- \frac{13+2k+s}{12}\cdot( \frac{4k+15}{2(k+4)}-\frac{7}{8}) \leq 0.
	\end{align*}
	Equivalently,
	\begin{align*}
		&\frac{1}{192(k + 4)^2(s + 4)^2(13 + 2k + s)} \cdot ( -8k^4(s + 4)^2 + (-8s^3 - 262s^2 - 1552s - 2624)k^3\\
		& + (-2s^4 - 131s^3 - 2736s^2 - 13944s - 21072)k^2 + (-8s^4 - 504s^3 - 10184s^2 \\
		&- 49760s - 71296)k - 432s^3 - 11632s^2 - 58496s - 82432),
	\end{align*}
	which holds obviously.

	If  $\tau^{\mathrm{max}} = 14+2k + s, $ 
	\begin{align*}
		& \bar{\alpha} =\frac{1}{\mu_f -1}\cdot( \frac{7}{8}+\frac{11}{8}+\frac{13}{8} +\sum_{ u=1 } ^{ 2k+7} \frac{u+2k+8}{ 2(k+4)}+\sum_{ v=1 } ^{ s+4} \frac{2v+2s+7}{ 2(s+4)})\\
		&=\frac{163 + 24k + 12s}{112 + 16k + 8s}.
	\end{align*}
	\begin{align*}
		& (\mu_f -  1) \cdot \mathrm{Var}[\mathrm{Sp^{\tau^{\mathrm{max}} }}(f)] \\
		& =(\frac{7}{8}-\bar{\alpha})^2+(\frac{11}{8}-\bar{\alpha})^2+(\frac{13}{8}-\bar{\alpha})^2 + \sum_{u = 1}^{2k+7} (\frac{u+2k+8}{2(k+4)}-\bar{\alpha})^2 +\sum_{v =1 } ^{s+4} (\frac{2v+2s+7}{ 2(s+4)}-\bar{\alpha})^2\\
		&= \frac{1}{384(s + 4)(k + 4)(7 + k + s/2)} \cdot ((64s + 256)k^3 + (64s^2 + 1410s + 4584)k^2\\
		& + (16s^3 + 769s^2 + 9503s + 26444)k + 64s^3 + 2068s^2 + 19708s + 49200).
	\end{align*}
	And $ \alpha_{\mu_f -1}-\alpha_1= \frac{4s+15}{2(s+4)}-\frac{7}{8}. $
	Then it suffices to prove
	\begin{align*}
		&\frac{1}{384(s + 4)(k + 4)(7 + k + s/2)} \cdot ((64s + 256)k^3 + (64s^2 + 1410s + 4584)k^2\\
		& + (16s^3 + 769s^2 + 9503s + 26444)k + 64s^3 + 2068s^2 + 19708s + 49200)\\
		&- \frac{14+2k+s}{12}\cdot(\frac{4s+15}{2(s+4)}-\frac{7}{8}) \leq 0.
	\end{align*}
	Equivalently,
	\begin{align*}
		&\frac{1}{384(s + 4)(k + 4)(7 + k + s/2)} \cdot ((-8 - 2k)s^3 + (-8k^2 - 87k - 204)s^2 \\
		&+ (-8k^3 - 142k^2 - 873k - 1572)s - 24k^2 - 436k - 976),
	\end{align*}
	which holds obviously.

	\noindent\textbf{(9) $VB_{(-1)}^s$.}

	Since $  \tau_{f_0}  = \mu_f - 2 = 15+ s, $ we need to verify the cases of $\tau^{\mathrm{max}} = 15 +s$ or $ 16+s. $ 
	If  $\tau^{\mathrm{max}} = 15 + s, $ for $ s \leq 2, $
	\begin{align*}
		&\bar{\alpha} =\frac{1}{\mu_f -2}\cdot( \frac{6}{7}+\frac{15}{14}+\frac{17}{14} + \frac{9}{7} + \frac{19}{14} + \frac{10}{7}+\frac{3}{2}+\frac{11}{7}+\frac{23}{14}+\frac{12}{7}+\frac{25}{14}+\sum_{ u=1 } ^{ s+4} \frac{2u+2s+7}{ 2(s+4)})\\
		&=\frac{300 + 21s}{210 + 14s}.
	\end{align*}
	\begin{align*}
		& (\mu_f -  2) \cdot \mathrm{Var}[\mathrm{Sp^{\tau^{\mathrm{max}} }}(f)] \\
		= & (\frac{6}{7}-\bar{\alpha})^2+(\frac{15}{14}-\bar{\alpha})^2+(\frac{17}{14}-\bar{\alpha})^2+( \frac{9}{7}-\bar{\alpha})^2+(\frac{10}{7}-\bar{\alpha})^2 +(\frac{3}{2}-\bar{\alpha})^2 +(\frac{11}{7}-\bar{\alpha})^2\\
		&+(\frac{23}{14}-\bar{\alpha})^2+(\frac{12}{7}-\bar{\alpha})^2+(\frac{25}{14}-\bar{\alpha})^2+ \sum_{ u=1 } ^{ s+4} (\frac{2u+2s+7}{ 2(s+4)}-\bar{\alpha})^2\\
		= & \frac{49s^3 + 1658s^2 + 16029s + 40185}{588(s + 4)(15 + s)}.
	\end{align*}
	And $\alpha_{ \mu_f -2} -\alpha_1 = \frac{4s+15}{2(s+4)}-\frac{6}{7} .$
	Then it suffices to prove
	\begin{equation*}
		\frac{49s^3 + 1658s^2 + 16029s + 40185}{588(s + 4)(15 + s)}-\frac{15+s}{12}\cdot (\frac{4s+15}{2(s+4)}-\frac{6}{7}) \leq 0.
	\end{equation*}
	Equivalently,
	\begin{equation*}
		\frac{-14s^3 - 443s^2 - 5112s - 9405}{1176(s + 4)(15 + s)} \leq 0,
	\end{equation*}
	which holds obviously.

	for $ s \geq 3, $
	\begin{align*}
		&\bar{\alpha} =\frac{1}{\mu_f -2}\cdot( \frac{6}{7}+\frac{15}{14}+\frac{17}{14} + \frac{9}{7} + \frac{19}{14} + \frac{10}{7}+\frac{3}{2}+\frac{11}{7}+\frac{23}{14}+\frac{12}{7}+\frac{25}{14}+\frac{27}{14}+\sum_{ u=1 } ^{ s+3} \frac{2u+2s+7}{ 2(s+4)})\\
		&=\frac{21s^2 + 383s + 1203}{14(s + 4)(15 + s)}.
	\end{align*}
	\begin{align*}
		& (\mu_f -  2) \cdot \mathrm{Var}[\mathrm{Sp^{\tau^{\mathrm{max}} }}(f)] \\
		= & (\frac{6}{7}-\bar{\alpha})^2+(\frac{15}{14}-\bar{\alpha})^2+(\frac{17}{14}-\bar{\alpha})^2+( \frac{9}{7}-\bar{\alpha})^2+(\frac{10}{7}-\bar{\alpha})^2 +(\frac{3}{2}-\bar{\alpha})^2 +(\frac{11}{7}-\bar{\alpha})^2\\
		&+(\frac{23}{14}-\bar{\alpha})^2+(\frac{12}{7}-\bar{\alpha})^2+(\frac{25}{14}-\bar{\alpha})^2+(\frac{27}{14}-\bar{\alpha})^2 + \sum_{ u=1 } ^{ s+3} (\frac{2u+2s+7}{ 2(s+4)}-\bar{\alpha})^2\\
		= & \frac{49s^4 + 1815s^3 + 21965s^2 + 104364s + 167868}{588(s + 4)^2(15 + s)}.
	\end{align*}
	And 
	\begin{equation*}
		\alpha_{ \mu_f -2} -\alpha_1 = 
		\begin{cases}
			\frac{27}{14}-\frac{6}{7}=\frac{15}{14}, \qquad \forall 3 \leq s \leq 9 \\
			\frac{4s+13}{2(s+4)}-\frac{6}{7}, \qquad \forall s \geq 10
		\end{cases}.
	\end{equation*}
	Then it suffices to prove
	\begin{equation*}
		\begin{cases}
			\frac{49s^4 + 1815s^3 + 21965s^2 + 104364s + 167868}{588(s + 4)^2(15 + s)}- \frac{15+s}{12} \cdot \frac{15}{14}\leq 0, \qquad \forall 3 \leq s \leq 9 \\
			\frac{49s^4 + 1815s^3 + 21965s^2 + 104364s + 167868}{588(s + 4)^2(15 + s)}-\frac{15+s}{12}\cdot (\frac{4s+13}{2(s+4)}-\frac{6}{7}) \leq 0, \qquad \forall s \geq 10
		\end{cases}.
	\end{equation*}
	Equivalently,
	\begin{equation*}
		\begin{cases}
			\frac{-7s^4 - 360s^3 - 6575s^2 - 30672s - 42264}{1176(s + 4)^2(15 + s)} \leq 0, \qquad \forall 3 \leq s \leq 9 \\
			\frac{-14s^4 - 479s^3 - 4944s^2 + 4083s + 64836}{1176(s + 4)^2(15 + s)} \leq 0, \qquad \forall s \geq 10
		\end{cases},
	\end{equation*}
	which hold obviously.

	If  $\tau^{\mathrm{max}} = 16 + s, $ 
	\begin{align*}
		&\bar{\alpha} =\frac{1}{\mu_f -1}\cdot( \frac{6}{7}+\frac{15}{14}+\frac{17}{14} + \frac{9}{7} + \frac{19}{14} + \frac{10}{7}+\frac{3}{2}+\frac{11}{7}+\frac{23}{14}+\frac{12}{7}+\frac{25}{14}+\frac{27}{14}+\sum_{ u=1 } ^{ s+4} \frac{2u+2s+7}{ 2(s+4)})\\
		&=\frac{327 + 21s}{224 + 14s}.
	\end{align*}
	\begin{align*}
		& (\mu_f -  1) \cdot \mathrm{Var}[\mathrm{Sp^{\tau^{\mathrm{max}} }}(f)] \\
		= & (\frac{6}{7}-\bar{\alpha})^2+(\frac{15}{14}-\bar{\alpha})^2+(\frac{17}{14}-\bar{\alpha})^2+( \frac{9}{7}-\bar{\alpha})^2+(\frac{10}{7}-\bar{\alpha})^2 +(\frac{3}{2}-\bar{\alpha})^2 +(\frac{11}{7}-\bar{\alpha})^2\\
		&+(\frac{23}{14}-\bar{\alpha})^2+(\frac{12}{7}-\bar{\alpha})^2+(\frac{25}{14}-\bar{\alpha})^2+(\frac{27}{14}-\bar{\alpha})^2 + \sum_{ u=1 } ^{ s+4} (\frac{2u+2s+7}{ 2(s+4)}-\bar{\alpha})^2\\
		= & \frac{49s^3 + 1815s^2 + 19544s + 51684}{588(s + 4)(16 + s)}.
	\end{align*}
	And 
	\begin{equation*}
		\alpha_{ \mu_f -1} -\alpha_1 = 
		\begin{cases}
			\frac{27}{14}-\frac{6}{7}=\frac{15}{14}, \qquad \forall 0 \leq s \leq 2 \\
			\frac{4s+15}{2(s+4)}-\frac{6}{7}, \qquad \forall s \geq 3
		\end{cases}.
	\end{equation*}
	Then it suffices to prove
	\begin{equation*}
		\begin{cases}
			\frac{49s^3 + 1815s^2 + 19544s + 51684}{588(s + 4)(16 + s)}- \frac{16+s}{12} \cdot \frac{15}{14}\leq 0, \qquad \forall 0 \leq s \leq 2 \\
			\frac{49s^3 + 1815s^2 + 19544s + 51684}{588(s + 4)(16 + s)}-\frac{16+s}{12}\cdot (\frac{4s+15}{2(s+4)}-\frac{6}{7}) \leq 0, \qquad \forall s \geq 3
		\end{cases}.
	\end{equation*}
	Equivalently,
	\begin{equation*}
		\begin{cases}
			\frac{-7s^3 - 150s^2 - 1232s - 4152}{1176(s + 4)(16 + s)} \leq 0, \qquad \forall 0 \leq s \leq 16 \\
			\frac{-14s^3 - 353s^2 - 2352s + 1224}{1176(s + 4)(16 + s)} \leq 0, \qquad \forall s \geq 17
		\end{cases},
	\end{equation*}
	which hold obviously.
	
	%	Considering the average value $ \bar{\alpha}=\frac{3}{2}, $
	%	\begin{align*}
		%		& (\mu_f -  1) \cdot \mathrm{Var}[\mathrm{Sp^{\tau^{\mathrm{max}} }}(f)] \\
		%		= & (\frac{6}{7}-\bar{\alpha})^2+(\frac{15}{14}-\bar{\alpha})^2+(\frac{17}{14}-\bar{\alpha})^2+( \frac{9}{7}-\bar{\alpha})^2+(\frac{10}{7}-\bar{\alpha})^2 +(\frac{3}{2}-\bar{\alpha})^2+(\frac{11}{7}-\bar{\alpha})^2\\
		%		&+(\frac{23}{14}-\bar{\alpha})^2+(\frac{12}{7}-\bar{\alpha})^2+(\frac{25}{14}-\bar{\alpha})^2+(\frac{27}{14}-\bar{\alpha})^2 + \sum_{ u=1 } ^{ s+3} (\frac{2u+2s+7}{ 2(s+4)}-\bar{\alpha})^2\\
		%		= & \frac{49s^3 + 1080s^2 + 6533s + 11841}{588(s + 4)^2}.
		%	\end{align*}
	%	Then it suffices to prove
	%	\begin{center}
		%		$\begin{cases}
			%			\frac{49s^3 + 1080s^2 + 6533s + 11841}{588(s + 4)^2}- \frac{15+s}{12} \cdot \frac{15}{14}\leq 0, \qquad \forall 0 \leq s \leq 16 \\
			%			\frac{49s^3 + 1080s^2 + 6533s + 11841}{588(s + 4)^2}-\frac{15+s}{12}\cdot (\frac{4s+13}{2(s+4)}-\frac{6}{7}) \leq 0, \qquad \forall s \geq 17
			%		\end{cases}
		%		$.
		%	\end{center}
	%	Equivalently,
	%	\begin{center}
		%		$\begin{cases}
			%			\frac{-7s^3 - 255s^2 - 1214s - 1518}{1176(s + 4)^2} \leq 0, \qquad \forall 0 \leq s \leq 16 \\
			%			\frac{-49s^4 - 1650s^3 - 16409s^2 + 32292s + 284544}{4704(s + 4)^2(15 + s)} \leq 0, \qquad \forall s \geq 17
			%		\end{cases}
		%		$,
		%	\end{center}
	%	which hold obviously.
	%	

	\noindent\textbf{(10) $VB_{(0)}^s$.}
	
	Since $  \tau_{f_0}  = \mu_f - 2 = 16+ s, $ we need to verify the cases of $\tau^{\mathrm{max}} = 16 +s$ or $ 17+s. $ If  $\tau^{\mathrm{max}} = 16+s, $ 
	for $ s \leq 5, $
	\begin{align*}
		&\bar{\alpha} =\frac{1}{\mu_f -2}\cdot( \frac{17}{20}+\frac{21}{20}+\frac{6}{5} + \frac{5}{4} + \frac{27}{20} + \frac{7}{5}+\frac{29}{20}+\frac{31}{20}+\frac{8}{5}+\frac{33}{20}+\frac{7}{4}+\frac{9}{5}+\sum_{ u=1 } ^{ s+4} \frac{2u+2s+7}{ 2(s+4)})\\
		&=\frac{229 + 15s}{160 + 10s}. 
	\end{align*}
	\begin{align*}
		& (\mu_f -  2) \cdot \mathrm{Var}[\mathrm{Sp^{\tau^{\mathrm{max}} }}(f)] \\
		= & (\frac{17}{20}-\bar{\alpha})^2+(\frac{21}{20}-\bar{\alpha})^2+2 (\frac{6}{5}-\bar{\alpha})^2+( \frac{5}{4}-\bar{\alpha})^2+(\frac{27}{20}-\bar{\alpha})^2+ (\frac{7}{5}-\bar{\alpha})^2+(\frac{29}{20}-\bar{\alpha})^2\\
		&+(\frac{31}{20}-\bar{\alpha})^2+(\frac{8}{5}-\bar{\alpha})^2+(\frac{33}{20}-\bar{\alpha})^2+(\frac{7}{4}-\bar{\alpha})^2+(\frac{9}{5}-\bar{\alpha})^2\\
		& + \sum_{ u=1 } ^{ s+4} (\frac{2u+2s+7}{ 2(s+4)}-\bar{\alpha})^2\\
		= & \frac{25s^3 + 900s^2 + 9212s + 23748}{300(s + 4)(16 + s)}.
	\end{align*}
	And $ \alpha_{\mu_f -2 } - \alpha_1 =\frac{4s+15}{2(s+4)}-\frac{17}{20}. $
	Then it suffices to prove
	\begin{equation*}
		\frac{25s^3 + 900s^2 + 9212s + 23748}{300(s + 4)(16 + s)}- \frac{16+s}{12} \cdot (\frac{4s+15}{2(s+4)}-\frac{17}{20})\leq 0.
	\end{equation*}
	Equivalently,
	\begin{equation*}
		\frac{-15s^3 - 490s^2 - 5712s - 9968}{1200(s + 4)(16 + s)} \leq 0, 
	\end{equation*}
	which holds obviously.

	For $ s \geq 6, $
	\begin{align*}
		&\bar{\alpha} =\frac{1}{\mu_f -2}\cdot( \frac{17}{20}+\frac{21}{20}+\frac{6}{5} + \frac{5}{4} + \frac{27}{20} + \frac{7}{5}+\frac{29}{20}+\frac{31}{20}+\frac{8}{5}+\frac{33}{20}+\frac{7}{4}+\frac{9}{5}+\frac{39}{20}\\
		&+\sum_{ u=1 } ^{ s+3} \frac{2u+2s+7}{ 2(s+4)})\\
		&=\frac{30s^2 + 577s + 1838}{20(s + 4)(16 + s)}. 
	\end{align*}
	\begin{align*}
		& (\mu_f -  2) \cdot \mathrm{Var}[\mathrm{Sp^{\tau^{\mathrm{max}} }}(f)] \\
		= & (\frac{17}{20}-\bar{\alpha})^2+(\frac{21}{20}-\bar{\alpha})^2+2 (\frac{6}{5}-\bar{\alpha})^2+( \frac{5}{4}-\bar{\alpha})^2+(\frac{27}{20}-\bar{\alpha})^2+ (\frac{7}{5}-\bar{\alpha})^2+(\frac{29}{20}-\bar{\alpha})^2\\
		&+(\frac{31}{20}-\bar{\alpha})^2+(\frac{8}{5}-\bar{\alpha})^2+(\frac{33}{20}-\bar{\alpha})^2+(\frac{7}{4}-\bar{\alpha})^2+(\frac{9}{5}-\bar{\alpha})^2+(\frac{39}{20}-\bar{\alpha})^2\\
		& + \sum_{ u=1 } ^{ s+3} (\frac{2u+2s+7}{ 2(s+4)}-\bar{\alpha})^2\\
		= & \frac{100s^4 + 3943s^3 + 50345s^2 + 246176s + 402036}{1200(s + 4)^2(16 + s)}.
	\end{align*}
	And 
	\begin{equation*}
		\begin{cases}
			\frac{39}{20}-\frac{17}{20}=\frac{11}{10}, \qquad \forall 6 \leq s \leq 15  \\
			\frac{4s+13}{2(s+4)}-\frac{17}{20}, \qquad \forall s \geq 16
		\end{cases}.
	\end{equation*}
	Then it suffices to prove
	\begin{equation*}
		\begin{cases}
			\frac{100s^4 + 3943s^3 + 50345s^2 + 246176s + 402036}{1200(s + 4)^2(16 + s)}- \frac{16+s}{12} \cdot \frac{11}{10}\leq 0, \qquad \forall 6 \leq s \leq 15 \\
			\frac{100s^4 + 3943s^3 + 50345s^2 + 246176s + 402036}{1200(s + 4)^2(16 + s)}-\frac{16+s}{12}\cdot (\frac{4s+13}{2(s+4)}-\frac{17}{20}) \leq 0, \qquad \forall s \geq 16
		\end{cases}.
	\end{equation*}
	Equivalently,
	\begin{equation*}
		\begin{cases}
			\frac{-10s^4 - 457s^3 - 7735s^2 - 35424s - 48524}{1200(s + 4)^2(16 + s)} \leq 0, \qquad \forall 6 \leq s \leq 15\\
			\frac{-15s^4 - 507s^3 - 4975s^2 + 9376s + 84596}{1200(s + 4)^2(16 + s)} \leq 0, \qquad \forall s \geq 16
		\end{cases},
	\end{equation*}
	which hold obviously.

	If  $\tau^{\mathrm{max}} = 17+s, $ 
	\begin{align*}
		&\bar{\alpha} =\frac{1}{\mu_f -1}\cdot( \frac{17}{20}+\frac{21}{20}+\frac{6}{5} + \frac{5}{4} + \frac{27}{20} + \frac{7}{5}+\frac{29}{20}+\frac{31}{20}+\frac{8}{5}+\frac{33}{20}+\frac{7}{4}+\frac{9}{5}+\frac{39}{20}\\
		&+\sum_{ u=1 } ^{ s+4} \frac{2u+2s+7}{ 2(s+4)})\\
		&=\frac{497 + 30s}{340 + 20s}. 
	\end{align*}
	\begin{align*}
		& (\mu_f -  1) \cdot \mathrm{Var}[\mathrm{Sp^{\tau^{\mathrm{max}} }}(f)] \\
		= & (\frac{17}{20}-\bar{\alpha})^2+(\frac{21}{20}-\bar{\alpha})^2+2 (\frac{6}{5}-\bar{\alpha})^2+( \frac{5}{4}-\bar{\alpha})^2+(\frac{27}{20}-\bar{\alpha})^2+ (\frac{7}{5}-\bar{\alpha})^2+(\frac{29}{20}-\bar{\alpha})^2\\
		&+(\frac{31}{20}-\bar{\alpha})^2+(\frac{8}{5}-\bar{\alpha})^2+(\frac{33}{20}-\bar{\alpha})^2+(\frac{7}{4}-\bar{\alpha})^2+(\frac{9}{5}-\bar{\alpha})^2+(\frac{39}{20}-\bar{\alpha})^2\\
		& + \sum_{ u=1 } ^{ s+4} (\frac{2u+2s+7}{ 2(s+4)}-\bar{\alpha})^2\\
		= & \frac{100s^3 + 3943s^2 + 44896s + 121596}{1200(s + 4)(17 + s)}.
	\end{align*}
	And 
	\begin{equation*}
		\begin{cases}
			\frac{39}{20}-\frac{17}{20}=\frac{11}{10}, \qquad \forall 0 \leq s \leq 5 \\
			\frac{4s+15}{2(s+4)}-\frac{17}{20}, \qquad \forall s \geq 6
		\end{cases}.
	\end{equation*}
	Then it suffices to prove
	\begin{equation*}
		\begin{cases}
			\frac{100s^3 + 3943s^2 + 44896s + 121596}{1200(s + 4)(17 + s)}- \frac{17+s}{12} \cdot \frac{11}{10}\leq 0, \qquad \forall 0 \leq s \leq 5 \\
			\frac{100s^3 + 3943s^2 + 44896s + 121596}{1200(s + 4)(17 + s)}-\frac{17+s}{12}\cdot (\frac{4s+15}{2(s+4)}-\frac{17}{20}) \leq 0, \qquad \forall s \geq 6
		\end{cases}.
	\end{equation*}
	Equivalently,
	\begin{equation*}
		\begin{cases}
			\frac{-10s^3 - 237s^2 - 1854s - 5564}{1200(s + 4)(17 + s)} \leq 0, \qquad \forall 0 \leq s \leq 5\\
			\frac{-15s^3 - 377s^2 - 2279s + 3106}{1200(s + 4)(17 + s)} \leq 0, \qquad \forall s \geq 6
		\end{cases},
	\end{equation*}
	which hold obviously.
	
	%	Considering the average value $ \bar{\alpha}=\frac{3}{2}, $
	%	\begin{align*}
		%		& (\mu_f -  1) \cdot \mathrm{Var}[\mathrm{Sp^{\tau^{\mathrm{max}} }}(f)] \\
		%		= & (\frac{17}{20}-\bar{\alpha})^2+(\frac{21}{20}-\bar{\alpha})^2+2 (\frac{6}{5}-\bar{\alpha})^2+( \frac{5}{4}-\bar{\alpha})^2+(\frac{27}{20}-\bar{\alpha})^2+(\frac{7}{5}-\bar{\alpha})^2+(\frac{29}{20}-\bar{\alpha})^2\\
		%		&+(\frac{31}{20}-\bar{\alpha})^2+(\frac{8}{5}-\bar{\alpha})^2+(\frac{33}{20}-\bar{\alpha})^2+(\frac{7}{4}-\bar{\alpha})^2+(\frac{9}{5}-\bar{\alpha})^2+(\frac{39}{20}-\bar{\alpha})^2\\
		%		& + \sum_{ u=1 } ^{ s+3} (\frac{2u+2s+7}{ 2(s+4)}-\bar{\alpha})^2\\
		%		= & \frac{100s^3 + 2343s^2 + 14444s + 26388}{1200(s + 4)^2}.
		%	\end{align*}
	%	Then it suffices to prove
	%	\begin{center}
		%		$\begin{cases}
			%			\frac{100s^3 + 2343s^2 + 14444s + 26388}{1200(s + 4)^2}- \frac{15+s}{12} \cdot \frac{11}{10}\leq 0, \qquad \forall 0 \leq s \leq 25 \\
			%			\frac{100s^3 + 2343s^2 + 14444s + 26388}{1200(s + 4)^2}-\frac{15+s}{12}\cdot (\frac{4s+13}{2(s+4)}-\frac{17}{20}) \leq 0, \qquad \forall s \geq 26
			%		\end{cases}
		%		$.
		%	\end{center}
	%	Equivalently,
	%	\begin{center}
		%		$\begin{cases}
			%			\frac{-10s^3 - 297s^2 - 1396s - 1772}{1200(s + 4)^2} \leq 0, \qquad \forall 0 \leq s \leq 25\\
			%			\frac{-15s^3 - 267s^2 + 884s + 6548}{1200(s + 4)^2} \leq 0, \qquad \forall s \geq 26
			%		\end{cases}
		%		$,
		%	\end{center}
	%	which hold obviously.
	%	

	\noindent\textbf{(11) $VB_{(-1)}^{\#,2k}$.}

	Since $  \tau_{f_0} = \mu_f - 2 = 15+ 2k, $ we need to verify the cases of $\tau^{\mathrm{max}} = 15 +2k$ or $ 16+2k. $ If  $\tau^{\mathrm{max}} = 15+2k, $ 
	for $ k \leq 3, $
	\begin{align*}
		&\bar{\alpha} =\frac{1}{\mu_f -2}\cdot( \frac{13}{15}+\frac{16}{15}+\frac{19}{15} + \frac{4}{3} + \frac{22}{15} +\sum_{ u=1 } ^{ 2k+7} \frac{u+2k+8}{ 2(k+4)})\\
		&=\frac{643 + 90k}{450 + 60k}.
	\end{align*}
	\begin{align*}
		& (\mu_f -  2) \cdot \mathrm{Var}[\mathrm{Sp^{\tau^{\mathrm{max}} }}(f)] \\
		= & (\frac{13}{15}-\bar{\alpha})^2+(\frac{16}{15}-\bar{\alpha})^2+(\frac{19}{15}-\bar{\alpha})^2+( \frac{4}{3}-\bar{\alpha})^2+(\frac{22}{15}-\bar{\alpha})^2 + (\frac{23}{15}-\bar{\alpha})^2\\
		&+(\frac{5}{3}-\bar{\alpha})^2+(\frac{26}{15}-\bar{\alpha})^2 + \sum_{ u=1 } ^{ 2k+7} (\frac{u+2k+8}{ 2(k+4)}-\bar{\alpha})^2\\
		= & \frac{300k^3 + 5560k^2 + 32391k + 60329}{1800k^2 + 20700k + 54000}.
	\end{align*}
	And $ \alpha_{\mu_f -2} - \alpha_1 = \frac{4k+15}{2(k+4)}-\frac{13}{15}.$
	Then it suffices to prove
	\begin{equation*}
		\frac{300k^3 + 5560k^2 + 32391k + 60329}{1800k^2 + 20700k + 54000}- \frac{15+2k}{12} \cdot (\frac{4k+15}{2(k+4)}-\frac{13}{15})\leq 0.
	\end{equation*}
	Equivalently,
	\begin{equation*}
		\frac{-80k^3 - 1500k^2 - 9768k - 15467}{3600k^2 + 41400k + 108000} \leq 0,
	\end{equation*}
	which holds obviously.
	
	for $ k \geq 4, $
	\begin{align*}
		&\bar{\alpha} =\frac{1}{\mu_f -2}\cdot( \frac{13}{15}+\frac{16}{15}+\frac{19}{15} + \frac{4}{3} + \frac{22}{15} + \frac{23}{15}+\frac{5}{3}+\frac{26}{15}+\frac{29}{15}+\sum_{ u=1 } ^{ 2k+6} \frac{u+2k+8}{ 2(k+4)})\\
		&=\frac{90k^2 + 1001k + 2579}{60k^2 + 690k + 1800}.
	\end{align*}
	\begin{align*}
		& (\mu_f -  2) \cdot \mathrm{Var}[\mathrm{Sp^{\tau^{\mathrm{max}} }}(f)] \\
		= & (\frac{13}{15}-\bar{\alpha})^2+(\frac{16}{15}-\bar{\alpha})^2+(\frac{19}{15}-\bar{\alpha})^2+( \frac{4}{3}-\bar{\alpha})^2+(\frac{22}{15}-\bar{\alpha})^2 + (\frac{23}{15}-\bar{\alpha})^2\\
		&+(\frac{5}{3}-\bar{\alpha})^2+(\frac{26}{15}-\bar{\alpha})^2+(\frac{29}{15}-\bar{\alpha})^2 + \sum_{ u=1 } ^{ 2k+6} (\frac{u+2k+8}{ 2(k+4)}-\bar{\alpha})^2\\
		= & \frac{300k^4 + 6648k^3 + 53663k^2 + 191245k + 253244}{900(15 + 2k)(k + 4)^2}.
	\end{align*}
	And 
	\begin{equation*}
		\alpha_{\mu_f -2} - \alpha_1 =
		\begin{cases}
			\frac{29}{15}-\frac{13}{15}=\frac{16}{15}, \qquad \forall 4 \leq k \leq 10 \\
			\frac{2k+7}{k+4}-\frac{13}{15}, \qquad \forall k \geq 11
		\end{cases}.
	\end{equation*}
	Then it suffices to prove
	\begin{equation*}
		\begin{cases}
			\frac{300k^4 + 6648k^3 + 53663k^2 + 191245k + 253244}{900(15 + 2k)(k + 4)^2}- \frac{15+2k}{12} \cdot \frac{16}{15}\leq 0, \qquad \forall 4 \leq k \leq 10 \\
			\frac{300k^4 + 6648k^3 + 53663k^2 + 191245k + 253244}{900(15 + 2k)(k + 4)^2}-\frac{15+2k}{12}\cdot (\frac{2k+7}{k+4}-\frac{13}{15}) \leq 0, \qquad \forall k \geq 11
		\end{cases}.
	\end{equation*}
	Equivalently,
	\begin{equation*}
		\begin{cases}
			\frac{-20k^4 - 712k^3 - 7857k^2 - 29555k - 34756}{900(15 + 2k)(k + 4)^2} \leq 0, \qquad \forall 4 \leq k \leq 10\\
			\frac{-20k^4 - 436k^3 - 3001k^2 - 4240k + 7372}{450(15 + 2k)(k + 4)^2} \leq 0, \qquad \forall k \geq 11
		\end{cases},
	\end{equation*}
	which hold obviously.

	If  $\tau^{\mathrm{max}} = 16+2k, $ 
	\begin{align*}
		&\bar{\alpha} =\frac{1}{\mu_f -1}\cdot( \frac{13}{15}+\frac{16}{15}+\frac{19}{15} + \frac{4}{3} + \frac{22}{15} + \frac{23}{15}+\frac{5}{3}+\frac{26}{15}+\frac{29}{15}+\sum_{ u=1 } ^{ 2k+7} \frac{u+2k+8}{ 2(k+4)})\\
		&=\frac{701 + 90k}{480 + 60k}.
	\end{align*}
	\begin{align*}
		& (\mu_f -  1) \cdot \mathrm{Var}[\mathrm{Sp^{\tau^{\mathrm{max}} }}(f)] \\
		= & (\frac{13}{15}-\bar{\alpha})^2+(\frac{16}{15}-\bar{\alpha})^2+(\frac{19}{15}-\bar{\alpha})^2+( \frac{4}{3}-\bar{\alpha})^2+(\frac{22}{15}-\bar{\alpha})^2 + (\frac{23}{15}-\bar{\alpha})^2\\
		&+(\frac{5}{3}-\bar{\alpha})^2+(\frac{26}{15}-\bar{\alpha})^2+(\frac{29}{15}-\bar{\alpha})^2 + \sum_{ u=1 } ^{ 2k+7} (\frac{u+2k+8}{ 2(k+4)}-\bar{\alpha})^2\\
		= & \frac{300k^3 + 6048k^2 + 38765k + 78092}{1800(k + 4)(8 + k)}.
	\end{align*}
	And 
	\begin{equation*}
		\alpha_{\mu_f -1} - \alpha_1 =
		\begin{cases}
			\frac{29}{15}-\frac{13}{15}=\frac{16}{15}, \qquad \forall 0 \leq k \leq 3 \\
			\frac{4k+15}{2(k+4)}-\frac{13}{15}, \qquad \forall k \geq 4
		\end{cases}.
	\end{equation*}
	Then it suffices to prove
	\begin{equation*}
		\begin{cases}
			\frac{300k^3 + 6048k^2 + 38765k + 78092}{1800(k + 4)(8 + k)}- \frac{16+2k}{12} \cdot \frac{16}{15}\leq 0, \qquad \forall 0 \leq k \leq 3 \\
			\frac{300k^3 + 6048k^2 + 38765k + 78092}{1800(k + 4)(8 + k)}-\frac{16+2k}{12}\cdot (\frac{4k+15}{2(k+4)}-\frac{13}{15}) \leq 0, \qquad \forall k \geq 4
		\end{cases}.
	\end{equation*}
	Equivalently,
	\begin{equation*}
		\begin{cases}
			\frac{-20k^3 - 352k^2 - 2195k - 3828}{1800(k + 4)(8 + k)} \leq 0, \qquad \forall 0 \leq k \leq 3\\
			\frac{-40k^3 - 602k^2 - 2355k + 652}{1800(k + 4)(8 + k)} \leq 0, \qquad \forall k \geq 4
		\end{cases},
	\end{equation*}
	which hold obviously.

	\noindent\textbf{(12) $VB_{(0)}^{\#,2k}$.}

	Since $ \tau_{f_0}  = \mu_f - 2 = 16+ 2k, $ we need to verify the cases of $\tau^{\mathrm{max}} = 16 +2k$ or $ 17+2k. $ If  $\tau^{\mathrm{max}} = 16+2k, $ 
	for $ k \leq 6, $
	\begin{align*}
		&\bar{\alpha} =\frac{1}{\mu_f -2}\cdot( \frac{19}{22}+\frac{23}{22}+\frac{27}{22} + \frac{29}{22} + \frac{31}{22} + \frac{3}{2}+\frac{35}{22}+\frac{37}{22}+\frac{39}{22}+\sum_{ u=1 } ^{ 2k+7} \frac{u+2k+8}{ 2(k+4)})\\
		&=\frac{252 + 33k}{176 + 22k}.
	\end{align*}
	\begin{align*}
		& (\mu_f -  2) \cdot \mathrm{Var}[\mathrm{Sp^{\tau^{\mathrm{max}} }}(f)] \\
		= & (\frac{19}{22}-\bar{\alpha})^2+(\frac{23}{22}-\bar{\alpha})^2+(\frac{27}{22}-\bar{\alpha})^2+( \frac{29}{22}-\bar{\alpha})^2+(\frac{31}{22}-\bar{\alpha})^2 + (\frac{3}{2}-\bar{\alpha})^2+(\frac{35}{22}-\bar{\alpha})^2 \\
		&+(\frac{37}{22}-\bar{\alpha})^2+(\frac{39}{22}-\bar{\alpha})^2+(\frac{3}{2}-\bar{\alpha})^2+(\frac{35}{22}-\bar{\alpha})^2 + \sum_{ u=1 } ^{ 2k+7} (\frac{u+2k+8}{ 2(k+4)}-\bar{\alpha})^2\\
		= & \frac{242k^3 + 4733k^2 + 28949k + 56040}{1452(k + 4)(8 + k)}.
	\end{align*}
	And $ \alpha_{\mu_f -2 } - \alpha_1 = \frac{2k+7}{k+4}-\frac{19}{22}.$
	Then it suffices to prove
	\begin{equation*}
		\frac{242k^3 + 4733k^2 + 28949k + 56040}{1452(k + 4)(8 + k)}- \frac{16+2k}{12} \cdot (\frac{2k+7}{k+4}-\frac{19}{22})\leq 0.
	\end{equation*}
	Equivalently,
	\begin{equation*}
		\frac{-33k^3 - 646k^2 - 4315k - 6616}{1452(k + 4)(8 + k)} \leq 0,
	\end{equation*}
	which holds obviously.

	for $ k \geq 7, $
	\begin{align*}
		&\bar{\alpha} =\frac{1}{\mu_f -2}\cdot( \frac{19}{22}+\frac{23}{22}+\frac{27}{22} + \frac{29}{22} + \frac{31}{22} + \frac{3}{2}+\frac{35}{22}+\frac{37}{22}+\frac{39}{22}+\frac{43}{22}+\sum_{ u=1 } ^{ 2k+6} \frac{u+2k+8}{ 2(k+4)})\\
		&=\frac{66k^2 + 767k + 2023}{44(k + 4)(8 + k)}.
	\end{align*}
	\begin{align*}
		& (\mu_f -  2) \cdot \mathrm{Var}[\mathrm{Sp^{\tau^{\mathrm{max}} }}(f)] \\
		= & (\frac{19}{22}-\bar{\alpha})^2+(\frac{23}{22}-\bar{\alpha})^2+(\frac{27}{22}-\bar{\alpha})^2+( \frac{29}{22}-\bar{\alpha})^2+(\frac{31}{22}-\bar{\alpha})^2 + (\frac{3}{2}-\bar{\alpha})^2+(\frac{35}{22}-\bar{\alpha})^2 \\
		&+(\frac{37}{22}-\bar{\alpha})^2+(\frac{39}{22}-\bar{\alpha})^2+(\frac{43}{22}-\bar{\alpha})^2+(\frac{3}{2}-\bar{\alpha})^2+(\frac{35}{22}-\bar{\alpha})^2 + \sum_{ u=1 } ^{ 2k+6} (\frac{u+2k+8}{ 2(k+4)}-\bar{\alpha})^2\\
		= & \frac{484k^4 + 11276k^3 + 95051k^2 + 350764k + 476733}{2904(k + 4)^2(8 + k)}.
	\end{align*}
	And 
	\begin{equation*}
		\alpha_{ \mu_f -2 } -\alpha_1 =
		\begin{cases}
			\frac{43}{22}-\frac{19}{22}=\frac{12}{11}, \qquad \forall 7 \leq k \leq 17 \\
			\frac{2k+7}{k+4}-\frac{19}{22}, \qquad \forall k \geq 18
		\end{cases}.
	\end{equation*}
	Then it suffices to prove
	\begin{equation*}
		\begin{cases}
			\frac{484k^4 + 11276k^3 + 95051k^2 + 350764k + 476733}{2904(k + 4)^2(8 + k)}- \frac{16+2k}{12} \cdot \frac{12}{11}\leq 0, \qquad \forall 7 \leq k \leq 17 \\
			\frac{484k^4 + 11276k^3 + 95051k^2 + 350764k + 476733}{2904(k + 4)^2(8 + k)}-\frac{16+2k}{12}\cdot (\frac{2k+7}{k+4}-\frac{19}{22}) \leq 0, \qquad \forall k \geq 18
		\end{cases}.
	\end{equation*}
	Equivalently,
	\begin{equation*}
		\begin{cases}
			\frac{-44k^4 - 1396k^3 - 14773k^2 - 54740k - 63939}{2904(k + 4)^2(8 + k)} \leq 0, \qquad \forall 7 \leq k \leq 17\\
			\frac{-22k^4 - 480k^3 - 3223k^2 - 3228k + 12479}{968(k + 4)^2(8 + k)} \leq 0, \qquad \forall k \geq 18
		\end{cases},
	\end{equation*}
	which hold obviously.
	
	If  $\tau^{\mathrm{max}} = 17+2k, $ 
	\begin{align*}
		&\bar{\alpha} =\frac{1}{\mu_f -1}\cdot( \frac{19}{22}+\frac{23}{22}+\frac{27}{22} + \frac{29}{22} + \frac{31}{22} + \frac{3}{2}+\frac{35}{22}+\frac{37}{22}+\frac{39}{22}+\frac{43}{22}+\sum_{ u=1 } ^{ 2k+7} \frac{u+2k+8}{ 2(k+4)})\\
		&=\frac{547 + 66k}{374 + 44k}.
	\end{align*}
	\begin{align*}
		& (\mu_f -  1) \cdot \mathrm{Var}[\mathrm{Sp^{\tau^{\mathrm{max}} }}(f)] \\
		= & (\frac{19}{22}-\bar{\alpha})^2+(\frac{23}{22}-\bar{\alpha})^2+(\frac{27}{22}-\bar{\alpha})^2+( \frac{29}{22}-\bar{\alpha})^2+(\frac{31}{22}-\bar{\alpha})^2 + (\frac{3}{2}-\bar{\alpha})^2+(\frac{35}{22}-\bar{\alpha})^2 \\
		&+(\frac{37}{22}-\bar{\alpha})^2+(\frac{39}{22}-\bar{\alpha})^2+(\frac{43}{22}-\bar{\alpha})^2+(\frac{3}{2}-\bar{\alpha})^2+(\frac{35}{22}-\bar{\alpha})^2 + \sum_{ u=1 } ^{ 2k+6} (\frac{u+2k+8}{ 2(k+4)}-\bar{\alpha})^2\\
		= & \frac{484k^3 + 10308k^2 + 69335k + 144477}{2904k^2 + 36300k + 98736}.
	\end{align*}
	And 
	\begin{equation*}
		\alpha_{ \mu_f -1 } -\alpha_1 =
		\begin{cases}
			\frac{43}{22}-\frac{19}{22}=\frac{12}{11}, \qquad \forall 0 \leq k \leq 6 \\
			\frac{4k+15}{k+4}-\frac{19}{22}, \qquad \forall k \geq 7
		\end{cases}.
	\end{equation*}
	Then it suffices to prove
	\begin{equation*}
		\begin{cases}
			\frac{484k^3 + 10308k^2 + 69335k + 144477}{2904k^2 + 36300k + 98736}- \frac{17+2k}{12} \cdot \frac{12}{11}\leq 0, \qquad \forall 0 \leq k \leq 6 \\
			\frac{484k^3 + 10308k^2 + 69335k + 144477}{2904k^2 + 36300k + 98736}-\frac{17+2k}{12}\cdot (\frac{4k+15}{k+4}-\frac{19}{22}) \leq 0, \qquad \forall k \geq 7
		\end{cases}.
	\end{equation*}
	Equivalently,
	\begin{equation*}
		\begin{cases}
			\frac{-44k^3 - 780k^2 - 4717k - 8115}{2904k^2 + 36300k + 98736} \leq 0, \qquad \forall 0 \leq k \leq 6\\
			\frac{-132k^3 - 2000k^2 - 7377k + 6023}{5808k^2 + 72600k + 197472} \leq 0, \qquad \forall k \geq 7
		\end{cases},
	\end{equation*}
	which hold obviously.

	\noindent\textbf{(13) $VB_{(1)}^{\#,2k}$.}
	
	Since $ \tau_{f_0}  = \mu_f - 2 = 17+ 2k, $ we need to verify the cases of $\tau^{\mathrm{max}} = 17 +2k$ or $ 18+2k. $ If  $\tau^{\mathrm{max}} = 17+2k, $ 
	for $ k \leq 13, $
	\begin{align*}
		&\bar{\alpha} =\frac{1}{\mu_f -2}\cdot( \frac{31}{36}+\frac{37}{36}+\frac{43}{36} + \frac{47}{36} + \frac{49}{36} + \frac{53}{36}+\frac{55}{36}+\frac{59}{36}+\frac{61}{36}+\frac{65}{36}+\sum_{ u=1 } ^{ 2k+7} \frac{u+2k+8}{ 2(k+4)})\\
		&=\frac{439 + 54k}{306 + 36k}. 
	\end{align*}
	\begin{align*}
		& (\mu_f -  2) \cdot \mathrm{Var}[\mathrm{Sp^{\tau^{\mathrm{max}} }}(f)] \\
		= & (\frac{31}{36}-\bar{\alpha})^2+(\frac{37}{36}-\bar{\alpha})^2+(\frac{43}{36}-\bar{\alpha})^2+(\frac{47}{36}-\bar{\alpha})^2+( \frac{49}{36}-\bar{\alpha})^2 +(\frac{53}{36}-\bar{\alpha})^2+(\frac{55}{36}-\bar{\alpha})^2\\
		&+(\frac{59}{36}-\bar{\alpha})^2+(\frac{61}{36}-\bar{\alpha})^2+(\frac{65}{36}-\bar{\alpha})^2+ \sum_{ u=1 } ^{ 2k+7} (\frac{u+2k+8}{ 2(k+4)}-\bar{\alpha})^2\\
		= & \frac{216k^3 + 4450k^2 + 28527k + 57218}{1296k^2 + 16200k + 44064}.
	\end{align*}
	And $\alpha_{\mu_f -2} -\alpha_1 = \frac{4k+15}{2(k+4)}-\frac{31}{36}. $
	Then it suffices to prove
	\begin{equation*}
		\frac{216k^3 + 4450k^2 + 28527k + 57218}{1296k^2 + 16200k + 44064}- \frac{17+2k}{12} \cdot (\frac{4k+15}{2(k+4)}-\frac{31}{36})\leq 0.
	\end{equation*}
	Equivalently,
	\begin{equation*}
		\frac{-60k^3 - 1216k^2 - 8277k - 12146}{2592k^2 + 32400k + 88128} \leq 0, 
	\end{equation*}
	which holds obviously.

	for $ k \geq 14, $
	\begin{align*}
		&\bar{\alpha} =\frac{1}{\mu_f -2}\cdot( \frac{31}{36}+\frac{37}{36}+\frac{43}{36} + \frac{47}{36} + \frac{49}{36} + \frac{53}{36}+\frac{55}{36}+\frac{59}{36}+\frac{61}{36}+\frac{65}{36}+\frac{71}{36}+\sum_{ u=1 } ^{ 2k+6} \frac{u+2k+8}{ 2(k+4)})\\
		&=\frac{108k^2 + 1309k + 3526}{72k^2 + 900k + 2448}. 
	\end{align*}
	\begin{align*}
		& (\mu_f -  2) \cdot \mathrm{Var}[\mathrm{Sp^{\tau^{\mathrm{max}} }}(f)] \\
		= & (\frac{31}{36}-\bar{\alpha})^2+(\frac{37}{36}-\bar{\alpha})^2+(\frac{43}{36}-\bar{\alpha})^2+(\frac{47}{36}-\bar{\alpha})^2+( \frac{49}{36}-\bar{\alpha})^2 +(\frac{53}{36}-\bar{\alpha})^2+(\frac{55}{36}-\bar{\alpha})^2\\
		&+(\frac{59}{36}-\bar{\alpha})^2+(\frac{61}{36}-\bar{\alpha})^2+(\frac{65}{36}-\bar{\alpha})^2+(\frac{71}{36}-\bar{\alpha})^2+ \sum_{ u=1 } ^{ 2k+6} (\frac{u+2k+8}{ 2(k+4)}-\bar{\alpha})^2\\
		= & \frac{216k^4 + 5279k^3 + 46357k^2 + 176576k + 245532}{648(k + 4)^2(17 + 2k)}.
	\end{align*}
	And 
	\begin{equation*}
		\alpha_{\mu_f -2} -\alpha_1 = 
		\begin{cases}
			\frac{71}{36}-\frac{31}{36}=\frac{10}{9}, \qquad \forall 14 \leq k \leq 31 \\
			\frac{2k+7}{k+4}-\frac{31}{36}, \qquad \forall k \geq 32
		\end{cases}.
	\end{equation*}
	Then it suffices to prove
	\begin{equation*}
		\begin{cases}
			\frac{216k^4 + 5279k^3 + 46357k^2 + 176576k + 245532}{648(k + 4)^2(17 + 2k)}- \frac{17+2k}{12} \cdot \frac{10}{9}\leq 0, \qquad \forall 14 \leq k \leq 31 \\
			\frac{216k^4 + 5279k^3 + 46357k^2 + 176576k + 245532}{648(k + 4)^2(17 + 2k)}-\frac{17+2k}{12}\cdot (\frac{2k+7}{k+4}-\frac{31}{36}) \leq 0, \qquad \forall k \geq 32
		\end{cases}.
	\end{equation*}
	Equivalently,
	\begin{equation*}
		\begin{cases}
			\frac{-24k^4 - 721k^3 - 7463k^2 - 27424k - 31908}{648(k + 4)^2(17 + 2k)} \leq 0, \qquad \forall 14 \leq k \leq 31\\
			\frac{-5(12k^4 + 262k^3 + 1709k^2 + 892k - 9432)}{1296(k + 4)^2(17 + 2k)} \leq 0, \qquad \forall k \geq 32
		\end{cases}.
	\end{equation*}
	which hold obviously.

	If  $\tau^{\mathrm{max}} = 18+2k, $ 
	\begin{align*}
		&\bar{\alpha} =\frac{1}{\mu_f -1}\cdot( \frac{31}{36}+\frac{37}{36}+\frac{43}{36} + \frac{47}{36} + \frac{49}{36} + \frac{53}{36}+\frac{55}{36}+\frac{59}{36}+\frac{61}{36}+\frac{65}{36}+\frac{71}{36}+\sum_{ u=1 } ^{ 2k+7} \frac{u+2k+8}{ 2(k+4)})\\
		&=\frac{949 + 108k}{648 + 72k}. 
	\end{align*}
	\begin{align*}
		& (\mu_f -  1) \cdot \mathrm{Var}[\mathrm{Sp^{\tau^{\mathrm{max}} }}(f)] \\
		= & (\frac{31}{36}-\bar{\alpha})^2+(\frac{37}{36}-\bar{\alpha})^2+(\frac{43}{36}-\bar{\alpha})^2+(\frac{47}{36}-\bar{\alpha})^2+( \frac{49}{36}-\bar{\alpha})^2 +(\frac{53}{36}-\bar{\alpha})^2+(\frac{55}{36}-\bar{\alpha})^2\\
		&+(\frac{59}{36}-\bar{\alpha})^2+(\frac{61}{36}-\bar{\alpha})^2+(\frac{65}{36}-\bar{\alpha})^2+(\frac{71}{36}-\bar{\alpha})^2+ \sum_{ u=1 } ^{ 2k+7} (\frac{u+2k+8}{ 2(k+4)}-\bar{\alpha})^2\\
		= & \frac{432k^3 + 9694k^2 + 68253k + 146636}{2592(k + 4)(9 + k)}.
	\end{align*}
	And 
	\begin{equation*}
		\alpha_{\mu_f -1} -\alpha_1 = 
		\begin{cases}
			\frac{71}{36}-\frac{31}{36}=\frac{10}{9}, \qquad \forall 0 \leq k \leq 13 \\
			\frac{4k+15}{k+4}-\frac{31}{36}, \qquad \forall k \geq 14
		\end{cases}.
	\end{equation*}
	Then it suffices to prove
	\begin{equation*}
		\begin{cases}
			\frac{432k^3 + 9694k^2 + 68253k + 146636}{2592(k + 4)(9 + k)}- \frac{18+2k}{12} \cdot \frac{10}{9}\leq 0, \qquad \forall 0 \leq k \leq 13 \\
			\frac{432k^3 + 9694k^2 + 68253k + 146636}{2592(k + 4)(9 + k)}-\frac{18+2k}{12}\cdot (\frac{4k+15}{k+4}-\frac{31}{36}) \leq 0, \qquad \forall k \geq 14
		\end{cases}.
	\end{equation*}
	Equivalently,
	\begin{equation*}
		\begin{cases}
			\frac{-48k^3 - 866k^2 - 5187k - 8884}{2592(k + 4)(9 + k)} \leq 0, \qquad \forall 0 \leq k \leq 13\\
			\frac{-60k^3 - 914k^2 - 3135k + 4724}{2592(k + 4)(9 + k)} \leq 0, \qquad \forall k \geq 14
		\end{cases},
	\end{equation*}
	which hold obviously.
	\end{proof}

\end{document}